\documentclass[11pt, a4paper, leqno]{article}

\usepackage{verbatim}
\usepackage{hyperref}
\usepackage{amsmath}
\usepackage{enumerate, amsthm}
\usepackage{mathrsfs}
\usepackage{rotating}
\usepackage{xcolor}
\usepackage[left=1.0 in, right=1.0 in, top=0.8 in, bottom=0.9in]{geometry}
\usepackage[margin={1.5cm, 1.5cm},font=small, labelsep=endash]{caption}
\usepackage{mathtools}
\usepackage{tikz}
\usetikzlibrary{patterns}
\usetikzlibrary{arrows}
\usepackage{tikz-dimline}
\usepackage{ifthen}

\usepackage[nottoc]{tocbibind}

\newcommand{\nocontentsline}[3]{}
\let\origcontentsline\addcontentsline
\newcommand\stoptoc{\let\addcontentsline\nocontentsline}
\newcommand\resumetoc{\let\addcontentsline\origcontentsline}

\usepackage{bm}

      \usepackage{amssymb}
      \usepackage{amsmath}
      \usepackage{centernot}

      \numberwithin{equation}{section}
      \theoremstyle{plain}
      \newtheorem{theorem}{Theorem}[section]

            \newtheorem{thm}[theorem]{Theorem}
      \newtheorem{lemma}[theorem]{Lemma}
      \newtheorem{lem}[theorem]{Lemma}
      
      \newtheorem{proposition}[theorem]{Proposition}
      \newtheorem{prop}[theorem]{Proposition}

      \theoremstyle{definition}
      
      \newtheorem{definition}[theorem]{Definition}

      \theoremstyle{remark}
      \newtheorem{remark}[theorem]{Remark}

\renewcommand{\P}{\mathbb P}
\newcommand{\R}{\mathbb R}
\newcommand{\E}{\mathbb E}

\newcommand{\F}{\mathcal F}

\newcommand{\G}{\mathcal G}

\newcommand{\Z}{\mathbb Z}

\newcommand{\N}{\mathbb N}

\renewcommand{\bm}{\boldsymbol}

\newcommand{\red}{\color{red}}

\newcommand{\lr}[4]{#3\xleftrightarrow[#1]{#2} #4}
     \newcommand{\nlr}[4]{#3\mathrel{\mathop{\centernot\longleftrightarrow}_{#1}^{#2}} #4}

\usepackage{constants}
\usepackage{titletoc}
\usepackage{dsfont}
\usepackage{stackrel}
\usepackage{etex}
\usepackage[shortlabels]{enumitem}
\usepackage{enumerate}
\usepackage{vwcol}
\usepackage[utf8]{inputenc}
\usepackage[T1]{fontenc}
\usepackage{dsfont}
\usepackage{times}

\usepackage{verbatim}
\usepackage{hyperref}
\usepackage{amsmath}
\usepackage{amsthm}
\usepackage{mathrsfs}
\usepackage{rotating}
\usepackage{xcolor}
\usepackage{mathtools}
\usepackage{tikz}
\usetikzlibrary{patterns}
\usetikzlibrary{arrows}
\usepackage{tikz-dimline}
\usepackage{ifthen}

\usepackage{stackengine}

\newcommand{\vertiii}[1]{{\left\vert\kern-0.25ex\left\vert\kern-0.25ex\left\vert #1 
		\right\vert\kern-0.25ex\right\vert\kern-0.25ex\right\vert}}

\dottedcontents{<section>}[<left>]{<above-code>}
 {<label width>}{<leader width>}

\newconstantfamily{c}{symbol=c}
\makeatother

\hypersetup{linktocpage,
	colorlinks   = true,
	urlcolor     = blue,
	linkcolor    = blue,
	citecolor   = black
}

\usepackage{titlesec}
\titleformat{\subsection}[runin]{\normalfont\bfseries}{\thesubsection.}{.5em}{}[.]\titlespacing{\subsection}{0pt}{2ex plus .1ex minus .2ex}{.8em}
\titleformat{\subsubsection}[runin]{\normalfont\bfseries}{\thesubsubsection.}{.5em}{}[.]
\titlespacing{\subsubsection}{0pt}{2ex plus .1ex minus .2ex}{.8em}

\newcommand{\savefootnote}[2]{\footnote{\label{#1}#2}}
\newcommand{\repeatfootnote}[1]{\textsuperscript{\ref{#1}}}

\usepackage{multicol}
\usepackage{parcolumns}

\usepackage{float}

\title{
\normalsize
\textbf{
	STRONG LOCAL UNIQUENESS	FOR THE VACANT SET OF \\ RANDOM INTERLACEMENTS  
}}
\author{}
\date{}
\makeatletter
\usepackage{color}
\definecolor{Red}{rgb}{1,0,0}
\definecolor{Blue}{rgb}{0,0,1}
\definecolor{Yarok}{rgb}{0,0.5,0}

\def\red{\color{Red}}

\newcommand{\PF}[1]{{\red{[\small PF: #1]}}}

\begin{document}
\thispagestyle{empty}
\maketitle
\vspace{0.01cm}
\begin{center}
\vspace{-1.8cm}
Subhajit Goswami\footnote{School of Mathematics, Tata Institute of Fundam.~Research, 400005
Mumbai, India. \url{goswami@math.tifr.res.in}}, Pierre-Fran\c{c}ois Rodriguez\footnote{Centre for Mathematical Sciences, University of Cambridge, CB3 0WB Cambridge, UK. \url{pfr26@cam.ac.uk}}$^{,}$\savefootnote{imp-address}{Imperial College London, Department of Mathematics, SW7 2AZ London, UK. 
\url{yuriy.shulzhenko16@imperial.ac.uk}} and Yuriy Shulzhenko\repeatfootnote{imp-address}

\end{center}
\begin{abstract}
We consider the vacant set $\mathcal V^u$ of random interlacements on $\Z^d$ in 
dimensions $d \ge 3$. For varying intensity $u > 0$, the connectivity properties of $\mathcal 
V^u$ undergo a percolation phase transition across a critical parameter $u_*=u_*(d) \in (0,\infty)$. In this article, we prove that this phase transition is sharp in the supercritical 
phase $u < u_*$. 
This follows from a certain \emph{strong local uniqueness} property (SLU)  introduced in the 
present work, which we prove $\mathcal{V}^u$ satisfies. 
In itself, 
this property furnishes the missing ingredient needed to 
deduce a number of desirable {quenched} results characterizing the large-scale geometry of 
the infinite cluster. Moreover, SLU entails a sought-after local and monotone criterion 
amenable to renormalization arguments below $u_*$.
\end{abstract}

\vspace{-0.6cm}

{\setcounter{tocdepth}{2} \tableofcontents}
\thispagestyle{empty}
\vspace{0.3 cm}
\newpage




\newpage

\setcounter{page}{1}

\section{Introduction}\label{sec:intro}
 {\em Local uniqueness} 
 refers to the structural property of percolation models in their supercritical phase, by which the infinite cluster is witnessed quantitatively in regions of finite size.  This 
property is instrumental for unraveling the anatomy of the infinite 
cluster, see, e.g.~\cite{AntPis96, Bar04, MR3650417} and refs.~therein.  Local uniqueness is known in a handful of cases, and notoriously thorny to prove, especially in higher dimensions. To quote a few instances, see~\cite{GriMar90} for Bernoulli percolation, \cite{Bod05} for the random-cluster model at $q=2$, and, more recently, \cite{DCGRS20} for Gaussian free field excursion sets.

{In this article, we establish the local uniqueness for the vacant set $\mathcal{V}^u$ 
of random interlacements on $\Z^d$, $d \geq 3$, thereby proving the sharpness of the phase 
transition for this model in the supercritical phase. More so, we deduce a strong form of local 
uniqueness, see Theorem~\ref{T:ri-main} below, where the underlying event is also {\em monotone}. In contrast with 
all the models mentioned above, interlacement percolation has distinctive features: whereas $\mathcal{V}^u$ 
undergoes a percolation 
transition in $u$, its complement $\mathcal{I}^u =\Z^d \setminus \mathcal{V}^u$ is 
supercritical for \emph{all} $u>0$!  In fact $\mathcal{I}^u$ consists of one single 
unbounded cluster \cite[Corollary 2.3]{MR2680403} for all $u>0$. This 
\emph{global} topological constraint and the ensuing rigidity features of the model 
(e.g.~absence of any finite energy) are the chief reasons why proving local uniqueness for 
$\mathcal{V}^u$ has remained open until now, despite substantial progress on the nature of its phase transition over the last fifteen years. 

\subsection{Main results}\label{subsec:results}
Random interlacements form a Poissonian cloud 
of bi-infinite transient $\Z^d$-valued trajectories modulo time-shift carrying a 
time-like label $u \ge 0$ called the \emph{level} (see~\S\ref{subsec:occupation} for a precise description of the 
setup). The interlacement 
set $\mathcal I^u$ at level $u$ is defined as the 
range of all trajectories in this Poisson cloud with label at most $u$ and 
 $\mathcal V^u = \Z^d \setminus \mathcal I^u$ is the corresponding vacant set.
As such, $\mathcal V = (\mathcal V^u)_{u > 0}$ forms a decreasing family of random subsets of 
$\Z^d$, 
 which undergo a percolation transition across a critical threshold 
$u_\ast = u_\ast(d)  \in (0, \infty)$ (see~\cite{MR3602841, MR2512613, 
	10.1214/ECP.v20-3734}). This phase transition is fundamentally linked to   the behaviour of the random walk/Brownian motion itself: indeed, $u_*$ enters a growing number of formulas describing various characteristics of its trace; see, e.g.,~\cite{zbMATH05054008, benjamini2008giant, MR2561432, MR2838338, 
	zbMATH06797082, MR3602841, zbMATH07227743, li2024sharp}.


Let $(\Omega,\mathcal{A},\P)$ denote the canonical law of the interlacement process; we refer to the original article~\cite{MR2680403} for precise definitions. In view of the recent series of works \cite{RI-I, RI-II, RI-III}, one now knows that the phase 
transition of $\mathcal V^u$ around $u_\ast$ exhibits the following sharpness features. For all $u > 
u_{\ast}$, there exist $c = c(d)$ and $C = C(u, d)$ in $(0, \infty)$ such that, for all $L \ge 1$,
\begin{equation}\label{eq:subcritical0}
	\P\big[\lr{}{\mathcal V^u}{0}{\partial B_L}\big] \le C e^{-L^c}, 
\end{equation}
where, with hopefully obvious notation, the event in \eqref{eq:subcritical0} 
refers to a (nearest-neighbor) path in $\mathcal V^u$ connecting $0$ and $\partial B_L$ with 
$B_L = [-L, L]^d \cap \Z^d$ and $\partial B_L$ the inner (vertex) boundary of $B_L$. On the 
other hand, for all $0 < v < u < u_\ast$, there exist $c = c(d)$ and $C = C(u, v, d)$ in $(0, 
\infty)$ such that
\begin{equation}\label{eq:supcrit0}
	\P[\mathrm{Exist}(L, u)]  \wedge \P\left[\mathrm{Unique}(L, u, v)\right] \ge 1-C{ e}^{-L^c}
\end{equation}
for all $L \ge 1$, where the two events in \eqref{eq:supcrit0} are defined as
\begin{equation}
	\begin{split}
		\text{Exist}(L,u)&=\left\{ \,\text{there is a cluster in $\mathcal V^u \cap B_L$ of diameter at least  $\tfrac L5$}\right\}, \text{ and}
		\\[0.3cm]
		\text{Unique}(L,u,v)&= \left\{\begin{array}{c}\text{any two clusters in $ \mathcal V^u \cap B_{L}$ having diameter at}\\ \text{least $\frac L{10}$ are connected to each other in $ \mathcal V^v \cap B_{2L}$} \end{array}\right\}.\label{eq:UNIQUE1}
	\end{split}
\end{equation}
The bound \eqref{eq:supcrit0} (with $v<u$), {which is the subject of  \cite{RI-II}, already 
	provides valuable information towards 
understanding the supercritical phase of $\mathcal V^u$ and} has several interesting 
applications (see \cite[\S1.2]{RI-I} for 
details). Yet \eqref{eq:supcrit0} leaves much to be desired in comparison with local uniqueness, which essentially asks~whether $u$ and $v$ can be chosen to be {\em 
equal}. This improvement is paramount in order to access 
{properties of clusters (both infinite and finite) in the supercritial phase at any 
\emph{fixed} level $u$. For instance, prior to the present work, no meaningful bound on the 
(truncated) two-point function $\tau_u^{\text{tr}}$ 
was known when $u<u_*$ except in the perturbative regime $u \ll 1$; 
see \cite{Tei11} in dimensions $d \geq 5$ when $u \ll 1$, see also 
\cite{MR2838338} for corresponding results for the random walk, and \cite{MR3269990} for the extension to all $d \geq 3$ 
 (and $ u \ll 1$). We refer to \S\ref{subsec:apps} for further important implications of local uniqueness.

 The main contribution of the present article is the proof that $\mathcal V^u$  indeed possesses the local 
uniqueness property throughout 
the entire supercritical regime $0<u<u_*$. {Deriving such a property from a statement like \eqref{eq:supcrit0} for $\mathcal{V}^u$ is by no means 
innocuous: for, \eqref{eq:supcrit0} was specifically designed to avoid issues (e.g.~in \cite{MR3602841, https://doi.org/10.48550/arxiv.2105.12110, RI-II}) related to topological rigidity mentioned at the start of this introduction -- see \S\ref{subsec:overview} for a 
glimpse of what making this leap entails.}  

\medskip
Our first theorem, which 
forms the foundation of other results in the paper, shows that a strengthened version of the 
local uniqueness event holds with a probability stretched exponentially close to one. This also 
implies the purported equality between $u_\ast$ and another critical parameter $\widehat u$ 
introduced previously in the literature in connection with uniform bounds on the two-point 
function: following 
\cite[(0.2)-(0.3)]{zbMATH07395560}, we say that $\textnormal{NLF}(u)$, the \textit{no large finite cluster} property in $[0,u]$, holds when
\begin{equation}
\label{eq:NLF}
\begin{split}
&\text{there exist $\Cl[c]{c:NLF}(u) > 0$, $\Cl{C:NLF}(u) < \infty$ and $\gamma(u) \in (0,1]$, such that}\\
&\text{for all $v \in [0,u]$ and $x,y \in \Z^d$, $\tau_{v}^{\textnormal{tr}}(x,y) \leq  \Cr{C:NLF}e^{-\Cr{c:NLF}|x-y|^{\gamma}}$}
\end{split}
\end{equation}
(all of $\Cr{c:NLF}, \Cr{C:NLF}, \gamma$ may implicitly also depend on the dimension $d$), where
\begin{equation}
	\label{eq:tau-RI}
	\tau_{v}^{\textnormal{tr}}(x,y) =  \P\big[x \stackrel{\mathcal{V}^v}{\longleftrightarrow}y,\nlr{}{\mathcal{V}^v}{x}{\infty} \big], \quad x,y \in \Z^d.
\end{equation}
Note that \cite{zbMATH07395560} employs a slightly different quantity than 
\eqref{eq:tau-RI} to define $\textnormal{NLF}(u)$ as in \eqref{eq:NLF} but the two can be 
related by a straightforward union bound, and the resulting $\textnormal{NLF}$-properties 
are in fact identical. Noting that $\textnormal{NLF}(u)$ is monotone in $u$, let
\begin{equation}\label{eq:u-hat}
\widehat{u} = \widehat{u}(d) \coloneqq 
\sup\{ u \in [0, u_\ast) :  \, \textnormal{NLF}(u) \text{ holds} \}.
\end{equation}
One deduces from \cite{MR3269990} that $\widehat{u} \geq \Cl[c]{c:uhat} $ for some 
$\Cr{c:uhat}= \Cr{c:uhat}(d) \in (0,1)$ and $d \geq 3$. Recall 
that $u_*= u_*(d)$ denotes the critical parameter describing the 
percolation phase transition of $ \mathcal{V}^u$.

\begin{theorem}
	 \label{T:ri-main} Defining for $u  \geq 0$ and $L \geq 1$ the `strong local uniqueness' event 
	\begin{equation}
		\label{eq:slu}
		{\rm SLU}_{L}(u) = \left\{\, \textnormal{for all $v \in [0,u]$, the event $\text{Unique}(L,v,v)$ occurs} \, \right\}
	\end{equation}
	(see~\eqref{eq:UNIQUE1} for notation),
	there exist $C=C(d,u)$, $c=c(d)$ in $(0,\infty)$ such that 
	\begin{equation}
		\label{eq:ri-unique}
		\P[{\rm SLU}_L(u)] \geq 1-Ce^{-L^c}, \,  \text{ for all $L \geq 1$, $u < u_*$ and $d \geq3$.}
	\end{equation}
Moreover,
\begin{equation}
\label{eq:ri-equal}
\widehat{u}(d) =u_*(d), \, \text{ for all $d \ge 3$.}
\end{equation}
\end{theorem}

The equality \eqref{eq:ri-equal} establishes the supercritical sharpness of 
$(\mathcal{V}^u)_{u>0}$ 
and constitutes the first meaningful bounds 
on $\tau_{u}^{\textnormal{tr}}$ at large distances valid for \emph{all} $u < u_*$. We return to 
this in Theorem~\ref{T:ri-2point} below. Theorem~\ref{T:ri-main} has far-reaching 
implications, by giving access to  quenched properties characterizing the well-behavedness 
of $\mathcal{V}^u$ in the entire supercritical phase $u< u_*$; see \S\ref{subsec:apps} for 
more on this.

The strong local uniqueness event ${\rm SLU}_{L}(u)$ in \eqref{eq:slu} is of central interest: 
for, it defines a {\em finite-size criterion} for the infinite cluster at level $u$, which -- 
crucially -- is renormalizable. That is, the family ${\rm SLU} \equiv ({\rm SLU}_{L}(u))_{u>0}$, being 
monotone, feeds into certain multi-scale arguments in $L$ that involve renormalisation in $u$; 
see \cite{MR2680403,MR2891880} for pioneering works. In stark contrast, the event 
$\text{Unique}(L,u,u)$ is completely unfit for these techniques. For a sample application 
illustrating the 
interesting consequences of \eqref{eq:ri-unique} and renormalizability of ${\rm SLU}$, we 
refer to Theorem~\ref{T:ri-2point} below.

 The presence of the quantifier `\emph{for all $v \in (0,u]$}' in \eqref{eq:slu} and inside the 
 probability in \eqref{eq:ri-unique} hints at the strength of our methods, which in fact allow to prove a 
version of \eqref{eq:ri-unique} not just involving $\mathcal V^v$ uniformly in $v \in [0,u]$, but 
any not too degenerate (in a sense to be made precise) subset $\mathcal{V} \subset \mathcal{V}^u$ formed out of 
excursions at scale $L$ (of which $\mathcal V = \mathcal V^v$ for any  $0< v \leq u$ are just examples). This change of perspective, by which one considers a 
much larger class of events (not at all measurable only with respect to the original configuration
$(\mathcal{V}^{u})_{u>0}$ alone) is a novel and conceptually important part of our methods. It 
yields a new way to couple 
non-monotone events involving 
$(\mathcal V^u)_{u > 0}$ without introducing any undesirable sprinkling, which is of independent interest. We discuss this new technique in more detail as part of \S\ref{subsec:overview} below.


\subsection{Proof overview}\label{subsec:overview}
 We now outline the strategy to prove Theorem~\ref{T:ri-main}, thereby highlighting the main novel aspects of this work, itemized as 1), 2) and 3) below. The equality \eqref{eq:ri-equal} is 
a straightforward consequence of \eqref{eq:ri-unique} combined with the disconnection 
estimate from \cite{MR3602841} (see also \cite[display~(5.73)]{gosrodsev2021radius}), so we focus on \eqref{eq:ri-unique} from here on. 

 Our starting point is the following. For $\zeta$ a configuration of excursions at scale $L$, let $\mathscr G_L(\zeta)$ denote the event that the box $B_L \coloneqq [-L ,L]^d$ contains an `ambient' cluster in the complement of $\text{range}(\zeta)$ whose coarse-graining at scale $L_0 \ll L$ consists of boxes $B$, translates of $[0,L_0)^d$, in each of which an (a-priori free to choose) `local' event $F_{B}(\zeta)$ occurs. In the interest of time, we won't delve into the precise meanings of technical terms like \emph{ambient, local} etc. It will be critical to keep track of the structural conditions $F_B$ must satisfy, see \ref{condC1}-\ref{condC3} 
 below. By state-of-the-art renormalization techniques \cite{MR2891880,PT12,drewitz2018geometry} (by which $L=L_n$ for a rapidly sequence of scales $(L_n)_{n \geq 0}$), one can ensure that 
\begin{equation}\label{eq:goodG}
\P[\mathscr G_L(\zeta)] \ge 1 - C e^{-L^{c}}, \, L \geq 1,
\end{equation}
provided $F_{B}(\zeta)$ is:
\begin{enumerate}[noitemsep, label = {(C\arabic*)}]
\item
\label{condC1} sufficiently likely as $L_0 \to \infty$, and 
\item
\label{condC2}`renormalizable' (e.g.~monotone in the configuration $\zeta$). 
\end{enumerate}
For now (but see below, around \eqref{eq:boosted0}) the reader can think of $\zeta$ in \eqref{eq:goodG} as the excursions induced by $\mathcal{I}^u$ in $B_L$. The probability in \eqref{eq:goodG} roughly corresponds to the situation where a certain recursively defined family of events `cascades' down to scale $L_0$, giving rise to the features defining $\mathscr G_L(\zeta)$.

\medskip
Equipped with \eqref{eq:goodG}, a promising avenue to prove, for simplicity say, that ${\rm 
Unique}(L, u, u)$ is likely, consists of devising an exploration process for 
any given cluster $\mathscr{C}$ of $\mathcal V^u$ inside $B_{L}$ in such a way that, every time it comes close to the ambient cluster in some $L_0$-box $B$  (call this an `encounter time' $\tau$),
 \begin{enumerate}[noitemsep, label = {(C\arabic*)}, start=3]
 \item
 \label{condC3} the occurrence of $F_{B}$ ensures a uniformly positive conditional probability at time $\tau$ (w.r.t.~the filtration induced by the exploration) for $\mathscr{C}$ to connect to the ambient cluster. 
\end{enumerate}
In principle, this then yields a lower bound like \eqref{eq:ri-unique}, since the `ambient' cluster is precisely designed to guarantee many encounter times when the explored component is macroscopic in $B_L$. 


In the past, a strategy along the above lines has been made to work \cite{DCGRS20, 
gosrodsev2021radius} for level-set percolation of the Gaussian free field (GFF), a model with similar algebraic decay of correlations. However, 
it crucially relies in one way or another on a form of `ellipticity' present in the model, which is instrumental to define $F_{B}$. Indeed, 
\ref{condC1}-\ref{condC3} are  
achieved for the GFF by {\em laying aside} part of the randomness in the form of a non-degenerate residual white noise stemming from a suitable orthogonal decomposition of the GFF over scales; see for instance \cite[(5.30)]{gosrodsev2021radius}, which plays the role of $F_B$; see also \cite{dembin2023supercritical}, which exploits similar ideas for Voronoi percolation. The presence of i.i.d.~noise notably ensures that \ref{condC3} can be fulfilled using a standard cluster exploration algorithm upon freezing all but the residual randomness. Unfortunately, this noise is also the very reason the GFF retains an insertion tolerance property (albeit non-uniform), which is absent for $\mathcal{V}^u$ owing to the rigidity features mentioned at the start of this introduction. Hence, separation of randomness at scale $1$ is precluded for $\mathcal{V}^u$.



\bigskip

\noindent{\bf 1) Pseudo insertion tolerance.} 
To cope with this issue, we 
start by proving a {\em restricted} insertion tolerance property for $\mathcal V^u$ to the effect that, for all $r > 0$ and $B=B_r$,  the bound
\begin{equation}\label{eq:insertion_tol0}
\P[B \subset \mathcal V^u \, | \, \mathcal F^{\phi}_{B} ] \ge c\cdot 1_{F_{B}} 
\end{equation}
holds, where $\mathcal F^{\phi}_{B}= \sigma\big(\mathcal 
V^u \cap (B_{2r})^c, \phi(\mathcal V^u \cap A_{2r}) \big)$ 
with $A_{2r}\coloneqq B_{2r} \setminus B_r$, the event $F_{B} \in \mathcal F^{\phi}_{B}$ has high probability as $r \to \infty$ and $c=c(u ,r,d) > 0$. 

The purpose of the functional $\phi: \{0, 1\}^{A_{2r}} \to \{0, 1\}^{A_{2r}}$ is 
to 
{\em hide} some information in $\mathcal V^u \cap B_{2r}$ to facilitate a lower bound like \eqref{eq:insertion_tol0}. It  
needs to be chosen carefully: without restriction to any event 
on the right-hand side and $\phi \equiv {\rm id}$, \eqref{eq:insertion_tol0} fails.  The following result, albeit technically insufficient for our purposes, is already of interest, so we state it formally here. We refer to Proposition~\ref{lem:conditional_prob1} below for a more general statement.

\begin{lemma}[Restricted insertion tolerance] \label{L:FE} For all $u,r>0$, letting
\begin{equation}\label{def:phi}
\phi(\mathcal V^u \cap A_{2r}) \coloneqq \text{the revealment of the cluster of $\partial B_{2r}$ 
inside $A_{2r} \setminus \mathcal V^u $},
\end{equation}
 there exists $F_B \in \mathcal F^{\phi}_{B}$ with $\P[F_B] \geq 1- C(u)e^{-L^c}$ such that \eqref{eq:insertion_tol0} holds.
\end{lemma}
By \emph{revealment} in \eqref{def:phi}, we mean the configuration comprising all of $\partial B_{2r}$, all the clusters in $A_{2r} \setminus \mathcal V^u $ that intersect it, and their outer boundaries in $A_{2r}$.

Lemma~\ref{L:FE} constitutes a 
significant improvement over \cite[Proposition~3.1]{RI-II}, where a {\em sprinkled} version of  \eqref{eq:insertion_tol0} was 
proved with $\mathcal F^{\phi}_{B}$ replaced by the {\em smaller} 
configuration $\sigma(\mathcal V^{u + \varepsilon} \cap (B_r)^c)$ for some sprinkling parameter $\varepsilon > 0$; see the beginning of Section~\ref{sec:insertion} for more on this.

At first glance, Lemma~\ref{L:FE}, which is bereft of any sprinkling, seems to provide a candidate event $F_B$ to salvage the general strategy outlined above, and the bound  
on $\P[F_B]$ below \eqref{def:phi} means that \ref{condC1} is easily met for $r=L_0$. However, pseudo insertion tolerance comes at a high price: 
\begin{itemize}[noitemsep]
\item the event $F_B$ in Lemma~\ref{L:FE} (which is made explicit in Section~\ref{sec:insertion}) is not monotone, so \ref{condC2} fails (hence \eqref{eq:goodG} is compromised, because the renormalisation can't be initiated at all);
\item the $\sigma$-algebra $\mathcal F^{\phi}_{B}$ with $\phi$ given by \eqref{def:phi} is (necessarily!) restricted (i.e.~$\phi \neq \text{id}$), hence the  exploration alluded to in the context of \ref{condC3}, if at all feasible, won't follow via standard arguments.
\end{itemize}
We deal with each of these issues separately in items 2) and 3) below. As with 1), the ideas we introduce there are of independent interest. Because they touch on fundamental issues, we expect the techniques described in items 2) and 3) to have many applications, some of which are described in \S\ref{subsec:apps}; see, in particular, Theorem~\ref{T:ri-2point}.
\bigskip

\noindent{\bf 2)  
Renormalizing non-monotone events.} 
We now describe how to address the first bullet point above. This 
is an instance of a well-known issue, namely, how to renormalize non-monotone events; see e.g.~\cite{Tei11, MR3269990, drewitz2018geometry} for delicate work-arounds in the presence of a small parameter, which is absent here. Our methods deal not only with the specific issue relating to $F_{B}$ but with the problem in more generality. 

Let us briefly recall this key issue: the essence of renormalization 
is to adjust the parameters $u=u_n $ etc.~along a growing sequence of spatial scales $L=L_n$ as $n \to \infty$, so as to overcome the algebraic correlations and decouple distant events. For 
$\mathcal V^u$, 
decoupling involves {\em monotone} couplings between relevant sequences of excursions. 
 For non-monotone 
events, like ${\rm Unique}(L,u,v)$ in \eqref{eq:UNIQUE1}, or $F_{B}$ for that matter, these couplings invariably force the 
parameters $u,v$ to move in different directions. For instance, getting a good lower bound on $\P[{\rm Unique}(KL, u_0, u_0)]$ for $K \gg 1$ would rely on a bound for $\P[{\rm Unique}(L, u, v)]$ with $u<u_0 < v$, which is not available (recall that $u>v$ in \eqref{eq:supcrit0})! Actually, the framework of the recent work \cite{RI-II}, which led to~\eqref{eq:supcrit0}, is specifically designed to avoid this issue.

To 
get out of this apparent impasse, we adopt a  
new perspective. We illustrate this with the example of 
${\rm U}_L( u) \equiv {\rm Unique}(L, u, u)$, which already conveys the gyst of our approach. Similar observations can be made about $F_{B}$. As hinted above, one can view the event ${\rm U}_L(u)$ as 
a function of a (finite) sequence of excursions $Z^u = (Z_k)_{1 \le k \le N^u}$ between $B_L$ and a larger concentric box $B_{CL}$, {\em factoring} through 
its corresponding interlacement set $\mathcal I(Z^u) = \bigcup_{1 \le k \le N^u}\,{\rm 
range}(Z_k)$. This 
enables us to define an analogue of ${\rm U}_L(u)$ for any sequence of excursions $Z$ which we call 
${\rm U}_L(Z)$ with a slight abuse of notation.  We refer to $Z$ as a \emph{packet} (of excursions).  Now let $\zeta$ denote the collection of 
subsequences of $Z^u$ containing {\em any} sequence $(Z_k)_{1 \le k \le n}$ where $N^v \le n 
\le N^u$ for some $v \in (0, u)$ and consider a {\em strengthening} ${\rm U}_L( \zeta) \subset {\rm U}_L( u)$ of the event ${\rm U}_L( 
u)$ as follows:
\begin{equation}\label{eq:boosted0}
{\rm U}_L( \zeta) \coloneqq \textstyle\bigcap_{Z \in \zeta} {\rm U}_L(Z) \ \big( \subset {\rm U}_L(
u) \big).
\end{equation}
It is clear 
from \eqref{eq:boosted0} that the event ${\rm U}_L( \zeta)$ is {\em monotonically decreasing} 
in the collection $\zeta$ and one verifies that this event is `well-behaved' {across} the couplings mentioned previously. 

As with \eqref{eq:boosted0}, the true definition of $\mathscr G_L(\zeta)$ in \eqref{eq:goodG} involves a family $\zeta$ of excursion packets. This trickles down to the `seed' event $F_B$, and we prove an 
extended version of \eqref{eq:insertion_tol0} in Section~\ref{sec:insertion} to packets of 
excursions, with $\mathcal V^u$ in  replaced by $\mathcal V(Z)= \Z^d \setminus \bigcup_k 
\text{range}(Z_k)$, for $Z = (Z_k)_{1 \le k \le n} \in \zeta$ a sequence of excursions. The 
$\sigma$-algebra $\mathcal F^{\phi}_{B}(Z)$ and the event $F_{B}(Z) \supset F_B(\zeta)$, are 
generalized similarly as in \eqref{eq:boosted0}. Triggering the renormalization now requires 
enhanced a-priori estimates, for events of the type appearing in \eqref{eq:boosted0} (for 
instance, $F_B(\zeta)$). We prove these separately in \S\ref{sec:apriori}. With these at hand, 
\eqref{eq:goodG} now follows via renormalization, thus addressing the issue with \ref{condC2}.

\medskip

The use of packets is quite malleable, starting with the choice of $\zeta$, and it will be applied to various non-monotone events, cf.~\eqref{def:boosted} and 
\S\ref{subsubsec:supcrit} for further details; for a concrete application in relation with the event ${\rm 
SLU}_L(u)$ from \eqref{eq:slu}, see for instance \eqref{eq:V-boosted} together with Lemma~\ref{lem:VzSLU_inclusion}, which is instructive. 

\bigskip

\noindent{\bf 3) Exploration and gluing of large clusters.} 
We now explain how to deal with the problem in the second bullet point above, which relates to \ref{condC3}. As a result of Step 2, we can now take for granted the occurrence of $\mathscr G_L(\zeta)$ in \eqref{eq:goodG}, ensuring 
the existence of an ambient cluster 
strewn with copies of the `insertion-tolerance good' event $F_{B}(\zeta)$ with 
very high probability via renormalization. 

The $\sigma$-algebra in \eqref{eq:insertion_tol0} (more precisely, its extension $\mathcal F^{\phi}_B(Z)$ to excursion packets $Z \in \zeta$) does not 
always allow us to {\em reveal} the points in 
$\mathcal V(Z)$ {\em all the way} up to the (outer) boundary of $B$ owing 
to the partial information provided by the functional $\phi$ from  \eqref{def:phi} in the vicinity of $B$ (see below 
\eqref{eq:insertion_tol0}). This means that we can {\em not} follow the 
conventional wisdom of trying to connect to the ambient cluster 
inside a good box $B$ when the exploring cluster 
arrives at its boundary, thus making the exploration schemes used in previous works, see e.g.~the proof of \cite[Lemma~5.1]{gosrodsev2021radius} or \cite[Proposition~1.5]{DCGRS20}, 
inadequate for our purpose. In the same vein, the sprinkling present in \cite{RI-II} effectively {avoids} this issue. 


To address this, we design in \S\ref{subsec:reduction}-\ref{subsec:exploration} a delicate exploration algorithm suited to work in this terrain  when $\phi$ is given by \eqref{def:phi}. This 
exploration method is another novel contribution of this work. We refer to \S\ref{subsec:reduction} regarding the intricacies of its construction. Overall, this leads to the following result, which summarizes in a slightly informal way the combined effects of Theorem~\ref{prop:exploration_RI-I} and Proposition~\ref{prop:reduction} below. 

\begin{proposition}\label{P:explo-informal} For each excursion packet $Z \in \zeta$ and cluster $\mathscr{C}\subset (\mathcal{V}(Z)\cap B_L)$ , there exists a random sequence $(\tau_k)_{k \geq 0}$ of ``good encounter times,'' along with a sequence of points $(Y_k)_{k \geq 0}$ in $B_L \cap L_0 \Z^d$, both defined measurably in $(Z, \mathscr{C})$, 
such that all of the following hold: with $B_k\coloneqq Y_k + [0,L_0^d), $ 
\begin{itemize}
\item[i)] $\{ \tau_k < \infty\}$ occurs if and only if $F_{B_k}$ does, `the' ambient cluster $\mathscr{A} \subset \mathcal{V}(Z)$ intersects $B_k$, and $\mathscr{C}$ intersects the twice bigger box $\tilde{B}_k$. If moreover $B_k \subset \mathcal{V}(Z)$, then $\mathscr{C}$ is connected to $\mathscr{A}$ in $\tilde{B}_k$;
\item[ii)] the event $\{\tau_j < \infty, B_j \not\subset \mathcal V(Z), j < k\} \cap \{\tau_{k} < \infty, B_k=B\}$ is $\mathcal{F}_B^\phi(Z)$-measurable;
\item[iii)] With $m= \lfloor L^c \rfloor$, one has
\begin{equation}\label{eq:sprinkling_removed}
\P[({\rm SLU}_L(\zeta))^c] \le \P[\mathscr G_L(\zeta)^c] + \P\big[ \exists (Z, \mathscr{C}): Z \in \zeta, \, \textnormal{diam}(\mathscr{C})\geq cL, \, \tau_{m}(Z, \mathscr{C})< \infty \big].
\end{equation}
\end{itemize}
\end{proposition}

Proposition~\ref{P:explo-informal} is subtle: whereas item i) and the choice of $m$ in item iii) seem to straightforwardly harness the features of $\mathscr G_L(\zeta)$ described above \eqref{eq:goodG} (as the bound \eqref{eq:sprinkling_removed} also suggests), the difficulty is to satisfy the required $\mathcal{F}_B^\phi(Z)$-measurability in item ii).  

In doing so, Proposition~\ref{P:explo-informal} meets condition \ref{condC3} by supplying the missing exploration alluded to in the second bullet point above. Indeed, items i) and ii) together now imply that one can apply Lemma~\ref{L:FE} repeatedly to deduce, upon applying a union bound over $(Z, \mathscr{C})$, that the second term in \eqref{eq:sprinkling_removed} is bounded by $Ce^{-L^c}$. The bound over $\mathscr{C}$ is polynomial in $L$, but it is crucial that $\zeta$ does not have too high complexity, all the while ensuring that ${\rm SLU}_L(\zeta) \subset {\rm SLU}_L(u)$, cf.~\eqref{eq:slu}. The choice of excursion packets constituting $\zeta$ is actually quite a bit more involved than what appears above \eqref{eq:boosted0}; the interested reader is referred to \eqref{eq:SLU-reduc0} and \eqref{eq:V-first-inclusion} (see also \eqref{eq:Z^+}-\eqref{def:Z_tr_RI} and \eqref{eq:V_z^I}-\eqref{eq:V_z^II} regarding the important events ${\rm V}_z^{\rm i}$, ${\rm i}= {\rm I},{\rm II}$). Notwithstanding, we will typically choose $\zeta$ such that that $|\zeta| =O( \text{cap}(B_L)) =O(L^{d-2})$. Overall, \eqref{eq:sprinkling_removed} thus leads to \eqref{eq:ri-unique}.



%


\subsection{Applications} \label{subsec:apps}
Both Theorem~\ref{T:ri-main} and the techniques introduced in \S\ref{subsec:overview} have important further consequences, which we now discuss. Theorem~\ref{T:ri-main} itself has direct, far-reaching implications, by giving access to intrinsic `quenched' properties of $\mathcal{V}^u$ 
valid in the {entire} supercritical phase $u< u_*$. We first briefly discuss those. A classical way to test the geometry of the supercritical phase is to probe the large-scale 
behaviour of the random walk on the infinite cluster. Questions of homogenisation in porous 
media have a long tradition, see for instance \cite{Bar04, 
SidSzn04,MatPia07,BerBis07,MR3568036,zbMATH06930205}. 

Theorem~\ref{T:ri-main} supplies the missing ingredient to prove an invariance principle on the 
infinite cluster $\mathcal{C}_{\infty}^u$ of $\mathcal{V}^u$ for all $u< u_*$. Indeed, combining Theorem~\ref{T:ri-main} and~\cite[Theorem 
1.1]{MR3568036}, one now knows that for $\P[\cdot | 0 \in \mathcal{C}_{\infty}^u ]$-a.e.~$\omega$, the diffusively rescaled random walk on $\mathcal{C}_{\infty}^u(\omega)$ (see for instance \cite[(0.6) or (0.7)]{MR2512613}) converges to a $d$-dimensional Brownian motion with non-degenerate diffusion matrix $\sigma^2I$, with $\sigma^2= \sigma^2(u)>0$. This is obtained as follows. With the exception of \textbf{S1}, all 
the conditions \textbf{P1}-\textbf{P3} and \textbf{S2} appearing in \cite{MR3568036} can be 
checked in the same way as in \cite[Section 2.3]{MR3390739}, and the outstanding condition 
\textbf{S1} that postulates a finite-size criterion for 
$\mathcal{C}_{\infty}^u$ is implied by \eqref{eq:ri-unique}. Previous results of CLT-type were restricted in the case of $\mathcal{V}^u$ to the regime $u\ll 1$, 
see \cite{Tei11,MR3269990,MR3568036}.

The validity of a quenched invariance principle is but one emblematic example of the realm of Theorem~\ref{T:ri-main} 
as concerns the geometry of the infinite cluster, which feeds into various other works that 
successfully exploited conditions \textbf{P1}-\textbf{P3} and 
\textbf{S1}-\textbf{S2}. These results now apply to 
$(\mathcal{V}^u)_{u \in (0,u_*)}$ as a consequence of Theorem~\ref{T:ri-main}, which supplies the outstanding condition \textbf{S1}. They include, 
among others,  the validity of a shape 
theorem~\cite{MR3390739}, and quenched (Gaussian) heat kernel estimates on $\mathcal{C}_{\infty}^u$, see~\cite{MR3650417,Bar04}. Underlying these results is a 
structural result, the validity of a certain {isoperimetric 
inequality} on $\mathcal{C}_\infty^u$, which in its currently strongest available form is stated in \cite[Theorem 
5.10]{MR3650417}. Its proof employs a delicate coarse-graining scheme which 
utilises \eqref{eq:ri-unique} as a crucial ingredient. This isoperimetric inequality now holds for $\mathcal{V}^u$ at \emph{all} values of $u \in (0,u_*)$ as a consequence of Theorem~\ref{T:ri-main}.

 In the same vein as \eqref{eq:ri-unique}, \eqref{eq:ri-equal} can be viewed as extending the 
 string of equalities $ \bar u= u_* = u_{**}  $ between various critical parameters established as 
 part of \cite{RI-I}; cf.~\S1.3 and (1.21) therein. 
This extension is of independent interest. For instance, together with \cite[Theorem~1]{zbMATH07395560}, \eqref{eq:ri-equal} implies that
\begin{equation}
 \label{eq:C1}
\text{$u \in [0, u_\ast) \mapsto \theta (u)\coloneqq \P[\lr{}{\mathcal V^u}{0}{\infty}]$ is $C^1$ with $\tfrac{\mathrm{d}\theta(u)}{\mathrm{d}u} < 0$.}
\end{equation}
The regularity of $\theta(\cdot)$ asserted in \eqref{eq:C1}, previously only known for $u< \widehat{u}$ (cf.~\eqref{eq:u-hat}), answers a question of \cite{zbMATH07395560} and is relevant for 
certain constrained variational problems: it implies in combination with 
\cite[Theorem~3]{zbMATH07395560}~and the main results of 
\cite{zbMATH07483480,https://doi.org/10.48550/arxiv.2105.12110, MR4650161} 
that the minimizer(s) for the variational problem associated to the cost of having an excessive 
density $\nu (> 1 - \theta(u))$ of points in the complement of $\mathcal{C}_\infty^u$, has a \emph{small excess} phase in $\nu$ for every $u < u_*$. This phase, now known to be non-trivial for all $u<u_*$ using \eqref{eq:C1}, is characterized by the fact that the formation of `droplets' secluded from the infinite cluster is ruled out; see~\cite[Remark~1.3)]{zbMATH07395560} and the discussion around  \cite[(0.12)]{zbMATH07483480} for more on this.

\medskip

We now present our second main result, which relies on both Theorem~\ref{T:ri-main} itself, as well as the techniques outlined in \S\ref{subsec:overview}. Theorem~\ref{T:ri-2point} below is concerned with the actual decay rate of the truncated two-point function $\tau_u^{{\rm tr}}(x)= \tau_u^{{\rm tr}}(0,x)$ from \eqref{eq:tau-RI} at large Euclidean distances $|x|$. This question is particularly challenging in the supercritical regime $u< u_*$, in which the decay crucially hinges on the truncation in the form of the additional disconnection from infinity, giving rise to a non-monotone event (on the contrary, for $u>u_*$ this condition can be safely ignored). 

By Theorem~\ref{T:ri-main}, see \eqref{eq:ri-equal}, along with \eqref{eq:NLF} and \eqref{eq:u-hat}, one now knows that the decay is at least stretched exponential in $|x|$ when $u<u_*$. This is already significant. Indeed, the recent series of works \cite{RI-I,RI-II,RI-III} involving the first two authors, 
established a similar decay for $\tau_{u}^{{\rm tr}}(x)$ in $|x |$ for $u > u_*$, but no meaningful bound on $\tau_{u}^{{\rm tr}}(x)$ was obtained for $u<u_*$ outside perturbative regimes $u\ll 1$, cf.~\cite{Tei11,MR3269990}. The closest to getting a bound on $\tau_{u}^{{\rm tr}}(x)$ for $u$ near $u_*$ is the main result of \cite{RI-I}, which yields a stretched exponential bound on a {sprinkled} version of 
$\tau_{u}^{{\rm tr}}(\cdot)$ for 
$u < u_\ast$ where the disconnection in \eqref{eq:tau-RI} happens at a {\em strictly} lower 
intensity 
$v < u$; cf.~also the discussion in Step 2 of \S\ref{subsec:overview} for the reason why. Our next theorem pins down the rate of decay for $\tau_{u}^{{\rm tr}}(\cdot)$ at all non-critical values 
of $u$. 

\begin{theorem}\label{T:ri-2point}
For $u \neq u_*$, 
with $|\cdot|$ denoting the Euclidean distance on $\Z^d$,
\begin{align}
&\text{when $d \geq 4$, 
} \sup_{x \in \mathbb{Z}^d}\frac{1}{|x|}\log \tau_u^{{\rm tr}}(x) \le  -c(u,d) \ (\in (0,1));\label{eq:ri-2point-4d}\\
&\text{when $d = 3$, } \limsup_{|x|\to \infty} \frac{\log |x|}{|x|}\log \tau_u^{{\rm tr}}(x) \le  -\frac{\pi}{3}(\sqrt{u} - \sqrt{u_*})^2. \label{eq:ri-2point-3d}
\end{align}
\end{theorem}
\begin{remark}\label{remark:ub}
Similarly as with Theorem~\ref{T:ri-main}, where $\text{SLU}(u)$ in \eqref{eq:slu} strengthens $\text{Unique}(L,u,u)$ from \eqref{eq:UNIQUE1}, Theorem~\ref{T:ri-2point} can be strengthened as follows. For all $u \neq u_*$, the conclusions \eqref{eq:ri-2point-4d}, \eqref{eq:ri-2point-3d} remain true with $\tau_u^{{\rm tr}}(x)$ replaced by $\P\big[\bigcup_{v \le u} \{x \stackrel{}{\longleftrightarrow}y,\nlr{}{}{x}{\infty} \text{ in } \mathcal{V}^v\}\big],$ cf.~\eqref{eq:tau-RI}.
\end{remark}

The bounds \eqref{eq:ri-2point-4d} and \eqref{eq:ri-2point-3d} valid for all $u \neq u_*$ constitute a substantial improvement over \eqref{eq:NLF}, but their full strength only becomes apparent in combination with our companion results derived in \cite[Theorem 1.1 and 
Corollary~1.2]{GRS24.2}, which comprise matching lower bounds, exhibiting the exponential decay of $\tau_u^{{\rm tr}}(x)$ in 
\eqref{eq:ri-2point-4d} and effectively allowing to replace the $\limsup$ by a $\lim$ in 
\eqref{eq:ri-2point-3d}. 

Both upper bounds in Theorem~\ref{T:ri-2point} follow from 
corresponding estimates for a truncated radius observable, whereby the point $x$ in the event 
defining $\tau_u^{{\rm tr}}(x)$ in \eqref{eq:tau-RI} is replaced by the inner vertex boundary of 
the Euclidean ball $B_r^2$ centered at $0$ of radius $r=|x|$, see 
Theorems~\ref{thm:d=4} and~\ref{thm:d=3} below. These theorems also supply pertinent 
bounds on other quantities of interest, such as the two-arms event 
(
see around \eqref{def:locuniq}). Incidentally, in the subcritical regime $u>u_*$, the tail of the radius observable was derived independently in \cite{Prevost23}, building on the 
sharpness result \eqref{eq:subcritical0} of \cite{RI-I}. We further refer to \cite{gosrodsev2021radius} (building on \cite{DCGRS20}) for similar results in the context of GFF level-set 
percolation, to~\cite{MR4749810, MR4736694} for similar results in the {\em subcritical} phase for a more general class of Gaussian fields, and to 
\cite{DPR22,drewitz2023arm,drewitz2024criticalonearmprobabilitymetric} for quantitative results in both $u$ and $R$ for the GFF on the cable system. 
Bounds akin to \eqref{eq:ri-2point-4d} in the supercritical regime 
have been proved recently in \cite{dembin2023supercritical} for Voronoi percolation, inspired 
by the ideas in \cite{RI-II}; see also \cite{Bod05,MR4756366,DCGR20,gunaratnam2025supercritical} for the (FK-)Ising and related $\phi^4$ 
models, and \cite{MR4734554} for a robust proof in the case of Bernoulli percolation, extending \cite{GriMar90} to transitive graphs of polynomial volume growth. One common feature of these models is that correlations  
decay (at least) exponentially fast with distance. 

We now briefly comment on the proof of Theorem~\ref{T:ri-2point}. 
With Theorem~\ref{T:ri-main} at hand, we achieve this via {bootstrapping} 
arguments that use an inequality akin to \eqref{eq:sprinkling_removed}
in order to effectively bootstrap the event $\rm{SLU}$ across scales when $u<u_*$ (which is the main focus of 
the current article). These arguments crucially rely on the excursion packet technology
delineated in \S\ref{subsec:overview}, which is combined with ideas  drawn from \cite{gosrodsev2021radius} and \cite{MR3602841}. As alluded to below Theorem~\ref{T:ri-main}, when $u<u_*$ it is absolutely critical to have access to ${\rm SLU}_L(u)$ (rather than ${\rm Unique}(L,u,u)$ for instance) as a 
triggering estimate, which removes the obstruction to initializing the bootstrap (cf.~the discussion atop \eqref{eq:boosted0}).

The bootstrapping mechanism not only results in a progressive improvement on the bound of a probability of interest (say, involving ${\rm SLU}$). It also serves to improve the `quality' of the family $\zeta$ of excursion packets involved in the bootstrap, which becomes, in a sense, increasingly fine-tuned at larger scales: that is, the bootstrap also needs to ensure that the number of good encounter times (corresponding to $m$ in item iii) of Proposition~\ref{P:explo-informal}) increases at larger and larger scales. This rests on a careful design of the relevant family $\zeta$ of excursion packets, which depends on a parameter $\nu$ that controls their `quality';  we refer the interested reader to~\eqref{def:Z_tr_RI} and \eqref{eq:V_z^I}-\eqref{eq:V_z^II}.

\subsection{Organization}\label{subsec:organization} We now give an overview of the organization of this article.
In Section~\ref{sec:prelims}, we set up notation, collect preliminary facts about random walks and interlacements, 
and introduce several systems of excursions, together with  related couplings, as a 
means to `localize' the interlacement set. They play a prominent role in the forthcoming sections.

\medskip

 In Section~\ref{sec:insertion} we state and prove Proposition~\ref{lem:conditional_prob1}, which corresponds to the extended version of Lemma~\ref{L:FE}. It constitutes a first major contribution of this paper. 

\medskip
Sections~\ref{sec:cg} and \ref{sec:coarse_graining} develop 
a robust framework for bootstrapping, amenable to all later purposes. The probabilistic estimates it entails are the content of Section~\ref{sec:cg} (see Proposition~\ref{prop:bootstrap_prob}), while the deterministic part of the scheme is described in 
Section~\ref{sec:coarse_graining} (see Proposition \ref{prop:bootstrap_events}). The latter refers to the 
propagation of an ambient cluster with `good' properties (such as the event $\mathscr 
G_L$ above) 
{across} a single renormalization step. 
To track various quantities with enough precision, we also rely on certain 
(surrogate) harmonic averages largely inspired by ideas from \cite{MR3602841}, which have been 
streamlined along the way (cf.~Proposition~\ref{prop:Xutail_bnd}). As a straightforward first application of this scheme, we prove in \S\ref{subsubsec:subcrit} the bounds of Theorem~\ref{T:ri-2point} in the sub-critical phase $u>u_*$. Our \emph{unified} treatment of all values $u \neq u_*$ via the framework of Sections~\ref{sec:cg}-\ref{sec:coarse_graining} is very much in line with the physical intuition by which the near-critical behavior of the system is agnostic to the side from which $u_*$ is approached, as reflected by the fact that the right hand side of \eqref{eq:ri-2point-3d} vanishes for $u=u_*(1\pm \varepsilon)$ as $O(\varepsilon^2)$ when $\varepsilon \downarrow 0$ regardless of the sign.

\medskip

Section~\ref{sec:supcrit_upper_bnd} is the cornerstone of the proof. In~\S\ref{subsubsec:supcrit}, we introduce excursion packets and 
feed them into the framework of Sections~\ref{sec:cg}-\ref{sec:coarse_graining} in the specific context of $\rm{SLU}$. The main result is then Theorem~\ref{prop:exploration_RI-I}, which roughly corresponds to the bound \eqref{eq:sprinkling_removed} in item iii) of Proposition~\ref{P:explo-informal} (when combined with Lemma~\ref{L:FE}). For technical reasons relating to small values of $u$, we need to distinguish two cases, referred to as \emph{types} ${\rm I}$ and ${\rm II}$. Strictly speaking, Theorem~\ref{prop:exploration_RI-I} deals with the more difficult type ${\rm I}$. The companion result for the simpler type ${\rm II}$ is Theorem~\ref{prop:exploration_RI-II}. \S\ref{subsec:reduction} comprises the proof of Theorem~\ref{prop:exploration_RI-I}. A stepping stone is Proposition~\ref{prop:reduction}, which subsumes all matters relating to the exploration and the good encounter times (cf.~items i) and ii) in Proposition~\ref{P:explo-informal}). Proposition~\ref{prop:reduction} is proved separately in \S\ref{subsec:exploration}.

\medskip

The pieces are put together in Section~\ref{sec:denouement}, where we prove our main results, Theorems~\ref{T:ri-main} and \ref{T:ri-2point}.

\medskip

Our policy with constants is as follows. Throughout the article $c, c', C, C', \ldots$ 
etc. denote 
finite, positive constants 
which are allowed to change from place to place. All constants may implicitly depend on the 
dimension $d \ge 3$. Their dependence on other parameters will be made explicit. Numbered 
constants $c_1,c_2,\dots$ and $C_1,C_2,\dots$
remain fixed 
after 
their first appearance within the text.

\section{Useful facts}\label{sec:prelims}

We gather here a few preliminaries that will be used throughout. In \S\ref{subsec:occupation}, we recall a few facts around random walks and interlacements. 
In \S\ref{subsec:excursion} we discuss excursion decompositions and couplings with independent excursions, see in particular Lemma~\ref{L:RI_basic_coupling}, 
which will fit all our purposes. This leads to three important notions of vacant sets $\overline{\mathcal V}^u_z$, ${\mathcal V}^u_z$ and $\widetilde{\mathcal 
V}^u_z$, see \eqref{eq:3Vs}, and corresponding systems of excursions, that are increasingly `localized'. 
They will play a central role in the sequel. 


The following notation is used throughout this article. The lattice $\Z^d, d \ge 3$, is equipped with the usual nearest-neighbor graph structure. We denote by $|\cdot|$ and $|\cdot|_{\infty}$ the $\ell^2$ and $\ell^{\infty}$-norms on $\Z^d$.  
We write $U \subset \subset \Z^d$ to denote a finite subset, $U^c\coloneqq \Z^d \setminus U$ and $|U|$ is the cardinality of $U$. A \textit{box} is a $\Z^d$-translate of $([0,L)\cap \Z)^d$ for some $L \geq 1$. For $U  \subset \Z^d$ we denote by $\partial U \coloneqq \{x \in U: \exists y\in U^c \text{ s.t. } |y-x|=1\}$ and $\partial^{\text{out}} U \coloneqq \partial (U^c)$ its \textit{inner and outer boundary}, and write $\overline{U} \coloneqq U \cup \partial^{\text{out}} U$ for its \textit{closure}. The 
set $U$ is \textit{connected} if any points $x, y \in U$ can be joined 
by a path whose range is contained in $U$. A \textit{component} (or \textit{cluster}) 
\textit{of $U$} is a maximal connected subset of $U$; we omit the attribute ``of $U$'' when 
$U=\Z^d$. 


\subsection{Setup}\label{subsec:occupation}
We consider the continuous-time {random walk} on $\Z^d$ with unit jump rate. We write $P_x$ for the law of this process under which $Y = 
(Y_n)_{n \ge 0}$ is a discrete-time random walk on $\Z^d$ with $Y_0=x$, and 
$(\zeta_n)_{n \ge 0}$ are i.i.d.~unit mean exponential variables. The continuous-time  walk 
$X = (X_t)_{t \ge 0}$ attached to this sequence is defined via $X_t =Y_k$, for $t \geq 0$ such that $\sum_{ 0 \leq i < k} \zeta_i  \le t < \sum_{0 \leq i  \leq k} \zeta_i$, 
where the empty summation is interpreted as 0. For any positive measure $\mu$ on $\Z^d$ we write $P_\mu = \sum_{x\in \Z^d} \mu(x)P_x$. We use $E_x$ for the 
expectation with respect to $P_x$ and similarly $E_{\mu}$ (although $P_{\mu}$ is not necessarily a probability measure). To a set $K \subset \Z^d$ we associate the 
stopping times $H_K, \widetilde{H}_K$, where $H_K =  \inf \{t \geq 0: X_t \in K \}$ and $\widetilde{H}_K = \inf \{ t \geq \zeta_0: X_t \in K\}$. The exit time from $K$, 
i.e.~$H_{\Z^d \setminus K}$ is denoted as $T_K$.

We briefly recall some potential theory associated to $X$ that is used throughout.
We denote by 
$g(x,y) \coloneqq E_x[ \int_{0}^{\infty} 1_{\{X_s = y \}} ds ]$, for $x, y \in \Z^d,$ the Green's function of $X$. 
For 
any finite set $K \subset \Z^d$, the equilibrium measure of $K$ 
is defined as
\begin{equation}\label{eq:equilib_K} 
e_K(x) = P_x[\widetilde{H}_K=\infty]1\{x \in K\},
\end{equation}
with $\widetilde{H}_K = \inf \{ t \geq \zeta_0: X_t \in K\}$. It is a measure supported on $\partial K$. We denote by
$\text{cap}(K) = \sum_x e_{K}(x)$ its total mass, the capacity of $K$, 
and by 
$\bar{e}_{K} = \tfrac{e_{K}}{ \text{cap}(K)}$ the normalized equilibrium measure. 

Following is a mixing result for entrance  distributions of the walk $X$ from afar. It appears e.g.~as Proposition~2.5 in \cite{MR3602841}. Inspection of that proof yields the quantitative dependence on $K$ stated below.
\begin{proposition}
\label{prop:entrance_time_afar}
For all $K \ge \Cl{C:Kbnd}, L \ge 1$, 
non-empty $A \subset B_{4L}$  
and 
$B \subset \subset \Z^d$ such that $B \cap B_{KL} = A$ 
with $\Z^d \setminus B$ 
connected, one has for any $y \in A$ and $x \in { \Z^d \setminus }(B \cup B_{KL})$,
\begin{equation}\label{eq:eAbar_condentrance}
 \big| P_x\left[ X_{H_B} = y \, \big|  \, H_B < \infty,  
X_{H_B} \in A\right] -  \bar e_A(y)
\big| \le {\Cr{c:equil}}{K}^{-1} \bar e_A(y)
\end{equation}
and
\begin{eqnarray}\label{eq:eBeB}
 \Big|  \frac{\bar e_B(y)}{\bar e_B(A)} - \bar e_A(y) \Big| \le  {\Cl{c:equil}}{K}^{-1}\bar e_A(y).
\end{eqnarray}
\end{proposition}	

We come to interlacements. Throughout this article, we work with the continuous-time interlacement point process, defined on its canonical space $(\Omega, 
\mathcal{A}, \P)$. We briefly review its construction. Let $\widehat{W}$ denote the set of doubly-infinite, $\Z^d \times (0, \infty)$-valued sequences $\widehat{w} = (w_n, \zeta_n)_{n \in \mathbb{Z}}$ such that 
$(w_n)_{n \in \Z}$ forms a doubly infinite, nearest-neighbor transient trajectory in $\Z^d$, and the sequence 
$(\zeta_n)_{n \in \Z}$ has infinite forward and backward sums. 
We endow 
$\widehat{W}$ with its canonical $\sigma$-algebra $\widehat{\mathcal{W}}$, generated by the 
evaluation maps $(X_t, \sigma_t)_{t\in \mathbb{R}}$ , defined by setting $X_t(\widehat{w})= w_n$, 
$\sigma_t(\widehat{w})=\zeta_n$ with $n \in \mathbb{Z}$ uniquely determined such that 
$\sum_{i <n} \zeta_i  \le t < \sum_{i \leq n} \zeta_i $. The discrete time-shifts $\theta_n$, $n 
\in \N$ naturally act on $\widehat{W}$ and we let $\widehat{W}^* \coloneqq 
\widehat{W}/\sim$, where $\widehat w \sim \widehat w'$ if $\widehat w = \theta_n \widehat w'$ for some $n\in \mathbb{Z}$. We write $\pi^*: \widehat W\to \widehat 
W^*$ for the corresponding canonical projection and $\widehat W^*_K \subset \widehat W^*$ is the set of trajectories modulo time-shift whose first coordinate 
visits~$K \subset \mathbb{Z}^d$. We equip $\widehat W^*$ with the quotient $\sigma$-algebra under $\pi^\ast$, denoted as $\widehat {\mathcal W}^*$. The 
space $\widehat{W}_+$ is defined analogously as 
above but comprising one-sided trajectories $(w_n, \zeta_n)_{n \geq 0}$ instead. The above laws $P_x$, $x \in \Z^d$, are defined on $\widehat{W}_+$.

The measure $\mathbb{P}$ is the probability governing the Poisson point process $\omega$ on $\widehat{W}^* \times \mathbb{R}_+$ with intensity measure 
$\nu(\mathrm{d}\widehat{w}^*) \mathrm{d}u$, where $\R_+= [0,\infty)$, $\mathrm{d}u$ denotes the Lebesgue measure and $\nu$ is a canonical measure on $(\widehat W^*, \widehat {\mathcal W}^*)$ 
(see~\cite[Theorem~1.1]{MR2680403}).   
 Given any $u \ge 0$, the interlacement set is defined as 
$\mathcal{I}^u= \mathcal{I}^u(\omega)=\bigcup_{(\widehat w^*,v) \in \omega, \, v \leq u}$ 
where $\widehat w^* = (w^*,\zeta^*)$ 
and, slightly abusing the notation, we implicitly identify the point measure $\omega$ with its 
support in writing $(\widehat w^*,v) \in \omega$. 

When only interested in $\mathcal{I}^u$ within a region $\Sigma \subset \Z^d$, it is  convenient to project $\omega$ onto the effective (Poisson) measure $\mu_{\Sigma, 
u}(\omega)$ on $\widehat W_+$, defined as the push-forward of $\omega$ obtained 
by retaining only the points $(\widehat w^*, v)\in \omega$ such that $v\leq u$ and 
${\color{black}\widehat w^*} \in {\color{black}\widehat W_\Sigma^*}$ and mapping them to 
the onward trajectory $(\in \widehat W_+)$ upon their first entrance in $\Sigma$. It follows that 
\begin{equation}
	\label{eq:law_muKu}
	\text{under $\P$, $\mu_{\Sigma, u}$ is a Poisson process on $\widehat W_+$ with intensity $
		uP_{e_\Sigma}[\, \cdot \,]$}
\end{equation}
and on account of 
definition of $\mathcal I^u$, one sees that
	$\textstyle\mathcal{I}^u\cap \Sigma =\bigcup_{\widehat w \in \mu_{\Sigma,u}} \text{range}(w) \cap \Sigma$.
With hopefully obvious notation, we write $\mu_{\Sigma}(\omega)$ for the pushforward measure on $\widehat W_+ \times (0,\infty)$ defined similarly as $\mu_{\Sigma, 
	u}(\omega)$, but which retains the labels $u$. For a measurable function $f: \widehat W_+ \to \R_+$ we write 
\begin{equation}\label{eq:can-pair}
	\langle \mu_{\Sigma, u}, f \rangle \stackrel{\text{def.}}{=} \int_{\widehat W_+} f d \mu_{\Sigma, u} =  \sum_{\widehat w \in  \mu_{\Sigma, u}} f(\widehat{w})
\end{equation}
for its canonical pairing with $\mu_{\Sigma,u}$ 
The sum on the right-hand side of \eqref{eq:can-pair} is finite $\P$-a.s.~when $\Sigma$ is a finite set on account of \eqref{eq:law_muKu}.
The set $\Sigma$ will typically consist of many separated boxes, cf.~\eqref{def:Sigma} below. 

\subsection{Localization and couplings}\label{subsec:excursion}
We now set up the framework to decompose trajectories into excursions between a pair of 
nested sets, denoted as 
{\color{black}$ D$} and 
{\color{black}$ U$} below,  which will later allow us to split various local connectivity events 
into two parts --- one possessing very good decoupling properties and the other involving an atypical 
number of excursions (see Sections~\ref{sec:coarse_graining} and~\ref{sec:supcrit_upper_bnd}). 
From here on we assume that for any realization $\omega =\sum_{i \geq 0} \delta_{(\widehat 
w_i^*, u_i)} \in \Omega $ (see 
above \eqref{eq:law_muKu}) the labels $u_i$, $i \geq 0$ are pairwise distinct 
and that $\omega(\widehat W_K^*\times \R_+)=\infty$ and $\omega(\widehat 
W_K^*\times[0,u])<\infty$ for all $u \geq 0$ and $K \subset \subset \Z^d$. We do not incur 
any loss of generality with these assumptions since these sets have full $\P$-measure.

Now let 
{\color{black}${D} , {U}$} be finite subsets of $\mathbb{Z}^d$ with 
{\color{black}$\emptyset \neq  D \subset  U$} and denote by 
{\color{black}$W_{ D,  U}^+$} the set of all excursions between 
{\color{black}$ D$} and 
{\color{black}$\partial^{{\rm out}}   U$}, i.e.~all finite $\Z^d$-valued nearest-neighbor 
piecewise constant right-continuous trajectories 
starting in 
{\color{black}$\partial  D$}, ending in 
{\color{black}$\partial^{{\rm out}}  U$} and not 
exiting 
{\color{black}$ U$} in between. 
By `finite' we 
mean that the trajectory makes finitely many jumps. The interlacement point measure 
$\omega$ gives rise to a sequence of excursions in 
{\color{black}$W_{ D,  U}^+$}, as follows. For 
{\color{black}$\widehat w \in {\rm supp}(\mu_{ D})$} (see below \eqref{eq:law_muKu} for notation) the infinite transient trajectory 
$(X_t(\widehat w) : t \ge 0)$ induces excursions between 
{\color{black}$ D$} and 
{\color{black}$\partial^{{\rm out}}  U$} according to the successive return times $R_k$ 
and departure times $T_k$ between these sets, defined recursively as $T_0 =0$ 
and 
$R_k = T_{k-1} + H_{ D} \circ \theta_{T_{k-1}}$, $T_k = R_k + T_{ U} \circ \theta_{R_k}$, for $k \geq 1$, where 
{\color{black}$T_{ U} = H_{ U^c}$} and all of $T_k, R_j, T_j$, $j > k$ are understood to be $=\infty$ whenever 
$R_k=\infty$ for some $k$. Given 
{\color{black}$\mu_{ D}(\omega)=\sum_{i \geq 0} \delta_{(\widehat w_i, u_i)}$}, we order the excursions from {\color{black}$ D$ to $\partial^{{\rm out}}   U$}, first by increasing value of $\{u_i: 
\widehat w_i \in \widehat W^+\}$, then by order of appearance within a  trajectory 
$\widehat w_i \in \widehat W^+$. This yields the sequence 
{\color{black}$Z^{ D,  U}(\omega)= ( Z^{ D,  U}_k(\omega) )_{k \geq 
		1}$} given by
{\color{black}\begin{equation}
		\label{eq:RI_Z}
		\big( Z^{ D,  U}_k(\omega) \big)_{k \geq 1} \coloneqq \big(X^0[R_1, T_1], \dots, X^0[R_{N_{ D,  U}}, T_{N_{ D,  U}}],\, X^1[R_1, T_1], \dots \big),
\end{equation}}
of 
{\color{black}$W^+_{ D,  U}$}-valued random variables under $\P$, encoding the successive 
excursions of $\omega$; here, with hopefully obvious notation $X^i = X(\widehat w_i)$, 
$X^i[t_1, t_2]$ is the trajectory 
given by $X^i[t_1, t_2](s) = X^i((s + t_1) \wedge t_2)$ for $s \in [0, \infty)$ and 
{\color{black}$N_{ D,  U}= N_{ D,  U}(\widehat w_0)$} is the total number of excursions from 
{\color{black}$ D$ to $\partial^{{\rm out}}  U$} in $\widehat w_0$, i.e. 
{\color{black}$N_{ D,  U}(\widehat w_0) = \sup\{j : T_j(\widehat w_0) < \infty\}$.} 
We further define (see \eqref{eq:law_muKu} and \eqref{eq:can-pair} for notation)
\begin{equation}\label{def:N_AU}
	N_{ D,  U}^u= N_{ D,  U}^u(\omega) = \langle \mu_{ D, u}, N_{ D,  U}\rangle(\omega), 
\end{equation}
the total number of excursions from 
{\color{black}$ D$ to 
	$\partial^{{\rm out}}  U$. By construction of 
	$Z^{ D,  U}$ (cf.~below \eqref{eq:law_muKu}) if follows that $\mathcal{I}^u 
	\cap  D$ can be written as $\bigcup_{k} \text{range}(Z_k) \cap  
	D$ with the union ranging over $1 \leq k \leq N_{ D,  U}^u$. }


Now suppose that ${{D}} \subset  { U} \subset\subset \Z^d$ are 
such that $ D \subset  \check{D}$ and $ U \subset \check{U}$. We can then define the 
successive return and departure times $(R_k^\ell, T_k^\ell)_{k \ge 1}$ 
between $D$ and $\partial^{{\rm out}} U$ as above \eqref{eq:RI_Z} for 
any excursion $Z_{\ell}^{\check{D},  \check{U}}$. Since $D \subset \check D$ and $U \subset \check U$, any excursion $X[R_k, T_k]$ of 
a 
trajectory $X$ between $D$ and $\partial^{{\rm out}} U$ is a segment 
of a unique excursion $X[\check R_{k'}, \check T_{k'}]$ of $X$ between $\check D$ and 
$\check U$. 
Thus, letting $M^{\ell} =\sup\{j : T_j^\ell < \infty\}$ and $N= N^u_{ D,  U}$, cf.~\eqref{def:N_AU}, 
we 
have in view of \eqref{eq:RI_Z} that for all $u \ge 0$,
\begin{equation}\label{eq:nested_excursion}
		  \big(Z_{1}^{ \check{ D},  \check{ U}}[R_1^1, T_1^1], \ldots, Z_{1}^{ \check{ 
				D},  \check{ U}}[R_{M
				^1}^1, T_{M
				^1}^1], \ldots, Z_{N}^{ \check{ D},  \check{ U}}[R_{M^N}^N, T_{M^N}^N]\big) = \, \big(Z_{1}^{{ D}, { U}}, \ldots, Z_{N}^{{ 
				D}, { U}} \big).
\end{equation}
We refer to $(Z_{\ell}^{ \check{ D},  
	\check{ U}}[R_k^\ell, T_k^\ell])_{1 \le k \le M^\ell}$ between $ D$ 
and $ U$ as the sequence of excursions {\em induced} by $Z_{\ell}^{ \check{ D}, 
\check{ U}}$. 


We now couple the excursions \eqref{eq:RI_Z}, which are highly correlated, with a family of i.i.d.~excursions between 
$D$ and $\partial^{{\rm out}}  U$. We call this \emph{localization}. 
We will need to localize systems of excursions as in \eqref{eq:RI_Z} jointly for several choices of sets $( D,U)$ 
that are sufficiently spread-out. Thus let $\mathcal C$ be a finite set and 
$\emptyset \ne {\color{black} D_z \subset  U_z} \subset \subset \Z^d$ be pairs of 
sets indexed by $z \in \mathcal C$ satisfying 
\begin{equation}\label{eq:UzUz'}
	{\color{black}\overline{ U}_z \cap \overline{ U}_{z'}} = \emptyset, \text{ for all $z \ne z' \in \mathcal C$.}
\end{equation}
For a given collection 
$( D_z,  U_z : z \in \mathcal{C})$
satisfying \eqref{eq:UzUz'}, the desired coupling will be between $\P$ 
and an auxiliary probability $\widetilde\P_{\mathcal{C}}$ governing a collection of 
independent right-continuous, Poisson counting functions $(n_z(0, t))_{t \ge 0}$, $z \in \mathcal C$, with unit intensity, vanishing at 0, along with independent collections of i.i.d.~excursions $(\widetilde 
Z_k^{ D_z,  U_z})_{k \ge 1}$, $z \in \mathcal C$, having for each $z$ the same law as $X_{\cdot \wedge T_{ U_z}}$ under $P_{\bar e_{ D_z}}$.
We write $\widetilde \P_z = \widetilde\P_{\{z \}}$ when $\mathcal{C}=\{ z\}$ is a singleton. For $m_0 \ge 1$ and  
$\varepsilon \in (0, 1)$, the event (declared under $\widetilde\P_{\mathcal{C}}$)
\begin{equation}\label{def:Uz}
		\mathcal{U}_z^{\varepsilon, m_0} \coloneqq\big\{n_z(m, (1 + \varepsilon) m) < 2\varepsilon m, 
		(1 - \varepsilon) m < n_z(0, m) < (1 + \varepsilon) m,  \text{ for } m \ge m_0\big\},
\end{equation}	
for $z \in \mathcal{C}$, will play a central role in the sequel.  For later reference, we record that 
\begin{equation}\label{eq:bnd_Uzepm}
	\widetilde{\P}_{\mathcal{C}}[\mathcal{U}_z^{\varepsilon, m_0}] \ge 1 - {C}{\varepsilon^{-2}}e^{-c m_0 \varepsilon^2}, \text{ for all $z \in \mathcal{C}$, $\varepsilon \in (0, 1)$ and $m_0 \geq 1$,}\end{equation}
which follows from standard Poisson tail bounds. For $\mathbb Q$ any coupling of $\P$ and $(\widetilde 
Z_k^{ D_z,  U_z})_{k \ge 1}$ we define the inclusion event
${\rm Incl}_z^{\varepsilon, m_0}$ as
\begin{equation}\label{eq:RI_basic_coupling2}
	\begin{split}
		&\Big\{\{\widetilde Z_1^{ D_z,  U_z}, \dots, \widetilde Z^{ D_z,  U_z}_{\lfloor (1-\varepsilon)m \rfloor} \} \subset  \{Z_1^{ D_z,  U_z}, \dots, Z^{ D_z,  U_z}_{\lfloor(1+3\varepsilon)m\rfloor} \}, \: \mbox{and}\\
		&\{Z_1^{ D_z,  U_z}, \dots, Z^{ D_z,  U_z}_{\lfloor (1 - \varepsilon)m \rfloor} \} \subset \{\widetilde Z_1^{ D_z,  U_z}, \dots, \widetilde Z^{ D_z,  U_z}_{\lfloor (1 + 3\varepsilon)m \rfloor} \} \mbox{ for $m \ge m_0$}\Big\}.
	\end{split}
\end{equation}
 In \eqref{eq:RI_basic_coupling2} and throughout the remainder of this article, with a 
slight abuse of notation, inclusions involving sets of excursions of the form $\{ Z_1,\dots, Z_n\} 
\subset \{ Z_1',\dots Z_{n'}'\}$ are understood as inclusions between \emph{mutlisets}; that is, 
the plain inclusion of sets holds and moreover if $Z_k=Z'_{k'}$ then $\text{mult}(Z_k) \leq 
\text{mult}(Z'_{k'}) $, where $\text{mult}(Z_k) = |\{ j \in \{1,\dots ,n\} : Z_j =Z_k\}| $ is the 
multiplicity of $Z_k$ in the sequence $(Z_j)_{1\leq j \leq n}$.

Our aim is to devise a coupling $\mathbb Q_{\mathcal{C}}$ of $\P$ and $\widetilde{\P}_{\mathcal{C}}$ rendering \eqref{eq:RI_basic_coupling2} likely for all $z \in 
\mathcal{C}$ (when $m_0$ is large). The next lemma asserts in essence that a coupling with this property exists, provided the sets $( D_z,  U_z)$ satisfy certain conditions, see \eqref{eq:RI_cond_Q} below. These requirements are best stated in terms of an auxiliary random variable $Y$ 
defined as follows. For $x\in \Z^d$, let $Q_x$ be the joint law of two independent simple random walks $X^1$, 
$X^2$ on $\Z^d$, respectively sampled from $P_x$ and from $P_{\bar{e}_{ D}}$, and 
define
\begin{equation*}
	Y = \begin{cases}
		X^1_{H_{ D}^1}, \quad & \text{if $H_{ D}^1\coloneqq\inf\{ t \geq 0: X^1_t \in A\} < \infty$}\\
		X^2_0, & \text{otherwise.}
	\end{cases}
\end{equation*}
The following result is a restatement of Lemma~2.1 in \cite{MR3126579} adapted to our context and 
uses the soft local time technique from \cite{PT12} (see also \cite[Section~5]{MR3602841}).  Although the event ${\rm Incl}_z^{\varepsilon, m_0}$ does not include multiplicities in the context of \cite{MR3126579}, the inclusion \eqref{eq:inclusion} below continues to hold for this stronger notion of ${\rm Incl}_z^{\varepsilon, m_0}$, as follows directly from the soft local time technique, which entails a domination of point measures (that account for multiplicities).  We omit the proof.

\begin{lem}[Coupling $Z$ and $\widetilde{Z}$]\label{L:RI_basic_coupling} 
	For any finite collection $ D_z \subset  U_z \subset\subset \Z^d$, $z \in \mathcal C$, satisfying 
	\eqref{eq:UzUz'}, there exists a coupling $\mathbb Q_{\mathcal C}$ of $\P$ and $\widetilde{\P}_{\mathcal 
		C}$ with the following property. If, for some $\varepsilon \in (0,1)$,
	\begin{equation}
	\begin{split}
		&\textstyle \big(1-\frac{\varepsilon}{3}\big)  \bar{e}_{ D_z}(y) \leq Q_x\left [Y = y \, |  \, Y \in   D_z \right] \leq \textstyle \big(1+\frac{\varepsilon}{3}\big)  \bar{e}_{ D_z}(y), \label{eq:RI_cond_Q} \\
		&\text{for all $z \in \mathcal C$ and $x \in \partial^{{\rm out}}  U_z$,} 
	\end{split}
	\end{equation}
	then in the probability space underlying $\mathbb Q_{\mathcal C}$, 
	\begin{equation}\label{eq:inclusion}
		\mathcal{U}_z^{\varepsilon, m_0} \subset {\rm Incl}_z^{\varepsilon, m_0}, \text{ for all $z \in \mathcal C$ and $m_0 \geq 1$.}
	\end{equation}
\end{lem}	

We will use Lemma~\ref{L:RI_basic_coupling} to switch back and forth between interlacement sets comprising different sets of excursions, Three such interlacement 
sets, which we now introduce, will play a central role. 
If $Z=(Z_k)_{1\leq k \leq n_Z}$ is a sequence of excursions (i.e.~$Z_k \in W_{ D,  U}^+$ for some $ D \subset  U$) and $n_Z \in \mathbb{N}$, which we sometimes call a 
\emph{packet}, we write
\begin{equation} \label{eq:I(Z)}
	\mathcal I( Z) \coloneqq \bigcup_{1 \le k \le n_Z}\,\,{\rm range}(Z_k) \mbox{ and } \mathcal V(Z) \coloneqq \Z^d \setminus \mathcal I(Z) 
\end{equation}
to denote the {\em interlacement} and {\em vacant} set corresponding to the $Z$, respectively. 
The number $n_Z$ may or may not be random, i.e. {\em vary} with $\omega$, and the $Z_k$'s will typically be excursions between boxes which we now introduce,. Given a length scale $L \geq 1$ and a rescaling parameter $K \geq 100$, both integer, we consider boxes $C_z \subset \tilde{C}_z \subset \tilde{D}_z \subset D_z  \subset 
{U}_z$ attached to points $z \in \Z^d$, where
\begin{equation}\label{def:CDU}
\begin{split}
&C_z \coloneqq z+[0,L)^d ,\; \ \tilde{C}_z \coloneqq z+[-L,2L)^d, \\
&\tilde{D}_z \coloneqq 
z+[-2L,3L)^d,\; \  D_z \coloneqq z+[-3L,4L)^d, 
\text{ and}\\
&U_z \coloneqq z+[-KL+1, L+KL-1)^d 
\end{split}		
\end{equation}
(all tacitly viewed as subsets of $\Z^d$). In case we work with more than one scale in a given 
context we sometimes explicitly refer to the associated length scale $L$ by writing $C_{z,L}= 
C_z, \tilde{C}_{z,L}=\tilde{C}_z$ etc.

For $z \in \mathbb{L}= L\Z^d$ and 
$u>0$, abbreviating $N_{D, U}^u$ in \eqref{def:N_AU} as $N^u_{z} = N^u_{z, L}$  when 
$( D,  U) = ( D_{z},  U_{z})$ as in \eqref{def:CDU}, we introduce in the 
notation of \eqref{eq:I(Z)}, \begin{equation}\label{eq:3Vs}
	\begin{array}{ll}
		\overline{Z}_z^u= (Z^{ D_z, U_z}_k)_{1 \leq k \leq N_z^u}, & 
		\overline{\mathcal{V}}_z^u = \mathcal{V}(\overline{Z}_z^u)\, \big (  = \mathcal{V}^u 
		\text{ on } D_z \big), \\[0.3em]
		Z_z^{u}= (Z^{ D_z, U_z}_k)_{1 \leq k \leq 
			u \, \text{cap}( D_z) 
		}, & \mathcal{V}_z^{u} = \mathcal{V}(Z_z^{u})\\[0.3em]
		\widetilde Z_z^u= (\widetilde Z^{ D_z, U_z}_k)_{1 \leq k \leq  
			u \, \text{cap}( D_z)  
		}, & \widetilde{\mathcal{V}}_z^u = \mathcal{V}(\widetilde Z_z^u);
	\end{array}
\end{equation} 
The quantities $\overline Z_z^u, Z_z^{u}$ and $\overline{\mathcal{V}}_z^u, \mathcal{V}_z^{u} $ are a priori defined under $\P$ (see \eqref{eq:RI_Z}) and $\widetilde 
Z_z^u, \widetilde{\mathcal{V}}_z^u $ under $\widetilde \P_{\mathcal{C}}$, and all quantities
in \eqref{eq:3Vs} are naturally declared under $\mathbb Q_{\mathcal{C}}$ any coupling of $(\P, \widetilde \P_{\mathcal{C}})$ (such as the one 
from Lemma~\ref{L:RI_basic_coupling}). We seldom add subscripts $L$, e.g.~write $ Z_{z,L}^u$ or $\mathcal{V}_{z,L}^u$ instead of 
$Z_z^u$ or $\mathcal{V}_z^u$, to insist on the scale $L$ used in defining $\mathbb{L}$ and the sets $ D_z$ and $ U_z$.

Given a sequence $Z=(Z_k)_{1\leq k \leq n_Z}$ of excursions and $x$ in $\Z^d$, we denote by $\ell_x({Z})$ their 
{\em discrete} occupation time at $x$, i.e.~the total 
number of times $x$ is visited by the sicrete skeleton of any 
trajectory in ${Z}$. Note that $\mathcal{V}(Z)= \{x: \ell_x(Z)=0\}$. We abbreviate $\ell_x^u=\ell_x( \overline{Z}_z^u)$, $x \in  D_z$.

We use $\overline{\mathcal V}_{\mathbb L}$, $\mathcal V_{\mathbb L}$ and $\widetilde{\mathcal V}_{\mathbb L}$ to refer collectively to the configurations 
$\{\overline{\mathcal V}^u_z : z \in \mathbb L, u > 0\}$, $\{{\mathcal V}^u_z : z \in \mathbb L, u > 0\}$ and
$\{\widetilde{\mathcal V}^u_z : z \in \mathbb L, u > 0\}$,  respectively, and use $\widehat{\mathcal V}_{\mathbb L}$ when referring to any one of them. 
Accordingly, we use $\widehat{\mathcal V}^u_z$ to denote either $\overline{\mathcal V}^u_z$, ${\mathcal V}^u_z$ or $\widetilde{\mathcal V}^u_z$  depending on 
whether $\widehat {\mathcal V}_{\mathbb L} = \overline{\mathcal V}_{\mathbb L}$, ${\mathcal V}_{\mathbb L}$ or $\widetilde{\mathcal V}_{\mathbb L}$.

The following important class of events will facilitate switching back and forth between the configurations appearing in \eqref{eq:3Vs}. For any $u, v \ge 0$ and $z \in \Z^d$, let 
{\begin{equation}\label{eq:F}
		\mathcal F_{z}^{u, v} = \mathcal F_{z, L}^{u, v}  \stackrel{{\rm def.}}{=} 
		\begin{cases}
			\big\{ N^u_{z} \le v\, {\rm cap}( D_{z})\big\}, & \text{ if } u \leq v\\
			\big\{ N^u_{z} \ge v\, {\rm cap}( D_{z})\big\}, & \text{ if } u > v
		\end{cases}
\end{equation}}
As with $\mathcal{U}_z^{\varepsilon, m_0}  $ in \eqref{def:Uz}, the 
events $\mathcal F_{z}^{u, v}$ in \eqref{eq:F} have been set up in a way that they will in practice always be likely. 
For reference, we record the following tail bounds from \cite{MR3602841} (see (3.18) and (3.22) 
in the proof of Theorem~3.3, therein), valid for any $\varepsilon \in (0, 1)$, if one chooses
$D_z$ and 
${U}_z$ as in \eqref{def:CDU}, then
\begin{equation}\label{eq:Nuz_tail_bnd}
	\P\left[({\mathcal F}_z^{u, v})^c \right] \le e^{-c(\varepsilon) u N^{d-2}} \mbox{ for $u, v > 0$ s.t.~$\tfrac{u\vee v}{u \wedge v} \ge 1 + \varepsilon$ and $K \ge C(\varepsilon)$}.
\end{equation}

Finally, we shall 
often work with a certain {\em noised} version of the configurations $\widehat{\mathcal V}^u_z$. To this effect we assume by suitable extension that $\mathbb{Q}_{\mathcal{C}}$ coupling of $\P$ and $\widetilde\P_{\mathcal{C}}$ (and a fortiori also $\P$) carries independent i.i.d~uniform random variables $\mathsf U = \{\mathsf 
U_x : x \in \Z^d\}$. For $\delta \in [0, 1)$ and $\mathcal V \subset \Z^d$, let 
$(\mathcal V)_\delta \subset \Z^d$ denote the set with the occupation variables
\begin{equation}\label{def:noised_set}
	1_{\{x \in (\mathcal V)_\delta\}} = 1_{\{x \in \mathcal V\}}1_{\{\mathsf U_x \ge \delta \}}.\end{equation}
One immediate but important consequence of this definition is that
\begin{equation}\label{def:noised_set_inclusion}
	\mbox{
		$(\mathcal V)_\delta$ is increasing in $\mathcal V$ and decreasing in $\delta$ 
		w.r.t. set inclusion.}
\end{equation}

\section{Restricted insertion tolerance for $\mathcal V^u$}\label{sec:insertion}

In this section we present a result bearing on the insertion tolerance property of the 
vacant set $\mathcal V(Z)$ (see \eqref{eq:I(Z)}) for {\em suitable} sequences of excursions $Z$. 
As noted in the discussion leading to \eqref{eq:insertion_tol0}, a naive 
insertion tolerance property does not hold for such vacant sets owing to their structural 
rigidity. 

Proposition~\ref{lem:conditional_prob1} below is the main result of this section and applies when the underlying sequence $Z$ is `large'. It is 
reminiscent of Proposition~3.1 in \cite{RI-II} 
which proves a {\em sprinkled} insertion tolerance 
property for $\mathcal V^u$. More precisely, in \cite[Proposition~3.1]{RI-II}, it is shown 
that a box $B$ can be opened in $\mathcal V^{u - \varepsilon}$ with probability $c = c({\rm 
rad}(B), \varepsilon) > 0$ conditionally on $\mathcal V^u$ (everywhere) and $\mathcal V^{u - 
\varepsilon}$ outside a {\em strictly larger} box $\widehat{B}$ provided some good event 
$\widetilde F_B$ occurs. In order to remove the sprinkling inherent in this result, 
one would 
want such a bound to hold with $\varepsilon = 0$ conditionally on $\mathcal V^u$ outside $B$ {\em itself} rather than a 
larger box $\widehat B$. We accomplish this goal partly in 
Proposition~\ref{lem:conditional_prob1} in the sense that we can retain partial information on 
$\mathcal V^u \cap (\widehat B \setminus B)$ in the conditioning. However, this partial information (see~\eqref{def:C_partial1} below) 
turns out to be sufficient in many practical instances; cf.~the proof of Theorem~\ref{prop:exploration_RI-I}.

We now set the stage for 
Proposition~\ref{lem:conditional_prob1}. 
Given a sequence $Z = (Z_j)_{1 \le j \le n_Z}$ of excursions (see above 
\eqref{eq:I(Z)}) 
and $y \in \mathbb L_0 = L_0\Z^d$ for some $L_0 \geq 10$, employing the notation from 
\eqref{def:CDU}, let
\begin{equation}\label{def:C_partial1}
	\begin{split} 
		\mathscr C_{\partial D_{y, L_0}}(Z) \coloneqq \,\text{the union of components of points in $\partial D_{y, L_0}$ in $\mathcal I(Z) \cap (D_{y, L_0} \setminus C_{y, L_0}).$}
	\end{split}
\end{equation}
Using the set $\mathscr{C}= \mathscr C_{\partial D_{y, L_0}}(Z)$, we define a `local uniqueness' event in a smaller annulus as
\begin{equation}\label{def:tildeLU1}
	\begin{split}
		\widetilde {\rm LU}_{y, L_0}(Z) = \textstyle  \bigcap_{x, x'}  \{\lr{}{ }{x}{x'} \text{ in } \mathscr C
			\setminus \partial D_{y, L_0}\},
	\end{split}
\end{equation}
with $x, x' \in (\tilde D_{y, L_0} \setminus \tilde C_{y, L_0}) \cap 
			\mathscr C$.
In words, $\widetilde {\rm LU}_{y, L_0}(Z)$ is the event that the set $\mathscr C \setminus \partial D_{y, L_0}$ has at most one component intersecting 
$\tilde D_{y, L_0} \setminus \tilde C_{y, L_0}$. 
Recalling the discrete occupation times $(\ell_x(Z))_{x\in \Z^d}$ from below \eqref{eq:3Vs}, let
\begin{equation}\label{def:Conn0}
{\rm O}_{y, L_0}({Z}) \coloneqq \textstyle \bigcap_{x \in \partial D_{y, L_0}}
\{ \ell_x({Z})\le L_0\}.	
\end{equation}	
Next, for any  (finite)  $J \subset \N^\ast =\{ 1,2,\dots\}$ and $z \in N\Z^d$ for integer $N \geq 1$, using the notation from \eqref{eq:RI_Z} and \eqref{def:CDU} we consider the sequence of excursions  $Z_J = {Z}_J^{ D_{z, N}, U_{z, N}}=\big(Z_j^{ D_{z, N},  U_{z, N}}\big)_{j \in J}$, 
$\delta \in (0, \frac12)$ a noise parameter as in \eqref{def:noised_set} and $u' \ge u \in 
(0, \infty)$. Using this data, we define the $\sigma$-algebra that will be used for 
conditioning in Proposition~\ref{lem:conditional_prob1}. For $y \in 
\mathbb L_0$, define $ \mathcal F \equiv \mathcal F_{y, L_0}(Z_J, \delta, u, u')$ as
\begin{equation}\label{def:FEsigma_algebra1}
		\mathcal F =
		\sigma\big((\mathcal V^{u})_\delta, (\mathcal V^{u'})_{2\delta}, N_{z, N}^u,  \mathcal I(Z_J)_{| (\mathring D_{y, L_0})^c}, \mathscr C_{\partial D_{y, L_0}}(Z_J), \big\{\ell_{x}^{u} : \,  x \in (\mathring D_{y, L_0})^c\, \big\} \big),
\end{equation}
where $\mathring D_{y, L_0} 
	\coloneqq D_{y, L_0} \setminus \partial D_{y, L_0}$ and 
	$\sigma(\cdot)$ denotes the $\sigma$-algebra generated by a set of random variables. 
In \eqref{def:FEsigma_algebra1} and elswhere, we tacitly identify $\mathcal I(Z_J)$ or any other 
subset of $\Z^d$ (e.g.~$\mathscr C_{\partial D_{y, L_0}}(Z_J)$) with its occupation function $\{ 1_{\{x \in \mathcal I(Z)\}} : x \in \Z^d\}$. Observe now that the `good' events $\widetilde {\rm LU}_{y, L_0}(Z_J)$  
and ${\rm O}_{y, L_0}(\overline Z_{z, N}^u)$ (see \eqref{eq:3Vs} regarding $\overline Z_{z, 
N}^u$) are both measurable relative to $\mathcal F_{y, L_0}$, as is the event $\{J \subset [1, N_{z, N}^u]\}$ {\em when} $J$ is measurable relative to 
$N_{z, N}^u$ (e.g., when it is {\em deterministic}). We are now ready to state 
the main result of this section, which entails a certain form of insertion tolerance.

\begin{prop}\label{lem:conditional_prob1}
	Let $L_0 \geq 100$ and $\delta \in (0, \frac12)$. There exists $c = c(\delta, 
	L_0) \in (0, 1)$ such that for all $u' \ge u \in (0, \infty)$, $K \geq 100$, $z \in N\Z^d$ for some $N \ge 10^3 L_0$,
	 $y \in L_0 \Z^d$ such that $ D_{y, L_0} \subset  D_{z, N}$ and any $J \subset \N^\ast$ measurable relative to $N_{z, N}^u$, abbreviating $Z_J = Z_J^{ 
	D_{z, N},  U_{z, N}}$ we have
	\begin{equation}\label{eq:conditional_prob1}
		\P \text{-a.s.,  }\P\left[C_{y, L_0} \subset \mathcal V(Z_J) \, \big | \, \mathcal F \, 
		\right] \ge c \,
		1_{G},  
	\end{equation}
	with the ``good'' event $G=\widetilde {\rm LU}_{y, L_0}(Z_J) \, \cap \, {\rm O}_{y, L_0}(\overline{Z}_{z, N}^{u}) \, \cap \, \{J \subset [1, N_{z, N}^u]\}$.
\end{prop}
%
\begin{remark}
Although insufficient for our purposes, the choice $J = [1, N_{z, N}^u]$ is perfectly valid, whence $Z_J= \overline{Z}_{z,N}^u$, cf.~below \eqref{eq:3Vs}, and \eqref{eq:conditional_prob1} yields a bound 
on the conditional probability that $C_{y, L_0} \subset \mathcal{V}^u$.
\end{remark}
\begin{proof}
 We first introduce a certain sigma-algebra which corresponds to the `correct' conditioning 
 outside $\mathring D_{y, L_0}$, see $\widehat{\mathcal F}$ in \eqref{eq:measurability} below.  
 For a bi-infinite transient trajectory  
$w = (w(n))_{n \in \mathbb{Z}}$ in $\Z^d$ and a pair of sets $\emptyset \ne D  \subset U 
\subset\subset \Z^d$, let $-\infty < R_1(w) < T_1(w) < R_2(w) < \ldots$ denote the successive 
return and departure times of $w$ between $D$ and $\partial^{{\rm out}} U$. Note that possibly $R_1(w)=\infty$. 
Let $w_{D, U}^-$ denote the sequence of segments $(w(-\infty, R_1-1], w[T_1, R_2-1], \ldots)$ 
where $w[n_1, n_2](n) = w(n_1 + n)$ for $0 \leq n  \leq n_2 - n_1$, $w(-\infty, n_2](n) = w(n_2 
-  n)$ for $-\infty < n  \leq n_2$ and 
$w(-\infty, \infty]$ ($= w(-\infty, \infty)$) is understood as $(w(n))_{n \in \mathbb{Z}}$. 
In words, $w_{D, U}^-$ is the sequence of segments of $w$ obtained after deleting all its 
excursions between $D$ and $\partial^{{\rm out}}U$ {\em minus} their endpoints, i.e.~the 
paths $w[R_k, T_k-1]$ for $k \ge 1$. Since we forget the endpoints $R_k, T_k - 1$ in 
defining the segments $w[R_k, T_k-1]$, it is clear from this definition that $w_{D, U}^- = \tilde 
w_{D, U}^-$ if $\tilde w = \theta_n w$ for some $n \in \Z$ and thus $w_{D, U}^-$ is 
well-defined as a {function} of the equivalence class $w^\ast$ of any trajectory $w$ 
under discrete time-shift. 
Our case of interest is 
\begin{equation}\label{eq:DU-FE}D = U = \mathring D_{y, L_0} \ (= D_{y, L_0} \setminus \partial D_{y, L_0}).
\end{equation}
For the choice \eqref{eq:DU-FE} we abbreviate $w_{D, U}^-= w_y^-$. Using this, we introduce the point process 
	\begin{equation*}
		\omega_{y}^{-} = 
		\textstyle\sum_{(\widehat w^\ast, v) \in \omega} \delta_{(w_y^{ -}, v)}
	\end{equation*}
 (see 
\S\ref{subsec:occupation} regarding $\omega$), where $w_y^-$ is uniquely defined as a function of $ \widehat w^\ast$ in view of 
our previous observation. 
Then, abbreviating $\mathcal I = \mathcal I(Z_J)$, $\mathcal V = \mathcal V(Z_J)$ and $D= 
\mathring D_{y, L_0} $, as above, we claim that 
	\begin{equation}\label{eq:measurability}
		\begin{split}
			&\big({(\mathcal V^{u})_{\delta}}_{|D^c}, \, {(\mathcal V^{u'})_{2\delta}}_{|D^c},  \,
			N_{z, N}^u, \, \mathcal I_{|D^c}, \,
			\{\ell_{x}^u : x \in D^c\} \big) \text{ belong to }
			\widehat{\mathcal F} \coloneqq\sigma\big(\omega_y^-, \{\mathsf U_x : x \in D^c\}\big)
		\end{split}
	\end{equation}
	(see~\eqref{def:FEsigma_algebra1} to compare with the definition of $\mathcal F_{y, L_0}(Z_J, \delta, u, u')$). We now explain \eqref{eq:measurability}. Clearly, the discrete occupation times $\{\ell_{x}^{u} : x \in D^c\}$ are measurable relative to 
	$\omega_{y}^{-}$. Since 
	$ D_{y, L_0} \subset  D_{z, N}$, 
	$K \ge 100$ and $N \ge 10^3L_0$, it follows from the definitions of the relevant boxes in \eqref{def:CDU} that 
	  $U = D = \mathring D_{y, L_0} \subset  U_{z, N}$ and that 
no excursion between $D$ and $\partial^{{\rm out}}U$ can intersect $ U_{z, N}^c$. This 
implies that for any $v > 0$, the (finite, possibly empty) set of labels $\{v_1 < \ldots < v_k\} 
\subset (0, v)$ of any $(\widehat w, v') \in \mu_{ D_{z,N}}(\omega)$ with label at most $v$ is measurable relative to 
$\omega_y^-$. In particular, the random variable $N_{z, N}^u$ and the set of 
labels $\{v_j : j \in J\}$ for any and for any (random) $J$ measurable w.r.t.~$N_{z, N}^u$, 
are both measurable relative to $\omega_{y}^{-}$. Since $\omega_y^-$ retains all pieces of 
trajectories in the support of 
$\omega$ outside $D$, it is also clear from the definitions of the sets $\mathcal{I}$ 
and $(\mathcal V^v)_{\delta'}$ that $\mathcal I_{|D^c}$ is measurable relative to 
$\omega_y^-$ whereas ${(\mathcal V^v)_{\delta'}}_{|D^c}$ is measurable with respect to 
	$\big(\omega_y^-, \{\mathsf U_x : x \in D^c\}\big)$ for any $v \ge 0$ and $\delta' \in [0, 1)$. Overall, \eqref{eq:measurability} follows.

We denote by  $\Gamma^{v}$ the {\em multiset} of all pairs of points $(w(R_k(w) 
-1), w(T_k(w))) \in (\partial D_{y, L_0})^2; k \ge 1$  such that $(\widehat w, v') \in \omega$ for 
some $v' \le v$, i.e. we take into account the number of times each such pair appears over all 
pairs $(\widehat w, v')$. 	It is clear that $\Gamma^{v}=\Gamma^{v}(\omega_y^-)$ is indeed a (measurable) 
function of $\omega_y^-$. 
In analogous manner, we define $\Gamma_{J}=\Gamma_{J}(\omega_y^-)$ for (finite) $J \subset \N^\ast$ measurable w.r.t. $N_{z, N}$ by restricting $v'$ to the set $\{v_j : j \in J\}$,  where $v_j$ are the ordered labels of trajectories in the support of $\mu_{ D_{z,N}}(\omega)$. 

	With $\widehat{\mathcal{F}}$ from \eqref{eq:measurability}, 
	by the Markov property of random walks 
	and the definition of the excursions $Z_{j}^{ D_{z, N},  U_{z, N}}$, 
	we have the following 
	description of the law of excursions of 
	trajectories underlying $\mu_{ D_{z,N}}(\omega)$ as well as the variables 
	$\{\mathsf U_x: x \in D\}$ {conditionally} on 
	$\widehat{\mathcal{F}}$:
\begin{equation}\label{eq:conditional_law}
\text{\begin{minipage}{0.90\textwidth} 
 $\{\bm \gamma_{x, x'}: (x, x') \in 
\Gamma^{u'}(\omega_y^-)\}$ under 
$\P\big[ \cdot \,|\, 
\widehat{\mathcal F}\big]$ are distributed as independent random walk bridges where $\bm 
\gamma_{x, x'}$ is conditioned to start at $x$, end at $x'$ and lie inside $D$ except at the final 
point, independently of $\{\mathsf U_x : x \in 
D\},$ which are i.i.d. uniform random variables. \end{minipage}}
\end{equation}
In view of \eqref{eq:measurability} and \eqref{def:FEsigma_algebra1}, we have
$\mathcal F = \sigma(\widehat{\mathcal F}, (\mathcal V^{u})_\delta \cap D , (\mathcal V^{u'})_{2\delta} \cap D, \mathscr C_{\partial D_{y, L_0}}  (Z_J) )$. 
Since the last three of these random objects take values in a finite set, given any realization $\zeta_y$ of $(\omega_y^-, \{\mathsf U_x : x \in D^c\})$ as well as possible 
realizations $\eta, \eta'$ and $\xi$ of $(\mathcal V^{u})_\delta \cap D$, $(\mathcal V^{u'})_{2\delta} \cap D$ and $\mathscr C = \mathscr C_{\partial 
D_{y, L_0}}(Z_J))$ resp., we can write a regular conditional law $\P[ \cdot \,|\, \mathcal  F](\zeta_y, \eta, \eta', \xi)=  \mathbb Q_{\zeta_y}[ \, \cdot \, | \, \mathsf V(\eta, \eta') \cap \mathsf 
C(\xi) ]$ where $\frac 00$ is interpreted as $0$, $\mathbb Q_{\zeta_y}$ is the conditional law $\P\big[ \cdot \,|\, \widehat{\mathcal F}\big]$ described in 
\eqref{eq:conditional_law} evaluated at $\zeta_y \coloneqq (\omega_y^-, 
	\{\mathsf U_x : x \in D^c\})$, and 
$\mathsf V(\eta, \eta') =
			\{(\mathcal V^{u})_\delta \cap D= \eta, (\mathcal V^{u'})_{2\delta} \cap D = \eta'\}$, $\mathsf C(\xi) = \{\mathscr C 
			= \xi \}$.
				With this, \eqref{eq:conditional_prob1} follows if 
we prove that $\mathbb Q_{\zeta_y}\left[C_{y, L_0} \subset \mathcal V 
		\, | \, \mathsf V(\eta, \eta') \cap \mathsf C(\xi) \right] \ge c(\delta, L_0)$ holds	for each $(\eta, \eta', \xi)$, almost every $\zeta_y$ such that $\mathbb 
	Q_{\zeta_y}\big[\mathsf V(\eta, \eta') \cap \mathsf C(\xi)\big] > 0$ and $(\zeta_y, \eta, \eta', 
	\xi)$ belonging to the event $G$. We prove an even stronger inequality. 
	Given any collection $\{\gamma_{x, x'} : (x, x') \in \Gamma^{u'} 
	\setminus \Gamma_J\}$, where each $\gamma_{x, x'}$ is an excursion 
	between $D$ and $\partial^{{\rm out}} U$ (with $D,U$ as in \eqref{eq:DU-FE}) starting and ending at $x$ and $x'$ respectively, we show that
		\begin{equation}\label{eq:conditional_lower_bnd}
		\mathbb Q_{\tilde \zeta_y}\left[C_{y, L_0} \subset \mathcal V 
		\, | \, \mathsf V(\eta, \eta') \cap \mathsf C(\xi)\right] \ge 
		c(\delta, L_0) \, ( >0),\end{equation}
for almost every $\tilde \zeta_y \coloneqq (\zeta_y, \big\{ \bm \gamma_{x, x'}: (x, 
x') \in \Gamma^{u'} \setminus \Gamma_J\big\})$ and all $(\eta, \eta', 
\xi)$ 
such that $(\zeta_y, \eta, \eta', \xi)$ belongs to the event $G$ and $\mathbb Q_{\tilde \zeta_y}[\mathsf V(\eta, \eta') \cap \mathsf C(\xi) 
] > 0$.
The desired lower bound follows immediately from \eqref{eq:conditional_lower_bnd} 
by integrating the latter over all 
realizations of $\{\bm \gamma_{x, x'} : (x, x') \in \Gamma^{u'} \setminus 
\Gamma_J\}$. 
	

The remainder of the proof 
is devoted to showing \eqref{eq:conditional_lower_bnd}. Since 
$\zeta_y$ (and hence $\tilde \zeta_y$) satisfies the event $\{J \subset [1, N_{z, N}^u]\}$ and $u \le u'$, 
the events $\mathsf V(\eta, \eta')$ and $\mathsf C(\xi)$ are measurable relative to the 
excursions 
$\{\bm\gamma_{x, x'}: (x, x') \in \Gamma_J\}$ and the noise variables 
$(\mathsf U_x)_{x\in D}$ 
{given} $\tilde \zeta_y$. 
By the definition of conditional probability, there exist choices of excursions 
$\{\gamma_{x, x'}: (x, x') \in 
\Gamma_{J}\}$  
and occupation configurations $(b_x)_{x\in D}, (b_x')_{x\in D}  \in \{0, 1\}^{D}$ for almost 
all realizations of $\tilde \zeta_y$ with $\mathbb Q_{\tilde \zeta_y}[\mathsf V(\eta, \eta') \cap \mathsf C(\xi) 
] > 0$ such that
\begin{equation*}
\textstyle	\bigcap_{(x, x') \in \Gamma
		_J}\{\bm \gamma_{x, x'} = \gamma_{x, x'}\} \, \cap \, 
	\bigcap_{x \in D} \big(\left\{1_{\{\mathsf U_x \ge \delta\}} = b_x\right\}  \cap \left\{1_{\{\mathsf U_x \ge 2\delta\}} = b_x'\right\} \big) 
\end{equation*} 
is contained in $\mathsf V(\eta, \eta') \cap \mathsf C(\xi)$, 
 and the event in the display has positive $\mathbb Q_{\tilde \zeta_y}$-probability.  
We can 
specify the values of $b_x$ and $b_x'$, which are informed only by $\eta, \eta'$, 
using the 
properties of the noised sets in \eqref{def:noised_set}, 
whereby the occupied vertices are 
{explained} by triggering suitable noise variables: letting 
\begin{equation}\label{def:mathcalU}
\mathcal U(\eta, \eta') 
\coloneqq \textstyle \bigcap_{x \in 
	D \setminus \eta
}\{ \mathsf U_x < \delta\} \, \cap \, \bigcap_{x \in \eta \setminus \eta'}\{ \mathsf U_x \in  [\delta, 
2\delta)\} \, \cap \, \bigcap_{x \in \eta' }\{ \mathsf U_x \ge 2\delta\},
\end{equation}
we have
$\bigcap_{(x,x')}\{\bm \gamma_{x, x'} = \gamma_{x, x'}\}  \cap 
\mathcal U(\eta, \eta') \subset {\mathsf V}(\eta, \eta') \cap \mathsf C(\xi)$
(
given $\tilde \zeta_y$); here and in the sequel, any unspecified union or intersection is over $(x, x') \in \Gamma_J$. In fact, as we now briefly explain, the same argument yields that for {\em any} choice of 
excursions $\{\bar \gamma_{x, x'} : (x, x') \in \Gamma_J\}$ with 
\begin{equation}\label{eq:gamma_bargamma}
\begin{split}
& \textstyle\bigcup_{(x,x')} {\rm range}(\bar \gamma_{x, x'}) \subset \bigcup_{(x, 
x') } {\rm range}( \gamma_{x, x'}), \text{ and } 
\bigcap_{(x,x')}\{\bm \gamma_{x, x'} = \bar \gamma_{x, x'}\} \subset \mathsf C(\xi),
\end{split}
\end{equation}
one has
\begin{equation}\label{eq:mathcalU_include}
\textstyle\bigcap_{(x,x')}\{\bm \gamma_{x, x'} = \bar \gamma_{x, x'}\}  \cap 
\mathcal U(\eta, \eta') \subset {\mathsf V}(\eta, \eta') \cap \mathsf C(\xi).
\end{equation}
To see \eqref{eq:mathcalU_include}, note that the inclusion in \eqref{eq:gamma_bargamma} entails that the choice of $\bar \gamma_{x, x'}$ over $ \gamma_{x, x'}$ can only increase the vacant sets in the event $\mathsf V(\eta, \eta')$, but the occurrence of 
$\mathcal U(\eta, \eta')$ precludes this on account of  \eqref{def:noised_set}.

Our goal in the rest of the 
proof to reroute the trajectories $\gamma_{x, x'}$ into suitably chosen $\bar \gamma_{x, x'}$ so 
that \eqref{eq:gamma_bargamma} holds and $C_{y, L_0} \subset \mathcal V$. To this end, let us call $\gamma_{x, x'}$ \emph{crossing} if it intersects $C_{y, L_0}$. If none of the excursions in $\{\gamma_{x, x'}: (x, x') \in \Gamma_{J}\}$ (as in the display above \eqref{def:mathcalU}) 
is crossing, we have
$\bigcap_{(x, x') \in \Gamma
	_J}\{\bm \gamma_{x, x'} = \gamma_{x, x'}\} 
\subset \{C_{y, L_0} \subset \mathcal V 
 \}$
	given $\zeta_y$. 
	Together with 
\eqref{eq:gamma_bargamma} and \eqref{eq:mathcalU_include}, this implies 
\eqref{eq:conditional_lower_bnd} in this case. So  suppose that at least one $\gamma_{x, x'}$ with $(x, x') \in 
	{\Gamma}_J$ is crossing. Our strategy is to reroute the crossing excursions 
	into non-crossing ones, thereby vacating $C_{y,L_0}$, all the while explaining the events 
	$\mathsf V(\eta, \eta'), \mathsf C(\xi)$ as well as the configuration $\tilde \zeta_y$. 
		
	 Recall from \eqref{def:C_partial1} that $\xi$, the realisation of $\mathscr C= \mathscr C_{\partial D_{y, L_0}}$,   
	 is a disjoint union of connected subsets of $D_{y, L_0} \setminus C_{y, L_0}$ each of which 
	 intersects $\partial D_{y, L_0}$. Since $(\tilde \zeta_y, \eta, \eta', \xi)$ satisfies 
	 $\widetilde {\rm LU}_{y, L_0}$, it follows from~\eqref{def:tildeLU1} that there exists exactly 
	 one component of $\xi \setminus \partial D_{y, L_0}$ intersecting $\tilde D_{y, L_0}$, 
	 say $\mathcal C(\xi)$, which also contains $\xi \cap (\tilde D_{y, L_0} \setminus \tilde C_{y, 
	 L_0})$. 
	 But because any crossing excursion $\gamma_{x, x'}$ with $(x, x') \in 
	{\Gamma}_J$ belongs to a component of $\mathscr C$ 
	that also intersects $\tilde D_{y, L_0}$, it follows that all such crossing excursions are 
	part of $\mathcal C(\xi)$. Based on this observation and 
	using that the event $\widetilde {\rm LU}_{y, L_0}$ from \eqref{def:tildeLU1} is in force, we 
	can for each crossing excursion $\gamma_{x, x'}$ with $(x, x') \in 
	{\Gamma}_J$ find a non-crossing excursion $\bar \gamma_{x, x'}$ having the 
	same endpoints and such that 
\begin{equation}\label{eq:gamma-bar}
\text{$|\bar \gamma_{x, x'}| \le (10 L_0)^d$ and $\mathcal C(
\xi) = {\rm range}(\bar \gamma_{x, x'}) \cap D$,}
\end{equation}
simply by making $\bar \gamma_{x, x'}$ exhaust all of $\mathcal C(\xi)$ while reaching the 
desired endpoints. 
 Any other component of $\xi \setminus \partial D_{y, L_0}$, i.e. any component 
$\mathcal C$ that does not intersect $\tilde D_{y, L_0}$ necessarily satisfies $\mathcal C \subset (D \setminus \tilde D_{y, L_0}) \cap \bigcup_{(x, x')} {\rm range}(\gamma_{x, x'})$, where the union ranges over $(x, x') \in \Gamma_{J}$ such that  $\gamma_{x, x'}$ is non-crossing. Define $\bar \gamma_{x, x'}=  \gamma_{x, x'}$ for such $(x,x')$. 
Putting the previous observation together with the second item in \eqref{eq:gamma-bar}, and 
using the fact that all excursions $\bar \gamma_{x, x'}$ are non-crossing, it follows that (given 
$\tilde \zeta_y$)
\begin{equation*}
		\begin{split}
			&\textstyle\bigcap_{(x, x')}\{\bm \gamma_{x, x'} = \bar \gamma_{x, x'}\}
			\subset \big( \mathsf C(\xi) \cap  \{C_{y, L_0} \subset \mathcal V
			\} \big), \text{ and } \bigcup_{(x, x')} {\rm range}(\bar \gamma_{x, x'}) \subset \bigcup_{(x, x')} {\rm range}( \gamma_{x, x'}).
		\end{split}	
	\end{equation*}
By combining \eqref{eq:gamma_bargamma}, \eqref{eq:mathcalU_include}  and the previous display, and using that $\mathbb Q_{\tilde \zeta_y}[\mathcal U(\eta, \eta')] \ge c'= c'(\delta, L_0)$, see \eqref{def:mathcalU} and \eqref{eq:conditional_law}, 
we deduce that 
$\mathbb Q_{\tilde \zeta_y}\left[\{ C_{y, L_0} \subset \mathcal V\} 
			\cap \mathsf V(\eta, \eta') \cap \mathsf C(\xi)\right]$ is bounded from below by	
			\begin{equation*}
				\mathbb Q_{\tilde \zeta_y}\big[\textstyle\bigcap_{(x, x') }\{\bm \gamma_{x, x'} = 
			\bar \gamma_{x, x'}\} \, \cap \, \mathcal U(\eta, \eta')  \big]
			\stackrel{\eqref{eq:conditional_law}, \eqref{def:mathcalU}}{\ge}  c' \cdot\mathbb Q_{\tilde \zeta_y}\big[\bigcap_{(x, x') }\{\bm \gamma_{x, x'} = \bar \gamma_{x, x'}\}\big] \,
			 \stackrel{\eqref{eq:conditional_law}}{\ge} c'\left({2d}\right)^{-CL_0^{2d}},
	\end{equation*}
where the last step also uses the bound from \eqref{eq:gamma-bar} and the fact that $|\Gamma_{J}|	\le |\Gamma^u(\omega_y^-)|	 \le C L_0^{d-1}\cdot L_0$, since 
$\zeta_y$ satisfies ${\rm O}_{y, L_0}(\overline{Z}_{z, N}^u)$ (recall \eqref{def:Conn0}) and 
$\{J \subset [1, N_{z, N}^u]\}$. Overall, this yields \eqref{eq:conditional_lower_bnd}.
\end{proof}

\section{The observable $h^u$ and  coarse-graining 
}
\label{sec:cg}
We now lay the foundation for the upper bounds in the upcoming sections.  In \S\ref{sec:hu}, we introduce the scalar random variable $h^u$, see \eqref{def:Xu}, attached to the interlacement in a system $\Sigma$ of well-separated boxes. 
In Proposition~\ref{prop:Xutail_bnd}, we collect bounds on the probabilities that $h^u$ is atypical 
in terms of ${\rm cap}(\Sigma)$ 
and quantitative in $u$. The system $\Sigma$ of boxes will arise from a coarse-graining scheme presented in 
\S\ref{subsec:admissible}. 
The coarse-graining leads to a certain `good' event $\mathscr{G}$, introduced in \eqref{def:script_G}. 
As will be apparent in Section~\ref{sec:coarse_graining}, $\mathscr{G}$ will be used to propagate 
certain bounds from a (base) scale $L$ to scale $N \gg L$. The event $\mathscr{G}$ is 
sufficiently generic to fit all our purposes. 
The main result then comes in 
\S\ref{subsec:bootsbounds}, see
Proposition~\ref{prop:bootstrap_prob}. It yields a deviation estimate on the probability of $\mathscr{G}^c$ involving the aforementioned bounds on $h^u$. The proof of Proposition~\ref{prop:bootstrap_prob} could be omitted at first reading.


Our framework involves two parameters, a length scale 
$L \geq 1$ and a rescaling parameter $K \geq 100$, both integers. 
The scale $L$ induces the renormalized lattice $\mathbb L \coloneqq L \Z^d$ 
and we 
consider the boxes $C_z \subset \tilde{C}_z \subset \tilde{D}_z \subset D_z  \subset 
{U}_z$ \sloppy attached to points $z \in \mathbb L$ (or $\Z^d$) as defined in \eqref{def:CDU}. 
Now, let \begin{equation}\label{def:C}
\begin{split}
&\text{$\mathcal{C}\subset\mathbb L$ be a non-empty collection of sites with mutual 
$|\cdot|_{\infty}$-distance at least {
	$10KL$}}
	\end{split}
\end{equation}
and define
{\color{black}\begin{equation}\label{def:Sigma}
		\Sigma = \Sigma(\mathcal C) = \textstyle \bigcup_{z \in \mathcal C}  D_z.
\end{equation}}
In view of \eqref{def:C}, the parameter $K$ controls the separation between boxes 
{\color{black}$ D_z$} comprising $\Sigma$ in \eqref{def:Sigma}. 

\subsection{Deviation estimate for $h^u$} \label{sec:hu}
We now introduce a scalar random variable $h^u=h^u(\mathcal{C})$ that will play a central role in the sequel.  Consider the function $V$ on $\Z^d$ defined as (cf.~\eqref{eq:equilib_K} for notation)
{\color{black}\begin{equation}\label{def:V}
V(x) =\textstyle \sum_{ D \in \mathcal C} e_{\Sigma}( D) \bar e_{ D}(x), \quad x \in 
\Z^d,
\end{equation}}
where  
{\color{black}$e_{\Sigma}( D) = 
\sum_{y \in  D} e_{\Sigma}(y)$}  
and the sum ranges over all 
{\color{black}$ D$} such that 
{\color{black}$ D =  D_z$} for some  $z \in \mathcal C$. Notice that 
$V$ in \eqref{def:V} has the same support as $e_{\Sigma}$. 
Moreover, $V$ is well-approximated by $e_{\Sigma}$ as $K$ becomes large, in the sense that or all $\varepsilon \in (0, 1)$, $L \ge 1$, 
$K \ge \tfrac{\Cr{c:equil}}{\varepsilon}$ and 
$\mathcal C$ as in \eqref{def:C}, one has (pointwise on $\Z^d$)
\begin{equation}\label{eq:compareVe}
(1 - \varepsilon) e_\Sigma \le V \le (1 + \varepsilon) e_\Sigma.
\end{equation}
The proof of \eqref{eq:compareVe} follows similarly as in \cite[Prop.~4.1]{MR3602841} using Proposition~\ref{prop:entrance_time_afar}.
Now define 
(cf.~\eqref{eq:can-pair})
\begin{equation}\label{def:Xu}
h^u = h^u( \mathcal C) \coloneqq \left\langle \mu_{\Sigma, u}, \textstyle\int_{0}^{\infty} V(X_s) ds\right\rangle
\end{equation}
with $\Sigma=\Sigma(\mathcal{C})$ as in \eqref{def:Sigma}, and where $\int_{0}^{\infty} V(X_s) 
ds$ is short for the map $\widehat{w} \mapsto  \int_{0}^{\infty} V(X_s(\widehat w)) ds$ 
($\widehat{w} \in \widehat{W}_+ $). In case $\mathcal{C}=\{z\}$ is a singleton, we write 
$h^u(z)= h^u(\{z\})$. A quantity akin to \eqref{def:Xu} was already implicit in the work 
\cite{MR3602841}. 
Our presentation is somewhat streamlined 
and it includes two-sided estimates, which is intimately related to the non-monotone nature 
of the events we consider (contrary to the disconnection events in 
\cite{MR3602841}). 
The following result prepares the ground for tail bounds on $h^u$.

\begin{lemma}\label{lem:laplace_trans}
For all $u > 0$,   $a < 1$, $L \ge 1$ and $\mathcal C$ as in \eqref{def:C}, one has that (cf.~\eqref{eq:can-pair} for notation)
\begin{equation}\label{eq:laplace_transforma>0}
\textstyle \E \left [\exp\left(a \, \left\langle \mu_{\Sigma, u}, \int_{0}^{\infty} e_\Sigma(X_s) ds \right \rangle\right) \right] = \exp \big(\frac{u\,a\,{\rm cap}(\Sigma)}{1-a}\big).
\end{equation}
\end{lemma}
\begin{proof}
This follows from an application of \cite[display~(2.5)]{SznitmanRIGFF12} combined with 
\cite[Lemma~2.1]{MR3602841}. 
\end{proof}
Following are the announced deviation estimates for $h^u$ in \eqref{def:Xu}.
\begin{proposition}\label{prop:Xutail_bnd}
Let  $\varepsilon \in (0, 1)$ and $0 < \frac{u_{-}}{1 - \varepsilon} < u < \frac{u_{+}}{1 + \varepsilon}$. 
Then, for any $K \ge \tfrac{\Cr{c:equil}}{\varepsilon}$, $L \ge 1$, and $\mathcal C$ as in 
\eqref{def:C}, one has that
\begin{equation}\label{eq:Xutail_bndlow}
\textstyle \P [\pm h^u \ge \pm u_{\pm} \,{\rm cap}(\Sigma)] \le \exp \Big( - \left(\sqrt{\frac{u_{\pm}}{1 \pm \varepsilon}} - \sqrt{u}\right)^2{\rm 
cap}(\Sigma)\Big).
\end{equation}
\end{proposition}
\begin{proof}
\eqref{eq:Xutail_bndlow} follows by applying a Chernoff-type bound to $h^u$ using  Lemma~\ref{lem:laplace_trans} and \eqref{eq:compareVe}.
\end{proof}

\subsection{Admissible coarsenings}\label{subsec:admissible} 

We now introduce more precisely the collections $\mathcal{C}$ satisfying \eqref{def:C} that will be of interest.
Below, we write $B^2_r(x) \subset \Z^d $ for the closed 
$\ell^2$-ball of radius $r\geq 0$ around $x \in \Z^d$, and $B_r(x)$ 
for the corresponding $\ell^\infty$-ball. We abbreviate $B^2_r=B^2_r(0)$ and $B_r=B_r(0)$. 

For $U\subset V \subset \subset \Z^d$ where $V$ is simply connected in $\Z^d$, we say that a path $\gamma$ in $\Z^d$ 
\textit{crosses} $V\setminus U$, or that $\gamma$ is a 
{\em crossing} of $V \setminus U$, if it intersects both $U$ and 
$\partial V$. If $U=\{0\}$, we omit the reference to $U$; e.g.~when $\gamma$ crosses $B_r^2$ we 
mean that $\gamma$ intersects both $0$ and $\partial B_r^2$. In what follows, we always tacitly assume that $\Lambda_N \subset \Z^d$ is of the form (see \eqref{def:CDU} for 
notation) 
\begin{equation}
\label{eq:scriptS_N}
\Lambda_N \in \text{$\mathcal{S}_N \stackrel{{\rm def}.}{=} \{ B_N^2,  \, B_N^2 \setminus B_{\sigma N}^2$,   $\sigma \in (0,\tfrac13)$, $B_{2N}^2 \setminus  B_N^2,$ $ \tilde{D}_{0,N}\setminus 
\tilde{C}_{0,N}\}$.}
\end{equation}
To allow for a uniform presentation it will be convenient to define 
$\sigma=\sigma(\Lambda_N)$ for all choices in \eqref{eq:scriptS_N} by setting $\sigma(B_N^2 
\setminus B_{\sigma N}^2)=\sigma$ and $\sigma(\Lambda_N)=0$ otherwise, so that $\sigma \in [0,\tfrac13)$ for any choice of $\Lambda_N$. 
Borrowing the notion of coarsenings from \cite[Definition~4.2]{gosrodsev2021radius} 
we consider collections $\mathcal{C}$ that are well-behaved with respect to a {given} entropy rate $\Gamma$. Let $\Gamma: [1, \infty) \mapsto [0, \infty)$ be increasing and $a \in (0, 1)$. For $L \geq 1$, $K \geq 100$ and $N \geq h(KL)$, where
$h(x)=x(1 + (\log x)^21_{d \geq 4})$, a  family 
$\mathcal{A} = \mathcal{A}_{L}^K(\Lambda_N)$ of collections $\mathcal{C} \subset \mathbb{L}$ satisfying \eqref{def:C}
is {\em$(a, \Gamma)$-admissible} if, 
\begin{align}
&\text{$\log |\mathcal{A}| \leq \Gamma(N/L),$} \label{def:coarse_admissible3}\\
&\text{${D}_z={D}_{z,L} \subset 
\Lambda_N$ for all $z \in \mathcal{C}$,} \label{def:coarse_admissible0}\\ 
&
\text{\begin{minipage}{0.9\textwidth} all $\mathcal C \in \mathcal A$ have equal cardinality $n=|\mathcal{C}|$, which lies in the interval $\big[\tfrac{a 
			{(1-\sigma)} N}{h(KL)}, \tfrac{
			{(1-\sigma)} N}{ h(KL)}\big]$, and\end{minipage}}
\label{def:coarse_admissible1}\\
	&
		\text{\begin{minipage}{0.9\textwidth}for any 
			crossing $\gamma$ of $\Lambda_N$, there exists $\mathcal{C} \in 
		\mathcal{A}$ such that $\gamma$ crosses 
		${D}_{z}\setminus C_z$ for all $z \in \mathcal{C}$.\end{minipage}}
	\label{def:coarse_admissible2}
	\end{align}
We are now ready to state our result on the existence of coarsenings with good capacity bound. 
In the sequel, we let $T_N$ denote the line segment $([0, 
N] \cap \Z) \times \{0\}^{d-1}$.
\begin{proposition}[Admissible coarsenings]\label{prop:coarse_paths} There exist $\Cr{C:cg_complexity} \in [1,\infty)$ and $\Cl[c]{c:nLB} \in (0, 
	{\frac{1}{100{d}}})$ such that, for all 
	$K \geq 100$, $L \geq 1$, 
	$N \geq {\color{black}\Cr{c:nLB}^{-1} \, h(KL)}$ and $\Lambda_N \in \mathcal{S}_N$ (see 
	\eqref{eq:scriptS_N}), there exists a 
	$(\Cr{c:nLB}, \Gamma)$-admissible collection $\mathcal{A}=\mathcal{A}_{L}^K(\Lambda_N)$ 
	with the following properties. 
\begin{itemize}
\item[i)] If $d=3$, one has for all $\rho \in (0,1)$ and with $\Gamma(x)= {\Cl{C:cg_complexity}}{K}^{-1} x \log ex$,
\begin{align}
\label{eq:coarse_capacity_d=3}
	\liminf_{ N \to \infty}\; \inf_{\substack{ K \in [K_-, K_+],\\ L \in [L_-, L_+]}}  \; 
		&\inf_{\mathcal{C}\in \mathcal{A}} \; \inf_{\substack{\tilde{\mathcal{C}}\subset \mathcal{C} \\ |\tilde{\mathcal{C}}|\geq (1-\rho)|\mathcal{C}|}} \; \frac{\mathrm{cap}(\Sigma(\tilde{\mathcal{C}}))}{\big(1 - \tfrac {\Cl{C:Kinv}}{K}\big)\mathrm{cap}\big(T_{ 
				{(1-\sigma)} N}\big)} \geq  (1 - \rho),
	\end{align}
where $\Cr{C:Kinv} \in [200, \infty)$, $\Sigma(\tilde{\mathcal{C}})$ is as in \eqref{def:C}
and $K_\pm =K_{\pm}(N)$, $ L_{\pm}= L_{\pm}(N)$ satisfy
\begin{equation}\label{eq:quant_LK} 
\begin{split}
&\text{$K_-(N) = 100$, $\lim_N L_-(N) = \infty$, and }\\
&\text{$\lim_N \big(\tfrac{\log (K_+(N)L_+(N))}{\log 
N}\big)^{{1}/{K_+(N)}} = 0$,  \, 
$
{   L_+(N) \leq { \Cr{c:nLB} N/K_+(N)}}$.}
\end{split}
\end{equation}

\item[ii)] If $d \geq 4$, choosing instead $\Gamma(x) = {\Cr{C:cg_complexity}}\, x$, the bound \eqref{eq:coarse_capacity_d=3} remains valid with $L_-(N)=1$, any 
fixed $K \ge 100$, $L_+(N) =  \Cr{c:nLB} N/K$ and $\big(1 - \tfrac {\Cr{C:Kinv}}{K}\big)$ 
replaced by some $c(K) \in (0, 1]$.
\end{itemize}
\end{proposition}

Proposition~\ref{prop:coarse_paths} follows by adapting the arguments in the proof of \cite[Proposition~4.3]{gosrodsev2021radius}, extended to the case of Euclidean balls while keeping track of the quantitative dependence on $K \in [K_-, K_+]$ and $L \in [L_-, L_+]$. We omit further details. 
From here onwards, we refer to \emph{admissible}
 collections $\mathcal{C} \in \mathcal{A}=\mathcal{A}_{L}^K(\Lambda_N)$ without mention of $(a,\Gamma)$, which are set to $a=\Cr{c:nLB}$ and $\Gamma$ as supplied by Proposition~\ref{prop:coarse_paths}.

\subsection{The event $\mathscr{G}$} \label{subsec:bootsbounds}
In the sequel, we work under $\P$ (see the paragraph below 
Proposition~\ref{prop:entrance_time_afar}) and extensions thereof. The specification of the 
event $\mathscr{G}$ which plays a key role in our proofs involves two families of events 
$\mathcal{F} =\{ \mathcal{F}_{z,L}: z\in \mathbb{L}\}$ and $\mathcal{G} =\{ 
\mathcal{G}_{z,L}: z\in \mathbb{L}\}$. 
Whereas  $\mathcal F_{z,L}$ will be specified shortly, see \eqref{eq:Fext}, the events $\mathcal{G}_{z,L}$ will be generic and sufficiently versatile to fit all our purposes. We comment further on the role of $\mathcal{F}$ and 
$\mathcal{G}$ in Remark~\ref{R:FvsG} at the end of this section. Given families 
$\mathcal{F},\mathcal{G} $ and for any $\rho \in (0,1)$ and $\Lambda_N \in  \mathcal{S}_N$ 
as in \eqref{eq:scriptS_N}, let 	\begin{equation}\label{def:script_G}
\mathscr G= \mathscr G(\Lambda_N, \G, \mathcal{F};\rho) \stackrel{{\rm def.}}{=} 
		\bigcap_{\mathcal C \in \mathcal A} \, 
			\bigcup_{\substack{\tilde{\mathcal C} \subset \mathcal C, \\ |\tilde {\mathcal C}| \ge  \rho |\mathcal C|}} 
			\,\, \bigcap_{z \in \tilde{\mathcal C}} \,\, (\mathcal F_{z,L} \cap \mathcal{G}_{z, L} ),
			\end{equation}
where 
$\mathcal A= \mathcal A_{L}^{K}(\Lambda_N)$ is the family of admissible coarsenings supplied by 
Proposition~\ref{prop:coarse_paths}, which implicitly requires that $L \geq 1, K \geq 100$ and 
$N \geq \Cr{c:nLB}^{-1} \, h(KL)$.

 The (good) event $\mathscr{G}$ will typically be likely in upcoming applications. Following is an `umbrella bound' in this direction which subsumes all the events of our interest in this paper. We start by specifying the relevant events $\mathcal{F}=\{\mathcal{F}_{z,L}: z\in \mathbb{L}\}$. Let $k \geq 1$ and $u_i, v_i \in (0,\infty)$ with $u_i \neq v_i$ for all $1 \leq i \leq k$. The parameters $\bm u = (u_1, \ldots, u_k)$ and $\bm v = (v_1, \ldots, v_k)$ represent various interlacement levels that will be involved in our construction. 
Extending \eqref{eq:F}, let
\begin{equation}
\label{eq:Fext}
\mathcal{F}_{z,L} = \mathcal{F}_{z,L}^{\bm u, \bm v} \stackrel{{\rm def.}}{=} \bigcap_{1\leq i \leq k }\mathcal{F}_{z,L}^{ u_i,  v_i},
\end{equation}
so that $\mathcal{F}_{z,L}^{\bm u, \bm v} = \mathcal{F}_{z,L}^{ u_1,  v_1}$ in case $k=1$, 
i.e.~\eqref{eq:Fext} boils down to \eqref{eq:F} in this case. The events in \eqref{eq:Fext} 
comprise the family $\mathcal{F}= \mathcal{F}_L^{\bm u, \bm v} =\{ \mathcal{F}_{z,L}^{\bm u, \bm v} : z\in 
\mathbb{L}\}$.

As to the events forming the family $\mathcal{G}= \mathcal{G}_L= \{ \G_{z,L}: z \in 
\mathbb{L}\}$, we assume that there exists an event $\widetilde {\mathcal G}_{z, L}$ for each $z 
\in \mathbb L$ measurable relative to the 
i.i.d. excursions $\widetilde{Z}^{ D_{z},  U_{z}}= (\widetilde Z_k^{ D_{z}, 
 U_{z}})_{k \ge 1}$ 
governed by the law 
$\widetilde \P_{z} = \widetilde \P_{\{z\}}$ (see above \eqref{def:Uz} for notation) and an 
independent collection of i.i.d.~uniform random variables $\mathsf U = \{\mathsf U_x : x \in 
 D_{z,L}\}$, an integer $m_L > 0$ and $\varepsilon_L \in (0, 1)$ such that the inclusion
	\begin{equation}\label{eq:bootstrap_prob_coupling}
		\begin{split}
			&(\widetilde{\mathcal G}_{z, L} \cap {\rm Incl}_{z}^{\varepsilon_L, m_L}) \subset \mathcal G_{z, L} \mbox{  
				holds under any coupling $\mathbb Q$ of $\mathbb P$ and 
			$\widetilde {\mathbb P}_{z}$}
		\end{split}
	\end{equation}
	(recall the event ${\rm Incl}_{z}^{\varepsilon_L, m_L}$ from \eqref{eq:RI_basic_coupling2}). This condition on $\G$ depends implicitly on $K$ via $ U_z$; see \eqref{def:CDU}. 	 
\begin{proposition}[Estimate for $\mathscr{G}^c$]\label{prop:bootstrap_prob}
	For any choice of $\Lambda_N \in  \mathcal{S}_N$ (see \eqref{eq:scriptS_N}), 
	$\rho \in (0, 1]$, $ k \geq 1$
	and families $\mathcal{F}= \mathcal{F}^{\bm u, \bm v}_L$, $\G=\G_L$ as above, the following hold. If, for some $K_0, L_0 \geq 1$ and $\beta' \in (0, 1)$, one has $\sup_{z \in \mathbb L}\P\big[\widetilde{\mathcal G}^c_{z, L}\big] \le p_{L}$ for all $K \ge K_0$ and $L \geq L_0$ with 
	\begin{equation}\label{eq:initial_limit_prob}
		\sup_{L \geq L_0}  {L^{-\beta'}}
		\log \big( p_L \vee \P\big[ (\mathcal{U}_{0}^{\varepsilon_L, m_L})^c\big]\big)   
		 \leq -1	\end{equation}
	(see~\eqref{def:Uz} for notation), then:
	\begin{itemize}
	\item[i)] for $d = 3$, there exists $\alpha = \alpha(\beta') \in (0, \infty)$ such that with $L(N) = \lfloor (\log N)^{\alpha} \rfloor$,  
	one has for all $\delta \in (0,1)$ and $N \ge C(\beta',  
	\bm u,\bm v, k, \rho, K_0, L_0, \delta)$,
	\begin{multline}\label{eq:bootstrapped_limit_prob1d=3}
			 \sup_{ K \in [K_-, K_+]} \big(1 - \tfrac {\Cr{C:Kinv}}{K}\big)^{-1} {\log \P\big[(\mathscr G(\Lambda_N, \mathcal G_{L(N)}, \mathcal{F}^{\bm u, \bm v}_{L(N)}; \rho))^c\big]} \\
		\textstyle	\le - (1-\delta)(1 - \sigma)(1 - \Cl{C:rho}\rho)\frac{\pi}{3k}\big[\min_{1 \le i \le k}(\sqrt{u_i} - \sqrt{v_i})^2 \big]\frac N{(\log N)^{\beta}}
	\end{multline}
for some $\beta = \beta(\beta') \in (1, \infty)$ if $\beta' \le \frac12$ and $\beta = 1$ otherwise, where $K_- = \frac{3\Cr{c:equil}}{\varepsilon_{L(N)}} \vee C(\delta, \bm u,\bm v)\vee K_0$, $K_+= \sqrt{\log \log e^2N}$,  
$\Cr{c:equil}$ and 
$\Cr{C:Kinv}$ are 
from Props.~\ref{prop:entrance_time_afar} and \ref{prop:coarse_paths}, respectively, and $\Cr{C:rho} \in (1, \infty)$. 
\item[ii)] For $d \geq 4$ and $\varepsilon_L = \varepsilon \in (0, 1)$ uniformly in $L$, we have 
for any fixed $L \ge C(\bm u, \bm v, \varepsilon, L_0)$, 
$K = C({\bm u,\bm v}, \varepsilon) \vee K_0$  and $N 
\ge C({\bm u,\bm v}, \varepsilon, K_0, L_0)$,
\begin{equation}\label{eq:bootstrapped_limit_probd=4}
\begin{split}
N^{-1} \log  \P\big[(\mathscr G(\Lambda_N, \mathcal G_{L}, \mathcal{F}^{\bm u, \bm v}_{L}; \tfrac12))^c\big] & \le -c(\bm u, \bm v, \varepsilon, K_0, L_0) \, (< 0).
\end{split}
\end{equation}
	\end{itemize}
\end{proposition}

The proof of Proposition~\ref{prop:bootstrap_prob} will exhibit the observable $h^u(\mathcal{C})$ introduced in Section~\ref{sec:hu}, for certain subsets $\mathcal{C}$ of admissible collections (in $\mathcal{A}$). This is not obvious at all (the generic event $\mathscr{G}$ does not involve $h^u$) and will require some work. A key step is a certain dichotomy, see \eqref{eq:G-bound-pf6} below, which will make $h^u$ appear (cf.~the event $\widetilde{E}_{2.2}$ below). The crucial properties of collections in $\mathcal{A}$ gathered in Proposition~\ref{prop:coarse_paths} then come into play to produce the leading-order decay in \eqref{eq:bootstrapped_limit_prob1d=3} when combined with Proposition~\ref{prop:Xutail_bnd}, which in particular requires a lower bound on $\text{cap}(\Sigma)$ for $\Sigma=\Sigma(\mathcal{C})$. The bound \eqref{eq:coarse_capacity_d=3} thus ensures that the coarse-grained path $\mathcal{C}$ does not ``loose too much'' capacity.

\begin{proof} We treat both $d = 3$ and $d\geq4$ simultaneously. Assume that $K \geq 100$, $L \geq 1$, and
$N \geq \Cr{c:nLB}^{-1} \, h(KL)$ so 
that Proposition~\ref{prop:coarse_paths} is in force. In particular, this entails that
a $(\Cr{c:nLB}, \Gamma)$-admissible collection $\mathcal{A}=\mathcal{A}_{L}^K(\Lambda_N)$ with the properties listed in Proposition~\ref{prop:coarse_paths} exists, and $\mathscr{G}= \mathscr G(\Lambda_N, \mathcal G_{L}, \mathcal{F}^{\bm u, \bm v}_{L}; \rho) $ is well-defined. For all such $K,L,N$, applying a union bound over $\mathcal{C}$ in \eqref{def:script_G} and using \eqref{def:coarse_admissible3} yields that
\begin{equation}
\label{eq:G-bound-pf1}
\textstyle\log \P[\mathscr{G}^c] \leq \Gamma ({N}/{L}) + \sup_{\mathcal{C} \in \mathcal{A}} \log \P\big[ \big( \bigcup_{\tilde{\mathcal C}} 
			\,\, \bigcap_{z \in \tilde{\mathcal C}} \,\, ( \mathcal F_{z,L}^{\bm u, \bm v} \cap \mathcal{G}_{z, L})   \big)^c \big],
\end{equation}
where the union is over $\tilde{\mathcal C} \subset \mathcal C$ having cardinality $|\tilde {\mathcal C}| \ge  \rho |\mathcal C|$. In the sequel we always tacitly assume that $K,L,N$ satisfy the requirements above \eqref{eq:G-bound-pf1}. Additional conditions on any of these parameters will be mentioned explicitly. We will deal with the term $\Gamma(\cdot)$ at the end of the proof and focus on the probability appearing on the right-hand side of \eqref{eq:G-bound-pf1}, which is wherein the work relies. All subsequent considerations hold uniformly in $\mathcal{C} \in \mathcal{A}$. Let $z \in \mathbb{L}$ be a \emph{good} point if the event $\mathcal{G}_{z, L} \cap \mathcal F_{z,L}$ occurs, and \emph{bad} otherwise. For any collection $\mathcal{C} \in \mathcal{A}$, the event appearing on the right of \eqref{eq:G-bound-pf1} asserts that there is no sub-collection $\tilde{\mathcal{C}} \subset \mathcal{C}$ of cardinality at least $\rho |\mathcal C|$ consisting of good points only. Thus, on this event $\mathcal{C}$ contains at least $(1-\rho) |\mathcal C|$ bad points. It follows that for any $\mathcal{C} \in \mathcal{A}$, 
\begin{equation}
\label{eq:G-bound-pf2}\textstyle
\P\big[ \big( \bigcup_{\tilde{\mathcal C}} 
			\,\, \bigcap_{z \in \tilde{\mathcal C}} \,\, ( \mathcal F_{z,L}^{\bm u, \bm v} \cap \mathcal{G}_{z, L})  \big)^c \big] \leq  \mathbb{Q}_{\mathcal{C}}[E_1] +  \mathbb{Q}_{\mathcal{C}}[E_2], 
\end{equation}
where $ \mathbb{Q}_{\mathcal{C}}$ is the extension of $\P$ supplied by Lemma~\ref{L:RI_basic_coupling} and
\begin{align*}
&E_1= \big\{ \exists \mathcal C_1 \subset \mathcal C, \, |\mathcal C_1| \geq \rho |\mathcal C| : \, (\mathcal{G}_{z,L})^c \text{ occurs for all } z \in C_1  \big\},\\
&E_2= \big\{ \exists \mathcal C_2 \subset \mathcal C, \, |\mathcal C_2| \geq (1-2\rho) |\mathcal C| : \, (\mathcal{F}_{z,L}^{\bm u, \bm v})^c 
 \text{ occurs for all } z \in C_2  \big\}.
\end{align*}
We will bound the two probabilities on the right-hand side of \eqref{eq:G-bound-pf2} individually. We start by observing that the inclusion \eqref{eq:inclusion} obtained in Lemma~\ref{L:RI_basic_coupling} holds for the choices $\varepsilon= \varepsilon_L$ and $m_0=m_L$ whenever $ K \geq \Cr{c:equil}(\varepsilon_L)^{-1}$; indeed the relevant condition \eqref{eq:RI_cond_Q} holds on account of Proposition~\ref{prop:entrance_time_afar} (see~\eqref{eq:eAbar_condentrance}). Using the inclusion \eqref{eq:inclusion}, recalling that the events $\mathcal{U}_{z}^{\varepsilon_L, m_L}$ are independent as $z \in \mathbb{L}$ varies, and applying the relevant bound from \eqref{eq:initial_limit_prob}, it follows that for all $L \geq L_0$ and $K \geq K_0 \vee \Cr{c:equil}(\varepsilon_L)^{-1}$,
\begin{equation}
\label{eq:G-bound-pf2.1}
\mathbb{Q}_{\mathcal{C}}[ ( {\rm Incl}_{z}^{\varepsilon_L, m_L})^c,  z \in \mathcal{D}] \leq e^{-L^{\beta'} |\mathcal{D}|}, \, \text{for all } \mathcal{D}\subset \mathbb{L}.
\end{equation}
To bound $\mathbb{Q}_{\mathcal{C}}[E_1]$, one then proceeds as follows. First one applies a union bound over the choice of $\mathcal{C}_1$, using the elementary estimate ${n\choose k} \leq (\frac{en}{k})^k$ for all $1 \leq k \leq n$ (implied by the bound $\frac{k^k}{k!} \leq e^k$), applied with $n=|\mathcal{C}|$ and $k=\lfloor \rho n \rfloor$. Then one uses
the inclusion $( \mathcal G_{z, L})^c \subset (\widetilde{\mathcal G}_{z, L})^c \cup ( {\rm Incl}_{z}^{\varepsilon_L, m_L})^c $ implied by \eqref{eq:bootstrap_prob_coupling} together with a union bound, \eqref{eq:G-bound-pf2.1} and the control on the decay of $p_L$ from \eqref{eq:initial_limit_prob}. All in all, this yields, for $L \geq L_0 \vee C(\beta')$ and $K \geq K_0 \vee \Cr{c:equil}(\varepsilon_L)^{-1}$,
\begin{equation}
\label{eq:G-bound-pf3}
\mathbb{Q}_{\mathcal{C}}[E_1] 
\leq  \exp\big\{-\rho |\mathcal C| \big( c L^{\beta'} - C|\log \rho \,| \big) \big\}.
\end{equation}

The case of $E_2$ is more involved, and, as will turn out, produces the leading-order contribution to the right-hand side of \eqref{eq:G-bound-pf1}. We start by modifying the event $E_2$ to make it easier to handle. Recall from \eqref{eq:Fext} that $(\mathcal{F}_{z,L}^{\bm u, \bm v})^c$ involves a union over events at $k \geq 1$ different pairs of levels $(u_i,v_i)$, $1\leq i \leq k$. The collection $\mathcal C_2$ involved in $E_2$ must therefore contain a sub-collection of cardinality at least ${|\mathcal C_2|}/k = {(1-2\rho)|\mathcal C|}/k$ for which $(\mathcal{F}_{z,L}^{ u,  v})^c$ occurs for some $(u,v) \in \{ (u_i,v_i): \, 1\leq i \leq k \}$ (the choice of $(u,v)$ depends on $\mathcal C_2$ of course).
By further sacrificing a fraction ${\rho |\mathcal C|}/{k}$ from this new collection we may assume that for each  $z$, the event ${\rm Incl}_{z}^{\tilde \varepsilon, \tilde m}$ occurs, where $\tilde\varepsilon$, $\tilde m$ are chosen for given $\delta \in (0,1)$ as 
\begin{equation}
\label{eq:G-bound-pf4}
\begin{split}
\tilde \varepsilon = \tfrac{\delta}{100}  \min_{i}|u_i-v_i|
, \,  \tilde m= \tilde \varepsilon^{-3} \vee 2^{-1}\min \{ u_i, v_i, 1 \leq i \leq k\}\text{cap}( D_0).
\end{split}
\end{equation}
Thus, defining the event
$
\widetilde{E}_2(\mathcal{D})= \bigcap_{z \in \mathcal{D}} 
(\mathcal{F}_{z,L}^{ u,  v})^c \cap {\rm Incl}_{z}^{\tilde \varepsilon, \tilde m}
$
it follows from the above discussion by means of appropriate union bounds that
\begin{equation}
\label{eq:G-bound-pf5}
\mathbb{Q}_{\mathcal{C}}[E_2] \leq (C(\rho^{-1} \vee k))^{\rho |\mathcal{C}|}\big( \sup_{(u,v),\,  \mathcal{D}}  \mathbb{Q}_{\mathcal{C}}\big[\widetilde{E}_2(\mathcal{D})\big] +
e^{-c\tilde \varepsilon^{2}\tilde m \rho |\mathcal{C}|/k} \big)
\end{equation}
 for $K \geq \Cr{c:equil}\tilde\varepsilon^{-1}$, where the supremum ranges over all $k$ choices for $(u,v)$ and all subsets $\mathcal{D}\subset \mathcal C$ having cardinality $|\mathcal{D}| \geq (1-3\rho) |\mathcal C|/k$, and the last term in \eqref{eq:G-bound-pf5} is a bound for probability of jointly occurring events of type $({\rm Incl}_{z}^{\tilde \varepsilon, \tilde m})^c$, for $z$ ranging over a collection of cardinality ${\rho |\mathcal C|}/{k}$; this bound is obtained similarly as in \eqref{eq:G-bound-pf2.1}, exploiting \eqref{eq:inclusion}, using independence of $\mathcal{U}_{z}^{\tilde \varepsilon,\tilde m}$ over $z$ and applying \eqref{eq:bnd_Uzepm}.

It remains to deal with $\widetilde{E}_2(\mathcal{D})$, for $\mathcal{D}$ as above. To this effect, we introduce the following events. For a given $Z^u=(Z^u_k)_{1\leq i \leq n_Z}$ with $Z^u \in \{ \overline{Z}_z^u, {Z}_z^u, \widetilde{Z}_z^u \}$ for some $z \in \mathbb{L}$ (see \eqref{eq:3Vs} for notation), let
 \begin{equation}\label{def:Euv+} \textstyle
\mathcal E^{v}(Z^u) \stackrel{{\rm def.}}{=} \big\{\sum_{1\leq i \leq n_Z} \int_0^{T_{U}} e_{ D_z}\big(Z_i(s)\big)ds \leq v \, {\rm 
cap}( D_z) \big\} 
\end{equation}
 if $u \leq v$ and with opposite inequality when $u > v$. Notice in particular that $\mathcal E^{v}(\overline{Z}_z^u) =\{ h^u(z) \leq v \, {\rm cap}( D_z)\}$ when $u \leq v$ (and similarly when $u>v$) on account of \eqref{def:Xu} and the first line of \eqref{eq:3Vs}. The events $\mathcal E^{v}(Z^u)$ are defined in such a way that they are typical in practice, i.e.~likely to occur.
 
 Following is a crucial dichotomy, which brings into play deviations of the type considered in Proposition~\ref{prop:Xutail_bnd}. As we now argue, if $u \leq v$, we claim that 
  \begin{equation}
\label{eq:G-bound-pf6}
\widetilde{E}_2(\mathcal{D}) \subset \widetilde{E}_{2.1}(\mathcal{D}) \cup \widetilde{E}_{2.2}(\mathcal{D}),  
\end{equation}
for any collection $\mathcal{D}$ with $|\mathcal{D}| \geq (1-3\rho) |\mathcal C|/k$ and $v' \in (u,v) \, ( = (u \wedge v, u \vee v))$,
 where
\begin{align*}
& \widetilde{E}_{2.1}(\mathcal{D})\stackrel{\text{def.}}{=}\bigcup_{\substack{\mathcal{D}' \subset \mathcal{D}: \\ |\mathcal{D}'| \geq  \rho |\mathcal{C}|/k }}     \bigcap_{z\in  \mathcal{D}'}   (\mathcal E^{v'}(Z_z^v))^c,  \qquad \widetilde{E}_{2.2}(\mathcal{D}) \stackrel{\text{def.}}{=} \bigcup_{\substack{\mathcal{D}' \subset \mathcal{D}: \\ |\mathcal{D}'| \geq (1-4 \rho )|\mathcal{C}|/k }}   \big\{ h^u(\mathcal{D'}) \geq v' \text{cap}(\Sigma(\mathcal{D}')) \big\}.
\end{align*}
The inclusion \eqref{eq:G-bound-pf6} continues to hold in case  $u > v$ with our above convention on $\mathcal E^{v}(Z^u)$, but now  for any $v'\in (v,u)$ and provided one defines $\widetilde{E}_{2.2}(\mathcal{D})$ with the inequality reverted in case $u > v'$. 

Let us now explain \eqref{eq:G-bound-pf6}. We focus on the case $u<v$ for concreteness, the other case follows a similar reasoning. Suppose $\widetilde{E}_2(\mathcal{D})$ occurs but $\widetilde{E}_{2.1}(\mathcal{D})$ doesn't. Define $\mathcal{D}'\subset \mathcal{D}$ as the collection of $z \in \mathcal{D}$ such that $\mathcal E^{v'}(Z_z^v)$ occurs. We will show that with this choice of $\mathcal{D}'$, one has i) $|\mathcal{D}'| \geq (1-4 \rho )|\mathcal{C}|/k$, and ii) $h^u(\mathcal{D}') \geq v' \text{cap}(\Sigma(\mathcal{D'})$. Thus, $\widetilde{E}_{2.2}(\mathcal{D})  $ occurs and \eqref{eq:G-bound-pf6} follows. To see i), recall that $|\mathcal{D}| \geq (1-3 \rho )|\mathcal{C}|/k$ so if i) were not to hold then the set of points $z \in \mathcal{D}$ such that $(\mathcal E^{v'}(Z_z^v))^c$ occurs would have cardinality exceeding $ \rho |\mathcal{C}|/k$, implying $\widetilde{E}_{2.1}(\mathcal{D})$, which is precluded. To see ii), notice that by joint occurrence of $(\mathcal{F}_{z,L}^{ u,  v})^c$  (as implied by $\widetilde{E}_2(\mathcal{D})$) and of $\mathcal E^{v'}(Z_z^v)$ for each $z \in \mathcal{D}'$, one has, abbreviating $\Sigma= \Sigma(\mathcal{D}')=\bigcup_{z\in \mathcal{D}'}  D_z$ (see \eqref{def:Sigma} for notation), that
\begin{multline*}
h^u(\mathcal{D}') \stackrel{\eqref{def:Xu},\eqref{def:V}}{=}
\sum_{z \in \mathcal D'} \frac{e_{\Sigma}( D_z)}{\text{cap}( D_0)}\sum_{1\leq i \leq N_z^u} \int_0^{T_{U}} e_{ D_z}\big(Z^{ D_z, U_z}_i(s)\big)ds \\\stackrel{\eqref{eq:F}}{\geq} \sum_{z \in \mathcal D'} \frac{e_{\Sigma}( D_z)}{\text{cap}( D_0)}\sum_{1\leq i \leq v \text{cap}( D_0)} \int_0^{T_{U}} e_{ D_z}\big(Z^{ D_z, U_z}_i(s)\big)ds \stackrel{\eqref{def:Euv+}, \eqref{eq:3Vs}}{\geq} v' \text{cap}(\Sigma);
\end{multline*}
in the last line, when using occurrence of $\mathcal E^{v'}(Z_z^v)$, recall that $v'< v$ since we are in the case $u<v$, so the event corresponds to the one defined
below \eqref{def:Euv+} with opposite inequality.  Overall, ii) thus holds and the verification of \eqref{eq:G-bound-pf6} is complete.

We now use \eqref{eq:G-bound-pf6} to bound $\mathbb{Q}_{\mathcal{C}}[\widetilde{E}_2(\mathcal{D})]$ uniformly in $(u,v)$ and $\mathcal{D}$ as in \eqref{eq:G-bound-pf5}. From here on,
\begin{equation}
\label{eq:G-bound-pf7}
v' = v(1+ 3\tilde\varepsilon (1_{ u > v}-1_{ u < v})).
\end{equation}
(see \eqref{eq:G-bound-pf4} regarding $\tilde\varepsilon$).
For concreteness we assume again that $u< v$, so $v' =v(1- \tilde\varepsilon)$. We first deal with $ \widetilde{E}_2(\mathcal{D}) \cap \widetilde{E}_{2.1}(\mathcal{D})$, and aim to decouple the (unlikely) events $ (\mathcal E^{v'}(Z_z^v))^c$ as $z$ varies in $\mathcal{D}'\subset \mathcal{D}$. To this effect, we use the occurrence of ${\rm Incl}_{z}^{\tilde \varepsilon, \tilde m}$ implied by $\widetilde{E}_2(\mathcal{D})$ and a localization argument similar to the one below \eqref{eq:F}. 
By monotonicity of \eqref{def:Euv+} in $u$, recalling \eqref{eq:RI_basic_coupling2} and the choices of parameters in \eqref{eq:G-bound-pf4} and \eqref{eq:G-bound-pf7}, it follows that for $L \geq C(u,v)$, when $u <v$ the inclusion
\begin{equation}
\label{eq:G-bound-pf8}
\big({\rm Incl}_{z}^{\tilde \varepsilon, \tilde m} \cap (\mathcal E^{v'}(Z_z^v))^c \big) \subset (\mathcal E^{v'}(\widetilde{Z}_z^{v(1-\tilde\varepsilon)}))^c
\end{equation}
holds $\mathbb{Q}_{\mathcal{C}}$-a.s.~The events on the right-hand side of \eqref{eq:G-bound-pf8} are independent as $z$ varies by construction (see above \eqref{def:Uz}). For a single $z$, the probability of the event in question is best bounded by restituting $h^{v(1-2\tilde\varepsilon)}(z)$ from the functional entering $\mathcal E^{v'}(\widetilde{Z}_z^{v(1-\tilde\varepsilon)})$. This is achieved by de-localizing, i.e.~suitably coupling tilted with untilted trajectories and controlling the relevant number $N_{z}^{v(1-2\tilde\varepsilon)}$ to effectively replace $\widetilde{Z}_z^{v(1-\tilde\varepsilon)}$ by $\overline{Z}_z^{v(1-2\tilde\varepsilon)}$. It follows that for all $L \geq C(u,v)$, $K \geq C \tilde{\varepsilon}^{-1}$ and $z \in \mathbb{L}$ (when $u<v$),
\begin{equation*}
\mathbb{Q}_{\mathcal{C}}[(\mathcal E^{v'}(\widetilde{Z}_z^{v(1-\tilde\varepsilon)}))^c] \leq e^{-c\tilde{\varepsilon}^2 v \text{cap}( D_0)} + \P[h^{v(1-2\tilde\varepsilon)}(z) \leq v' \text{cap}( D_0)] \stackrel{\eqref{eq:Xutail_bndlow}, \eqref{eq:G-bound-pf7}}\leq e^{-c'\tilde{\varepsilon}^2 v \text{cap}( D_0)}.
\end{equation*}
Combining this with \eqref{eq:G-bound-pf8} and a union bound over $\mathcal{D}'$ (cf.~below \eqref{eq:G-bound-pf6}) yields the desired bound on $\mathbb{Q}_{\mathcal{C}}[\widetilde{E}_2(\mathcal{D}) \cap \widetilde{E}_{2.1}(\mathcal{D})]$. With regards to $\mathbb{Q}_{\mathcal{C}}[\widetilde{E}_{2.2}(\mathcal{D})]$, one performs a similar union bound and applies Proposition~\ref{prop:Xutail_bnd} with $\Sigma=\Sigma(\mathcal{D}')$ and $\mathcal{D}' \subset \mathcal{C}$ satisfying $ |\mathcal{D}'| \geq (1-4 \rho )|\mathcal{C}|/k$. Altogether, this gives,  for all $u,v,\mathcal{D}$ as  in the sup of \eqref{eq:G-bound-pf5}, all $L \geq C(u,v, \delta)$ (recall $\delta \in (0,1)$ is implicit in $\tilde{\varepsilon}$) and $K \geq C \tilde{\varepsilon}^{-1}$,
 \begin{multline}
\label{eq:G-bound-pf9}
\mathbb{Q}_{\mathcal{C}}[\widetilde{E}_2(\mathcal{D})]\stackrel{\eqref{eq:G-bound-pf6}}{ \leq } \mathbb{Q}_{\mathcal{C}}[\widetilde{E}_2(\mathcal{D}) \cap \widetilde{E}_{2.1}(\mathcal{D})]+ \mathbb{Q}_{\mathcal{C}}[ \widetilde{E}_{2.2}(\mathcal{D})] \\
\leq  (C\rho^{-1})^{ k^{-1}\rho |\mathcal C|}   \Big( \exp\big\{- ck^{-1}\rho |\mathcal C|  {  v \tilde{\varepsilon}^2 \text{cap}( D_0)}\big\} +\textstyle \sup_{\mathcal{D}'} \exp\big\{- (1- 2\tilde{\varepsilon}) (\sqrt{u}-\sqrt{v})^2 \text{cap}(\Sigma(\mathcal{D}'))\big\} \Big)
\end{multline}
with the supremum ranging over $\mathcal{D}' \subset \mathcal{C}$ satisfying $ |\mathcal{D}'| \geq (1-4 \rho )|\mathcal{C}|/k$.

We now assemble the various pieces, and in the process aim to apply \eqref{eq:coarse_capacity_d=3} to control the term involving $\text{cap}(\Sigma(\mathcal{D}'))$ in \eqref{eq:G-bound-pf9}. We first focus on the case $d=3$, which is more intricate. In that case recall from \eqref{def:coarse_admissible1} that $c(1-\sigma)N/KL \leq |\mathcal{C}| \leq (1-\sigma)N/KL$ and that $\sigma \leq \frac12$. Define $L=L(N)=  \lfloor (\log N)^{\alpha} \rfloor$ for $\alpha>0$ to be determined. Then \eqref{eq:G-bound-pf3} yields that for all $N \geq C( \beta', \rho)$ (so that in particular $c L^{\beta'} - C'|\log \rho \,|$) and $K \geq K_0 \vee \Cr{c:equil}(\varepsilon_L)^{-1}$, 
\begin{equation}
\label{eq:G-bound-pf10}
\log \mathbb{Q}_{\mathcal{C}}[E_1] \leq - \frac{\Cl[c]{c:iidpart}N}{K (\log N)^{\alpha(1-\beta')}}.
\end{equation}
As to $\mathbb{Q}_{\mathcal{C}}[E_2] $, we now examine the sizes of the various terms involved in \eqref{eq:G-bound-pf5} and \eqref{eq:G-bound-pf9}. Since $\text{cap}( D_0) \geq cL^{d-2}=cL$ when $d=3$ and due to the choices of $\tilde{m}$ and $\tilde{\varepsilon}$ in \eqref{eq:G-bound-pf4}, one readily finds that the last term in \eqref{eq:G-bound-pf5} decays to leading exponential order as $cN/ (\log L)^{\theta}$, with $L=L(N)$. The same conclusions can be reached of the first term in the last line of \eqref{eq:G-bound-pf9}. All in all, these two terms are negligible relative to the decay in \eqref{eq:G-bound-pf10} as soon as $N \geq C( \beta', \rho, k, \delta, {\bm u,\bm v}, L_0 ) $ and $K \geq K_0 \vee C \tilde{\varepsilon}^{-1} $. 

The second term in \eqref{eq:G-bound-pf9} is bounded using \eqref{eq:coarse_capacity_d=3}. Note to this effect that \eqref{eq:quant_LK} is satisfied for the choice $L=L(N) (=L_-=L_+) $ with $K_+= \sqrt{\log \log e^2N}$. Overall, this yields, for all $\delta \in (0,1)$,  $N \geq C( \beta', \rho, k, \delta, {\bm u,\bm v}, L_0 )  $ and $ K_0 \vee C \tilde{\varepsilon}^{-1} \leq K \leq K_+ $,
\begin{equation}
\label{eq:G-bound-pf11}
\log \mathbb{Q}_{\mathcal{C}}[E_2] \leq \frac{(1-\frac\delta 2) \gamma N}{\log N}, \quad \gamma= (1 - \sigma)(1 - C\rho)\frac{\pi}{3k}\bigg[\min_{(u,v)}(\sqrt{u} - \sqrt{v})^2 \bigg],
\end{equation}
with $(u,v)$ ranging among $(u_i,v_i)$, $1\leq i \leq k$.
Returning to \eqref{eq:G-bound-pf1}-\eqref{eq:G-bound-pf2}, the bounds in \eqref{eq:G-bound-pf10}, \eqref{eq:G-bound-pf11} are now pitted against $\Gamma(N/L(N))$. Since $\Gamma(N/L) \leq {\Cr{C:cg_complexity}}N (\log N) /KL  $, see Proposition~\ref{prop:coarse_paths}, item~i), it follows from \eqref{eq:G-bound-pf1} that for all $N \geq C( \beta', \rho, k, \delta, {\bm u,\bm v}, L_0 ) $ and $K \geq K_0 \vee C \tilde{\varepsilon}^{-1} \vee \Cr{c:equil}(\varepsilon_{L(N)})^{-1} $,
 \begin{equation}
\label{eq:G-bound-pf12}
N^{-1} \log \P[\mathscr{G}^c] \leq  \frac{\Cr{C:cg_complexity}}{K_-(N)(\log N)^{\alpha-1}}- \bigg( \frac{\Cr{c:iidpart}}{K_+(N) (\log N)^{\alpha(1-\beta')}}  \vee \frac{(1-\tfrac\delta 2) \gamma}{\log N}\bigg).
\end{equation}
In order for the term in parenthesis to be larger than the first term on the right of \eqref{eq:G-bound-pf12}, and because $K_- \geq 1$ whereas $K_+ \leq C \vee \sqrt{\log \log N} $, it is sufficient that $\alpha -1 > \alpha(1- \beta') \vee 1$. In particular this requires $\alpha > 2$. There are now two cases to consider, depending on the value of $\beta' \in (0,1)$. If $0< \beta'  \leq \frac12$, one simply picks any $\alpha > \frac1{\beta'} (>2)$, for instance $\alpha= \frac1{\beta'}+1$. Since for such $\beta'$, one has $\alpha(1- \beta') \ge 1$, the right-hand side of \eqref{eq:G-bound-pf12} is bounded by $- {\Cr{c:iidpart}}/({2K_+(N) (\log N)^{\alpha(1-\beta')}})$ for large $N$. By choosing any $\beta$ ever so slightly larger than $\alpha(1-\beta')$, one can easily absorb the factor $1/ 2K_+(N)$ for large $N$ and instead produce the desired pre-factor, leading to \eqref{eq:bootstrapped_limit_prob1d=3} in this case. If instead $\frac12< \beta'  \leq \frac12$, one picks a value of $\alpha$ satisfying $2< \alpha < (1-\beta')^{-1}$, for instance, $\alpha= 1 + \frac12(1-\beta')^{-1} $. The decay in \eqref{eq:G-bound-pf12} is then governed by the second term in parenthesis (since now $\alpha(1- \beta') <1$), and for suitably large $N$ one ensures that the right-hand side of \eqref{eq:G-bound-pf12} is at most $\frac{(1-\delta ) \gamma}{\log N}$, yielding \eqref{eq:bootstrapped_limit_prob1d=3}. 

The case  $d \geq 4$ is simpler, notably because the complexity $\Gamma(N/L) \leq C N/L$ never requires fine-tuning of $L$ beyond choosing $L$ large (in a manner depending on the various parameters). For instance, in the case of $\mathbb{Q}_{\mathcal{C}}[E_1]$, recalling that $|\mathcal{C}| \geq {cN}/L \log(KL)^2$ from \eqref{def:coarse_admissible1}, one obtains a bound on $\log \mathbb{Q}_{\mathcal{C}}[E_1]$ effectively of the form $c(N/L) \frac{ L^{\beta'}}{\log(KL)^2}$ and the second fraction is more than enough for large $L$ to produce a decay of exponential order $N/L$ with arbitrary large pre-factor. The case of $\mathbb{Q}_{\mathcal{C}}[E_2]$ is handled similarly, using that the capacity of a box of side-length $L$ grows at least quadratically when $d>3$ to handle both the second term in \eqref{eq:G-bound-pf5} and the first term in \eqref{eq:G-bound-pf9}, and appealing to item ii) of Proposition~\ref{prop:coarse_paths} for the remaining one in \eqref{eq:G-bound-pf9}. Notice in particular that Proposition~\ref{prop:coarse_paths} yields an exponential decay in $N$ rather than $N/L$ in this case. Overall, \eqref{eq:bootstrapped_limit_probd=4} follows. 
\end{proof}

\begin{remark}[The events $\mathcal{F}$ and $\mathcal{G}$]\label{R:FvsG}
We briefly return to the different roles played by the events $\F$ and $\G$ in defining the (good) event $\mathscr G$ in \eqref{def:script_G}. The family $\G$ will be further specified in the next section, but remains largely flexible. In the simplest cases of interest $\G_{z,L}$ will correspond to a (dis-)connection event inside $\tilde{D}_{z,L}$, see for instance \eqref{eq:disconnect_Gz} below, but more complex choices for $\G_{z,L}$ will also be required. The events 
$\mathcal{F}$ specified in \eqref{eq:Fext} may superficially look like a means to localization (cf.~
\S\ref{subsec:excursion}), but inspection of the proof of Proposition~\ref{prop:bootstrap_prob} (in particular the bounds on  $E_2$ defined below \eqref{eq:G-bound-pf2}, and later on $\widetilde{E}_{2.2}$, cf.~\eqref{eq:G-bound-pf5}, \eqref{eq:G-bound-pf6}, \eqref{eq:G-bound-pf9} and \eqref{eq:G-bound-pf11}) reveals that they generate the leading-order contribution to \eqref{eq:bootstrapped_limit_prob1d=3}.
\end{remark}

\section{Bootstrapping}\label{sec:coarse_graining}
In the previous section, we introduced an event $\mathscr{G}$, see  \eqref{def:script_G}, 
which is at the center of our coarse-graining mechanism. Roughly speaking, the event 
$\mathscr{G}= \mathscr G(\Lambda_N, \G, \mathcal{F};\rho)$, which lives at scale $N \gg 
L$, ensures that, for any choice of $\Lambda_N$ in \eqref{eq:scriptS_N}, any (admissible) 
coarse-grained path at scale $L$ crossing $\Lambda_N$ will meet a large number of good 
$L$-boxes anchored at points $z$, in the sense that the corresponding event 
$\G_{z,L}$ occurs. Whereas the events $\mathcal F_{z,L}$ are specified up to 
the choice of parameters, see \eqref{eq:Fext}, so far the family $\mathcal{G}= \{ \G_{z,L}: z 
\in \mathbb{L}\}$ was completely generic, save for some localization properties (see 
\eqref{eq:bootstrap_prob_coupling}). 

In \S\ref{subsec:bootstrap}, we give more structure to the events defining $\mathcal{G}$, and 
show that if $\mathcal{G}$ is chosen from a suitable class, specified by 
Definition~\ref{def:good_events}, the associated event $\mathscr{G} $ implies an event of type $\G$ at scale $N$. This is the 
content of Proposition~\ref{prop:bootstrap_events} below, see in particular 
\eqref{eq:grand_inclusion2}, which is entirely deterministic, and constitutes the main result 
of this section. The event $\mathscr G$ thus acts as a vehicle to propagate estimates for 
the events $\G$ from scale $L$ to scale $N$, which is the \emph{bootstrapping} alluded to 
in the header. The probability for this mechanism to fail will eventually be 
controlled by the previously derived Proposition~\ref{prop:bootstrap_prob}. 

As a first (simple) application of this framework, we prove in \S\ref{subsubsec:subcrit} sharp 
upper bounds on the one-arm probability for $\mathcal V^u$ in the subcritical regime $u > 
u_*$. This result is already known from Theorem~3.1 in \cite{PT12} for $d \ge 4$ and 
recently from Theorem~1.3 in \cite{Prevost23} for $d = 3$. The full strength of our framework, however, will be harnessed in the forthcoming sections where we deal with 
the supercritical regime.

\subsection{Bootstrapping with the events of type $\mathcal G$}\label{subsec:bootstrap}
Let us start by adding one more scale $L_0 \ll L (\ll N)$ to our setup. Thus, for the remainder 
of Section~\ref{sec:coarse_graining} we work with three concurrent scales $N, L, L_0 \in 
\N^\ast =\{1,2,\dots\}$ and an integer scaling factor $K$ which are always (tacitly) assumed to satisfy
\begin{equation}\label{eq:L_k_descending}
	\begin{split}
		&K \geq 100,
	N \ge 
	\Cr{c:nLB}^{-1}{\color{black}{10}^d}
	\rho^{-1}\,h(K L)  \mbox{ and } L > 100L_0, 
	\end{split}
\end{equation}
cf.~above \eqref{def:coarse_admissible3} regarding the function $h(\cdot)$, the statement of Proposition~\ref{prop:coarse_paths} regarding $\Cr{c:nLB}$ and \eqref{eq:bootstrapped_limit_prob1d=3} regarding $\Cr{C:rho}$. We also introduce $\mathbb L_0 =L_0\Z^d$ and for $A \subset \Z^d$
the set $\mathbb L_{0}(A) = \{z \in 	\mathbb L_0: A \cap C_{z, L_0} \ne \emptyset\}$; see below \eqref{def:CDU} for notation. If $\gamma$ is a path in $\Z^d$ we abbreviate $\mathbb L_{0}(\gamma)= \mathbb L_{0}(\text{Range}(\gamma))$. 
In bootstrapping from scale $L$ to $N$, the parameters $L_0$ and $K$ will remain fixed. For this reason, the dependence of quantities on $L_0$ and $K$ will be implicit in our notation.

We now introduce the family of (likely) events $\G_L =\{ \G_{z,L} : z\in \mathbb{L}\}$ of 
interest. Their definition also depends on the scale $L_0$, which, in accordance with the 
previous paragraph, will not appear explicitly in our notation; below, when passing from one scale $L$ to another for the family $\G_L$, thus varying $L$, the 
(base) scale $L_0$ will not change. The events in $\G_{L}$ are specified in 
terms of a `data' 
\begin{equation}\label{eq:data}
({\rm V}, {\rm W}, \mathscr{C}),
\end{equation}
where ${\rm V} = \{{\rm V}_z:  z \in \mathbb L\}$ and ${\rm W} = \{{\rm W}_{{
z,}y} : {
z \in \mathbb{L}, } \,  y \in \mathbb L_0\}$ are two families of events indexed by $\mathbb L$ 
and $\mathbb L\times \mathbb L_0$, respectively, and $\mathscr C = \{\mathscr C_z: z \in \mathbb 
L\}$ is a family of finite subsets of $\Z^d$. 

\begin{definition}[The events $\G= \G_L =\{ \G_{z,L} : z\in \mathbb{L}\}$]\label{def:good_events}	
	For $a \geq 0$ and $({\rm V}, {\rm W}, \mathscr{C})$ as in \eqref{eq:data}, let 
		\begin{equation}\label{def:Gz}
		\mathcal{G}_{z}( {\rm V}, {\rm W}, \mathscr C; a) =		G_z({\rm W}, \mathscr C; a) \cap {\rm V}_z
	\end{equation}
	where $\mathcal{G}_z=\mathcal{G}_{z, L} (= \mathcal{G}_{z, L, L_0}) $ and the 
	event $G_z=G_{z,L} (= G_{z,L,L_0})$ is defined as
	\begin{equation}\label{def:interface1}
		\left\{
		\begin{array}{c}
			\text{for any crossing $\gamma$ of $\tilde D_z \setminus \tilde C_z$, there exists a collection of points $S_\gamma \subset \mathbb L_{0}(\gamma)$}\\ \text{such that $|S_\gamma| \ge a$ and foreach $y \in S_\gamma$, ${\rm W}_{{
						z}, y}$ occurs and $C_{y, L_0} \cap \mathscr C_z \ne \emptyset$}
		\end{array}
		\right\}.
	\end{equation}
\end{definition}

Notice $\G_z$ depends on ${\rm V}, {\rm W}$ and $\mathscr{C}$ only through ${\rm V}_z$, 
${\rm W}_{z,\cdot}$ and $\mathscr{C}_z$. Moreover, in the case $a=0$, all of ${\rm W}$, 
$\mathscr{C}$ and $G_z$ become superfluous in view of \eqref{def:interface1}, and $\G_z$ 
coincides with ${\rm V}_z$. This simplified setup is already non-trivial and will be pertinent in 
the (simpler) subcritical regime; cf.~
\S\ref{subsubsec:subcrit}.

Our next result shows that one can relate events $\G$ of the form postulated by Definition~\ref{def:good_events} at two different 
scales $L$ and $N$ using the event $\mathscr G$ from \eqref{def:script_G}.
Even though the events $\mathcal F_{z,L}$ appearing in \eqref{def:script_G} will in practice be of the form \eqref{eq:Fext}, for the purposes of the present section it is sufficient that the inclusion in condition \eqref{eq:bootstrap_inclusion} below holds. Thus, the reader need not think beyond \eqref{def:script_G} about a specific space on which the events $\mathcal{F}$ and $\mathcal{G}$ are realized for the purposes of the next result.

The next proposition also includes a change of configurations when passing from scale $L$ to $N$, manifest by the presence of two sets of data $({\rm V}^1, {\rm W}^1, {\mathscr{C}}^1)$ and $(\tilde{\rm V}^1, \tilde{\rm W}^1, \tilde{\mathscr{C}}^1)$. The 
reader could however choose to {\em omit} this layer of complexity at first reading, i.e.~assume that the events $\mathcal{F}_{z,L}$ in \eqref{def:script_G} are 
trivial (i.e.~the full space), whence \eqref{eq:bootstrap_inclusion} below plainly holds with $(\tilde{\rm V}^1, \tilde{\rm W}^1, 
\tilde{\mathscr{C}}^1)=({\rm V}^1, {\rm W}^1, {\mathscr{C}}^1)$. 
For the topological component of the result (see~\eqref{eq:O-properties} below), we need to 
consider a different graph structure, by which $x, y \in \Z^d$ are $*$-{\em 
neighbors} if $|x - y|_\infty = 1$;  $\ast$-path, $\ast$-clusters etc.~are defined accordingly.

\begin{proposition}[Bootstrapping]\label{prop:bootstrap_events}
Under \eqref{eq:L_k_descending} and for any choice of $({\rm V}^1, {\rm W}^1, 
{\mathscr{C}}^1)$, $(\tilde{\rm V}^1, \tilde{\rm W}^1, \tilde{\mathscr{C}}^1)$ as in 
\eqref{eq:data}, all $\Lambda_N \in  \mathcal{S}_N$ (see \eqref{eq:scriptS_N}), $a^{(1)} \geq 0$ 
and 
$\rho \in (0, 1]$, the following hold. If 
\begin{equation}\label{eq:bootstrap_inclusion}
		\big(\mathcal G_{z,L}({\rm V}^1, {\rm W}^1, {\mathscr{C}}^1;a^{(1)})  \cap \mathcal F_{z,L} \big)\subset \mathcal G_{z,L}(\tilde{\rm V}^1, \tilde{\rm W}^1, \tilde{\mathscr{C}}^1;a^{(1)}), \text{ for all $ z \in \Lambda_N \cap \mathbb L$}
	\end{equation}
then there exists a non-empty set $\mathcal O \subset \mathbb L$ 
defined measurably in $\{1_{ \mathcal G_{z,L}(\tilde {\rm V}^1, \tilde{\rm W}^1, \tilde{\mathscr{C}}^1;a^{(1)})} : z \in  
\Lambda_N \cap \mathbb L\}$, 
such that (see App.~\ref{sec:dual_surface} for notation) 
\begin{equation}\label{eq:O-properties} 
\text{\begin{minipage}{0.85\textwidth} ${D}_{z, L}  \subset \Lambda_N$ for each $z \in 
	\mathcal{O}$  %
	and, writing $\Lambda_N = 
	V_N \setminus U_N$, each $\ast$-component $\mathcal O'$ of $\mathcal O$ satisfies $ 
	\{0\} \cup %
	(U_N 
	\cap \mathbb L) \preceq \mathcal O' \preceq \partial_{\mathbb L}(V_N \cap \mathbb L)$ as 
	subsets of the coarse-grained lattice $\mathbb L$ \end{minipage} }
	\end{equation}
and, abbreviating $\G^1=\{ \mathcal G_{z,L}({\rm V}^1, {\rm W}^1, {\mathscr{C}}^1;a^{(1)}) : z \in \mathbb{L}\}$, one 
	has the inclusion
	\begin{equation}\label{eq:grand_inclusion}
		\begin{split}
			& \textstyle \mathscr G(\Lambda_N, \G^1, \mathcal{F}; \rho)  \subset \bigcap_{z \in \mathcal O} \mathcal G_{z,L}(\tilde{\rm V}^1, \tilde{\rm W}^1, \tilde{\mathscr{C}}^1;a^{(1)}) \ \big( \subset \bigcap_{z \in \mathcal O} \tilde{\rm V}_{z}^1 \big),
		\end{split}
	\end{equation}
	where $\bigcap_{z \in \mathcal{O}} A_z \coloneqq\bigcap_{z \in \mathbb{L}} (A_z \cup \, \{\mathcal O \not\ni z\}).$ 
	Furthermore, if $\tilde{{\rm W}}^1_{z,\cdot}= \tilde{{\rm W}}^1_{0,\cdot}$ for all $z \in \mathbb{L}$, then
	\begin{equation}\label{eq:grand_inclusion2}
		\mathscr G(\Lambda_N, \G^1, \mathcal{F}; \rho) \subset \mathcal G_{0, N}({\rm V}^2, {\rm W}^2, \mathscr{C}^2;a^{(2)}), 
	\end{equation}
	where $a^{(2)} \coloneqq  {10^{-d}} \lfloor\tfrac{
		\rho (1-\sigma)%
		\Cr{c:nLB} N}{h(KL)}\rfloor \cdot a^{(1)}$, 
	${{\rm V}}^2_{0} \coloneqq \bigcap_{z \in \mathcal O} \tilde{{\rm V}}^1_{ z}$, ${{\rm W}}^2_{0,y}= \tilde{{\rm W}}^1_{ 0,y}$ for all $y \in \mathbb{L}_0$ and ${\mathscr C}^2_{ 0} \coloneqq \bigcup_{z \in \mathcal O} \tilde{\mathscr C}^1_{z}$ (note that only ${{\rm V}}^2_{0} , {{\rm W}}^2_{0,\cdot} $ and ${\mathscr C}^2_{ 0}$ are required to define $\mathcal G_{0, N}$ in \eqref{eq:grand_inclusion2}; see below \eqref{def:interface1}).
\end{proposition}

The inclusion \eqref{eq:grand_inclusion2} is precisely expressing the announced {\em bootstrap} mechanism: on the event $\mathscr G$ defined in \eqref{def:script_G}, which is a certain composite of events of type $\G$ as in Definition~\ref{def:good_events} at scale $L$ (along with the events $\mathcal{F}$ but let us forego this point), one witnesses an event of the same type at scale $N$.

\begin{proof}
Write $\Lambda_N =V_N \setminus U_N $, where $\Lambda_N$ ranges among any of the choices in $\mathcal{S}_N$. Throughout the proof, we always tacitly assume that the event $\mathscr G= \mathscr G(\Lambda_N, \G^1, \mathcal{F}; \rho)$ occurs. Let $\Sigma \subset (\mathbb L \cap \Lambda_N)$ be defined as 
\begin{equation}\label{eq:boot-sigma}
\Sigma = \{ z \in  \mathbb L : {D}_{z,L} \subset \Lambda_N  \text{ and } 
 \mathcal G_{z,L}(\tilde {\rm V}^1, \tilde {\rm W}^1, \tilde {\mathscr{C}}^1;a^{(1)})  \text{ occurs}\}.
\end{equation} 
We claim that any path $ \gamma \subset \mathbb{L}$ connecting $\{ 0\} \cup (U_N \cap \mathbb{L})$ to $V_N \cap \mathbb{L}$ intersects $\Sigma$ in at least 
\begin{equation}\label{eq:boot-k}
k \coloneqq \lfloor{
	\rho(1-\sigma) \Cr{c:nLB} N}/{h(KL)}\rfloor 
\end{equation}
points. Indeed consider $\bar\gamma$ (path in $\Z^d$) any extension of $\gamma$ to a 
crossing of $\Lambda_N$. We choose $\bar{\gamma}$ in such a way that 
$\text{Range}(\bar{\gamma}) \subset \bigcup_{z\in \gamma} C_{z,L}$, which can always be 
arranged. By \eqref{def:coarse_admissible2}, there exists an admissible coarsening 
$\mathcal{C}= \mathcal{C}(\bar\gamma) \in \mathcal A_{L}^{K}(\Lambda_N)$ such that 
$\bar{\gamma}$ crosses $\tilde{D}_z \setminus C_z$ for all $z \in \mathcal{C}$. In particular it 
intersects $C_z$ whence $z \in \text{Range}(\gamma)$. But by definition of $\mathscr G$ 
(which is in force), see \eqref{def:script_G}, and owing to \eqref{def:coarse_admissible1}, one 
can extract from $\mathcal{C}$ a sub-collection $\tilde{\mathcal C}$ of cardinality at least $k$  given by \eqref{eq:boot-k} such that 
the inclusion in \eqref{eq:bootstrap_inclusion} occurs for all $z \in \tilde{\mathcal C}$. The claim follows. In light of it, 
Proposition~\ref{lem:surroundingInterfaces} applies 
on $\mathbb{L}$ (rather than $\Z^d$) with $\Sigma$ as in \eqref{eq:boot-sigma}, $U=U_N \cap 
\mathbb{L}$, $V=V_N \cap \mathbb{L}$ and $k$ as in \eqref{eq:boot-k}, yielding 
$*$-connected sets $O_1,\dots,O_{\ell}$ satisfying items (a)-(c). Letting 
$\mathcal{O}\coloneqq \bigcup_i O_i$, it then immediately follows from (a) that any 
component $\mathcal{O'}$ satisfies $U_N \cap \mathbb L \preceq \mathcal O' \preceq 
\partial_{\mathbb L}(V_N \cap \mathbb L)$. The other properties required in 
\eqref{eq:O-properties} (including the measurability requirements on $\mathcal{O}$ above 
\eqref{eq:O-properties}) plainly hold. Moreover, since $\mathcal{O} \subset \Sigma$, the first 
inclusion in \eqref{eq:grand_inclusion} is immediate on account of \eqref{eq:boot-sigma}. 
The second inclusion in \eqref{eq:grand_inclusion} follows plainly from~\eqref{def:Gz}.

It remains to prove \eqref{eq:grand_inclusion2}. The fact that ${\rm V}_0^2$ (see below 
\eqref{eq:grand_inclusion2}) occurs on $\mathscr{G}$ is already implied by 
\eqref{eq:grand_inclusion}, hence in view of \eqref{def:Gz} it remains to show that 
$\mathscr{G}$ implies the occurrence of $G_{0,N}=G_{0,N}({\rm W}^2, \mathscr C^2;a^{(2)})$ 
as defined in \eqref{def:interface1}. Thus let $\gamma$ be a crossing of $\tilde{D}_{0,N} 
\setminus \tilde{C}_{0,N} (= \Lambda_N)$. By definition of ${\rm W}^2, \mathscr C^2$  
below~\eqref{eq:grand_inclusion2}, the proof is complete once we extract a collection of points 
$S_{\gamma} \subset \mathbb L_{0}(\gamma)$ such that 
\begin{equation}\label{eq:boot-goal}
\text{$|S_\gamma| \ge a^{(2)}$ and for each $y \in S_\gamma$, $\tilde{\rm W}^1_{{0}, y}$ occurs and $C_{y, L_0} \cap \big( \textstyle \bigcup_{z \in \mathcal O} \tilde{\mathscr C}^1_{z}\big) \ne \emptyset$}.
\end{equation}
Consider $\gamma' \subset \mathbb{L}$, the $*$-path obtained from $\gamma$ by retaining the sequence of all $z$'s intersected by $\text{Range}(\gamma)$, in the order visited by $\gamma$. By construction of $\mathcal{O}$ and item~(b) in Proposition~\ref{lem:surroundingInterfaces}, $\gamma'$ intersects $\mathcal{O}$ in at least $k$ points, with $k$ as in \eqref{eq:boot-k}. If $z \in  \mathcal{O} \cap \text{Range}(\gamma') $ is any such point, using the fact that $\tilde{D}_{z,L}$ is contained in $\Lambda_N$ (see \eqref{eq:boot-sigma} and recall that $\mathcal{O}\subset \Sigma$), it follows that the path $\gamma$ must cross $\tilde{D}_{z,L} \setminus \tilde{C}_{z,L}$.
 Moreover, still using that $\mathcal{O} \subset \Sigma$, \eqref{eq:boot-sigma}, \eqref{eq:bootstrap_inclusion} and \eqref{def:Gz} yield that $G_{z,L}(\tilde{\rm W}^1, \tilde{\mathscr{C}}^1;a^{(1)}) $ occurs. By definition, see \eqref{def:interface1}, this implies that there exists a set $S_\gamma(z)\subset \mathbb{L}_0(\gamma)$ of cardinality at least $a^{(1)}$ and such that for each $y \in S_\gamma(z)$, the event ${\rm W}_{z,y} = {\rm 
 W}_{0,y}$ occurs and $C_{y,L_0} \cap \tilde{\mathscr C}^1_{z}$ is not empty. 
 
The claim \eqref{eq:boot-goal} now follows immediately by extracting  $S_{\gamma}$ from $\bigcup_{z\in \mathcal{O} \cap \text{Range}(\gamma')} S_\gamma(z)$, by retaining at least a fraction $10^{-d}$ of points $z$ in the union, thereby ensuring that the sets $\tilde{D}_{z,L}$ are disjoint. It follows that the cardinalities of $S_\gamma(z)$ as $z$ varies over this thinning of  $\mathcal{O} \cap \text{Range}(\gamma')$ are additive, yielding that $|S_{\gamma}| \geq 10^{-d} k a^{(1)} =a^{(2)}$, as required. The other 
requirements on $S_{\gamma}$ in \eqref{eq:boot-goal} are immediate.
\end{proof}

\subsection{A first application: bounds for the subcritical phase}\label{subsubsec:subcrit}
Recall that $B_{N}^2$ denotes the ball of radius $N$ around $0$ in the $\ell^2$ (Euclidean) 
norm. The aim of this short section is to prove the following:

\begin{thm}[Subcritical regime]\label{thm:d=4}
	For all $u> u_*$, 
	\begin{align}	
		&\sup_{N \ge 1}N^{-1}\log \P\big[ \lr{}{}{0}{\partial B_N^2} \text{ in } \mathcal V^u
		\big] \le -c(u), \quad \text{if $d \ge 4$;} \label{eq:d=4-sub}\\
		& \limsup_{N \to \infty} \frac{\log N}{N}\log \P\big[ \lr{}{}{0}{\partial B_N^2} \text{ in } \mathcal V^u
		\big] \leq -\frac{\pi}{3}(\sqrt{u} - \sqrt{u_*})^2, \quad \text{if $d = 3$.} \label{eq:d=3-sub}
	\end{align}
\end{thm}

\begin{proof} Let $u>0$ and $d \geq 3$. We start by specifying
	the collection ${\rm V} = \{{\rm V}_z:  z \in \mathbb L\}$ from \eqref{eq:data} by setting
	${\rm V}_z= {\rm Dis}_z(\widehat{\mathcal V}_{\mathbb L}, u)$, where, for any $z \in \mathbb L$ and $\widehat{\mathcal V}_{\mathbb L} \in \{\overline{\mathcal V}_{\mathbb L}, \mathcal V_{\mathbb L}, \widetilde{\mathcal 
		V}_{\mathbb L}\}$ (see \eqref{eq:3Vs} for notation),
	\begin{equation}\label{def:disconnect}
		{\rm Dis}_z(\widehat{\mathcal V}_{\mathbb L}, u) \stackrel{{\rm def.}}{=}\big\{\nlr{}{}{\tilde C_z}{\partial \tilde D_z}\text{ in } \widehat{\mathcal V}^{u}_z\big\}.
	\end{equation}
	Recalling Definition~\ref{def:good_events}, it follows that
	\begin{equation}\label{eq:disconnect_Gz}
		\mathcal G_z({\rm Dis}(\widehat{\mathcal V}_{\mathbb L}, u), a = 0, {\rm W}, \mathscr C) \stackrel{\eqref{def:Gz}, \eqref{def:interface1}}{=} {\rm Dis}_z(\widehat{\mathcal V}_{\mathbb L}, u)
	\end{equation}
	regardless of the choice of families ${\rm W}$ and $\mathscr C$, 
	which will be henceforth be omitted from all notation. This also makes superfluous the scale $L_0$ involved in the definition of $({\rm W},\mathscr C)$. In the sequel we always assume that the scales $N,L$ and the scaling factor $K$ satisfy \eqref{eq:L_k_descending} with $L_0=1$. Observe that \eqref{eq:disconnect_Gz} asserts that $\mathcal{G}= {\rm Dis}(\widehat{\mathcal V}_{\mathbb L}, u) (= {\rm V}) $, which feeds into the definition of the 
	event $\mathscr G= \mathscr G(\Lambda_N, \G, \mathcal{F};\rho)$ from \eqref{def:script_G}. 
	Also note, for any $0 < u < v$, the inclusions
	\begin{equation}\label{eq:disconnect_inclusion}
		\begin{split}
			&{\rm Dis}_z(\mathcal V_{\mathbb L}, u) \cap \mathcal F_z^{v, u} \subset {\rm Dis}_z(\overline{\mathcal V}_{\mathbb L}, v), \mbox{ and } {\rm 
				Dis}_z(\overline{\mathcal V}_{\mathbb L}, u) \cap \mathcal F_z^{u, v} \subset {\rm Dis}_z(\mathcal V_{\mathbb L}, v).
		\end{split}
	\end{equation}
	Indeed these follow from the observation that the event ${\rm Dis}_z(\widehat{\mathcal V}_{\mathbb L}, u)$ in 
	\eqref{def:disconnect} is decreasing in the configuration $\widehat{\mathcal V}_z^u$,
	along with the definitions of $\overline{\mathcal{V}}_z^u$ and ${\mathcal{V}}_z^u$ in 
	\eqref{eq:3Vs} and of the event ${\mathcal F}_z^{u, v}$ 
	in \eqref{eq:F}. 
	
	Focusing on the first inclusion \eqref{eq:disconnect_inclusion}, 
	we now set $\mathcal F_{z, L} = \mathcal F_z^{v, u}$, which is of the form \eqref{eq:Fext} with $k=1$ (and $u_1=v$, $v_1=u$). With this choice, the first inclusion in 
	\eqref{eq:disconnect_inclusion}, in combination with \eqref{eq:disconnect_Gz}, tells us that the condition 
	\eqref{eq:bootstrap_inclusion} of Proposition~\ref{prop:bootstrap_events} is satisfied by ${\rm 
		V}^1_{z} = {\rm Dis}_z(\mathcal V_{\mathbb L}, u)$, $\tilde {\rm V}^1_{z} = {\rm Dis}_z(\overline{\mathcal V}_{\mathbb L}, 
	v)$, $a^1 = 0$ (omitting references to ${\rm W}^1$, $\mathscr 
	C^1$,$ \tilde {\rm W}^1$, $\tilde{\mathscr C}^1$ and $a^1$), which are all declared under $\P$.
	Thus, we obtain from \eqref{eq:O-properties}--\eqref{eq:grand_inclusion} 
	that, for any $N, L, K$ satisfying \eqref{eq:L_k_descending} with $L_0=1$, any $\rho \in (0, \tfrac12]$, any $0 < u < v$ and any choice of $\Lambda_N \in \mathcal{S}_N$ 
	(recall \eqref{eq:scriptS_N}), writing $\Lambda_N = V_N \setminus U_N$, the following holds: there exists a $\ast$-connected $\mathcal O' \subset \mathbb L$ 
	such that $D_z={D}_{z, L}  \subset \Lambda_N$ for each $z \in \mathcal{O}'$, $\{0\} \cup U_N \cap \mathbb L \preceq \mathcal O' \preceq \partial_{\mathbb L}(V_N \cap 
	\mathbb L)$  and on $\mathscr G(\Lambda_N, {\rm Dis}(\mathcal V_{\mathbb L}, u), \mathcal F_L^{v, u}; \rho)$, the event  
	$\text{Dis}_{z}(\overline{\mathcal V}_{\mathbb L}, v)$ occurs for each $z \in \mathcal O'$. 
	
	Let $\gamma$ now be a crossing of $\Lambda_N$, i.e.~a (nearest neighbor) path on $\Z^d$ 
	intersecting both $U_N$ and $\partial V_N$ (see above \eqref{eq:scriptS_N}). Its 
	coarse-graining $\gamma_{\mathbb{L}}$, obtained as the sequence of points $z \in \mathbb{L}$ such that $\gamma$ visits 
	$C_z$, is a connected set in $\mathbb{L}$, which owing to the above must intersect 
	$\mathcal{O}'$ on the event $\mathscr G(\Lambda_N, {\rm Dis}(\mathcal V_{\mathbb L}, u,  \mathcal 
	F_L^{v, u}; \rho)$. Thus, let $z \in \text{range}( \gamma_{\mathbb{L}}) \cap \mathcal{O}'$. It 
	follows that $\gamma$ must cross $\tilde D_{z} \setminus \tilde C_{z}$ and that 
	$\text{Dis}_{z}(\overline{\mathcal V}_{\mathbb L}, v)$ occurs. In particular, this implies that 
	$\gamma$ cannot lie inside $\mathcal V^{v}$. %
	All in all, 
	\begin{equation}\label{eq:inclusion_disconnect}
		\mathscr G(\Lambda_N, {\rm Dis}(\mathcal V_{\mathbb L}, u), \mathcal F_L^{v, u}; \rho) \subset \big\{\nlr{}{}{ U_N^0}{V_N} \text{ in } \mathcal V^{v}\big\}, \quad U_N^0= U_N \cup \{0\}.
	\end{equation}

	In view of \eqref{eq:inclusion_disconnect}, we now apply Proposition~\ref{prop:bootstrap_prob}, from which the bounds \eqref{eq:d=4-sub}-\eqref{eq:d=3-sub} will eventually follow.  Let $u > u_\ast$ and consider any $\varepsilon \in \big(0, ((\frac{u}{u_\ast})^{\frac1{10}} - 1) \wedge \frac1{10}\big)$. 
	We proceed to verify conditions \eqref{eq:bootstrap_prob_coupling}--\eqref{eq:initial_limit_prob} inherent to Proposition~\ref{prop:bootstrap_prob}. 
	To this effect, from \cite[Theorem~1.2-(i)]{RI-I} and a straightforward union bound, we know that for any $z \in \mathbb L$ and $L \geq1$, abbreviating $\pi({\mathcal V}_{\mathbb L},u)\coloneqq  \P[{\rm Dis}_z({\mathcal V}_{\mathbb L}, u)]$,
	\begin{equation}\label{eq:trigger_subcrit_rnd1}
		\pi(\overline{\mathcal V}_{\mathbb L}, u_\ast(1 + \varepsilon)) \ge 1 - C(\varepsilon) e^{-L^c}.
	\end{equation}
	We aim to transfer the bound \eqref{eq:trigger_subcrit_rnd1} to the configuration $\widetilde{\mathcal V}_{\mathbb L}$, cf.~\eqref{eq:3Vs} at a slightly different level than $u_\ast(1 + \varepsilon)$. We do this in two steps, using the intermediate configuration ${\mathcal V}_{\mathbb L}$. In view of the second inclusion in \eqref{eq:disconnect_inclusion}, we obtain from \eqref{eq:trigger_subcrit_rnd1} that for any $K \ge C(\varepsilon)$,
	\begin{equation}\label{eq:V_to_V_L_subcrit_rnd1}
		\pi(\mathcal V_{\mathbb L}, u_\ast(1 + \varepsilon)^2)  
		\ge \pi(\overline{\mathcal V}_{\mathbb L}, u_\ast(1 + \varepsilon)) - \P[(\mathcal F^{u_\ast( 1 + \varepsilon), u_\ast( 1 + \varepsilon)^2}_{z})^c] \stackrel{\eqref{eq:Nuz_tail_bnd}}{\ge} 1 
		- C(\varepsilon) e^{-L^c}.
	\end{equation}
	To proceed, we obain from the definition of the 
	event ${\rm Incl}_z^{\varepsilon, m}$ in \eqref{eq:RI_basic_coupling2} and the previously 
	alluded monotonicity of 
	${\rm Dis}_z(\widehat{\mathcal V}_{\mathbb L}, u)$ in 
	$\widehat{\mathcal V}_z^u$ that under any coupling 
	$\mathbb Q$ of $\mathbb P$ and $\widetilde {\mathbb P}_{z}$, any $v>0$, $\epsilon \in (0, 
	\frac1{2})$ and $L \ge C(v, \epsilon)$, the following inclusions hold:
	\begin{equation}\label{eq:inclusion_Ztilde}
		\begin{split}	
			&{\rm Dis}_z(\widetilde{\mathcal V}_{\mathbb L}, v) \, \cap \,  {\rm Incl}_{z}^{\frac{\epsilon}{6}, \lfloor v\,{\rm cap}( D_z)\rfloor} \subset {\rm Dis}_z\big(\mathcal V_{\mathbb L}, (1 + \epsilon)v\big)  \mbox{ and } \\
			&{\rm Dis}_z(\mathcal V_{\mathbb L}, v) \, \cap \,  {\rm Incl}_{z}^{\frac{\epsilon}{6}, \lfloor 
				v\,{\rm cap}( D_z)\rfloor} \subset {\rm Dis}_z\big(\widetilde{\mathcal V}_{\mathbb 
				L}, (1 + \epsilon)v\big)
		\end{split}	
	\end{equation} 
	Using the second inclusion in \eqref{eq:inclusion_Ztilde} with $v= u_\ast(1 + \varepsilon)^2$ and $\delta=\varepsilon$, we can use the coupling $\mathbb Q_{\{z\}}$ from 
	Lemma~\ref{L:RI_basic_coupling} to deduce that for all $L \ge C$ and $K \ge \frac{18\Cr{c:equil}}{\varepsilon}$ (so~\eqref{eq:RI_cond_Q} is satisfied in view of Prop.~\ref{prop:entrance_time_afar}),
	\begin{equation}\label{eq:V_L_to_tildeV_L_subcrit_rnd1}
		\pi(\widetilde{\mathcal V}_{\mathbb L}, u_\ast(1 + \varepsilon)^3) \stackrel{\eqref{eq:inclusion}}{\ge} \pi(\mathcal V_{\mathbb L}, u_\ast(1 + \varepsilon)^2) - 
		\widetilde\P_{z}\big[\big(\mathcal{U}_z^{\frac{\varepsilon}6, \lfloor u_\ast(1 + \varepsilon)^2\,{\rm cap}( D_z)\rfloor}\big)^c\big] 
		\stackrel{\eqref{eq:V_to_V_L_subcrit_rnd1}, \eqref{eq:bnd_Uzepm}}{\ge} 1 - C(\varepsilon) e^{-L^c}.
	\end{equation}
	
	We have now gathered all the ingredients to apply Proposition~\ref{prop:bootstrap_prob} and conclude the proof. We choose 
	$\rho=\frac1{2\Cr{C:rho}}$ where 
	$\Cr{C:rho}(d) = 1$ for $d \ge 4$ (see \eqref{eq:bootstrapped_limit_prob1d=3}--\eqref{eq:bootstrapped_limit_probd=4}),  $k=1$, $\bm u = u_1= u_\ast(1 + \varepsilon)^{5}$, $\bm v = v_1= u_\ast(1 + \varepsilon)^{4}$ (cf.~\eqref{eq:Fext}) and $\widetilde {\mathcal G}_{z, 
		L} = {\rm Dis}_z(\widetilde{\mathcal V}_{\mathbb L}, u_\ast(1 + \varepsilon)^3)$ and ${\mathcal 
		G}_{z, L} = {\rm Dis}_z(\mathcal V_{\mathbb L}, u_\ast(1 + \varepsilon)^{4})=  {\rm Dis}_z(\mathcal V_{\mathbb L}, v_1)$. With these choices, applying \eqref{eq:inclusion_disconnect} with $v=u_1$ and $u=v_1$, the associated event $\mathscr{G}= \mathscr{G}(\Lambda_N, \mathcal G_{L}, \mathcal{F}^{ u_1,  v_1}_{L}; \rho)$ of concern in Proposition~\ref{prop:bootstrap_prob} satisfies $ \big\{\lr{}{\mathcal V^{u_1}}{U_N^0}{V_N}\big\} \subset \mathcal{G}^c$. Thus, provided the conditions \eqref{eq:bootstrap_prob_coupling}--\eqref{eq:initial_limit_prob} are met, the bounds \eqref{eq:bootstrapped_limit_prob1d=3} and \eqref{eq:bootstrapped_limit_probd=4} apply and yield an upper bound on the former connection event. The fact that \eqref{eq:bootstrap_prob_coupling} holds for $\widetilde {\mathcal G}_{z, L}$, ${\mathcal 
		G}_{z, L}$ as above and with the choices $\varepsilon_L = \frac{\varepsilon}6$, $m_L = \lfloor u_\ast(1 + \varepsilon)^3\,{\rm cap}( D_{z})\rfloor$ is an immediate consequence of the first inclusion in 
	\eqref{eq:inclusion_Ztilde} with $\delta=\varepsilon$ and $v= u_\ast(1 + \varepsilon)^3$.
	Having fixed, $\varepsilon_L$, $m_L$, the condition \eqref{eq:initial_limit_prob} holds by virtue of  
	\eqref{eq:V_L_to_tildeV_L_subcrit_rnd1} and \eqref{eq:bnd_Uzepm} for $p_L = C(\varepsilon) 
	e^{-L^c}$,  $K_0 = C(\varepsilon)$ and 
	$L_0  = C(\varepsilon)$ and suitable $\beta'=c > 0$ uniform in $\varepsilon$. 
	
	In dimensions $d \geq 4$, choosing $K=K(\varepsilon)$, $L=L(\varepsilon)$ sufficiently large and $\Lambda_N= B_N^2$ (recall \eqref{eq:scriptS_N}), \eqref{eq:bootstrapped_limit_probd=4} immediately yields \eqref{eq:d=4-sub} for $u= u_1= u_\ast(1 + \varepsilon)^{5}$. Since $\varepsilon \downarrow 0$ as $u \downarrow u_\ast$ and by monotonicity, this concludes the verification of  \eqref{eq:d=4-sub}. For $d=3$, \eqref{eq:bootstrapped_limit_prob1d=3} instead yields for $\Lambda_N= \tilde D_{0, N} \setminus \tilde C_{0, N} $ that
	\begin{equation}\label{eq:disconnection_improved}
		\P[\nlr{}{}{\tilde C_{0, N}}{\tilde D_{0, N}} \text{ in } \mathcal V^{u_1}] \stackrel{\eqref{def:disconnect}}{=} \pi(\overline{\mathcal V}_{N\Z^d}, u_1) \ge 1 - 
		C(\varepsilon)\exp\{- N(\log N)^{-\beta}\}
	\end{equation}
	for all $N \geq 1$ and some $\beta=\beta(\beta') \in (1, \infty)$. 
	To deduce \eqref{eq:d=3-sub}, we apply Proposition~\ref{prop:bootstrap_prob} again,
	now starting with this {\em improved} estimate instead of \eqref{eq:trigger_subcrit_rnd1} and running the 
	same procedure as above. 
	In particular, 
	by \eqref{eq:disconnection_improved}, \eqref{eq:initial_limit_prob} is now satisfied for any choice of $\beta'<1$, say $\beta' = \frac34$.  We then deduce from 
	\eqref{eq:bootstrapped_limit_prob1d=3} and the inclusion in \eqref{eq:inclusion_disconnect} with $\bm u = u$, $\bm v = 
	u_\ast(1 + \varepsilon)^{8}$ (by choice of $\varepsilon$, $u_\ast(1 + \varepsilon)^{8} < u$), any $\rho \in (0, \frac1{2\Cr{C:rho}})$ and 
	$\delta' \in (0, 1)$, $K = \sqrt{\log \log e^2 N}$ and $\Lambda_N = B_N^2$ that 
	\begin{equation*}
		\limsup_{N \to \infty} \frac{\log N}{N}\log \P[\lr{}{\mathcal V^{u}}{0}{\partial B_N^2}] \le -
		(1 - \delta)
		(1 - \Cr{C:rho}\rho)\frac{\pi}{3} (\sqrt{u} - \sqrt{u_\ast(1 + \varepsilon)^{8}})^2.
	\end{equation*}
	Sending $\delta$, $\rho$ and $\varepsilon$ to $0$ yields \eqref{eq:d=3-sub}.
\end{proof}

\begin{remark}\label{R:subcrit-variation} Although not optimal, the following result, which 
	incorporates noise and is obtained by a variation of the above argument, will be useful below. 
	Recall $\mathsf N_\delta(\mathcal V)$ from 
	\eqref{def:noised_set}--\eqref{def:noised_set_inclusion}. There exists $\Cl{C:log-sub} < \infty$ with $\Cr{C:log-sub} = 0$ 
	when $d \ge 4$
	such that, for all $u<u_\ast$, $\delta \leq  \Cl[c]{c:delta-sub}(u) (> 0)$, and $N \ge 1$,
	\begin{equation}\label{eq:connection_delta_d=3}
		\P\big[\lr{}{}{C_{0, N}}{\partial D_{0, N}} \text{ in } \mathsf N_{\delta}(\mathcal V^{u}) \big]
		\ge 1 - 
		C(u)\exp\{-c(u){N}/{(1 \vee \log N)^{\Cr{C:log-sub}}}\}.
	\end{equation}	
	To obtain  \eqref{eq:connection_delta_d=3}, it is enough to show an analogue of the 
	\emph{seed estimate}  \eqref{eq:trigger_subcrit_rnd1} but replacing the event 
	$\widehat{\mathcal{G}}_{z, L} = \mathcal{G}_{z,L}({\rm Dis}(\widehat{\mathcal V}_{\mathbb 
		L}), \ldots)$ in \eqref{eq:disconnect_Gz}  by an analogue of the local 
	uniqueness event \cite[(5.43)]{gosrodsev2021radius}, involving configurations in 
	$\mathsf{N}_{\delta}(\widehat{\mathcal V}_{\mathbb L}^{\cdot})$ at levels close to $u$ instead 
	of $(\{\chi^z \geq \cdot\})_{z\in \mathbb{L}}$ (cf.~also the event ${\rm V}_z$ defined in 
	\eqref{def:Fz} below). Once this seed estimate is 
	shown, \eqref{eq:connection_delta_d=3} follows in the same way as 
	\eqref{eq:disconnection_improved}. We will not give a proof of the required seed estimate, 
	see \cite[Lemma 5.16]{gosrodsev2021radius} for a similar argument. 
\end{remark}

\section{Coarse-graining for SLU} \label{sec:supcrit_upper_bnd}
The main focus of this section is Theorem~\ref{prop:exploration_RI-I}, which transfers the 
probability bounds in \eqref{eq:supcrit0} at scale $L$ to a similar lower bound on the 
probability of an `${\rm SLU}$-type' event at scale $N \gg L$. This result is instrumental for the proof of Theorem~\ref{T:ri-main} in \S\ref{subsec:slu}. Theorem~\ref{prop:exploration_RI-I} relies crucially on certain novel events, see for instance \eqref{def:Z_tr_RI}, involving the {\em packeting} of excursions. We lay this out in \S\ref{subsubsec:supcrit}. The next two subsections are devoted to the proof of 
Theorem~\ref{prop:exploration_RI-I}, which hinges on several new ideas. In 
\S\ref{subsec:reduction}, we deduce Theorem~\ref{prop:exploration_RI-I} 
from the presence of a large number of {\em good encounter points} between any crossing cluster 
of $\tilde D_{z, N} \setminus \tilde C_{z, N}$ in $\mathcal V(Z)$ and a suitable ambient cluster. Here $Z$ is any sequence of excursions included in the aforementioned {\em packets} and good encounter points are defined in terms of carefully 
designed stopping times 
that allow us to connect two clusters using Proposition~\ref{lem:conditional_prob1}.  
 In \S\ref{subsec:exploration}, we describe a delicate exploration scheme designed to ensure a 
large number of such good encounter points $y$, thus supplying the last missing 
ingredient to the proof of Theorem~\ref{prop:exploration_RI-I}. The most pressing issue is that this exploration needs to be compatible with $\mathcal{F}$ from \eqref{def:FEsigma_algebra1}, see \eqref{property:reduction3} below, which is restrictive.

\subsection{Framework}\label{subsubsec:supcrit} 	
We now prepare the ground for the statement of Theorem~\ref{prop:exploration_RI-I}. 
We will work with two specific sets of data in \eqref{eq:data}, namely $({\rm V}, {\rm W}^{\rm I}, \mathscr{C})$ and $({\rm V}, {\rm W}^{\rm II}, \mathscr{C})$, roughly speaking in 
order to deal with small and large numbers of excursions; 
cf.~\eqref{eq:V_z^I}-\eqref{eq:V_z^II} below. They determine the corresponding families of 
events $\G^{\rm I}$ and $\G^{\rm II}$ from Definition~\ref{def:good_events} via 
\eqref{def:Gz}. We will refer to I and II as \emph{types}. 

We start by collecting here all the parameters that will appear in the sequel. These are:
\begin{equation}\label{eq:params_RI}
	\text{\begin{minipage}{0.58\textwidth}$u_0 \in  (0, \infty)$,
			$u < u_1 < u_2  < u_3  \in  (0, u_\ast)$, 
			$a \in \N^*$, 
			$\nu \ge 0$, $\delta \in [0, \tfrac12)$, scales $N> L> L_0$, $K$ and $\rho$ satisfying \eqref{eq:L_k_descending}. 
			\end{minipage}}
\end{equation}
To keep notations reasonable, we will routinely suppress the dependence of quantities (in particular,~events or sets) on parameters which stay fixed in a given context. 
Recall that $\mathbb{L}=L\Z^d$, that $\widehat{\mathcal V}_{\mathbb L} \in 
\{\overline{\mathcal V}_{\mathbb L}, \mathcal V_{\mathbb L}, \widetilde{\mathcal V}_{\mathbb 
L}\}$ (see below \eqref{eq:3Vs} for notation), that 
$({\mathcal V})_{\delta}$ refers to a noised configuration (see below 
\eqref{def:noised_set_inclusion}) and the boxes $C_z,\tilde{C}_z, D_z,\dots$ from 
\eqref{def:CDU}. We also use $\widehat{Z} = \widehat{Z}_{\mathbb L} = \{ 
\widehat{Z}_z^u: z \in \mathbb{L}, u > 0\}$ to refer {\em collectively} to any of the three 
sequences in \eqref{eq:3Vs}. We parametrize events below in terms of (finite) 
sequences $Z=(Z_j)_{1 \leq j \leq n_Z}$ (and sometimes $Z',Z'',\dots$) of excursions rather 
than through their associated vacant sets $\mathcal{V}=\mathcal{V}(Z)$, 
cf.~\eqref{eq:I(Z)}--\eqref{eq:3Vs}. 
This paradigm shift will be important.
\begin{itemize}
\item 
\textbf{The sets $\mathscr C = \{\mathscr C_z: z \in \mathbb L\}$.} For $Z$ a (finite) sequence of excursions, the set $\mathscr 
C_{z}(Z, Z',\delta)$ ($= \mathscr C_{z, L}(Z, Z',\delta)$) is the subset of $\Z^d$ obtained as the union of 
all clusters in $D_z \cap  (\mathcal{V}(Z))_{\delta}$  containing a crossing of $\tilde D_z \setminus \tilde C_z$ in $({\mathcal V}(Z'))_{2\delta}$. We also set  $\mathscr C_z(Z ,\delta)= \mathscr C_z(Z ,Z,\delta)$, and let
\begin{equation}\label{eq:script_C}
\mathscr C_z 
( = \mathscr C_{z}(\widehat{Z}, \delta, u_1, u_3))
= \mathscr C_{z}(\widehat{Z}_z^{u_1} , \widehat{Z}_z^{u_3}, \delta). 
\end{equation}
\item 
\textbf{The events ${\rm V} = \{{\rm V}_z:  z \in \mathbb L\}$.} For $Z,Z',Z''$ any finite sequences of excursions, set 
\begin{equation*}
{\rm 
V}_{z, L}(Z,Z',Z'', \delta)=	\left\{
	\begin{array}{c}
		\text{$C_z$ is connected to $\partial D_z$ in $({\mathcal V}(Z''))_{2\delta}$ and all clusters of $D_z \cap ({\mathcal V}(Z'))_{2\delta}$}\\ 
		\text{
			 crossing $\tilde D_z \setminus \tilde C_z$ are connected inside $D_z \cap ({\mathcal V}(Z))_{\delta}$}
	\end{array}
	\right\}
\end{equation*}
and abbreviate ${\rm V}_{z,L}(Z, \delta) = {\rm V}_{z, L}(Z, Z,Z,\delta)$. We then set 
\begin{equation}\label{def:Fz}
{\rm 
V}_z
(= {\rm 
V}_z(\widehat{Z}, \delta, u_1, u_2, u_3)) = {\rm 
V}_z(\widehat{Z}_z^{u_1},\widehat{Z}_z^{u_2},\widehat{Z}_z^{u_3}, \delta)
\end{equation}
and abbreviate ${\rm 
V}_z(\widehat{Z}, \delta, u)= {\rm 
V}_z(\widehat{Z}, \delta, u, u, u)$. 
We remove $\delta $ from all notation when $\delta=0$.
\end{itemize}

The remaining events ${\rm W}$ 
in \eqref{eq:data} 
will be introduced shortly. Notice that the set 
$\mathscr C_{z}$ is a connected set on the event ${\rm V}_z$. We are 
ultimately interested in ${\rm 
V}_z(\overline{Z}, u)= {\rm 
V}_z(\overline{Z}, \delta=0, u)$, which deals with the actual vacant set of random interlacements (cf.~\eqref{eq:3Vs}) and entails that there is {\em no} sprinkling.

So far we could have expressed the quantities $\mathscr C_{z}$ and ${\rm V}_z$ directly in terms of vacant sets $\widehat{\mathcal{V}}_{\mathbb{L}}$ via the 
identification $\widehat{\mathcal{V}}^u_z= \mathcal{V}(\widehat{Z}_z^u)$, see below \eqref{eq:3Vs}. In the sequel we will deal with a more general class of events 
involving subsets of excursions for which this factorization property no longer holds. 
To this end, we consider two basic collections of (sub-)sequences of excursions. Given a (finite) sequence $Z=(Z_j)_{1 \leq j \leq n_Z}$  of excursions (see above \eqref{eq:I(Z)}) and $\nu \in [0, \infty]$, we introduce 
\begin{align}
\label{eq:Z^+}
&{Z}_+(\nu) \coloneqq
\text{the collection of all sequences~$(Z_j)_{j \in J}$, $J \subset \{1, \ldots, n_Z\}$ s.t.~$\{1, \ldots, \lfloor \nu \rfloor \} \subset J$}\\
&\label{eq:Z^-}
{Z}_-(\nu) \coloneqq
\text{the collection of all sequences 
$(Z_j)_{j \in J}$, $J \subset \{1, \ldots, n_Z\}$ s.t.~$ J \subset \{1, \ldots, \lfloor \nu \rfloor \} $.}
\end{align}
By convention, ${Z}_+(\nu=0)$ comprises all subsequences of $Z$   whereas  
${Z}_-(\nu=0)$ is empty. 

Rather than dealing directly with ${\rm V}_z(\widehat{Z}_z^u)$, we will bound the complement of a stronger (i.e.~smaller) event involving certain subsequences of the excursions forming ${\rm V}_z(\widehat{Z}_z^u)$, as follows. Given a (finite) sequence $Z=(Z_j)_{1 \leq j \leq n_Z}$  of excursions and any $\nu \in [0, \infty]$,  we let 
\begin{equation}\label{def:Z_tr_RI}
{Z} (\nu) \coloneqq \textstyle \bigcup_
{j \geq 0}\big(Z_+(j) \cap Z_-(j+ \lfloor \nu \rfloor) \big).
\end{equation} 
In words, $Z(\nu)$ denotes the collection of all sequences 
$(Z_j)_{j \in J}$ such that $J \subset \{1, \ldots, n_Z\}$ satisfies 
$\{1, \ldots, j\} \subset J \subset \{1, \ldots, j + \lfloor \nu \rfloor\}$ for some integer 
$j \ge 0$. Roughly speaking, $(Z_j)_{j \in J} \in Z(\nu)$ if $J$ is `almost an interval.' Note that 
$Z(\nu)$ is increasing in $\nu$ and that $Z \in Z(\nu)$ (pick $j=n_Z$) for any $\nu \in [0, 
\infty]$. 
Now, for $\nu \in [0,\infty]$, with the notation from above \eqref{def:Fz}, we introduce
\begin{equation}\label{eq:V-boosted} \textstyle
{\rm V}_z(\widehat{Z}_z^u(\nu)) \coloneqq \bigcap_{{Z} \in \widehat{Z}_z^u(\nu)}  {\rm V}_z({Z}) \ (\subset 
 {\rm V}_z( \widehat{Z},u) ).
\end{equation}
In line with the convention below \eqref{def:Fz}, $\delta=0$ is implicit in \eqref{eq:V-boosted}. In what follows, much as in \eqref{eq:V-boosted}, given an event $E(Z)$ and $\zeta$ a collection of subsequences of $Z$, the event $E(\zeta)$ is declared by setting 
\begin{equation}\label{def:boosted} \textstyle
E(\zeta)=\bigcap_{Z'\in \zeta}E(Z')
\end{equation}
(measurability is never an issue since $Z$ is always a finite sequence). To see the expedience of events like ${\rm V}_z(\widehat{Z}_z^u(\nu))$ 
we present a lemma which will later form the starting point of our proof of Theorem~\ref{T:ri-main}.
\begin{lemma}\label{lem:VzSLU_inclusion}
For any $L \ge 1$ and $u > 0$, with ${\rm V}_{z,L} \equiv {\rm V}_{z}(\overline{Z}_{z,L}^u(\nu 
= 0))$ as defined by \eqref{eq:V-boosted}, one has the inclusion (recall \eqref{eq:slu} regarding ${\rm SLU}_L(u)$)
\begin{equation}
	\label{eq:SLU-reduc0}
	\textstyle\bigcap_{z \in B_{2L} } {\rm V}_{z,{L}/{100}} \subset {\rm SLU}_L(u).
\end{equation}
\end{lemma}	
\begin{proof}
Let $\ell \coloneqq L/100$. Any cluster of $\mathcal{V}^v$ for $v \in 
[0,u]$ in $B_L$ having diameter at least $L/10$
crosses $\tilde D_{z,\ell} \setminus \tilde{C}_{z,\ell} $ for some (nearby) $z\in B_{2L}$. The 
occurrence of ${\rm V}_{z, \ell}$, cf.~above \eqref{def:Fz} and recall that $\delta=0$, along 
with that of other $z$'s in $B_{2L} $, thus allows to connect any two such clusters inside 
$B_{2L}$, provided $\mathcal{V}^v$ can be identified as the vacant set $\mathcal{V}(Z)$ for 
some 
$Z \in \zeta \coloneqq \overline{Z}_{z,\ell}^u(\nu=0)$. But by \eqref{eq:Z^+}-\eqref{def:Z_tr_RI}, 
$\zeta$ comprises all collections of excursions  of the form $(Z_1,\dots Z_j)$ for some $j \leq N_{z,\ell}^u$, where $(Z_j)_{j \geq 1}$ is the sequence from \eqref{eq:RI_Z} with $ D= D_{z,\ell}$ and $ U= U_{z,\ell}$, 
exactly one of which (when $j = N_{z,\ell}^v$) corresponds to the excursions underlying 
$\mathcal{V}^v \cap D_{z, \ell}$ between 
$D$ and $\partial^{{\rm out}} U$. The inclusion \eqref{eq:SLU-reduc0} follows.
\end{proof}
 Our goal is now to bound the probability of ${\rm 
 V}_z(\overline{Z}_z^u(\nu))$ (recall  \eqref{eq:3Vs} concerning $\overline{Z}$). We will 
split this task into two parts, which brings into play the \emph{types} I and II alluded to 
above. Roughly speaking, type I, resp.~II, 
corresponds to the cases that $J$ is `sizeable,' resp.~`small.' More precisely, we let (see 
\eqref{eq:V-boosted} and \eqref{eq:Z^+}--\eqref{eq:Z^-} for notation and recall 
$N_z^u=N_{z,L}^u$ from above \eqref{eq:3Vs})
\begin{align}
\label{eq:V_z^I} &{\rm V}_z^{\rm I}={\rm V}_{z,L}^{\rm I}(\nu; u_0, u) \coloneqq 
{\rm V}_z \big(\overline{Z}_{z}^u(\nu) \cap \big((\overline{Z}_{z}^u)_+(N_{z}^{{u_0}/{2}}) \big) \big), \\
\label{eq:V_z^II} &{\rm V}_z^{\rm II}={\rm V}_{z,L}^{\rm II}(\nu; u_0, u) \coloneqq {\rm V}_z 
\big(\overline{Z}_{z}^u(\nu) \cap \big((\overline{Z}_{z}^u)_-(N_{z}^{{3u_0}/{2}}) \big) \big).
\end{align}
In view of \eqref{eq:V_z^I}-\eqref{eq:V_z^II} and \eqref{eq:Z^+}-\eqref{eq:Z^-}, types I 
and II respectively deal with typical and small numbers of excursions. 
As we now explain, abbreviating ${\rm V}_z= {\rm V}_{z,L}(\overline{Z}_{z,L}^u(\nu))$, 
the event of interest, one has
\begin{equation}\label{eq:V-first-inclusion}
{\rm V}_z^c \cap \big\{ N_z^{{3u_0}/{2}}- N_z^{{u_0}/{2}} > \nu \big\} \subset
({\rm V}_z^{\rm I})^c \cup  ( {\rm V}_z^{\rm II})^c, \text{ for all $\nu \in [1,\infty]$.}
\end{equation}
In view of \eqref{eq:V-boosted}, let $Z \in \overline{Z}_{z}^u(\nu) 
=\overline{Z}_{z,L}^u(\nu) $ be such that $(V_z(Z))^c$ occurs. If $Z \in 
(\overline{Z}_z^u)_+(N_z^{{u_0}/{2}})$, then $({\rm V}_z^{\rm I})^c$ occurs. Otherwise, by 
\eqref{eq:Z^-}-\eqref{def:Z_tr_RI}, $Z=(Z_j)_{j \in J}$ is such that $\{1,\dots, j\} \subset 
\{1,\dots, j+\lfloor\nu \rfloor\}$ for some 
$j \ge 0$ and $j < N_z^{{u_0}/{2}}$. In particular, on 
the event appearing on the left of \eqref{eq:V-first-inclusion}, $J$ is contained in an 
interval of length at most $j+ \nu < N_z^{{u_0}/{2}} + \nu \leq N_z^{{3u_0}/{2}}$, i.e.~$Z 
\in (\overline{Z}_{z}^u)_-(N_z^{{3u_0}/{2}}) $,  and $({\rm V}_z^{\rm II})^c$
occurs.

Theorem~\ref{prop:exploration_RI-I} below yields a crucial estimate on $\P[({\rm V}_{z,N}^{\rm I})^c]$. The event ${\rm V}_z^{\rm II}$  is in fact easier to deal with, and for 
improved clarity all proceedings relating to type II are relegated to Appendix~\ref{subsec:smallu}; see in particular
 Theorem~\ref{prop:exploration_RI-II}, which plays a role analogous to Theorem~\ref{prop:exploration_RI-I} but concerns ${\rm V}_z^{\rm II}$.

We now focus on the event ${\rm V}_z^{\rm I}$. Its occurrence will be bounded in terms of a (good) event $\mathscr{G}_z^{\rm I}$ of the form \eqref{def:script_G}, whose constituent family $\mathcal{G}^{\rm I}$ will be of the form given by Definition~\ref{def:good_events} for suitable data $({\rm V}, {\rm W}^{\rm I}, \mathscr{C})$, with 
$ \mathscr{C}$, ${\rm V}$ as in \eqref{eq:script_C}-\eqref{def:Fz}, and 
events ${\rm W}^{\rm I}$ that we now introduce. 
We declare $ \hat{\nu}_z(u) =\hat{\nu}_{z,L}(u)$ for $u>0$, $z \in \mathbb{L}$ (where, as with $\widehat{Z}$, the hat is a placeholder for three possibilities; cf.~\eqref{eq:3Vs}) as
\begin{equation} \label{eq:nu-hat}
\nu_{z,L}(u)=\tilde{\nu}_{z,L}(u)= u \, \text{cap}( D_{z,L}), \quad \bar{\nu}_{z,L}(u)= N_{z,L}^u.
\end{equation}
\begin{itemize}
\item 
\textbf{The events ${\rm W}^{\rm I} = \{{\rm W}_{z, y}^{\rm I} : z \in \mathbb L, y \in \mathbb L_0\}$.} Let (see below \eqref{eq:V-boosted} for notation) 
\begin{equation} \label{eq:WI}
{\rm W}_{z, y}^{\rm I} \equiv {\rm W}_{z, y}^{\rm I} (\widehat{Z}, u_0, u_1) \coloneqq {\rm FE}_y\big( (\widehat{Z}_z^{u_1})_+ ( \hat{\nu}_z( \tfrac{u_0}{8}))\big),
\end{equation}
where, for any  sequence $Z$ of excursions 
we define the event ${\rm FE}_{y}({Z})$ as follows:
\begin{equation}\label{def:FE1}
{\rm FE}_{y}({Z}) =  
{\rm FE}_{y, L_0}({Z}) = 
{\rm LU}_{y}({Z})  \cap
{\rm O}_{y}({Z})
\end{equation}
with ${\rm O}_{y}({Z}) = {\rm O}_{y, L_0}({Z})$ as in \eqref{def:Conn0} and, 
for $x, x'$ ranging in $(\tilde 
	D_{y} \setminus \tilde C_{y}) \cap \mathcal I(Z)$ below, 
\begin{equation}\label{def:Conn}
\begin{split}
&{\rm LU}_{y}({Z}) = {\rm LU}_{y, L_0}({Z})\coloneqq\textstyle\bigcap_{x,x'} \{\lr{}{}{x}{x'} \text{ in } \mathcal I({Z}) \cap \, (D_{y} \setminus (\partial D_y \cup C_{y}))\}.
\end{split}	
\end{equation}
\end{itemize}
The above data set $({\rm V}, {\rm W}^{\rm I}, \mathscr{C})$  leads to the well-defined event (recall \eqref{def:Gz} for the right-hand side) 
\begin{equation}\label{def:good_events_supcrit}
\mathcal{G}_z^{\rm I}(\widehat{Z}, \delta, u_0, u_1, u_2, u_3; a) \coloneqq 
\mathcal G_z({\rm V}, {\rm W}^{\rm I}, \mathscr C ; a),
\end{equation}
with ${\rm V}_z= {\rm 
V}_z(\widehat{Z}, \delta, u_1, u_2, u_3)$ given by \eqref{def:Fz}, 
${\rm W}^{\rm I}= {\rm W}^{\rm I} (\widehat{Z}, u_0, u_1)$ given by \eqref{eq:WI}, and $
\mathscr C_z = \mathscr C_z(\widehat{Z}, \delta, u_1, u_3)$ given by \eqref{eq:script_C}. 
Finally, $\mathscr{G}_{z,N}^{\rm I} (\widehat{Z}, \delta, u_0, u_1, u_2, u_3; a)$,~$z \in N \Z^d$, is defined as
\begin{equation}\label{def:mathscrG_RI-gen}
\mathscr G\big(\tilde D_{z, N} \setminus \tilde C_{z, N},  
\mathcal{G}^{\rm I}= \big\{\mathcal{G}_{z'}^{\rm I}(\widehat{Z},  \delta, u_0, u_1, u_2, u_3; a): z' \in 
\mathbb L\big\}, \mathcal F_L^{\bm u_1, \bm u_2} ; \rho = \tfrac1{2\Cr{C:rho}}\big)
\end{equation}
(recall \eqref{def:script_G} and \eqref{def:good_events_supcrit}), where $\mathcal F_L^{\bm 
u_1, \bm u_2} $ is given by \eqref{eq:Fext} 
and \eqref{eq:bootstrapped_limit_probd=4}) and $\bm u_1$, $\bm u_2$ are as follows:
\begin{equation}\label{def:F_y_RI}
\bm u_1 = (u, u_{2,3} \coloneqq\tfrac{u_2+u_3}{2}, u_{2,3}, \tfrac{u_0}2) \mbox{ and } \bm u_2  = (u_1, u_2, u_3, \tfrac{u_0}8).
\end{equation}
The next proposition is the announced estimate for $\P[({\rm V}_{z,N}^{\rm I})^c]$, 
expressed in terms of $\P\big[(\mathscr G_{z,N}^{\rm I})^c\big]$ (later controlled by means of Proposition~\ref{prop:bootstrap_prob}), where (see \eqref{eq:3Vs} regarding $Z_{\mathbb L}=\{Z_{z',\mathbb L}^u: u > 0 , z' \in \mathbb{L} \}$)
\begin{equation}\label{def:mathscrG_RI}
\mathscr{G}_{z,N}^{\rm I} \coloneqq \mathscr{G}_{z,N}^{\rm I} ({Z}_{\mathbb L}, \delta, u_0, u_1, u_2, u_3; a), \quad z \in N\Z^d
\end{equation}
 
\begin{theorem}[Coarse-graining for ${\rm V}_{z,N}^{\rm I}$]
\label{prop:exploration_RI-I}
Under \eqref{eq:params_RI} and for $\delta > 0$, $\nu \ge 
0$,  $2u_0  < u_\ast$, 
and $L \ge C(u_0)$, there exists $c = c(\delta, L_0) > 0$ 
such that, for $z \in N\Z^d$, with $h(x) = x(1 + (\log x)^21_{d \ge 4})$ as in \S\ref{subsec:admissible},
\begin{equation}\label{eq:truncated_RI-I}
	\P[({\rm V}_{z,N}^{\rm I})^c] \leq  \P\big[(\mathscr G_{z,N}^{\rm I})^c\big] 
	   + \P\big[\nlr{}{}{ C_{z, N}}{\partial  D_{z, N}} \text{ in } (\mathcal V^{u_{2, 3}})_{2\delta}\big] 
	 + e^{-c(a 
	 	\frac N{h(KL)}\wedge N)  +  C(\nu + \log N)}.
\end{equation}
\end{theorem}

\subsection{
Gluing clusters via good encounter points}\label{subsec:reduction}
In this subsection, we derive Theorem~\ref{prop:exploration_RI-I} from the existence of a carefully designed sequence $\tau=(\tau_k)_{k \geq1}$ of {\em good 
encounter times} attached to the dynamics of a cluster exploration algorithm. Their existence is the content of Proposition~\ref{prop:reduction} below. Throughout \S\ref{subsec:reduction}--\ref{subsec:exploration}, we assume that 
the assumptions of Theorem~\ref{prop:exploration_RI-I} hold; in particular, all parameters (such as $N, L,\dots$) below satisfy 
\eqref{eq:params_RI}. Let $J \subset \N^\ast=\{1,2.,\dots\}$ be finite. We abbreviate $Z_J= (Z_j^{ D_{z, N},  U_{z, N}})_{j \in J}$, for a fixed $z \in N\Z^d$. We omit 
$N$ in all the subscripts involving $z$ like $\tilde C_{z, N}$, 
$N_{z, N}^{u}$ etc. For a point $x \in \partial \tilde{C}_z$, we let $\mathscr C_J(x)$ denote the (possibly empty)
cluster of $x$ inside $D_z \cap \mathcal V({Z}_J)$ with $\mathcal V({Z}_J)$ given by \eqref{eq:I(Z)}.
By convention, we set 
$\partial_{D_z}^{{\rm out}}\mathscr C_J(x)=\{x\}$ in the sequel whenever $\mathscr C_J(x) = \emptyset$.

The exploration algorithm consists of a sequence $(w_n)_{n \ge 1}$ of $\Z^d$-valued random variables on the space $(\Omega, 
\mathcal{A}, \P)$, see 
below \eqref{eq:eBeB}, where $w_1 \coloneqq x $ and, if $\textstyle \bigcup_{1 \le i \le n-1}\{w_i\}=\mathscr C_J(x) \cup \partial_{D_z}^{{\rm 
out}}\mathscr C_J(x)$  (an event measurable relative to the random variables $(w_i, 
1_{\{w_{i} \in \mathcal V(Z_J)\}}, 1 \leq i  < n)$) for some $n \geq 2$, then $w_n \coloneqq w_{n-1}$. Else, $w_n$ is the smallest point (in a given deterministic ordering of 
$\Z^d$) in $\textstyle \Z^d \setminus \bigcup_{1 \le i \le n-1}\{w_i\}$ that lies on 
$\partial_{D_z}^{{\rm out}} \mathscr C_{J, n}(x)$, where $\mathscr C_{J, n}(x)$ is the cluster 
of $x$ in $\textstyle \bigcup_{1 \le i \le n-1}\{w_i\} \cap \mathcal V(Z_J)$. It is clear that 
$w_n \notin \bigcup_{1 \le i \le n-1}\{w_i\}$ as long as $\bigcup_{1 \le i \le n-1}\{w_i\}$ is a 
proper subset of $\mathscr C_J(x) \cup \partial_{D_z}^{{\rm out}}\mathscr C_J(x)$. Since 
this set is finite, $\P$-a.s.~$w_{n+1}=w_n$ for large enough $n$, i.e.~the 
exploration is complete in finite time.

The aforementioned variables $(\tau_k)_{k \geq1}$ will be coupled to the exploration algorithm $(w_n)_{n \ge 1}$. They roughly act as 
{\em stopping times} for the underlying exploration process. Recall that $\mathbb L=L\Z^d$ and 
the event ${\rm V}_{z',L}$ given by \eqref{def:Fz}. 
For fixed $z$ (as in the statement of Theorem~\ref{prop:exploration_RI-I}), consider the
(random) set
\begin{equation}\label{def:S}
	\Sigma\coloneqq
	\big\{z' \in \mathbb L 
	:  
	\text{
		$D_{z', L} \subset \tilde D_z \setminus \tilde 
		C_z$ and  
		${\rm V}_{z', L}(\overline Z_{\mathbb L}, \delta, u, u_{2, 3}, u_{2, 3})$ occurs}\big\},
\end{equation}
where, as for the remainder of this section, the parameters $u_0, u, u_{2, 3}$ etc. carry the 
same meaning as in \eqref{eq:params_RI} and \eqref{def:F_y_RI}. 
We now apply Proposition~\ref{lem:surroundingInterfaces} on $\mathbb{L}$ instead of 
$\Z^d$ 
with the choices $U= \{ z' \in \mathbb{L}: C_{z',L}\cap \tilde C_z \neq \emptyset\} $, $V= \{ z' \in \mathbb{L}: C_{z',L}\cap \tilde D_z \neq \emptyset\} $ and $\Sigma$ as in \eqref{def:S}. We refer to 
$k(\Sigma)$ as the maximal value of $k \geq 0$ such that the assumptions of Proposition~\ref{lem:surroundingInterfaces} are met with these choices, and denote by $O_1, 
\ldots, O_\ell \subset \Sigma$ the $\ast$-connected (as subsets of $\mathbb{L}$) sets thus obtained when choosing $k = k(\Sigma)$. It may well be that $k(
\Sigma)=0$ in which case $\ell=0$ and $\bigcup_{1 \le j \le \ell} \, O_j=\emptyset$ by convention in what follows.
Using the sets $O_1, \ldots, O_\ell$, we can define a special property of a point $y \in \mathbb L_0 = L_0\Z^d$ (cf.~\eqref{eq:params_RI}) as follows:
\begin{equation}\label{def:Cy}
	\text{$C_{y, L_0}\cap \mathscr C_{z', L} \neq \emptyset$  for some $z' \in \textstyle \bigcup_{1 \le j \le \ell} \, O_j$ 
			and $\widetilde{{\rm FE}}_{y, L_0}(Z_{J})$ occurs,}
\end{equation}
where $\mathscr C_{z', L} = 
	\mathscr C_{z'}(\overline{Z}_{\mathbb L}, \delta, u, u_{2, 3})$ and $\widetilde{{\rm FE}}_{y, L_0}(Z_{J})\coloneqq \widetilde{\rm 
				LU}_{y, L_0}\left({Z}_J\right) \cap {\rm O}_{y, L_0}(\overline{Z}_z^u)$; recall \eqref{def:tildeLU1} and \eqref{def:Conn0} for $\widetilde{{\rm LU}}_{y, L_0}$ and ${\rm O}_{y, L_0}$ respectively and compare with ${\rm FE}_{y, 
	L_0}(Z)$ in \eqref{def:FE1}. 
	
It will be expedient to partition $\mathbb{L}_0$ as follows. 
For $y \in \mathbb{L}_0 \cap D_{0,L_0}$, let $\mathbb{L}_{0,y}= y +7L_0 \mathbb Z^d$. By \eqref{def:CDU}, 
the sets $\mathbb{L}_{0,y}$ partition $\mathbb{L}_0$ as $y$ ranges over $\mathbb{L}_0 \cap 
D_{0,L_0}$. Furthermore, given $y$, the boxes $D_{y',L_0}$ with $y' \in 
\mathbb{L}_{0,y} $ form a partition of $\Z^d$. Let ${\widehat{\mathbb{L}}}_{0,y}$ be obtained 
from ${\mathbb{L}}_{0,y}$ by removing all points $y'$ such that $D_{y',L_0}$ intersects 
$\widetilde{C}_{z}$. Let $y \in \mathbb 
L_0 \cap D_{0, L_0}$. We call  $\tau =(\tau_k)_{k \ge 1} \equiv(\tau_{k; J, y}(x))_{k \ge 1}$ a sequence of {\em good encounter times (for $\mathscr C_J(x)$)} if 
$\tau$ is a non-decreasing sequence of
$\N^\ast \cup \{\infty\}$-valued random variables on the space $(\Omega, 
\mathcal{A}, \P)$ defined 
in \S\ref{subsec:occupation} satisfying the following three conditions:
\begin{align}
	\text{\begin{minipage}{0.90\textwidth}  If $\tau_k < \infty$, then 
			$w_{\tau_k} \in \partial D_{Y_k, L_0}$ for some (unique) $Y_k \in \widehat{\mathbb{L}}_{0,y}$ s.t.~$Y_k$ 
			satisfies~\eqref{def:Cy}. 
			Conversely, if $y' \in \widehat{\mathbb{L}}_{0,y}$ satisfies~\eqref{def:Cy} and 
			$C_{y', L_0}\cap \mathscr C_J(x) \neq \emptyset$, $ \exists k \ge 1$ s.t.~$\tau_k < \infty$, $ Y_k=y'$. \end{minipage}}\label{property:reduction1}\\[0.5em]	
	\text{\begin{minipage}{0.9\textwidth} If $\tau_k < \infty$ and $C_{Y_k, L_0} \subset \mathcal V(Z_J)$, then $w_{\tau_k}$ is connected to $C_{Y_k, L_0}$ 
			in $D_{Y_k, L_0} \cap \mathcal V(Z_J)$. 
	\end{minipage}}\label{property:reduction2} \\[0.5em]	
	\text{\begin{minipage}{0.9\textwidth} For any $y' \in \widehat{\mathbb L}_{0, y}$ and $k \ge 0$, the event $\textstyle\bigcap_{1 \le j \le k}\{\tau_j < \infty, C_{Y_j, L_0} \not\subset \mathcal V(Z_J)\} \cap \{\tau_{k + 1} < \infty, Y_{k+1} = y'\}$ is measurable rel.~to the $\sigma$-algebra $\mathcal F_{y', L_0}(Z_J, \delta, u, u_{2, 3})$ defined in \eqref{def:FEsigma_algebra1}. \end{minipage}}\label{property:reduction3}
\end{align}

Following is the main result of this section. Recall that the assumptions of Theorem~\ref{prop:exploration_RI-I} hold.
\begin{prop}\label{prop:reduction}
	For all finite $J \subset \N^\ast$, $y \in \mathbb 
	L_0 \cap D_{0, L_0}$ and  $x \in \partial \tilde C_z$, there exists a sequence of good encounter times $(\tau_k)_{k \ge 1}$ for $\mathscr C_J(x)$.
\end{prop}
We will prove this result in \S\ref{subsec:exploration} and proceed with the proof of Theorem~\ref{prop:exploration_RI-I}. The 
{\em good encounter points} correspond to the points $Y_k$ in 
\eqref{property:reduction1} when $\tau_k < \infty$. So far we have not said much on the events $\mathscr 
G_{z}^{{\rm I}}$ and $\{\lr{}{}{ C_{z}}{\partial  
	D_{z}} \text{ in } (\mathcal V^{u_{2, 3}})_{2\delta}\}$ whose complements appear on the 
right-hand side of \eqref{eq:truncated_RI-I}. Our next lemma 
connects these two events to the finiteness of $\tau_k$ in Proposition~\ref{prop:reduction} for 
a {\em large} value of $k$.
\begin{lemma}\label{lem:reduction}
	There exists 
	$\Cl[c]{c:reduce} \in (0, \infty)$ such that, for $J,y,x$ as above, letting  
	\begin{equation}\label{def:abc}
		A_{J, y}(x) \coloneqq \left\{ \tau_{\lceil \Cr{c:reduce} am \rceil; J, y}(x)  < \infty\right\}, 
	\end{equation}
	where $a (\in \mathbb{N}^*)$ enters the definition of $\mathscr{G}_{z,N}^{\rm I}$ (see \eqref{def:mathscrG_RI}) and $m$ is the common cardinality of coarsenings in 
	$\mathcal A_L^K(\tilde D_{z} \setminus \tilde C_{z})$ (cf.~Prop.~\ref{prop:coarse_paths}),
	one has the inclusion, with $y$ ranging over $\mathbb L_0 \cap D_{0, L_0}$ below,
	\begin{equation}\label{eq:inclusion_tauk_fin}
		\mathscr G_z^{{\rm I}} \cap 
		\{\lr{}{}{x}{\partial \tilde D_z} \text{ in } (\mathcal V^{u_{2, 3}})_{2\delta} \} \cap \big\{ [1, N_z^{{u_0}/2}] \subset J \subset [1, N_z^{u}]\big\} \subset \textstyle
		\bigcup_{y} A_{J, y}(x).
	\end{equation}
\end{lemma}
In \eqref{eq:inclusion_tauk_fin} and below we tacitly identify an interval $[a,b] \subset \mathbb{R}$ with $[a,b]\cap \mathbb{Z}$ and $[1,0]=\emptyset$ by convention. 
We will prove Lemma~\ref{lem:reduction} at the end of this subsection. We now proceed with the:
\begin{proof}[Proof of Theorem~\ref{prop:exploration_RI-I}]
	For any finite $J \subset \N^\ast$, we claim that
	\begin{equation}\label{eq:explore_hard}
		\P\big [ ({\rm V}_z(Z_J))^c , \,  \mathscr G_z^{{\rm I}}, \,  \lr{}{}{ C_{z}}{\partial  D_{z}} \text{ in } (\mathcal V^{u_{2, 3}})_{2\delta}, \,  [1, N_z^{{u_0}/2}] 
		\subset J \subset [1, N^u_z] \big] 	 \le  C N^{d-1} e^{- c(\delta, L_0)am}
	\end{equation}	
	with $m$ as below \eqref{def:abc}. Let us quickly conclude the proof assuming this bound. 
	Recalling the definition of the event ${\rm V}_z^{{\rm I}}$ from \eqref{eq:V_z^I} which 
	depends on the collection $\overline{Z}_z^u(\nu) \cap 
	\big((\overline{Z}_z^u)_+(N_z^{{u_0}/{2}}) \big)$ defined in 
	\eqref{eq:Z^+}--\eqref{def:Z_tr_RI} (see also 
	\eqref{def:boosted}), we deduce the inclusion
	\begin{multline}\label{eq:inclusion_large}
		({\rm V}_z^{{\rm I}})^c \subset \big( (\mathscr G_z^{\rm I})^c \, \cup \, \big\{\nlr{}{}{ C_{z}}{\partial  D_{z}} \text{ in }(\mathcal V^{u_{2, 3}})_{2\delta}\big\} \, \cup \, \big(  \mathcal F_z^{u, 2u_\ast} \big)^c  \cup
		\\ \textstyle \bigcup_{J} \big(({\rm V}_z(Z_J))^c \cap \mathscr G_z^{\rm I} \cap \big\{\lr{}{}{ C_{z}}{\partial  D_{z}}\text{ in } (\mathcal V^{u_{2, 3}})_{2\delta}\big\} \cap \big\{ [1, N_z^{u_0}] \subset J \subset [1, N^u_z]\big\} \big)\big),
	\end{multline}
	with $J$ ranging over all sets of the form $\{1, \ldots, j \} \subset J \subset \{1 ,\ldots,\, j + \nu\}$ for some $j \geq 1$ with $j + \nu \le 2u_\ast {\rm cap}( D_z)$; to obtain~\eqref{eq:inclusion_large} it suffices to note that $N_z^u \leq 2u_* \text{cap}(D_z)$ on the event  
	$\mathcal F_z^{u, 2u_\ast}$, see \eqref{eq:F}, whence any package $Z_J=(Z_j)_{j \in J}$ 
	belonging to $\overline{Z}_z^u(\nu)$, which by definition comprises excursions drawn from 
	$\overline{Z}_z^u$ in \eqref{eq:3Vs},  have label at most $2u_* \text{cap}(D_z)$.  Taking 
	expectations and applying a union bound in \eqref{eq:inclusion_large}, the claim 
	\eqref{eq:truncated_RI-I} readily follows upon using that $\P[ (\mathcal F_z^{u, 2u_\ast})^c ] {\le} e^{-c N^{d-2}}$ by \eqref{eq:Nuz_tail_bnd} (and since $u \in (0, u_\ast)$, see 
	\eqref{eq:params_RI}), observing that the number of terms in the resulting summation over 
	$J$ is at most $C \, 2^\nu {\rm cap}( D_z) \leq  C 2^\nu N^{d-2},$
	and combining this with \eqref{eq:explore_hard} to deduce that the probability of 
	the event in the second line of \eqref{eq:inclusion_large} is bounded by (with the sum below 
	running over $J$ as in \eqref{eq:inclusion_large})
	$$
	\textstyle \sum_{J} C N^{d-1}e^{- c(\delta, L_0)am} \leq  \exp\left\{-c(\delta, L_0)(am \wedge N)  +  C(\nu + \log N) \right\}, 
	$$
	which leads to the last term in \eqref{eq:truncated_RI-I} on account of \eqref{def:coarse_admissible1}, by which $m \geq c N/h(KL)$.
	
	It remains to show \eqref{eq:explore_hard}. Towards this, we will show an intermediate statement
	which is formally the same as \eqref{eq:explore_hard} but with the event ${\rm V}_z(Z_J)$ replaced by $\widetilde {\rm V}_{z}(Z_J) \coloneqq 
	\bigcap_{x \in \partial \tilde C_z} \widetilde {\rm V}_{z, x}(Z_J)$, where \begin{equation}\label{def:widetildeVx}
		\begin{split}
			\widetilde {\rm V}_{z, x}({Z}_J) 
			\coloneqq
			\big\{\nlr{}{}{x}{\partial \tilde D_z} \text{ in }\mathcal V({Z}_J)\big\} \cup \big\{\lr{}{}{x}{\textstyle (\bigcup_{z'} \mathscr C_{z', L})}  \text{ in } D_z \cap \mathcal V({Z}_J)\big\};
		\end{split}
	\end{equation}
	here $\mathscr C_{z', L} = \mathscr C_{z'}(\overline{Z}_{\mathbb L}, \delta, u, u_{2, 3})$ and the union is over $z' \in \bigcup_{1 \le j \le \ell} \, O_j$ with $O_j$ as introduced below \eqref{def:S}. The bound \eqref{eq:explore_hard}~follows from its version for $\widetilde {\rm V}_{z}(Z_J)$  and the following inclusion 
	of events:
	\begin{equation}\label{eq:onearm_twoarm}
		\widetilde {\rm V}_{z}(Z_J)  
		\cap \{\lr{}{}{ C_z}{\partial  D_z} \text{ in }(\mathcal V^{u_{2, 3}})_{2\delta}\} \cap \{J \subset [1, N_z^u]\} \subset {\rm 
			V}_z(Z_J).
	\end{equation}
	Let us first derive the inclusion \eqref{eq:onearm_twoarm}. Recall the definition of the event ${\rm V}_{z', L} = {\rm V}_{z', L}(\overline Z_{\mathbb L}, \delta, u, u_{2, 3}, u_{2, 3})$ \sloppy from \eqref{def:Fz} and also 
	the 
	set $\mathscr C_{z', L} = \mathscr C_{z'}(\overline{Z}_{\mathbb L}, \delta, u, u_{2, 3})$ from 
	\eqref{eq:script_C} for $z' \in \mathbb L$. It follows from these two definitions combined with: (a) $\mathcal V(\overline{Z}_{z', L}^v) \cap D_{z', L} = \mathcal V^v \cap 
	D_{z', L}$ for any $v \ge 0$ (see~\eqref{eq:3Vs}), 
	(b) $(\tilde D_{z', L} \setminus \tilde C_{z', L}) \subset (D_{z'', L} \setminus C_{z'', L})$ \sloppy 
	for any $|z' - z''|_\infty \le L$ (recall \eqref{def:CDU}) and (c) the inclusion $(\mathcal V^u)_\delta \subset \mathcal V^u$ (see~\eqref{def:noised_set_inclusion}) that, whenever $|z' - z''|_{\infty} \le L$ 
				and ${\rm V}_{z', L} \cap {\rm V}_{z'', L}$ occurs,
	\begin{equation}\label{eq:z'z''}
		\text{$\mathscr C_{z', L}$ and $\mathscr C_{z'', L}$ are (non-empty and)
				connected in $(D_{z', L} \cup D_{z'', L}) \cap \mathcal V^u$}
	\end{equation}
	In particular, \eqref{eq:z'z''} applies by \eqref{def:S} when $z,z'$ are $*$-neighbors in $\Sigma$.
	Recalling from the paragraph containing \eqref{def:S} that each $O_j$ is a $\ast$-connected 
	subset of $\Sigma$, 
	\eqref{eq:z'z''} thus implies that
			the set $\bigcup_{z' \in O_j}\mathscr C_{z', L}$ is connected in $D_z \cap \mathcal V^u$ for each $1 \le j \le \ell$, 
			where we also used 
			that $D_{z', L} \subset D_z$ for any $z' \in \Sigma$ 
			(see~\eqref{def:S}). The sets $O_1, 
			\ldots, O_\ell$ also satisfy property~(a) in Proposition~\ref{lem:surroundingInterfaces} with 
			$\mathbb L$ as the underlying lattice and hence any crossing of $\tilde D_z \setminus \tilde 
			C_z$ must necessarily cross $\tilde D_{z', 
				L} \setminus \tilde C_{z', L}$ for some $z' \in O_j$ and each $1 \le j \le \ell$. Furthermore, if this 
			crossing lies in $(\mathcal V^{u_{2, 3}})_{2\delta}$, then it must be connected to $\mathscr 
			C_{z'}$ in $(\mathcal V^u)_{\delta}$ (and hence in $\mathcal V^u$) by the definition of 
			$\mathscr C_{z', L}$ (revisit \eqref{eq:script_C}). Combined with the last two 
			displays, this yields that the set $\bigcup_{z'}\mathscr C_{z', L} 
			\subset D_{z}$ with $z'$ ranging in $\bigcup_{1 \le j \le \ell} \, O_j$ is connected in $\mathcal V^u$ { on} the event $\{\lr{}{}{ C_z}{ \partial D_z} \text{ in } (\mathcal V^{u_{2, 
						3}})_{2\delta}\}$. 
			Now together with the definitions of 
			$\widetilde {\rm V}_{z, x}(Z_J)$ and $\widetilde {\rm V}_{z}(Z_J)$ in and above 
			\eqref{def:widetildeVx} and 
			of ${\rm V}_{z}(Z_J)$ above \eqref{def:Fz}, the previous display yields 
			\eqref{eq:onearm_twoarm}.

			It remains to prove \eqref{eq:explore_hard} in its version for $\widetilde{\rm V}_z$. In view of the definition of the event 
			$\widetilde{\rm V}_z(Z_J)$ above \eqref{def:widetildeVx} 
			and the inclusion \eqref{eq:inclusion_tauk_fin} in Lemma~\ref{lem:reduction}, it suffices to 
			show that for all $x \in \partial \tilde C_z$, $y \in \mathbb L_0 \cap D_{0, L_0}$,
			\begin{equation}\label{eq:explore_hard_tilde_x}
				\begin{split}	
					\P\big [ (\widetilde{\rm V}_{z, x}(Z_J))^c \cap A_{J, y}(x) \cap \big\{J \subset [1, N^u_z]\big\}\big] \le  e^{- c(\delta, L_0)am}.
				\end{split}
			\end{equation}
			The desired bound \eqref{eq:explore_hard} with $\widetilde{\rm V}_z$ instead of ${\rm V}_z$ 
			follows from \eqref{eq:explore_hard_tilde_x} via a union bound applied {\em first} 
			over $y \in \mathbb L_0 \cap D_{0, L_0}$ for {\em given} $x \in \partial \tilde C_z$ and 
			then over $x \in \partial \tilde C_z$ (recall \eqref{def:CDU} for the cardinality of these 
			sets).  To show \eqref{eq:explore_hard_tilde_x}, we first claim that for any $x \in \partial \tilde C_z$ and $y \in \mathbb L_0 \cap D_{0, L_0}$,
			\begin{equation}\label{eq:inclusion_bad_verybad}
				(\widetilde {\rm V}_{z, x}({Z}_J))^c \cap \{J \subset [1, N^u_z]\} \subset \left\{C_{Y_k, L_0} \not\subset \mathcal V(Z_J) \mbox{ for any $k \ge1$ s.t.~}\tau_{k} < \infty\right\},
			\end{equation}
			where $Y_k$ is as in \eqref{property:reduction1} and $\tau_k \equiv \tau_{k; J, y}(x)$ is supplied by Prop.~\ref{prop:reduction}. 
			To verify this, note that by \eqref{def:widetildeVx}, 
			\begin{equation}\label{eq:abc1}
				(\widetilde {\rm V}_{z, x}({Z}_J))^c \subset \big\{\nlr{}{}{x}{\mathscr C_{z', L}} \text{ in } D_z \cap \mathcal V({Z}_J), \mbox{ for any }z' \in \textstyle \bigcup_{1 \le j \le \ell} \, O_j\big\}.
			\end{equation}
			Now if $\tau_{k} < \infty$ and $C_{Y_k, L_0} \subset \mathcal V(Z_J)$, then by~\eqref{property:reduction2}, $w_{\tau_k}$ (part of the exploration process $(w_n)_{n \geq 1}$ for $\mathscr{C}_J(x)$, the cluster of $x$ in $D_z \cap \mathcal V({Z}_J)$) is connected to $C_{Y_k, L_0}$ in $D_{Y_k, L_0} \cap \mathcal V(Z_J)$. On the other hand, by~\eqref{property:reduction1}, $Y_k$ satisfies~\eqref{def:Cy} and therefore $C_{Y_k, L_0}$ intersects $\mathscr C_{z', L} \subset D_{z', L} \subset \tilde D_z= \tilde D_{z,N}$ for some $z' \in \bigcup_{1 \le j \le \ell} 
			\, O_j$ (see \eqref{eq:script_C} and \eqref{def:S} for the two inclusions). These two observations imply that 
			\begin{multline}\label{eq:abc2}
				\left\{C_{Y_k, L_0} \subset \mathcal V(Z_J) \mbox{ for some $k \ge1$ s.t.~}\tau_{k} < \infty\right\} \cap \{J \subset [1, N^u_z]\}\\
				\subset \big\{\lr{}{}{x}{\mathscr C_{z', L}} \text{ in }D_z \cap \mathcal V({Z}_J), \text{ for some } z' \in \textstyle \bigcup_{1 \le j \le \ell} \, O_j \big\},
			\end{multline}
			{\em provided} $D_{Y_k, L_0} \subset D_z$ on the event $\{\tau_{k} < \infty, C_{Y_k, L_0} \subset \mathcal 
			V(Z_J)\}$. But the inclusion $D_{Y_k, L_0} \subset D_z$ follows from our earlier observation that $C_{Y_k, L_0}$ intersects $\mathscr C_{z', L} \subset \tilde D_z$ together with the definitions of 
			the boxes $\tilde D_z = \tilde D_{z, N}$, $D_z = D_{z, N}$, $C_{Y_k, L_0}$ and $D_{Y_k, L_0}$ in \eqref{def:CDU} and the fact that $N \ge 10^3 L_0$ which 
			is a consequence of \eqref{eq:L_k_descending} as part of our assumption \eqref{eq:params_RI}. Together,  \eqref{eq:abc1} and \eqref{eq:abc2} imply 
			\eqref{eq:inclusion_bad_verybad}.
			
			With \eqref{eq:inclusion_bad_verybad} at hand, recalling $ A_{J, y}(x) $ from \eqref{def:abc}, we see that the intersection $\widetilde{\rm V}_{z, x}(Z_J))^c \cap A_{J, y}(x)$
			as appearing in \eqref{eq:explore_hard_tilde_x} implies $\mathcal E_{\lceil \Cr{c:reduce} a m\rceil}$, where $\mathcal E_b \coloneqq \left\{ \tau_{k} < \infty \mbox{ and } C_{Y_k, L_0} \not\subset \mathcal V(Z_J) \mbox{ for all $k \le b$}\right\}$, for any integer $b \geq 1$. Hence \eqref{eq:explore_hard_tilde_x} follows immediately from the bound
			\begin{equation}\label{eq:explore_tilde2}
				\P \big[  
				\mathcal E_{\lceil \Cr{c:reduce} a m\rceil} \cap \big\{J \subset [1, N^u_z]\big\} \big]  \le  e^{- 
					c(\delta, L_0)am}.
			\end{equation}
			We will set up a recursive inequality in $b$ for this probability (with $\lceil \Cr{c:reduce} a m\rceil$ replaced by $b$) using Proposition~\ref{lem:conditional_prob1} along the way. 
			Letting $p_{b}\equiv \P [  
			\mathcal E_{b} \cap \{J \subset [1, N^u_z]\} ]$, we have, with $\Sigma$ ranging over $y ' \in \widehat{\mathbb{L}}_{0,y}$,
			%
			\begin{multline*}
				p_{b+1} 
				=\sum
				\P\big[
				\mathcal E_{b}  \cap \big\{C_{y', L_0} \not\subset \mathcal V(Z_J)\big\} \cap \big\{J \subset [1, N_z^u], \tau_{b+1} < \infty,  Y_{b+1} = y'\big\}\big]\\
				\stackrel{\eqref{property:reduction3} + \eqref{def:FEsigma_algebra1}}{=} \sum
				 \E\left[ \P\big[C_{y', L_0} \not\subset \mathcal V(Z_J)  \, \big | \, \mathcal F_{y, L_0}(Z_J, \delta, u, u_{2, 3}) \big] 1_{
					\mathcal E_{b} \cap  \{J \subset [1, N_z^u], \tau_{b+1} < \infty,  Y_{b+1} = y'\}}\right]\\
				\stackrel{\eqref{eq:conditional_prob1}}{\le} (1 -  c) \sum
				 \P\big [
				\mathcal E_{b} \cap \big\{\tau_{b+1} < \infty,  Y_{b+1} = y'\big\} \cap \big\{J \subset [1, N_z^u]\big\} \big] 
				\le \,(1 -c) p_b,
			\end{multline*}
			where $c = c(\delta, L_0) \in (0, 1)$ is from Proposition~\ref{lem:conditional_prob1}. In the third step, we used that 
			$\{\tau_{b+1} < \infty, Y_{b+1} = y'\} \subset \widetilde{\rm LU}_{y'}({Z}_J) \cap {\rm O}_{y'}(\overline{Z}_z^{u})$
			as well as the set inclusion 
			$ D_{y, L_0} \subset  D_{z}$, 
			as needed for \eqref{eq:conditional_prob1} to apply. The inclusion of events is a 
			consequence of~\eqref{property:reduction1} and~\eqref{def:Cy}. The inclusions of the boxes, 
			on the other hand, follow from an argument similar to that used at the end of the paragraph containing \eqref{eq:abc2}. Iterating the previous inequality then yields that the left-hand side of \eqref{eq:explore_tilde2} is bounded by $ (1 - c)^{\Cr{c:reduce} am}$.
		\end{proof}
		
		It remains to prove Lemma~\ref{lem:reduction}. The following result will be useful.
		\begin{lemma}\label{lem:LULUtilde}
			For any sequence $Z = (Z_j)_{1 \le j \le n_Z}$ of excursions and $y \in \mathbb L_0$, one has 
			the inclusion
			\begin{equation}\label{eq:LULUtilde}
				{\rm LU}_{y, L_0}(Z) \subset \widetilde {\rm LU}_{y, L_0}(Z) \quad (\text{see \eqref{def:Conn} and \eqref{def:tildeLU1}}).
			\end{equation}
		\end{lemma}
		\begin{proof}
			Suppose we are on the event ${\rm LU}_{y, L_0}(Z)$ and $x, x' \in \mathscr C_{\partial D_{y, L_0}}(Z) \cap (\tilde D_{y, L_0} \setminus \tilde C_{y, L_0})$. Since $\mathscr C_{\partial D_{y, 
					L_0}}(Z) \subset \mathcal I(Z)$ (see \eqref{def:C_partial1}), it then follows from the definition 
			of ${\rm LU}_{y}(Z)$ in \eqref{def:Conn} that $x$ and $x'$ lie in the same component of $\mathcal 
			I(Z) \cap (D_{y, L_0} \setminus (\partial D_{y, L_0} \cup C_{y, L_0}))$. Recalling 
			\eqref{def:C_partial1} and that $x, x' \in \mathscr C_{\partial D_{y, L_0}}(Z)$, we 
			conclude that the aforementioned component must lie in 
			the same cluster of  $\mathscr 
			C_{\partial D_{y, L_0}}(Z)$. Therefore $x, x'$ are in fact connected in $\mathscr C_{\partial 
				D_{y, L_0}}(Z) \cap (D_{y, L_0} \setminus (\partial D_{y, L_0} \cup C_{y, L_0}))$. Since $x, x'$ are two arbitrary points inside $\mathscr C_{\partial D_{y, L_0}}(Z) \cap (\tilde 
			D_{y, L_0} \setminus \tilde C_{y, L_0})$, this yields 
			\eqref{eq:LULUtilde} in view of~\eqref{def:tildeLU1}.
		\end{proof}
		We are now ready to give the
		\begin{proof}[Proof of Lemma~\ref{lem:reduction}]
			We will harness the {\em full} strength of Proposition~\ref{lem:surroundingInterfaces} in this proof, including item~(c) therein. Let 
			us define a slightly reformulated version of the event $A_{J, y}(x)$ in \eqref{def:abc}, namely
				$\widetilde A_{J, y}(x) =\{ \textstyle \sum_{y' \in \widehat{\mathbb{L}}_{0,y}}  1\left\{ \text{$y'$ satisfies \eqref{def:Cy} and $C_{y',L_0}$ intersects $\mathscr C_J(x)$} \right\} \ge 
				\Cr{c:reduce} am 
				\}$. In view of  the second part of~\eqref{property:reduction1},  and since $k \mapsto \tau_{k; 
				J, y}$ is non-decreasing, it follows that
			$\widetilde A_{J, y}(x) \subset A_{J, y}(x)$.
			Also since $u_{2, 3} > u$ (see \eqref{eq:params_RI} and \eqref{def:F_y_RI}) and $(\mathcal 
			V)_{\delta}$ is decreasing in $\delta$ (see \eqref{def:noised_set_inclusion}), we have
			\begin{equation*}
				\big\{ \lr{}{(\mathcal V^{u_{2, 3}})_{2\delta}}{\tilde C_{z}}{\partial \tilde D_{z}},  [1, N_z^{{u_0}/2}] 
				\subset J \subset [1, N^u_z] \big\} 
				\subset \big\{ \lr{}{\mathcal V(Z_J)}{\tilde C_{z}}{\partial \tilde D_{z}},  [1, N_z^{{u_0}/2}] 
				\subset J \subset [1, N^u_z] \big\}.
			\end{equation*}
			Hence it is enough to show (cf.~\eqref{eq:inclusion_tauk_fin}), with $y$ ranging in $\mathbb L_0 
					\cap D_{0, L_0}$ below,
			\begin{equation}\label{eq:inclusion_tauk_fin_J}
				\mathscr G_z^{{\rm I}} \cap \{\lr{}{}{x}{\partial \tilde D_z} \text{ in } \} \cap \big\{ [1, 
				N_z^{{u_0}/2}] \subset J \subset [1, N_z^{u}]\big\} \subset \textstyle\bigcup_{y} \widetilde A_{J, y}(x)
			\end{equation}
			Towards showing \eqref{eq:inclusion_tauk_fin_J}, let us start with an inclusion of events which plays a crucial role. For any $z' \in \mathbb L$ 
			satisfying $ D_{z', L} \subset  D_{z} (=D_{z,N})$ and $ U_{z', L} \subset  U_{z}$, 
			and with $\bm u_1, \bm u_2$ are as in \eqref{def:F_y_RI}, we claim that
			\begin{multline}\label{eq:include_GGtilde}
				\mathcal{G}_{z'}^{{\rm I}}(Z_{\mathbb L},  \delta, u_0, u_1, u_2, u_3 ; a) \cap \mathcal F_{z', L}^{\bm u_1, \bm u_2} \cap \big\{[
				1, N_z^{{u_0}/2}] \subset J \subset [1, N_z^{u}]\big\} \\
				\subset 
				\mathcal{G}_{z'}\big({\overline Z}_{\mathbb L}, {\rm V}({\overline Z}_{\mathbb L}, \delta, u, u_{2, 3}, u_{2, 3}), \widetilde{{\rm W}}^{\rm I}, \mathscr C({\overline Z}_{\mathbb L}, \delta, u, u_{2, 3}) ; a\big) \equiv  \widetilde{\mathcal{G}}_{z'}^{{\rm I}}
			\end{multline}
			(see \eqref{def:good_events_supcrit} and \eqref{def:Gz} for notation), where $\widetilde{\rm W}^{\rm I} = \{\widetilde{\rm W}_{z', y'}^{\rm I} : z' \in \mathbb L, y' \in 
			\mathbb L_0\}$ with 
			$	\widetilde{\rm W}_{z', y'}^{\rm I} = {\widetilde{\rm FE}}_{y'}({Z}_J)$
			(cf.~\eqref{eq:WI}). Let us assume \eqref{eq:include_GGtilde} for the moment and finish the proof of Lemma~\ref{lem:reduction}, i.e.~deduce \eqref{eq:inclusion_tauk_fin_J}. 
			
			In essence, we will find many points satisfying \eqref{def:Cy}, as required for $\widetilde A_{J, 
				y}(x)$ to occur, along a path $\gamma$ realizing the crossing event on the left-hand side of 
			\eqref{eq:inclusion_tauk_fin_J}. More precisely, on the event $\{\lr{}{}{x}{\partial \tilde D_z} 
			\text{ in } \mathcal V(Z_J)\}$ appearing in \eqref{eq:inclusion_tauk_fin_J}, and since $x \in 
			\partial \tilde C_z$ by assumption, the component $\mathscr C_J(x)$ contains a 
			crossing $\gamma$ of $\tilde D_z \setminus \tilde C_z$. Let 
			$\gamma_{\mathbb L}$ denote the sequence of points $(z'_1, z'_2, \ldots)$ in 
			$\mathbb L$ such that $\gamma$ 
			visits the boxes $C_{z'_1, L}$, $C_{z'_2, L}$ etc. in that order. 
			Since $\gamma$ is a crossing of $\tilde D_z \setminus \tilde C_z$, 
			it follows that $\gamma_{\mathbb L}$ is itself a crossing on the lattice $\mathbb L$ of 
			$V \setminus U$, with $U,V$ as below \eqref{def:S}.
			For later reference, note that conversely, given any crossing $\bm \gamma'$ of $V\setminus U$, 
			one easily constructs a crossing $\gamma'$ of $\tilde D_z \setminus \tilde C_{z}$ in $\Z^d$ such that $\bm \gamma' = \gamma'_{\mathbb L}$. 
			
			Now recall from the paragraph below \eqref{def:S} that the sets $O_1, \ldots, O_\ell \subset 
			\Sigma$ satisfy property~(c) in Proposition~\ref{lem:surroundingInterfaces}, with $U,V$ as above, $\Sigma$ as in \eqref{def:S} and $\mathbb 
			L$ as the underlying lattice. We deduce from this property and the observations made in the previous 
			paragraph that there exists a crossing $\gamma'$ of $\tilde D_z \setminus \tilde C_z$ satisfying (with $\gamma$ as above) 
			$\text{range}(\gamma'_{\mathbb L}) \cap \Sigma = \text{range}(\gamma_{\mathbb L}) \cap O$, 
			where $O= \bigcup_{1 \le j \le \ell} \, O_j$. By Proposition~\ref{prop:coarse_paths}, there exists a coarsening  $\mathcal C_{\gamma'} \in \mathcal A_{L}^K(\tilde D_{z} \setminus \tilde C_{z})$ that satisfies~\eqref{def:coarse_admissible2} for 
			$\gamma'$. Since $\mathcal C_{\gamma'}$
			is necessarily a subset of $\gamma'_{\mathbb L}$ (this follows from \eqref{def:coarse_admissible2} and the definition of {\em crossing} above \eqref{eq:scriptS_N}) 
			it follows that $\mathcal C_{\gamma'}$ intersects $\Sigma$ {\em only} in $\text{range}(\gamma_{\mathbb L}) \cap O$. 
			Further, by~\eqref{def:coarse_admissible0}, the fact that $\widetilde{\mathcal{G}}_{z'}^{{\rm I}}
			\subset {\rm V}_{z'}({\overline Z}_{\mathbb L}, \delta, u, u_{2, 3}, u_{2, 3}),$
			which is a consequence of \eqref{eq:include_GGtilde} and \eqref{def:Gz}, and by definition of the set $\Sigma$ in 
			\eqref{def:S}, it follows that 
		any $z' \in \mathcal C_{\gamma'}$ such that $\widetilde{\mathcal{G}}_{z'}^{{\rm I}}$ occurs is contained in $\Sigma$ and hence also in $\text{range}(\gamma_{\mathbb L}) \cap O$.
			
			However, on the event $\mathscr G_z^{{\rm I}}$ which appears on the left of \eqref{eq:inclusion_tauk_fin_J}, 
			the number of points $z' \in 
			\mathcal C_{\gamma'}$ such that the event $\mathcal{G}_{z'}^{{\rm I}}(Z_{\mathbb L},  \delta, 
			u_0, u_1, u_2, u_3 ; a) \cap \mathcal F_{z', L}^{\bm u_1, \bm u_2}$ occurs is at least $\rho 
			m = \tfrac m{2\Cr{C:rho}}$, and these two events together with the control over $J$ (also 
			present on the left of \eqref{eq:inclusion_tauk_fin_J}) imply 
			$\widetilde{\mathcal{G}}_{z'}^{{\rm I}}$ by \eqref{eq:include_GGtilde}. Also since any $z' \in 
			\mathcal C_{\gamma'}$ satisfies $D_{z', L} 
			\subset \tilde D_z$ (see property~\eqref{def:coarse_admissible0}), one has $ D_{z', L} 
			\subset  D_{z}$ and $ U_{z', L} \subset  U_{z}$ by \eqref{def:CDU} and  
			\eqref{eq:L_k_descending} (the latter holds as \eqref{eq:params_RI} is in force). 
			All in all, it thus follows that on the event on the left-hand side of 
			\eqref{eq:inclusion_tauk_fin_J}, with $\gamma$ realizing the crossing in $\{\lr{}{}{x}{\partial \tilde D_z} \text{ in } \mathcal V(Z_J)\}$, 
			the following holds:
			there exists $\Sigma_{\gamma} \subset \text{range}(\gamma_{\mathbb L}) \cap O$ such that $|\Sigma_{\gamma}| \ge \tfrac{m}{2\Cr{C:rho}}$, and for each $z' \in \Sigma_\gamma$, the event $\widetilde{\mathcal{G}}_{z'}^{{\rm I}}$ occurs and $\gamma$ crosses $\tilde D_{z'} \setminus C_{z'}$ 
			(simply pick $\Sigma_{\gamma}= \mathcal{C}_{\gamma'}$ in the above construction). 
			In view of Definition~\ref{def:good_events} of the events $\mathcal G_{z'}(\cdot)$ 
			and the definition of $\widetilde{\rm W}^{{\rm I}}$ below \eqref{eq:include_GGtilde}, it follows in this case that there exists a set $S_{z'} \subset \mathbb L_0$ with $|S_{z'}| \ge a$ for any $z' 
			\in \Sigma_\gamma$ such that both $ \text{range} (\gamma) \, (\subset \mathscr C_J(x))$ and $\mathscr 
			C_{z', L} = \mathscr C_{z'}(\overline{Z}_{\mathbb L}, \delta, u, u_{2, 3})$
			intersect $C_{y', L_0}$ as well as $\widetilde{{\rm FE}}_{y'}(Z_J)$ occurs for each $y' \in S_{z'}$. 
			In other words, for each such $y'$, \eqref{def:Cy} holds and $C_{y', L_0}$ intersects 
			$\mathscr C_{J}(x)$. Now observe that any $C_{y', L_0}$ 
			can intersect at most 
			$10^d$-many boxes $\tilde D_{z', L}$ with $z' \in \mathbb L$ (see \eqref{def:CDU}  and 
			\eqref{eq:L_k_descending}). Also note that any $C_{y', L_0}$ with $y' \in \mathbb L_0$ and 
			intersecting $\tilde D_{z', L}$ for some $z' \in \Sigma_\gamma \subset \Sigma$ must 
			necessarily satisfy $D_{y', L_0}  \subset D_{z', L} \subset (\tilde C_z)^c$ (see \eqref{def:S} 
			and \eqref{eq:L_k_descending}), i.e. $y' \in \widehat{\mathbb L}_{0,y}$ for some $y \in 
			\mathbb L_0 \cap D_{0, L_0}$. 
			All in all we thus obtain that on the event 
			on the left-hand side of \eqref{eq:inclusion_tauk_fin_J},
			\begin{equation*}
				\textstyle \sum_{y'}1\left\{\text{
					\eqref{def:Cy} holds with $y'$ in place of $y$ and $C_{y',L_0}$ intersects $\mathscr C_J(x)$
				} \right\}  \ge \tfrac{1}{2\cdot10^d\Cr{C:rho}} a m
			\end{equation*}
			where the sum ranges over $ y' \in \widehat{\mathbb L}_0 = \bigcup_{y} 
			\widehat{\mathbb L}_{0, y}$ with $y$ ranging in  $\mathbb L_0 \cap D_{0, L_0}$. 
			From this \eqref{eq:inclusion_tauk_fin_J} (and hence Lemma~\ref{lem:reduction}) follows immediately with $\Cr{c:reduce} =  ({2\cdot10^d\Cr{C:rho}|\mathbb 
				L_0 \cap D_{0, L_0}|})^{-1} = (2\cdot70^d)^{-1}$.

			\smallskip
			
			%
			%
			%
			%
			%
			
			We still need to verify \eqref{eq:include_GGtilde}. 
			Using 
			our argument for \eqref{eq:supcrit_inclusion0} 
			for the second inclusion below,
			\begin{multline*}
				\mathcal{G}_{z'}^{{\rm I}}(Z_{\mathbb L},  \delta, u_0, u_1, u_2, u_3 ; a)  \cap \mathcal F_{z', L}^{\bm u_1, \bm u_2} 
				\stackrel{\eqref{eq:Fext}, \eqref{def:F_y_RI}}{\subset}
				\mathcal{G}_{z'}^{{\rm I}}(Z_{\mathbb L},  \delta, u_0, u_1, u_2, u_3 ; a)  \cap \mathcal F_{z', L}^{u, u_1} \cap 
				\mathcal F_{z', L}^{u_{2,3}, u_2} \cap \mathcal F_{z', L}^{u_{2, 3}, u_3}\\
				\stackrel{\eqref{def:good_events_supcrit}}{\subset}   
				\mathcal{G}_{z'}\big({\overline Z}_{\mathbb L}, {\rm V}_{z'}({\overline Z}_{\mathbb L}, \delta, u, u_{2, 3}, u_{2, 3}), {{\rm W}}^{\rm I}, \mathscr C_{z'}({\overline Z}_{\mathbb L}, \delta, u, u_{2, 3}) ; a\big).
			\end{multline*}
			Since $\mathcal G_{z'}(\cdot)$ from Definition~\ref{def:good_events} is increasing 
			w.r.t.~the events $\{{\rm W}_{z', y'}: z' \in \mathbb L, y' \in \mathbb L_0\}$, and because 
			\begin{multline}\label{eq:include_GGtilde2}
				\, {{\rm W}}_{z', y'}^{{\rm I}} \cap \mathcal F_{z', L}^{\bm u_1, \bm u_2} \stackrel{\eqref{eq:WI}}{=} {{\rm FE}}_{y', L_0}((Z_{z', L}^{u_1})_+(\tfrac{u_0}{8}\,{\rm cap}( D_{z', L}))) \cap \mathcal F_{z', L}^{\bm u_1, \bm u_2}\\
				\stackrel{\eqref{eq:Fext}, \eqref{def:F_y_RI}}{\subset}    	{{\rm FE}}_{y', L_0}((Z_{z', L}^{u_1})_+(\tfrac{u_0}{8}\,{\rm cap}( D_{z', L}))) \cap  
				\mathcal F_{z', L}^{\frac{u_0}2, \frac{u_0}8} 
			\end{multline}
			all that is left to show towards proving \eqref{eq:include_GGtilde} is to argue that the intersection of $\{[1, N_z^{{u_0}/2}] \subset J \subset [1, N_z^{u}]\}$ with the event in the second line of \eqref{eq:include_GGtilde2} implies $\widetilde{\rm FE}_{y', L}\left({Z}_J\right)$.
			But by Lemma~\ref{lem:LULUtilde}, the definitions of  ${\rm FE}$ in  \eqref{def:FE1} and of $\widetilde{\rm FE}$ in \eqref{def:Cy} and the monotonicity of ${\rm O}_{y'}(Z)$ 
			in $Z$ (with respect to inclusion 
			of the underlying sets, see \eqref{def:Conn0}), we have that
			$	{\rm FE}_{y', L_0}(Z_J) \cap \big\{J \subset [1, N_z^{u}]\big\}$ implies 
			$\widetilde{\rm FE}_{y', L_0}({Z}_J)$ for any $y' \in \mathbb L_0$. 
			Therefore it suffices to show that the intersection of $\{[1, N_z^{{u_0}/2}] \subset J \subset [1, N_z^{u}]\}$ with the event in the second line of \eqref{eq:include_GGtilde2} implies ${\rm FE}_{y', L}\left({Z}_J\right)$ rather than $\widetilde{\rm FE}_{y', L}\left({Z}_J\right)$.
			
			Since $ D_{z', L} \subset  D_z$ and $ U_{z', L} \subset  U_z$ by our assumption above \eqref{eq:include_GGtilde}, it follows from 
			\eqref{eq:nested_excursion} in 
			\S\ref{subsec:excursion} 
			that the sequence of excursions $Z_J$ between $ 
			D_z$ and $ U_z$ {induces} a sequence of excursions $Z_{J', L} = 
			\big(Z_{j}^{ D_{z', L},  U_{z', L}}\big)_{j \in J'}$ between $ D_{z', L}$ and 
			$\partial_{{\rm out}} U_{z', L}$ such that $\mathcal I(Z_J) \cap 
			D_{z', L} = \mathcal I(Z_{J', L})$ and $\ell_{x'}(Z_J) = \ell_x(Z_{J', L})$ for all $x' \in 
			D_{z', L}$. In particular, if $J=[1,N_z^v]= [1,N_{z,N}^v] $ for some $v>0$ then $J'= [1,N_{z',L}^v]$ on account of \eqref{def:N_AU}. Furthermore, recalling $\mathcal F_{z', L}^{u, v}$ from 
			\eqref{eq:F}, on the event $$\mathcal F_{z', L}^{\frac{u_0}2, \frac{u_0}8}  \cap \mathcal F_{z', L}^{u, u_1} \cap \big\{[1, N_z^{{u_0}/2}] \subset J \subset [1, 
			N_z^{u}]\big\},$$ 
			we have $\{1, \ldots, \tfrac{u_0}8\,{\rm cap}( D_{z', L})\} \subset J' \subset \{1, \ldots, 
			u_1\,{\rm cap}( D_{z', L})\}$. However, this means $Z_{J', L}$ lies in the family $(Z_{z', 
				L}^{u_1})_+(\tfrac{u_0}{8}\,{\rm cap}( D_{z', L}))$ by \eqref{eq:Z^+}, thus yielding 
			the desired inclusion.
		\end{proof}

		\subsection{Discovery of good encounter points}\label{subsec:exploration}
		This subsection is devoted to the proof of Proposition~\ref{prop:reduction}. 
		In the sequel, we drop $J, y$ and $x$ from the notations 
		$\tau_{k; J, y}(x)$ etc. and abbreviate $\mathcal I = 
		\mathcal I({Z}_J)$,  $\mathcal V = \mathcal V({Z}_J)$. We proceed to construct a sequence of 
		random times $(\tau_k)_{k \ge 1}$ 
		satisfying~\eqref{property:reduction1}--\eqref{property:reduction3}. Recall the exploration 
		sequence $(w_n)_{n \ge 1}$ of the cluster $\mathscr{C}_J(x)$ of $x$ in $\mathcal{V}\cap 
		D_z$  from the beginning of 
		\S\ref{subsec:reduction}.

		
		%
		
		We start with the sequence of successive times $(\tilde \tau_k)_{k \ge 1}$ 
		at which the exploration of $\mathscr C_J(x)$ visits $$\partial \coloneqq
		\textstyle \bigcup_{y' \in \widehat{\mathbb{L}}_{0,y} } \partial D_{y',L_0},$$
		and certain additional (good) properties are satisfied. Formally, with $\tilde \tau_0=1$, for $k 
		\geq 1$ we let 
		\begin{equation} \label{eq:x-tilde-k}
			\tilde \tau_k=\inf\{n > \tilde \tau_{k-1}: w_n \in \partial\cap \mathcal V \text{ and } (\ast) \text{ holds}  \}, \quad \tilde{X}_k = w_{\tilde \tau_k} \mbox{ if $\tilde \tau_k < \infty$}
		\end{equation}
		(with the convention $\inf \emptyset = \infty$), where $(\ast)$ refers to the property that if 
		$\tilde Y_k \in \widehat{\mathbb{L}}_{0,y}$ denotes the unique point such that $\tilde{X}_k 
		\in D_{\tilde Y_k,L_0}$ (when $\tilde \tau_k < \infty$), then $\tilde Y_k$ satisfies 
		property~\eqref{def:Cy}. By~\eqref{eq:x-tilde-k} and~\eqref{def:Cy},   
		\begin{equation}\label{eq:tildetaumeasur}
			\{\tilde \tau_k = n\} \in \mathcal F_n, \text{ for all $n \geq 1$,}
		\end{equation}
		where, letting $ \partial_n$ denote the set of all points $y' \in \widehat{\mathbb{L}}_{0,y}$  satisfying $\{ w_m\}_{1\leq m \leq n} \cap \partial D_{y', L_0} \cap \mathcal V \neq \emptyset$,  
		\begin{multline}\label{eq:F_n}
			\mathcal{F}_n \coloneqq \sigma \big( (w_1, 1_{\{w_1 \in \mathcal V\}}), \ldots, (w_n, 1_{\{w_n \in \mathcal 
				V\}}), ((\mathcal V^u)_{\delta} \cap D_z, (\mathcal V^{u_{2,3}})_{2\delta} \cap D_z),\\ (\ell_x^u : x \in \partial D_{y', L_0}, y' \in  \partial_n), \{\mathscr C_{\partial D_{y', L_0}}(Z_J) : y' \in  \partial_n\} \big).
		\end{multline}
		
		We will maintain, at each time $n$, 
		three sets $\mathsf{B}_n, \mathsf{W}_n$ and $\mathsf{G}_n$ of so 
		called \emph{black}, \emph{white} and \emph{grey} vertices, 
		whose key features are summarized in Lemma~\ref{L:BWG} below. When speaking of 
		\textit{revealing} a vertex $v \in \Z^d$ in the sequel, we mean disclosing the value of $1\{ v \in  
		\mathcal I \}$. 
		It will  always be the case that $\mathsf W_n$ and $\mathsf B_n$ are precisely the set of 
		vertices in $\mathcal V$ and $\mathcal I$ respectively that have been revealed up until time 
		$n$. In parellel to $ \mathscr{A}_n= (\mathsf{B}_n, \mathsf{W}_n, \mathsf{G}_n)$, we also 
		maintain another triplet $\tilde{\mathscr{A}}_n=(\widetilde{\mathsf{B}}_n, \widetilde{\mathsf{W}}_n, \widetilde{\mathsf{G}}_n)$ for each $n$ that will remain fixed unless $n=\tilde\tau_k$ for some $k\geq 1$ (assuming $\tilde \tau_k< \infty$), where they will serve to keep track of certain occurrences (see 
		\eqref{eq:B_n}--\eqref{eq:W_n} below). We start by defining ${\mathscr{A}}_0$ and $\tilde{\mathscr{A}}_0$ as
		\begin{equation}
			\label{eq:BWG0}
			\mathsf{B}_0 = \mathsf{W}_0 = \widetilde{\mathsf{B}}_0 = \widetilde{\mathsf{W}}_0 = \emptyset, \mbox{ and } \mathsf{G}_0 = \widetilde{\mathsf{G}}_0 = \Z^d.
		\end{equation}
		For each $n$ such that $1=\tilde \tau_0 \leq n < \tilde \tau_1$, we keep $\tilde{\mathscr{A}}_n= \tilde{\mathscr{A}}_0$ and update the triplet ${\mathscr{A}}_n$ inductively from ${\mathscr{A}}_{n-1}$ as follows (note that the following simplifies for $n < \tilde \tau_1$ but the formulation will generalize immediately to $n \in (\tilde \tau_k,\tilde \tau_{k+1})$ 
		for $k \geq 1$). If $w_n \notin \mathsf{G}_{n-1} $ (which cannot happen when $n< \tilde \tau_1$, as can be easily seen inductively) we 
		set ${\mathscr{A}}_n= {\mathscr{A}}_{n-1}$. Otherwise, we reveal $w_n$, remove $w_n$ from 
		$\mathsf{G}_{n-1}$ and add it to $\mathsf{B}_{n-1}$ if $ w_n \in  \mathcal I$ and to 
		$\mathsf{W}_{n-1}$ if $ w_n \in  \mathcal V$, thus yielding ${\mathscr{A}}_n$. We call such a step \emph{generic}. 
		
		Assume now that $n = \tilde \tau_1 < \infty$, recall $\tilde{X}_1$ from \eqref{eq:x-tilde-k} 
		and denote by $\tilde{Y}_1 \in  \widehat{\mathbb{L}}_{0,y} $ the point such that $\tilde{X}_1\in 
		\partial D_{\tilde{Y}_1,L_0}$. The triplets ${\mathscr{A}}_n$ 
		and $\tilde{\mathscr{A}}_n$ are 
		now obtained from ${\mathscr{A}}_{n-1}$ as follows (this will be the first time we update the sets $\tilde{\mathscr{A}}_n$). First, let 
		\begin{equation*}
			\begin{split}
				\mathsf S_n \cap  D_{\tilde{Y}_1,L_0}^c = \mathsf S_{n-1} \cap  D_{\tilde{Y}_1,L_0}^c \mbox{ and } \widetilde{\mathsf S}_n \cap  D_{\tilde{Y}_1,L_0}^c = \widetilde{\mathsf 
					S}_{n-1} \cap  D_{\tilde{Y}_1,L_0}^c
			\end{split}
		\end{equation*}
		for $\mathsf S \in \{\mathsf B, \mathsf W, \mathsf G\}$.
		It thus remains to specify changes to the sets $\mathsf{B}_{n-1}, \mathsf{W}_{n-1},\mathsf{G}_{n-1}$ and $\widetilde{\mathsf{B}}_{n-1}, 
		\widetilde{\mathsf{W}}_{n-1},\widetilde{\mathsf{G}}_{n-1}$ inside $D_{\tilde{Y}_1, L_0}$. To this effect, we first reveal $\mathscr C_{\partial D_{\tilde Y_1, L_0}}(Z_J)$ 
		(recall \eqref{def:C_partial1}). That is, we reveal all the 
		points in $\partial D_{\tilde{Y}_1,L_0} \cap \mathsf{G}_{n-1}$,  and then we explore the clusters of points in $\partial D_{\tilde{Y}_1,L_0} $ inside $\mathcal I \cap (D_{\tilde{Y}_1,L_0} \setminus C_{\tilde{Y}_1,L_0} )$, thereby revealing the points 
		in these clusters and their outer boundary in $D_{\tilde{Y}_1,L_0}\setminus 
		C_{\tilde{Y}_1,L_0} $. All the points in $D_{\tilde{Y}_1,L_0}\setminus C_{\tilde{Y}_1,L_0} $ thereby revealed 
		are removed from $\mathsf{G}_{n-1}$ and constitute $\mathsf{B}_{n}'$, resp.~$\mathsf{W}_{n}'$, depending on whether they are in $\mathcal{I}$, 
		resp.~$\mathcal{V}$. The remaining points in $\mathsf{G}_{n-1} \cap D_{\tilde{Y}_1,L_0}$ define the set $\mathsf{G}_{n}'$. Notice that $C_{\tilde{Y}_1, L_0} \subset 
		\mathsf{G}_{n}'$. Now set (with $n =\tilde \tau_1$)
		\begin{align}
			&\label{eq:B_n} \widetilde{\mathsf{B}}_{n} \cap D_{\tilde{Y}_1, L_0} =  \mathsf{B}_{n}',\\
			&\label{eq:G_n} \widetilde{\mathsf{G}}_{n} \cap D_{\tilde{Y}_1, L_0} = \text{the points in the component of $C_{\tilde{Y}_1, L_0} $ in $\mathsf{G}_{n}'$},\\
			&\label{eq:W_n} \widetilde{\mathsf{W}}_{n} \cap D_{\tilde{Y}_1, L_0} =  \mathsf{W}_{n}'  \cup (\mathsf{G}_{n}' \setminus \widetilde{\mathsf{G}}_{n}).
		\end{align}
		We have thus completely specified $\tilde{\mathscr{A}}_n$; the case of
		$\mathscr{A}_n$ takes a bit more work. 
		Notice that $\tilde{X}_1 \in \mathsf{W}_n' (\subset 
		\widetilde{\mathsf W}_n \cap D_{\tilde{Y}_1, L_0})$ since $\tilde{X}_1\in \mathcal{V}$ on account of \eqref{eq:x-tilde-k}. 
		As shown in Lemma~\ref{L:BWG} below, see \eqref{eq:BW}, we have $\widetilde{\mathsf W}_n \subset 
		\mathcal V$. We now inspect whether $\tilde{X}_1$ is connected to $C_{\tilde{Y}_1, L_0} $ 
		by a path in $(\widetilde{\mathsf{G}}_{n} \cup \widetilde{\mathsf{W}}_{n}) \cap D_{\tilde{Y}_1, L_0}$:
		\begin{itemize}
			\item If the answer is no (Case I), we set $\mathsf{G}_{n} \cap D_{\tilde{Y}_1,L_0} = \widetilde{\mathsf{G}}_{n} \cap D_{\tilde{Y}_1,L_0}$,  $\mathsf{B}_{n} \cap 
			D_{\tilde{Y}_1,L_0}= \widetilde{\mathsf{B}}_{n} \cap D_{\tilde{Y}_1,L_0}$ and $\mathsf{W}_{n} \cap 
			D_{\tilde{Y}_1,L_0}= \widetilde{\mathsf{W}}_{n} \cap D_{\tilde{Y}_1,L_0}$, completing the specification of $(\mathsf{B}_{n}, 
			\mathsf{W}_{n},\mathsf{G}_{n})$ in that case. 
			\item If the answer is yes (Case II), we reveal all the 
			vertices in $\widetilde{\mathsf{G}}_n$
			, add them to $\widetilde{\mathsf{B}}_{n}\cap D_{\tilde{Y}_1,L_0}$ and $\widetilde{\mathsf{W}}_{n}\cap D_{\tilde{Y}_1,L_0}$ 
			dep.~on their state to obtain $\mathsf{B}_{n}\cap D_{\tilde{Y}_1,L_0}$, resp.~$\mathsf{W}_{n}\cap D_{\tilde{Y}_1,L_0}$, and set $\mathsf{G}_n \cap D_{\tilde{Y}_1,L_0}=\emptyset$. 
		\end{itemize}
		Notice that, since $\widetilde{\mathsf W}_n \subset \mathcal V$ as already observed, we 
		have $\mathsf W_n \subset \mathcal V$ and $\mathsf B_n \subset \mathcal I$ in all cases. 
		As a consequence of the definitions \eqref{eq:B_n}--\eqref{eq:W_n}, it follows that the sets $\widetilde{\mathsf S}_n \cap D_{\tilde Y_1, L_0}; \mathsf S \in \{\mathsf B, \mathsf W, \mathsf G\}$ are measurable relative to $\mathscr C_{\partial 
			D_{\tilde Y_1, L_0}}(Z_J)$. Further, we note that the statement \begin{equation}\label{eq:wninclude2}
			\text{\begin{minipage}{0.85\textwidth}$(\mathsf G_m)_{0 \le m \leq \tilde \tau_k}$ form decreasing sets, and for all $x \in \Z^d$, $\mathsf S \in 
					\{\mathsf B, \mathsf W, \mathsf G\}$, integers $1 \le n_1 < \ldots < n_{k-1} < n$, $y' \in  \partial_n$ and $m \le n$, the events $\{x  \in \mathsf S_m, \mathsf G_m \cap D_{y', L_0} \ne \emptyset, \tilde \tau_1 = n_1, \ldots, \tilde \tau_k = n\}$ and $\{x  \in \widetilde{\mathsf S}_m, \tilde \tau_1 = n_1, \ldots, \tilde \tau_k = n\}$ are $\mathcal F_n$-measurable (cf.~\eqref{eq:F_n}), \end{minipage}}
		\end{equation}
		defined for integer $k \geq 1$, holds for $k=1$.

		By induction, suppose that for some $k \geq 1$, both $(\mathscr{A}_n)_{ 0 \leq n \leq \tilde 
			\tau_k}$ and $(\tilde{\mathscr{A}}_n)_{ 0 \leq n \leq \tilde \tau_k}$ have been 
		specified on the event $\{\tilde \tau_k< \infty\}$  and \eqref{eq:wninclude2} holds. Now we continue for times $\tilde 
		\tau_k +1  \leq n < \tilde \tau_{k+1}$ by performing generic steps of the exploration, as 
		defined above for $1 \leq n < \tilde \tau_1$. 
		When $n = \tilde \tau_{k+1}< \infty$, there are three possible scenarios based on which we determine the next course of 
		action. If $\tilde Y_{k+1} \ne \tilde Y_{l}$ for {\em all} $l \le k$, then we follow exactly the same procedure as described for $n=\tilde \tau_1$. Otherwise, and 
		if in addition $\mathsf{G}_{n - 1} \cap D_{\tilde Y_{k+1}, L_0} \ne \emptyset$, we set  
	\begin{equation}\label{eq:GBW2}
		\begin{split}	
			{\widetilde{\mathsf S}_{n} } \cap D_{\tilde Y_{k+1}, L_0}^c = {\widetilde{\mathsf 
					S}}_{n - 1} \cap D_{\tilde Y_{k+1}, L_0}^c \mbox{, } \widetilde{\mathsf S}_{n} \cap 
			D_{\tilde Y_{k+1}, L_0} = {{\mathsf S}}_{n - 1} \cap D_{\tilde Y_{k+1}, L_0}
		\end{split}
	\end{equation}
	for $\mathsf S \in \{\mathsf B, \mathsf W, \mathsf G\}$ and skip to the remaining steps for $n=\tilde \tau_1$ (starting with inspecting whether $\tilde{X}_{k+1}$ is connected to $C_{\tilde Y_{k+1}, L_0} $ by a path 
	in $(\tilde{\mathsf{\mathsf{G}}}_{n} \cup \tilde{\mathsf{\mathsf{W}}}_{n}) \cap D_{\tilde Y_{k+1}, L_0}$). 
	Finally if $\mathsf{G}_{n - 1} \cap D_{\tilde Y_{k+1}, L_0} = \emptyset$, we simply carry out a generic step, which in this case boils down to setting $\mathscr{A}_n=\mathscr{A}_{n-1}$ and $\tilde{\mathscr{A}}_n=\tilde{\mathscr{A}}_{n-1}$.  

Overall, this defines 
 $(\mathscr{A}_n)_{ n \geq 0}= (\mathsf B_n, \mathsf W_n, \mathsf G_n)_{n \ge 0}$ and 
$(\tilde{\mathscr{A}}_n)_{ n \geq 0}= (\widetilde{\mathsf B}_n, \widetilde{\mathsf W}_n, \widetilde{\mathsf 
	G}_n)_{n \ge 0}$, and the 
exploration effectively stops when discovering $\mathscr C_J(x)$, 
from which time on $\mathscr{A}_n$ and $\tilde{\mathscr{A}}_n $ remain fixed. Also by our induction hypothesis and the rules of updating $\mathscr{A}_n$ and $\tilde{\mathscr{A}}_n$, \eqref{eq:wninclude2} holds with $k+ 1$ in place of $k$, and thus \eqref{eq:wninclude2} holds for all $k \geq 1$. The next lemma collects key features 
of $(\mathscr{A}_n)_{ n \geq 0}$ 
and $(\tilde{\mathscr{A}}_n)_{ n \geq 
	0}$. 
\begin{lemma}\label{L:BWG}
	For all $n, k \geq 1$, 
	the following hold.
	\begin{align}
		&\label{eq:BW0} \text{\begin{minipage}{0.89 \textwidth} The sets $(\mathsf B_n, \mathsf W_n, \mathsf G_n)$ and $(\widetilde{\mathsf B}_n, \widetilde{\mathsf W}_n, \widetilde{\mathsf G}_n)$
				both form 
				partitions of $\Z^d$. 
		\end{minipage}}\\
		&\label{eq:BW} \text{\begin{minipage}{0.89 \textwidth}
				$\mathsf{B}_n$ and $\widetilde{\mathsf{B}}_n$ are subsets of $\mathcal I$ whereas $\mathsf{W}_n$ and $\widetilde{\mathsf{W}}_n$ are subsets of $\mathcal V$.
		\end{minipage}}\\
		&\label{eq:w2} \text{\begin{minipage}{0.78 \textwidth} 
				If $\mathsf{G}_{n-1} \cap D_{\tilde Y_k, L_0} \ne \emptyset$ whereas $\mathsf{G}_{n} \cap 
				D_{\tilde Y_k, L_0} = \emptyset$, then $n = \tilde \tau_l$ for some $l \ge 1$ such that $\tilde 
				Y_l = \tilde Y_k$ and $\tilde X_l$ is connected to $C_{\tilde Y_l, L_0}$ by a path in 
				$(\widetilde{\mathsf{G}}_{\tilde \tau_l} \cup \widetilde{\mathsf{W}}_{\tilde \tau_l}) \cap D_{\tilde Y_l, L_0}$.
		\end{minipage}}\\
		&\label{eq:CsubsetV}\text{\begin{minipage}{0.80 \textwidth} If $n = \tilde \tau_k$ such 
				that $\mathsf{G}_{n-1} \cap D_{\tilde Y_k, L_0} \ne \emptyset$ and $C_{\tilde Y_k, L_0} 
				\subset \mathcal V$, then $\widetilde{\mathsf{G}}_n \cap D_{\tilde Y_k, L_0}  \subset \mathcal V$.
		\end{minipage}}
	\end{align}
\end{lemma}
For improved readability, the proof of Lemma~\ref{L:BWG} is postponed to Appendix~\ref{sec:aux-lemmas}. Using it, we now conclude the proof of Proposition~\ref{prop:reduction}. Towards this, we define the sequence $(\tau_k)_{k \geq 0}$ inductively as 
\begin{equation}\label{def:tauk}
	\tau_0 = 0, \quad \tau_k = 	\inf \big\{\tilde \tau_l  > \tau_{k-1} :  \text{ \ref{def:**} holds
	} \big\}
\end{equation}
for any $k \ge 1$, 
where 
\begin{equation}\label{def:**}\tag{$(\ast\ast)_l$}
	\text{$\mathsf{G}_{\tilde \tau_l-1}  \cap D_{\tilde 
			Y_l,L_0}  \neq \emptyset$ and $\lr{}{}{\tilde{X}_l (= w_{\tilde \tau_l})}{C_{\tilde Y_l, L_0} }$ in $\widetilde{\mathsf{G}}_{\tilde 
			\tau_l} \cup \widetilde{\mathsf{W}}_{\tilde \tau_l}$}.
\end{equation}
We set $Y_k = \tilde Y_l$ if $\tau_k = \tilde \tau_l < \infty$. 
We now proceed to verify that $(\tau_k)_{k \geq 1}$ defines a sequence of good encounter times, i.e.,~that properties \eqref{property:reduction1}-\eqref{property:reduction3} hold. Lemma~\ref{L:BWG} readily yields the first two of these.

\begin{proof}[Proof of Proposition~\ref{prop:reduction}, items \eqref{property:reduction1} and \eqref{property:reduction2}]

	The first part of property~\eqref{property:reduction1} follows from the definition of $\tilde 
	\tau_l$ in \eqref{eq:x-tilde-k}. 
	For the converse part, consider a point $y' \in \widehat{\mathbb 
		L}_{0, y}$ satisfying property~\eqref{def:Cy} and with $\mathscr C_J(x) \cap C_{y', L_0} \neq 
	\emptyset$. Since $D_{y', L_0} \subset (\tilde C_z)^c$ (see the definition of 
	$\widehat{\mathbb L}_{0, y}$ below~\eqref{def:Cy}), it follows from the 
	properties of $y'$ together with definition~\eqref{eq:x-tilde-k} that one can find $l \ge 
	1$ such that $y' = \tilde Y_l$ and 
	$\tilde X_l$ is connected to $C_{y', L_0}$ in $\mathcal V \cap D_{y', L_0}$ 
	If $\mathsf G_{\tilde \tau_l - 1} \cap D_{y', L_0} \ne \emptyset$, then $\tilde \tau_l$ satisfies 
	\ref{def:**} as $\mathcal V \cap D_{y', L_0} \subset (\widetilde{\mathsf G}_{\tilde \tau_l} \cup 
	\widetilde{\mathsf W}_{\tilde \tau_l}) \cap 
	D_{y', L_0}$ by \eqref{eq:BW} and the second part of \eqref{eq:BW0} and hence $y' = Y_k$ 
	(defined via $w_{\tau_k} \in \partial D_{Y_k, L_0}$) for some $k \ge 1$. Otherwise if $\mathsf 
	G_{\tilde \tau_l - 1} \cap D_{y', L_0} = \emptyset$, then from \eqref{eq:w2} and 
	\eqref{eq:BWG0} we can deduce the existence of $l' \le l$ such that $y' = \tilde Y_{l'}$, 
	$\mathsf G_{\tilde \tau_{l'} - 1} \cap D_{y', L_0} \ne \emptyset$ and $\tilde X_{l'}$ is 
	connected to $C_{y', L_0}$ in $\widetilde{\mathsf G}_{\tilde \tau_{l'}} \cup \widetilde{\mathsf 
		W}_{\tilde \tau_{l'}}$. In other words $(\ast\ast)_{l'}$ holds and we get the same conclusion 
	as before. 
	
	Property~\eqref{property:reduction2} is a consequence of \eqref{eq:CsubsetV} and \eqref{eq:BW} on account of \eqref{def:tauk} and \ref{def:**}. \end{proof}

Verifying that  $(\tau_k)_{k \geq 1}$ defined by \eqref{def:tauk} 
satisfies~\eqref{property:reduction3} requires more work. The key is:

\begin{lemma}\label{L:BWG'}
	For all $n, k \geq 1$ and $\mathsf S \in \{\mathsf{W}, \mathsf{B},\mathsf{G}\}$, 
	\begin{align}
		&\label{eq:w} \text{\begin{minipage}{0.89 \textwidth} 
				$\mathsf S_n \cap D_{\tilde Y_k,L_0}= \mathsf S_{\tilde \tau_k} \cap D_{\tilde Y_k,L_0}$ provided $n \geq \tilde \tau_k$ 
				and 
				$\mathsf{G}_{n} \cap D_{\tilde Y_k, L_0} \ne \emptyset$.
		\end{minipage}}
	\end{align}
\end{lemma}

A proof of Lemma~\ref{L:BWG'} is included in Appendix~\ref{sec:aux-lemmas}. We now complete the proof of Proposition~\ref{prop:reduction}.

\begin{proof}[Proof of Proposition~\ref{prop:reduction}, item~\eqref{property:reduction3}] 
	Let $\mathcal F_{y}= \mathcal F_{y, L_0}(Z_J, \delta, u, u_{2, 3})$ denote the 
	$\sigma$-algebra in \eqref{def:FEsigma_algebra1}. We will argue that for all integers 
	$l \geq 1$, $1 \leq n_1< \dots < n_l$, $\sigma_m \in \{0,1\}$ and $y_m \in \widehat{\mathbb 
		L}_{0, y}$, $1 \leq m \leq l$,
	\begin{equation}\label{eq:7.5meas-intermediate}
		\big( \textstyle\bigcap_{1 \leq m < l }\big\{ \tilde\tau_m = n_m , \tilde Y_m= y_m, 1_{\{ (\ast\ast)_m\}}=\sigma_m\big\} \cap \{ \tilde\tau_l = n_l, \tilde Y_l= y_l, (\ast\ast)_l  \} \big)  \in \mathcal F_{y_l}.
	\end{equation}
	Let us first conclude \eqref{property:reduction3} from this. In view of \eqref{def:tauk}, it is 
	clear that for all $y', y_1,\dots, y_k \in \widehat{\mathbb L}_{0, y}$ and $k \ge 0$, the event 
	$\textstyle\bigcap_{1 \le j \le k}\{\tau_j < \infty, Y_j = y_j\} \cap \{\tau_{k + 1} < \infty, Y_{k+1} 
	= y'\}$ can be decomposed into a countable union of events of the type appearing in 
	\eqref{eq:7.5meas-intermediate} with $l$ such that  $\tilde{\tau}_l= \tau_{k+1}$ and 
	$y_l=y'$, whence this event is in $\mathcal F_{y'}$. The event appearing in 
	\eqref{property:reduction3} differs from this through the presence of the condition $C_{y_j, 
		L_0} \not\subset \mathcal V(Z_J)$ for each $1 \leq j \leq k$. 
	As we now explain, none of the $y_j$'s can be equal to $y'$ if the corresponding event is non-empty. For, otherwise this would mean in view 
	of \eqref{def:tauk} and the definition of the processes $(\mathscr{A}_n)_{n \ge 0}$ and 
	$(\tilde{\mathscr{A}}_n)_{n \ge 0}$ (see below \eqref{eq:W_n} and below \eqref{eq:GBW2}) 
	that Case II would have to have occurred at time $n= \tau_j (< \tau_{k+1})$, thus implying that 
	$\mathsf G_n \cap D_{y_j, L_0} = \mathsf G_n \cap D_{y', L_0}= \emptyset$. With $\mathsf 
	G_n$ being decreasing in $n$ by \eqref{eq:wninclude2}, 
	this prevails at all times $>n$, thus precluding the possibility that $\tau_{k+1}=\tilde{\tau}_l$ in 
	view of the first condition inherent to \ref{def:**} in \eqref{def:tauk}. But with $y_j \neq 
	y'$, the event $\{C_{y_j, L_0} \not\subset \mathcal V(Z_J)\}$ is measurable with respect to 
	$\mathcal I \cap (D_{y', L_0})^c$ (part of $\mathcal F_{y_l}$), and 
	\eqref{property:reduction3} follows.

	We now turn to \eqref{eq:7.5meas-intermediate} and abbreviate the event appearing there as $\mathcal E$. Let $\Xi$ denote the random field comprising the configurations $((\mathcal V^u)_{\delta} \cap D_z, (\mathcal V^{u_{2,3}})_{2\delta} \cap D_z)$, the 
	occupation times $(\ell_x^u : x \in \partial D_{y', L_0}, y' \in  \partial_{n_l})$ (recall $ \partial_{n}$ from below \eqref{eq:tildetaumeasur}) and the 
	clusters $\{\mathscr C_{\partial D_{y', L_0}}(Z_J) : y' \in \partial_{n_l}\}$. In view of definitions \eqref{eq:x-tilde-k} and \ref{def:**}, 
	we see that $\mathcal E$ is measurable relative to the variables $(w_1, 1_{\{w_1 \in \mathcal 
		V\}}), \ldots, (w_{n_l}, 1_{\{w_{n_l} \in \mathcal V\}})$, the field $\Xi$ and the occupancy 
	variables of the events $\{x  \in \mathsf S_m, \mathsf G_m 
	\cap D_{y', L_0} \ne \emptyset, \tilde \tau_1 = n_1, \ldots, \tilde \tau_l = n_l\}$ and $\{x  \in \widetilde{\mathsf S}_m, \tilde \tau_1 = n_1, \ldots, \tilde \tau_l = n_l\}$, where $x \in \Z^d, \mathsf S \in \{\mathsf B, \mathsf W, \mathsf G\}, y' \in  \partial_{n_l},$ and $1 \leq m \le n_l$. With this, \eqref{eq:wninclude2} and \eqref{eq:F_n} give  that 
	$\mathcal E \in \mathcal F_{n_l}$. 
	Now consider $x_1, \ldots, x_{n_l} \in \Z^d$, $s_1, \ldots, s_{n_l} \in \{0, 1\}$ and a realization $\xi$ of $\Xi$ satisfying
	\begin{equation}\label{eq:Eintersectnonemp}
		\mathcal E \cap \{w_1 = x_1, 1_{\{x_1 \in \mathcal V\}} = s_1, \ldots, w_{n_l} = x_{n_l}, 1_{\{x_{n_l} \in \mathcal V\}} = s_{n_l}, \Xi = \xi\} \ne \emptyset.
	\end{equation} 
	Note that $\partial_{n_l}$ introduced above \eqref{eq:F_n} (and part of $\Xi$) is deterministic 
	on the event in \eqref{eq:Eintersectnonemp}. 
	Since $\mathcal E \in \mathcal F_{n_l}$, the definition of the $\sigma$-algebra $\mathcal 
	F_{n_l}$ from \eqref{eq:F_n} and of $\mathcal{F}_{y_l}$ from 
	\eqref{def:FEsigma_algebra1}, it therefore suffices to show for proving 
	\eqref{eq:7.5meas-intermediate} that for any choice as above \eqref{eq:Eintersectnonemp}, 
	the event $\{w_1 = x_1, 1_{\{x_1 \in \mathcal V\}} = s_1, \ldots, w_{n_l} = x_{n_l}, 1_{\{x_{n_l} 
		\in \mathcal V\}} = s_{n_l}, \Xi = \xi\}$ is $\mathcal F_{y_l}$-measurable. But the definition of 
	the exploration sequence $(w_n)_{n \ge 1}$ from \S\ref{subsec:reduction} implies that 
		the variables $1_{\{w_1 = x_1, x_1 \in \mathcal V\}}, \ldots, 1_{\{w_{n_l} = x_{n_l}, x_{n_l} 
			\in \mathcal V\}}$ are measurable relative to $\{x_1, \ldots, x_{n_l}\} \cap \mathcal V$, 
		so we only need to prove that
		\begin{equation}\label{eq:requirement}
			\{1_{\{x_1 \in \mathcal V\}} = s_1, \ldots, 1_{\{x_{n_l} \in \mathcal V\}} = s_{n_l}, \Xi = \xi\} \in \mathcal F_{y_l}.
		\end{equation}
		To this end note that the value of $\mathscr C= \mathscr C_{\partial D_{y_l, L_0}}(Z_J)$ is 
		fixed by the event $\{\Xi = \xi\}$ (note that $y_l \in \partial_{n_l}$ due to 
		\eqref{eq:Eintersectnonemp}), and let $\widetilde{\mathscr C} = \overline{\mathscr C}  \cup 
		\{\text{the points outside the component of $C_{y_l, L_0}$ in $D_{y_l, L_0} \setminus 
			\overline{\mathscr C}$}\},$
		where $\overline{\mathscr C} = \mathscr C \cup (\partial^{{\rm out}}_{D_{y_l, L_0}}{{\mathscr 
				C}} )\cup \partial D_{y_l, L_0}$. Note that both $\widetilde {\mathscr C}$ and 
		$\widetilde{\mathscr C} \cap \mathcal V$ are determined on the event in 
		\eqref{eq:Eintersectnonemp}. 
		Therefore, in view of the definition of the $\sigma$-algebra $\mathcal F_{y_l} = \mathcal F_{y_l, L_0}(Z_J, \delta, u, u_{2, 3})$ in \eqref{def:FEsigma_algebra1} whence $\Xi$ and 
		$\mathcal I \cap ({\mathring D_{y_l, L_0}})^c$ are both measurable with respect 
		to $\mathcal F_{y_l}$, \eqref{eq:requirement} follows at once if
		\begin{equation}\label{eq:requirement1}
			\{x_1, \ldots, x_{n_l}\} \cap (\mathring D_{y_l, L_0} \setminus \widetilde{\mathscr C}) = \emptyset \mbox{ on the event in \eqref{eq:Eintersectnonemp}}.
		\end{equation}
		
		The remainder of the proof is devoted to showing \eqref{eq:requirement1}. 
		Recalling that $l_1 = \min \{1 \le m \le l: y_m = y_l\}$, we see that 
		$w_k \notin D_{y_l, L_0}$ for any $k < \tilde \tau_{l_1} = n_{l_1}$ and $w_{n_{l_1}} 
		\in \partial D_{y_{l}, L_0}$ on the event $\mathcal E$ (recall \eqref{eq:7.5meas-intermediate}). 
		Hence in view of \eqref{eq:Eintersectnonemp}, we immediately conclude that
		\begin{equation}\label{eq:requirement1_1}
			\{x_1, \ldots, x_{n_{l_1}}\} \cap (\mathring D_{y_l, L_0} \setminus \mathscr C) = \emptyset.
		\end{equation}
		As to $n_{l_1} < n < n_l$ (if $l_1 < l $), note that since \eqref{eq:Eintersectnonemp} holds and $\mathcal E  \subset \{\tilde \tau_l = n_l, \tilde Y_l = y_l, (\ast\ast)_l\}$, which itself is contained in $\{\tilde \tau_l = n_l, \tilde Y_l = y_l, \mathsf G_{n_l -1} \cap D_{y_l, L_0} \ne \emptyset, \lr{}{}{w_{n_l}}{C_{y_l, L_0} } \text{ in } \widetilde{\mathsf{G}}_{n_l} \cup \widetilde{\mathsf{W}}_{n_l}\}$ by \ref{def:**}, we have from \eqref{eq:w} and our rules for updating $\mathscr{A}_n$ and 
		$\tilde{\mathscr{A}}_n$ as described around \eqref{eq:GBW2} that 
		$\mathsf S_n \cap D_{y_l, L_0} = \mathsf S_{n_{l_1}} \cap D_{y_l, L_0}$ for all $n_{l_1} \le n < 
		n_l$ and $\mathsf S \in \{\mathsf B, \mathsf W, \mathsf G\}$  on the event $\mathcal E$. 
		Consequently, 
		\begin{equation}\label{eq:wGD}
			\{w_n\}_{n_{l_1} < n < n_{l}} \cap (\mathsf G_{n_{l_1}} \cap D_{y_l, L_0}) = \emptyset \text{ on the event } \mathcal E
		\end{equation}
		for otherwise the exploration would reveal some vertex in $\mathsf G_{n} \cap  D_{y_l, L_0} = 
		\mathsf G_{n_{l_1}} \cap D_{y_l, L_0}$ for some $n_{l_1} < n < n_{m}$ causing $\mathsf G_{n} 
		\cap  D_{y_l, L_0}$ to be different from $\mathsf G_{n_{l_1}} \cap  D_{y_l, L_0}$. Since $\mathcal E$ also implies the event $\{\tilde \tau_l = n_l, \tilde Y_l = y_l, \mathsf G_{n_l 
			-1} \cap D_{y_l, L_0} \ne \emptyset\},$
		it follows that we were in the scenario covered by Case~I at time $\tilde \tau_{l_1} = n_{l_1}$ 
		(see below \eqref{eq:W_n}) on the event $\mathcal E$ and consequently $\mathsf S_{n_{l_1}} \cap D_{y_l, L_0} = \widetilde{\mathsf S}_{n_{l_1}} \cap D_{y_l, L_0}$ for $\mathsf S \in \{\mathsf B, \mathsf W, \mathsf G\}$. But by \eqref{eq:B_n}--\eqref{eq:W_n}, we have $(D_{y_l, L_0} \setminus \widetilde{\mathsf G}_{n_{l_1}}) \subset  \widetilde{\mathscr C}$.  Together with \eqref{eq:wGD}, this implies that
		$\{w_n\}_{n_{l_1} < n < n_{l}}  \cap (\mathring D_{y_l, L_0} \setminus \widetilde{\mathscr C}) = \emptyset$ on the event in \eqref{eq:Eintersectnonemp}. In other words, on the event in \eqref{eq:Eintersectnonemp},
		\begin{equation}\label{eq:requirement1_2}
			\{x_{n_{l_1} +1}, \ldots, x_{n_{l} - 1}\} \cap (\mathring D_{y_l, L_0} \setminus \widetilde{\mathscr C}) = \emptyset.
		\end{equation}
		Finally, we have $w_{n_l} \in \partial D_{y_l, L_0}$ on the event $\mathcal E \subset \{\tilde 
		\tau_l = n_l\}$ and hence, again in view of \eqref{eq:Eintersectnonemp}, we have that $x_{n_l} 
		\notin \mathring D_{y_l, L_0}$. Combined with \eqref{eq:requirement1_1} and 
		\eqref{eq:requirement1_2}, this implies \eqref{eq:requirement1}, thus completing the proof. 
		\qedhere
	\end{proof}

\section{Denouement}\label{sec:denouement}
We now give the proofs of Theorems~\ref{T:ri-main} and \ref{T:ri-2point}, see \S\ref{subsec:slu} and \S\ref{subsec:supcrit_upper} 
respectively, drawing upon results developed in the previous three sections. Moreover, we derive in \S\ref{subsec:supcrit_upper} sharp upper 
bounds on the truncated one-arm probability for $\mathcal V^u$ in the supercritical regime $u < u_\ast$, see Theorem~\ref{thm:d=3} below, 
from which Theorem~\ref{T:ri-2point} follows immediately. In the intervening \S\ref{sec:apriori}, we provide 
an a priori estimate essential for the proofs of Theorems~\ref{T:ri-main} and \ref{thm:d=3}, which is interesting on its own.
\subsection{Proof of Theorem~\ref{T:ri-main}}\label{subsec:slu}
Theorem~\ref{T:ri-main} will be ultimately deduced from Theorems~\ref{prop:exploration_RI-I} and \ref{prop:exploration_RI-II} in view of 
\eqref{eq:V-first-inclusion}. Their usefulness hinges on suitable bounds for $\P[(\mathscr{G}_{z,N}^{\rm i})^c]$, ${\rm i} = {\rm I, II}$ (see~\eqref{def:mathscrG_RI} and 
\eqref{def:mathscrG_RI2}), for which we will rely on Proposition~\ref{prop:bootstrap_prob}.  This in turn requires verifying the condition \eqref{eq:initial_limit_prob} and 
notably to exercise control on a localized version of the events $\mathcal{G}_{z}^{\rm i}$, see 
\eqref{eq:bootstrap_prob_coupling}. The required estimate for ${\rm i} = {\rm I}$ is supplied by the following lemma. 
All parameters are assumed to satisfy \eqref{eq:params_RI}. 
\begin{lemma}[Seed estimates]\label{lem:mathcalG} 
	There exists a scale $L_0 = L_0(\bm u)$ where $\bm u = (u_0, u_1, u_2, u_3)$ and $\Cr{c:delta}(L_0) \in (0, \frac12)$ such that for all $\delta \in [0, \Cl[c]{c:delta}]$ and $L \ge 1$, with $\mathcal{G}_{0,L}^{\rm I}$ as in \eqref{def:good_events_supcrit},
	\begin{equation}\label{eq:mathcalGtrigger}
		\begin{split}
			&
			\P\left[\mathcal{G}_{0,L}^{\rm I}\big(\overline{Z}_{\mathbb{L}}, \delta, \bm u; a = 1 \big) \right] \ge 1 - C e^{-L^c}
		\end{split}
	\end{equation}
	for some $C = C(\bm u) \in (0, \infty)$ 
	with $\overline{Z}_{\mathbb{L}}=\{ \overline{Z}^u_z: z \in \mathbb{L}, u> 0\}$ as in \eqref{eq:3Vs}. 
\end{lemma}

Lemma~\ref{lem:mathcalG} follows by adapting the renormalization argument used to prove 
\cite[Lemma~5.16]{gosrodsev2021radius}, similarly as with \eqref{eq:connection_delta_d=3}. This hinges on a-priori 
estimates for $\P[({\rm FE}_{0, L_0})^c]$, which we prove in Lemma~\ref{lem:FE_likely}.

Next, we collect important localization properties of the events 
$\mathcal{G}_{0,L}^{\rm I}$ as one moves through the sequences of excursions in \eqref{eq:3Vs}. 
Recall that \eqref{eq:params_RI} is in force, and in particular that the scales $N,L, L_0$ and 
$K$ satisfy \eqref{eq:L_k_descending}. The following conclusions all hold uniformly in  $z 
\in \mathbb{L}$ 
without further mention. 
\begin{lemma}[Localization of 
	$\mathcal{G}_{z}^{\rm I}$]\label{L:loca-supercrit}
	For all 
	$u_0, v_0  \in [0, \infty), u_1 < u_2 < u_3 \in (0, \infty)$ and $v_1 < v_2 < v_3 \in (0, \infty)$ such 
	that  $u_0  < v_0$, $u_2 < v_2$, and $u_1 > v_1$, $u_3 > v_3$,  
	abbreviating 
	$\bm u = (u_0, \ldots, u_3)$, 
	$\bm u' = (u_0/8, u_1, \ldots, u_3)$ and 
	$\bm v = (v_0, \ldots, v_3)$, 
	$\bm v' = (v_0/8, v_1, \ldots, v_3)$, one has
	\begin{equation}\label{eq:supcrit_inclusion0}
		\begin{split}
			\mathcal{G}_{z}
			^{\rm I}(\overline{Z}_{\mathbb L}, \delta, \bm u; a)  \cap 
			\mathcal  F_{z}^{\bm u', \bm v'} \subset  \mathcal{G}_{z}
			^{\rm I}({Z}_{\mathbb L}, \delta, \bm  v;  a)
		\end{split}
	\end{equation}
	(see  \eqref{eq:F} and \eqref{eq:Fext} regarding the definition of $\mathcal F_{z}^{\bm u, \bm v}$). Moreover,  under any coupling $\mathbb Q$ of 
	$\mathbb P$ and 
	$\widetilde{\mathbb P}_{z}$,
	\begin{equation}\label{eq:supcrit_inclusion_Ztilde}
		\begin{split}	
			& \, \mathcal{G}_{z}
			^{{\rm I}}(\widetilde{Z}_{\mathbb L}, 
			\delta, \bm u; a) \, \cap \,  {\rm Incl}_{z}^{\frac{\varepsilon}
				{10}, \lfloor v\,{\rm cap}( 
				D_z)\rfloor} \subset \mathcal{G}_{z}
			^{{\rm I}}({Z}_{\mathbb L}, \delta, 
			\bm u(1, \varepsilon); a)  \mbox{ and } \\
			&\, \mathcal{G}_{z}
			^{{\rm I}}({Z}_{\mathbb L}, 
			\delta, \bm u; a) \, \cap \,  {\rm Incl}_{z}^{\frac{\varepsilon}
				{10}, \lfloor v\,{\rm cap}( 
				D_{z})\rfloor} \subset \mathcal{G}_{z}
			^{{\rm I}}(\widetilde{Z}_{\mathbb L}, \delta, 
			\bm u(1, \varepsilon); a)
		\end{split}	
	\end{equation} 
	for 
	$v  = \frac1{20} \, \min(u_0, u_1)$, 
	$\varepsilon \in (0, 1)$ and $L \ge C(v, \varepsilon)$,
	where 
	$
	\bm u(1, \varepsilon)$ stands for the tuple 
	$(u_{0, \varepsilon}, u_{1, -\varepsilon}, u_{2, \varepsilon}, u_{3, -\varepsilon})$ \sloppy with $u_{k, \epsilon} = u_k(1 + \epsilon)$ for any $\epsilon \in (-1, 1)$. 
\end{lemma}

\begin{proof}
	Let us recall from \eqref{def:good_events_supcrit} that $\mathcal{G}_{z}^{\rm I}$ is 
	increasing in all of ${\rm V}_z$, the events comprising ${\rm W}^{\rm I} $ and 
	$\mathscr{C}$ by \eqref{def:Gz}-\eqref{def:interface1}. We now inspect how each of these 
	events moves across the two sides of \eqref{eq:supcrit_inclusion0} starting with the events 
	${\rm W}^{\rm I}_{z, y}$. 
	Letting $u_0'=u_0/8$ and $v_0'=v_0/8$, the occurrence of the event 
	$\mathcal F_{z }^{u_0', v_0'} \cap \mathcal F_{z}^{u_1, v_1} (\supset \mathcal F_{z}^{\bm u', \bm v'})$, see \eqref{eq:F}, guarantees that any set of indices $J$ with $\{1,\dots, v_0' 
	\text{cap}( D_{z})\} \subset J \subset \{1,\dots, {v_1} \text{cap}( D_{z})\}$ 
	satisfies 
	$\{1,\dots, N_{z}^{u_0'}\} \subset J \subset \{1,\dots, N_{z}^{u_1}\}$. In other words, in view 
	of the definition of $Z_+$ in \eqref{eq:Z^+} and of \eqref{eq:3Vs},  on the event $\mathcal 
	F_{z}^{\bm u', \bm v'}$ we have the inclusion
	$	({Z}_{z}^{v_1})_+({\nu}_{z}( v_0')) \subset (\overline{Z}_{z}^{u_1})_+(\bar{\nu}_{z}( u_0'))$	(see also \eqref{eq:nu-hat} regarding $\nu_z$ and $\bar \nu_z$). By 
	\eqref{def:boosted} and the definition of ${\rm W}^{\rm I}$ in \eqref{eq:WI}, it then follows that
	\begin{equation}\label{eq:Winclusion}
		{\rm W}^{\rm I}_{z, y}(\overline{Z}_{\mathbb L}, u_0, u_1) \cap \mathcal F_{z}^{\bm u', \bm 
			v'} \subset {\rm W}^{\rm I}_{z, y}({Z}_{\mathbb L}, v_0, v_1)
	\end{equation}
	for any $z \in \mathbb L$ and $y \in \mathbb L_0$. Next, we deal with the event ${\rm V}_{z}={\rm V}_{z, 
		L}$ defined above \eqref{def:Fz}. It is clear from this definition that 
	${\rm V}_{z}(Z, Z', Z'', \delta)$ is decreasing in $Z$ and $Z''$ 
	and increasing in $Z'$ (w.r.t. inclusion of the underlying sets). Now on the event 
	$\mathcal F_{z}^{u_1, v_1} \cap \mathcal F_{z}^{u_2, v_2} \cap \mathcal F_{z}^{u_3, v_3} 
	(\supset \mathcal F_{z}^{\bm u', \bm v'})$, by \eqref{eq:3Vs} and \eqref{eq:F},
	\begin{equation}\label{eq:Z_include1}
		\begin{split}
			\{{Z}_{z}^{v_1}\} \subset \{\overline{Z}_{z}^{u_1}\}, \,  \{{Z}_{z}^{v_3}\} \subset 
			\{\overline{Z}_{z}^{u_3}\} \, \mbox{ and }\, \{\overline{Z}_{z}^{u_2}\} \subset \{{Z}_{z}^{v_2}\},
		\end{split}
	\end{equation}
	where $\{Z\} = \{Z_1, \ldots, Z_{n_Z}\}$ for any sequence $Z = (Z_j)_{1 \le j \le n_Z}$.  Although we don't need this for the purpose of dealing with ${\rm V}_{z}$ (but we'll use it shortly), the inclusions in \eqref{eq:Z_include1} do in fact hold as multisets (as per our convention below \eqref{eq:RI_basic_coupling2}) on the event $\mathcal F_{z}^{\bm u', \bm v'}$. 
	With \eqref{eq:Z_include1} at hand, in view of \eqref{def:Fz}, we have
	\begin{equation}\label{eq:Vinclusion}
		{\rm V}_{z}(\overline{Z}_{\mathbb L}, \delta, u_1, u_2, u_3) \cap 
		\mathcal F_{z}^{\bm u', \bm v'} \subset {\rm V}_{z}(Z_{\mathbb L}, \delta, v_1, v_2, v_3).
	\end{equation}
	By similar arguments (see \eqref{eq:script_C} regarding $\mathscr C_{z}$), we also obtain that 
	\begin{equation}\label{eq:Cinclusion}
		\mathscr C_{z}(\overline{Z}_{\mathbb L}, \delta, u_1, u_3) \cap \mathcal F_{z}^{\bm u', \bm v'} \subset \mathscr C_z({Z}_{\mathbb 
			L}, \delta, v_1, v_3).
	\end{equation}
	Together with the observation made in the line below \eqref{eq:supcrit_inclusion0} and the 
	definition of $\mathcal{G}_{z}^{\rm I}$ in~\eqref{def:good_events_supcrit}, 
	the displays \eqref{eq:Winclusion}, \eqref{eq:Vinclusion} and \eqref{eq:Cinclusion} yield \eqref{eq:supcrit_inclusion0}.

The arguments leading to 
	\eqref{eq:supcrit_inclusion0} also straightforwardly yield \eqref{eq:supcrit_inclusion_Ztilde}. Indeed, to pass to the $\widetilde{Z}_{\mathbb 
		L}$-version of the events $\mathcal G_z^{\rm i}$, the event ${\rm Incl}_z^{\varepsilon/10, m_0}$ 
	in \eqref{eq:RI_basic_coupling2} with $m_0 = \lfloor v\,{\rm cap}( 
	D_z)\rfloor$ and $v$ as below \eqref{eq:supcrit_inclusion_Ztilde} precisely ensures the desired inclusions between the relevant multisets of excursions belonging to ${Z}_{\mathbb 
		L}$ and $\widetilde{Z}_{\mathbb L}$. \end{proof}
We can now already conclude Theorem~\ref{T:ri-main}. 
\begin{proof}[Proof of Theorem~\ref{T:ri-main}]
	We prove \eqref{eq:ri-unique} and explain at the end of the proof how \eqref{eq:ri-equal} is 
	deduced. With Lemma~\ref{lem:VzSLU_inclusion} at hand, applying \eqref{eq:V-first-inclusion} with the 
	choices $u_0 = \bar u_0(u) > 0$ (to be specified) and $\nu=0$ at scale $L/100$ together with a union bound 
	over $z$ and using that the probability that $\big\{ N_{z, L/100}^{{3\bar u_0}/{2}}- N_{z, L/100}^{{\bar u_0}/{2}} =0 \big\} 
	$ (recall that the difference on the left 
	dominates a Poisson variable with mean $\bar u_0 \,\text{cap}(B_{L/100})$) is bounded from 
	above by $\exp\{-c\bar u_0L^{d-2}\} \le \exp\{-c L^{d-2}\}$,  the task of proving \eqref{eq:ri-unique} reduces to showing that 
	for suitable $c=c(d) >0$, all $u \in (0, u_*)$, $N \geq C(u)$ and ${\rm i} = {\rm I,II}$,
	\begin{equation}
		\label{eq:SLU-reduc1}
		\P[({\rm V}_{0,N}^{\rm i}(\nu = 0; \bar u_0, u))^c] \leq e^{-N^{c}},
	\end{equation}
	for some value $\bar u_0 = \bar u_0(u) \in (0, \frac{u_\ast}2)$. We will prove \eqref{eq:SLU-reduc1} for ${\rm i} = {\rm I}$ by application of 
	Theorem~\ref{prop:exploration_RI-I} aided by Lemmas~\ref{lem:mathcalG} and \ref{L:loca-supercrit}. In view of \eqref{eq:truncated_RI-I}, the first and the 
	main step is to derive a similar estimate on $\P[(\mathscr G_{0,N}^{\rm I})^c]$ which we will obtain by means of Proposition~\ref{prop:bootstrap_prob} 
	starting with the bound on $\P[(\mathcal G_{0, L}^{{\rm I}})^c]$ given by Lemma~\ref{lem:mathcalG} corresponding to the choice of parameters
	\begin{equation}\label{eq:params_RIep0}
		\text{\begin{minipage}{0.9\textwidth}
				$\bm u = (u_0 = \tfrac{\Cl[c]{c:u_4}}{10}, u_1, u_2, u_3)$, $L_0 = L_0(\bm u)$ and $\delta$ satisfying $\delta \in (0, \Cr{c:delta}(L_0)]$ 
				s.t.~\eqref{eq:params_RI} holds.
		\end{minipage}}
	\end{equation}
	Here $\Cr{c:u_4} \in (0, u_\ast)$ is a constant (see below). The case ${\rm i} = {\rm II}$ follows from the exact same 
	arguments, using Theorem~\ref{prop:exploration_RI-II} in lieu of 
	Theorem~\ref{prop:exploration_RI-I} which involves the `type II versions' of the events 
	$\mathcal G_z^{{\rm I}}$ and $\mathscr G_z^{{\rm I}}$, namely $\mathcal G_z^{{\rm II}}$ 
	and $\mathscr G_z^{{\rm II}}$, defined in \eqref{def:good_events_supcrit2} and 
	\eqref{def:mathscrG_RI2} respectively and analogous statements to \eqref{eq:mathcalGtrigger}--\eqref{eq:supcrit_inclusion_Ztilde} for the events $\mathcal 
	G_{z, L}^{{\rm II}}$ (the required a-priori estimate to prove the analogue of \eqref{eq:mathcalGtrigger} for type ${\rm II}$ is supplied by \cite[Corollary~3.7]{MR3269990}, which replaces Lemma~\ref{lem:FE_likely} below). The choice of the parameter $u_0 = \frac{\Cr{c:u_4}}{10}$ in \eqref{eq:params_RIep0} above is solely informed by the fact that the 
	analogue of \eqref{eq:mathcalGtrigger} holds for $\mathcal G_{0, L}^{{\rm II}}(\overline{Z}_{\mathbb L}, u_4; a = 1)$ (see 
	\eqref{def:good_events_supcrit2}) when $u_4 \in [0, \Cr{c:u_4}]$ thus reflecting the crossover between the type I and type II regimes (see below 
	\eqref{eq:V_z^II}).

	We now proceed with the proof of \eqref{eq:SLU-reduc1} for ${\rm i} = {\rm I}$. Combining \eqref{eq:mathcalGtrigger}
	with the inclusion \eqref{eq:supcrit_inclusion0}, we obtain that for any $z \in \mathbb 
	L$, 
	$L_0$ and $\delta$ as in \eqref{eq:params_RIep0}, and for $\varepsilon \in (0, 1)$,  $K \geq 
	C(\varepsilon)$, 
	\begin{multline}\label{eq:V_to_V_L_supcrit_rnd1}
		\P[\mathcal{G}_{z}^{\rm I}({Z}_{\mathbb L}, \delta, \bm u(1, \varepsilon); a = 1) ] \\
		\stackrel{\eqref{eq:supcrit_inclusion0}}{ \ge}  
		\P[\mathcal{G}_{z}^{\rm I}(\overline{Z}_{\mathbb L}, \delta, \bm u; a = 1) ] - 
		\P[(\mathcal F^{\bm u', {\bm u(1, \varepsilon)}'}_{z})^c]
		\stackrel{\eqref{eq:mathcalGtrigger},  \eqref{eq:Nuz_tail_bnd}}{\ge}  1 - 
		C(\varepsilon, \bm u) e^{-L^c},
	\end{multline}
	where 
	$\bm u(k, \varepsilon) = (u_0(1 + \varepsilon)^k, u_1(1 - \varepsilon)^k, u_2(1 + \varepsilon)^k, 
	u_3(1 - \varepsilon)^k)$ \sloppy for any $k \in \N$ (cf.~$\bm u(1, \varepsilon)$ below 
	\eqref{eq:supcrit_inclusion_Ztilde}) and $\bm u'$ (or ${\bm u(1, \varepsilon)}'$) is as above 
	\eqref{eq:supcrit_inclusion0}. 
	Now using the second inclusion in \eqref{eq:supcrit_inclusion_Ztilde}, we can use the coupling 
	$\mathbb Q_{\{z\}}$ from Lemma~\ref{L:RI_basic_coupling} 
	to deduce that 
	\begin{multline}\label{eq:V_L_to_tildeV_L_supcrit_rnd1}
		\P[\mathcal{G}_{z}^{\rm I}(\widetilde{Z}_{\mathbb L}, \delta, \bm u(2, \varepsilon); a = 1) ] 
		\stackrel{\eqref{eq:inclusion}}{\ge} 
		\P[\mathcal{G}_{z}^{\rm I}({Z}_{\mathbb L}, \delta, \bm u(1, \varepsilon); a = 1) ] 
		\widetilde \P_z[(\mathcal U_z^{\frac{\varepsilon}{ 10}, \lfloor (1-\varepsilon)v\,{\rm cap}(D_{z, L})\rfloor})^c] \\
		\stackrel{\eqref{eq:V_to_V_L_supcrit_rnd1}, \eqref{eq:bnd_Uzepm}}{\ge}  \, 1 - C(\varepsilon, \bm u) e^{-L^c} - C \varepsilon^2 e^{-c(\varepsilon, \bm u)L^{d-2}} \ge 1 - C(\varepsilon, \bm u) e^{-L^c},
	\end{multline}
	for all $L \ge C(\varepsilon)$ and $K \ge \frac{ 30\Cr{c:equil}}{\varepsilon} \vee C(\varepsilon)$ (ensures that condition 
	\eqref{eq:RI_cond_Q} holds on account of Proposition~\ref{prop:entrance_time_afar} and that \eqref{eq:V_to_V_L_supcrit_rnd1} applies) where 
	$v = \frac1{20}\, \min(u_0, u_1)$ as below \eqref{eq:supcrit_inclusion_Ztilde}. 

	Next, in view of the first inclusion in \eqref{eq:supcrit_inclusion_Ztilde} together with 
	Lemma~\ref{L:RI_basic_coupling} and Proposition~\ref{prop:entrance_time_afar}, and the 
	probability bounds in \eqref{eq:V_L_to_tildeV_L_supcrit_rnd1} above and 
	\eqref{eq:bnd_Uzepm}, we see that the conditions 
	\eqref{eq:bootstrap_prob_coupling}--\eqref{eq:initial_limit_prob} of 
	Proposition~\ref{prop:bootstrap_prob} are satisfied by the events 
	\begin{equation}\label{eq:supcrit_GtildeG}
		\begin{split}
			\widetilde {\mathcal G}_{z, L} = \mathcal{G}_{z}^{\rm I}(\widetilde{Z}_{\mathbb L}, \delta, \bm u(2, \varepsilon); a = 1)
			\text{ and } 
			{\mathcal G}_{z, L} = 
			\mathcal{G}_{z}^{\rm I}({Z}_{\mathbb L}, \delta, \bm u(3, \varepsilon); a = 1)
		\end{split}
	\end{equation}
	and for 
	$\varepsilon_L = \frac{\varepsilon}{10}$, 
	$m_L = \lfloor v(1 - \varepsilon)^2\,{\rm cap}( D_{z, L})\rfloor$, $\beta' = c \in (0, \infty)$, $K_0 = C(\varepsilon)$ and $L_0 = C(\varepsilon, \bm u)$ (as for \eqref{eq:initial_limit_prob} to hold). 
	Let us suppose for the remainder of this proof that, on top of the conditions specified 
	in 
	\eqref{eq:params_RIep0}, $\bm u$ and $\varepsilon$ also satisfy
	\begin{equation}\label{eq:params_RIep}
		\text{\begin{minipage}{0.7\textwidth}
				$u < u_1(1 - \varepsilon)^4$, $u_2(1 + \varepsilon)^3 < u_3(1 - \varepsilon)^4$ and 
				$2u_0(1 + \varepsilon)^3 < u_\ast$
		\end{minipage}}
	\end{equation}
	(cf.~\eqref{eq:params_RI} and also the assumptions underlying 
	Theorem~\ref{prop:exploration_RI-I}). 
	Then in dimension $d \ge 4$, choosing 
	$(\bm u, \bm v)  = ({\bm u(3, \varepsilon)}_1, {\bm u(3, \varepsilon)}_2)$ 
	for ${\rm i} = {\rm I}$, 
	where ${\bm u(3, \varepsilon)}_1$ and ${\bm u(3, \varepsilon)}_2$ are 
	defined 
	exactly as $\bm u_1$ and $\bm u_2$ with $\bm u(3, \varepsilon)$ replacing $\bm u$ 
	in \eqref{def:F_y_RI}, 
	$\widetilde{\mathcal G}_{z, L}, {\mathcal G}_{z, L}$ as in \eqref{eq:supcrit_GtildeG}, $K = K(\varepsilon, \bm u)$, $L = L(\varepsilon, \bm u)$ large enough and $\Lambda_N = \tilde D_{0, N} \setminus \tilde C_{0, N}$, we obtain from \eqref{eq:bootstrapped_limit_probd=4} that for 
	all $N \ge 1$, 
	\begin{equation}\label{eq:supcrit_GN_lower_bnd0d=4}
		\begin{split}
			\P[\mathscr G_{0, N}^{\rm I}(Z_{\mathbb L}, \delta, \bm u(3, \varepsilon); a = 1)]	\ge
			1 - 
			C (\varepsilon, \bm u) e^{-c(\varepsilon, \bm u) N}.
		\end{split}
	\end{equation}	
	On the other hand for $d = 3$, \eqref{eq:bootstrapped_limit_prob1d=3} yields with the choice 
	of $(\bm u, \bm v)$, $\widetilde{\mathcal G}_{z, L}$ and $\mathcal G_{z, L}$ as above, $K = 
	K(\varepsilon,\bm  u)$, $\delta = \frac12$ for the parameter appearing in \eqref{eq:bootstrapped_limit_prob1d=3}), $\rho = \frac1{2\Cr{C:rho}}$, $L = L(N) = \lfloor (\log 
	N)^\alpha \rfloor$ for some absolute constant $\alpha \in (0, \infty)$ and $\Lambda_N = \tilde D_{0, 
		N} \setminus \tilde C_{0, N}$ that for all $N \ge 1$, with $\beta \in (0, \infty)$ an absolute constant (determined by the choice of $\beta'=c$ from below \eqref{eq:supcrit_GtildeG}),
	\begin{equation}\label{eq:supcrit_GN_lower_bnd0d=3}
		\begin{split}
			\P[\mathscr G_{0, N}^{\rm I}(Z_{\mathbb L}, \delta, \bm u(3, \varepsilon) ; a = 1)]	\ge 1 - C(\varepsilon, \bm u)e^{- N{( 1 \vee \log N)^{-\beta}}}.
		\end{split}
	\end{equation}	
	
	Now 
	plugging the bounds \eqref{eq:supcrit_GN_lower_bnd0d=4} and \eqref{eq:supcrit_GN_lower_bnd0d=3} into the right-hand side of \eqref{eq:truncated_RI-I} 
	in Theorem~\ref{prop:exploration_RI-I} 
	with 
	$\nu = 0$ (the required conditions are ensured by \eqref{eq:params_RIep0} and \eqref{eq:params_RIep}), we get
	that for all $N \ge 2$ and $d \geq 3$, with ${\rm V}_{0,N}^{{\rm I}} = {\rm V}_{0,N}^{\rm 
		I}(\nu = 0; \bar u_{0} = u_0(1 + \varepsilon)^3, u)$, 
	\begin{equation}\label{eq:V_prob_lower_bnd_dge4}
		\P[({\rm V}_{0,N}^{\rm I})^c] 
		\le 
		\P\big[\text{Disc}_{0,N}^{\rm I}\big] + 
		C(\delta, \varepsilon, \bm u) \times \eta(N),
	\end{equation}
	where  $\text{Disc}_{0,N}^{\rm I} \coloneqq \big\{ \lr{}{}{ C_{0, N}}{\partial  D_{0, N}} \text{ in $(\mathcal V^{u_{2, 3}(\varepsilon)})_{2\delta}$} \big\}^c,$ $u_{2, 3}(\varepsilon) = \frac12(u_2(1 + \varepsilon)^3 + u_3(1 - \varepsilon)^3)$ (recall \eqref{def:F_y_RI}) and $\eta(N)= e^{- {N}{( \log N )^{-\beta}}} $, if $d=3$ and $\eta(N)= e^{-c(\delta, \varepsilon, \bm u) N}$ if $ d\geq 4$. 
		
		We already have from \eqref{eq:connection_delta_d=3} in Remark~\ref{R:subcrit-variation} a 
		bound on the disconnection probability in \eqref{eq:V_prob_lower_bnd_dge4} when $\delta \in 
		(0, \Cr{c:delta-sub}(u_{2, 3}(\varepsilon))]$. 
		To conclude the proof of \eqref{eq:SLU-reduc1} we set, for any given $u \in (0, u_*)$,
		\begin{multline}\label{eq:uchoice}
			u_0 = \tfrac{\Cr{c:u_4}}{10}, u_1 = u_\ast(1 - \varepsilon)^{10}, u_2 = u_\ast(1 - \varepsilon)^9, 
			u_3 = u_\ast(1 - \varepsilon), 
			\\ \text{and }\varepsilon = ((1 - (\tfrac{u}{u_\ast})^{\frac1{20}}) \wedge \tfrac1{10}, \, \text{$\delta = \tfrac{\Cr{c:delta-sub}(u_{2, 3}(\varepsilon))}2 \wedge  \Cr{c:delta}(L_0(\bm u)) \, (> 0)$}
		\end{multline}
		with $L_0(\cdot)$ provided by Lemma~\ref{lem:mathcalG}. We see that the conditions 
		\eqref{eq:params_RIep0} and \eqref{eq:params_RIep} are satisfied and the bound 
		\eqref{eq:connection_delta_d=3} 
		holds for the value $u = u_{2, 3}(\varepsilon)$. 
		Therefore we can plug 
		\eqref{eq:connection_delta_d=3} 
		into the right-hand side of \eqref{eq:V_prob_lower_bnd_dge4} to deduce 
		\eqref{eq:SLU-reduc1} for ${\rm i} = {\rm I}$ which, along with its type II analogue, concludes the proof of \eqref{eq:ri-unique}.

		It remains to argue that \eqref{eq:ri-equal} holds. The following inclusion of events follows from 
		the definition of ${\rm SLU}_L(u)$ in \eqref{eq:slu}. For any $v \in [0, u]$ and $x \in \Z^d$ such 
		that $|x|_{\infty} \ge 2$, one has
		\begin{equation}\label{eq:incl_tr_slu}
			\big\{0 \stackrel{}{\longleftrightarrow}x \nlr{}{}{}{\infty} \text{ in } \mathcal{V}^v\big\} 
			\subset \big({\rm SLU}_{|x|_{\infty}}(u) \, \cap \, \big\{B_{|x|_{\infty} / 4} 
			\stackrel{}{\longleftrightarrow}\infty \text{ in } \mathcal{V}^v\big\}\big)^c.
		\end{equation}
		Also following the derivation of (5.73) in \cite{gosrodsev2021radius}, by combining 
		\eqref{eq:ri-unique}, the disconnection estimate from \cite[Thm.~7.3]{MR3602841}, which holds for all $u < \bar u$ and the equality of $u_\ast$ and $\bar u$ of \cite[Thm.~1.2]{RI-I}, we obtain that the connection to infinity on the right of \eqref{eq:incl_tr_slu} has probability at least
		$1 - C(v)e^{-|x|^c}$, for all $v \in [0, u_\ast)$ and $x \in \Z^d$. Since the connection event in question is decreasing w.r.t.~$v$, 
		feeding the previous bound together with \eqref{eq:ri-unique} into \eqref{eq:incl_tr_slu} yields that 
		$\tau_v^{{\rm tr}}(x, y) = \tau_v^{{\rm tr}}(0,y - x) \le C(u) e^{-|x - y|^c}$, {\em uniformly} over all $v \in [0, u]$ and $x, y \in \Z^d$ when $u < u_\ast$ (see~\eqref{eq:tau-RI}). But this is precisely the asserted equality of $\widehat{u}$ 
		and $u_\ast$ in \eqref{eq:ri-equal} in view of the definition of $\widehat{u}$ in \eqref{eq:u-hat}.
	\end{proof}
\subsection{A-priori estimates}\label{sec:apriori}
In this section we prove the promised a-priori estimate required for the proof of Lemma~\ref{lem:mathcalG} (cf.~below the statement of Lemma~\ref{lem:mathcalG}).
\begin{lemma}[${\rm FE}_{y, L_0}$ is likely]
\label{lem:FE_likely}
For any $u_0, u  \in (0, \infty)$,  
$K \ge 100$ (from \eqref{eq:L_k_descending}) and $L_0\geq 1$,  
we have
\begin{equation}\label{eq:FE_likely}
\P \big[ {\rm FE}_{0, L_0}((\overline{Z}_{0, L_0}^{u})_+( N_{0, L_0}^{u_0}))
\big]
\ge 1 - 
C(u, u_0)\,e^{-c(u) L_0^c}.
\end{equation}
\end{lemma}
\begin{proof}[Proof of Lemma~\ref{lem:FE_likely}]
By \eqref{def:boosted} and \eqref{eq:Z^+}, and with a similar argument as above \eqref{eq:Winclusion}, we can write
\begin{equation} \label{eq:FE_likely_global_local}
{\rm FE}_{0, L_0}\big(({Z}_{0, L_0}^{2u})_+( \tfrac{u_0}2{\rm cap}(D_{0, L_0}))\big) \cap \mathcal F_{0, 
L_0}^{u, 2u} \cap \mathcal F_{0, L_0}^{u_0, {u_0}/2} \subset {\rm FE}_{0, L_0}((\overline{Z}_{0, L_0}^{u})_+( N_{0, L_0}^{u_0})).
\end{equation}
Throughout the remainder of this proof we omit the subscripts ``$0,L_0$'' from all notations, so 
$Z^u= {Z}_{0, L_0}^{u}$, ${\rm FE}(\cdot)= {\rm FE}_{0, L_0}(\cdot)$ etc. 
Recall from \eqref{def:FE1} that ${\rm FE}(\cdot)= {\rm LU}(\cdot) \cap {\rm O}(\cdot)$. To deal with the presence of $({Z}^{2u})_+( \tfrac{u_0}2{\rm cap}(D))$, we first observe that, since the event ${\rm O}(Z)$ is decreasing in $Z$ w.r.t. inclusion of the multiset  $\{Z\}$ 
(revisit definition~\eqref{def:Conn0} and our convention below \eqref{eq:RI_basic_coupling2}), it follows from \eqref{def:boosted} that
\begin{equation}\label{eq:FE_likely_O_inclusion}
{\rm O}(Z^{2u}) \subset {\rm O}\big(({Z}^{2u})_+( \tfrac{u_0}2{\rm cap}(D))\big). 
\end{equation}
As to ${\rm LU}(Z)$, by 
similar arguments as for 
\eqref{eq:supcrit_inclusion_Ztilde}, 
we obtain that under any coupling $\mathbb Q$ of $\mathbb P$ and 
$\widetilde{\P}_0$,
\begin{equation}\label{eq:FE_likely_LU_tildeinclusion}
\begin{split}
{\rm LU} \big((\widetilde{Z}^{4u})_+( \tfrac{u_0}4{\rm cap}(D)) \big)
\cap {\rm Incl}_0^{\frac1{10}, \frac{u_0}{8}\,{\rm cap}(D)} \subset {\rm 
LU} \big(({Z}^{2u})_+( \tfrac{u_0}2{\rm cap}(D))\big).
\end{split}
\end{equation}
for all 
$L_0 \ge C(u_0)$. 
Now we 
argue that for any two 
finite 
subsets $J$ and $ J'$ of $\N^\ast$, 
\begin{equation}\label{eq:FE_likely_LU_inclusion}
\begin{split}
\textstyle \bigcap_{J'' \subset J': |J''| \le 2}{\rm LU}\big(\widetilde{Z}^{D, U}_{J \cup J''}\big) \subset  {\rm LU}\big(\widetilde{Z}^{D, U}_{J \cup J'}\big)
\end{split}
\end{equation}
where $\widehat{Z}_J^{D, U} = (\widehat Z_j^{D, U})_{j \in J}$ (with $U=U_{0,L_0}$). 
To see this suppose that we are {on the event} at the left-hand side of \eqref{eq:FE_likely_LU_inclusion} 
and consider 
$x', x'' \in \mathcal I\big(\widetilde{Z}^{D, U}_{J \cup J'}\big) \cap 
(\tilde D \setminus \tilde C)$.  If neither $x'$ nor $x''$ lie in 
$\mathcal I(\widetilde{Z}^{D, U}_J)$, then there exist $j', j'' \in J'$ such that 
$x' \in \mathcal I(\widetilde{Z}^{D, U}_{j'})$ and $x'' \in \mathcal I(\widetilde{Z}^{D, 
U}_{j''})$. Recalling the definition of ${\rm LU}(\cdot)$ from \eqref{def:Conn}, and since 
${\rm LU}\big(\widetilde{Z}^{D, U}_{J \cup \{j', j''\}}\big)$ occurs, we 
have in this case that 
$x'$ and $x''$ are connected in 
$\mathcal I\big(\widetilde{Z}^{D, U}_{J \cup \{j', j''\}}\big) \cap 
(\mathring D \setminus C)$, which is contained in $\mathcal I\big(\widetilde Z^{D, U}_{J \cup J'}\big) \cap (\mathring D \setminus C)$ (recall that 
$\mathring D = D \setminus \partial D$). Similarly we can verify the cases $x', 
x'' \in \mathcal I(\widetilde Z^{D, U}_J)$, $x'\in 
\mathcal I(\widetilde Z^{D, U}_J)$ and $x'' \in 
\mathcal I(\widetilde Z^{D, U}_{J'})$. 
All in all, the inclusion in \eqref{eq:FE_likely_LU_inclusion} 
follows.

Now recall from 
\eqref{eq:Z^+} that any $Z \in (\widetilde{Z}^{4u})_+( \tfrac{u_0}4{\rm cap}(D))$ must 
necessarily contain 
$\widetilde{Z}^{{u_0}/4}$ as a subsequence (see \eqref{eq:3Vs} for 
notation). Therefore, we obtain from \eqref{eq:FE_likely_LU_inclusion} that 
\begin{equation}\label{eq:LU_simplify}\textstyle
\bigcap_{J \in \mathscr{J}_{u_0, u} } {\rm LU}(\widetilde{Z}_J^{D, U}) \subset {\rm LU}\big(
(\widetilde{Z}^{4u})_+( \tfrac{u_0}4{\rm cap}(D))\big) 
\end{equation}
where 
\begin{equation}\label{def:Ju0u}
	\begin{split}
		\mathscr{J}_{u_0, u} \coloneqq \ &\mbox{the collection of all subsets of $\{1, \ldots, \lfloor 
			4u \, {\rm cap}(D)\rfloor\}$ of the}\\ &\mbox{form $\{1 , \ldots, \lfloor \frac{u_0}{4} \,{\rm cap}(D)\rfloor\} \cup 
			J''$ for some 
			$J'' \subset \N^\ast$ such that
		$|J''| \le 2$.}	\end{split}
\end{equation}
The major gain from the inclusion in \eqref{eq:LU_simplify} is that the (complement of the) event on the 
left-hand side is now amenable to a union bound argument as $|
\mathscr J_{u_0, u}|$ is at most a power of $L_0$ (see below). With \eqref{eq:FE_likely_O_inclusion} and \eqref{eq:FE_likely_LU_tildeinclusion}, 
\eqref{eq:LU_simplify} implies in view of~\eqref{def:FE1} that under any coupling $
\mathbb Q$ of $\mathbb P$ and 
$\widetilde{\P}_0$, the inclusion 
\begin{equation*}
\begin{split}
{\rm O}(Z^{2u}) \cap \textstyle \big( \bigcap_{J \in \mathscr 
	J_{u_0, u} } {\rm LU}(\widetilde{Z}_J^{D, U})\big) \cap \, {\rm Incl}_0^{\frac1{10}, \frac{u_0}{8}\,{\rm cap}(D)} 
\subset 
{\rm FE}\big(
({Z}^{2u})_+( \tfrac{u_0}2{\rm cap}(D))\big)
\end{split}.	
\end{equation*}
holds. Finally, plugging this into the left-hand side of \eqref{eq:FE_likely_global_local} yields that 
\begin{equation}\label{eq:FE_likely_simplify}
{\rm O}
(Z^{2u}) \cap \textstyle  \big(\bigcap_{J \in \mathscr{J}_{u_0, u} } {\rm LU}(Z_J^{D, U})\big) \cap \, {\rm Incl}_0^{\frac1{10}, \frac{u_0}{8}\,{\rm cap}(D)} \cap \, \mathcal F^{u, 2u} \cap \mathcal F^{u_0, \frac{u_0}2}
 \subset \,  
{\rm FE}\big(
(\overline{Z}^{u})_+( N^{u_0} )\big)
\end{equation}
under any coupling $\mathbb Q$ of $\mathbb P$ and 
$\widetilde{\mathbb P}_0$ and for all $L_0 \ge C(u)$.
%
%
%
%
%
%
%
%
%
We will now bound from below the probabilities of each of the events on the left-hand side of \eqref{eq:FE_likely_simplify} starting with ${\rm O}(Z^{2u})$. From its definition 
 in \eqref{def:Conn0}, property \eqref{eq:law_muKu}, 
standard tail bound for a Poisson variables, 
and the exponential decay of the tail of occupation time for transient random walks, we 
readily obtain that
\begin{equation}\label{eq:Obound-apriori}
\P\big[{\rm O}(Z^{2u})\big] \ge 1 - e^{-c(u) L_0^c}.
\end{equation}
Next we want to prove that for any $K \ge 100$ 
and $J \in \mathscr J_{u_0, u}$, one has
\begin{equation}\label{eq:LU_Ztilde_bnd}
\P\big[{\rm LU}(\widetilde Z_J^{D, U})\big] \ge 1 - C(u_0)e^{-L_0^c}.
\end{equation}
To this end, for any $C > 0$, $R' > R \ge 1$ and $x \in \Z^d$, let us consider the event
$\overline{\rm LU}_{x, R, R'}(\widetilde Z_J^{D, U})$ defined as
$\textstyle \bigcap_{x, x'} \{\lr{}{ }{x}{x'} \text{ in } \mathcal I(\widetilde Z_J^{D, U}) \, \cap  \, B_{R'}(x)\}$, with $x, x' \in B_R(x) \, \cap \, \mathcal I(\widetilde Z_J^{D, U})$
(cf.~\eqref{def:Conn}), where $B_{R}(x) \subset \Z^d$ denotes the closed $\ell^\infty$-ball of radius $R$ centered at $x$ for any $R \ge 0$ and $x \in \Z^d$. It is enough to show that there exists $C > 0$ such that for any $R 
\ge 1$ and $x \in \Z^d$ satisfying $B_{CR}(x) \subset  D$, and for any $K \ge 100$ 
and $J \in \mathscr J_{u_0, u}$, the bound
\begin{equation}\label{eq:lower_bnd_tildeLU}
\P\big[ \overline{\rm LU}_{x, R, CR}(\widetilde Z_J^{D, U}) \big] \ge 1 - 
C(u_0)e^{-L_0^c}
\end{equation}
 holds; we can deduce \eqref{eq:LU_Ztilde_bnd} from this by a standard 
covering argument, 
see, e.g.,~the proof of Proposition~1 at the end of page~390 in \cite{RathSapoztransience11}. In view of the definition of $\mathscr J_{u_0, u}$ from \eqref{def:Ju0u}, 
we need to verify \eqref{eq:lower_bnd_tildeLU} when $J = \{1, 
\ldots, \lfloor\tfrac{u_0}4 \,{\rm cap}(D)\rfloor + k\}$, for $k\in \{0,1,2\}$. For any such $k$,  \eqref{eq:lower_bnd_tildeLU}  follows by adapting the arguments in the proof of \cite[(5.4) and (5.20)]{DPR22}. We omit further details.

Next, we note that Lemma~\ref{L:RI_basic_coupling} together  
with \eqref{eq:RI_basic_coupling2}, Proposition~\ref{prop:entrance_time_afar} and the 
bound in \eqref{eq:bnd_Uzepm} gives us that there is a coupling $\mathbb Q$ of $\mathbb P$ 
and 
$\widetilde{\mathbb P}_0$ such that
\begin{equation}\label{eq:bnd_incl}
\mathbb Q\big[ {\rm Incl}_0^{\frac1{10}, \frac{u_0}{4}\,{\rm cap}(\check D)}\big] \ge 1 - Ce^{-c(u_0)L_0^{d-2}}.
\end{equation}
Finally, from \eqref{eq:Nuz_tail_bnd} we have 
$\P\big[(\mathcal F_{}^{u_0, \frac{u_0}{2}} \cap 
\mathcal F_{}^{u, 2u})^c\big] \le 2 e^{-c\, (u_0 \wedge u) L_0^{d-2}}$.
Now plugging this along with the estimates \eqref{eq:Obound-apriori}, \eqref{eq:LU_Ztilde_bnd} and \eqref{eq:bnd_incl} 
 into \eqref{eq:FE_likely_simplify} 
after applying a union bound and using that $|{\mathscr J_{u_0, u}}| \le C (u\vee1)^3 L_0^C$ on account of \eqref{def:Ju0u} yields \eqref{eq:FE_likely}.\end{proof}

\subsection{Bounds for the supercritical phase}\label{subsec:supcrit_upper}
Let us introduce the \emph{local uniqueness} event
\begin{equation}\label{def:locuniq}
	{\rm LocUniq}(N, u) \coloneqq
	\left\{
	\text{$\mathcal{V}^u$ has a unique cluster crossing $B^2_{2N} \setminus B^2_N$}
	\right\} 
\end{equation}
as well as $2$-${\rm arms}(N, u)$, the two arms event in $B^2_{2N} \setminus B^2_N $, 
which refers to the existence of (at least) two 
crossings 
of $(B^2_{2N} \setminus B^2_N)$ in $\mathcal{V}^u$ that are not connected in $\mathcal 
V^u \cap (B^2_{2N} \setminus B^2_N)$. The two-arms event is clearly a subset of ${\rm LocUniq}(N, u)^c$. 
We can use any $\Lambda_N$ from \eqref{eq:scriptS_N} in place of $B^2_{2N} \setminus 
B^2_N$ and denote the resulting events as ${\rm LocUniq}(\Lambda_N, u)$ and $2$-${\rm 
arms}(\Lambda_N, u)$ respectively.
\begin{thm}[Supercritical regime] \label{thm:d=3}
	For all $u\in (0,u_*)$, 
	\begin{align}	
		&\sup_{N \ge 1}N^{-1}\log \P\big[ \lr{}{}{0}{\partial B_N^2}, \nlr{}{}{0}{\infty} 
	\text{ in } \mathcal V^u	\big] \le -c(u),  \text{ if $d \ge 4$;} \label{eq:d=4-sup}\\
		& \limsup_{N \to \infty} \frac{\log N}{N}\log \P\big[ \lr{}{}{0}{\partial B_N^2}, \nlr{}{}{0}{\infty} \text{ in } \mathcal V^u
		\big] \leq -\frac{\pi}{3}(\sqrt{u} - \sqrt{u_*})^2,  \text{ if $d = 3$.} \label{eq:d=3-sup}
	\end{align}
	Moreover, the bounds  \eqref{eq:d=4-sup} and \eqref{eq:d=3-sup} also hold for the events ${\rm LocUniq}(N, u)^c$ 
	and $2$-${\rm arms}(N, u)$. 
\end{thm}
We start with the case $d \ge 4$ which is simpler.
\begin{proof}[Proof of \eqref{eq:d=4-sup}]
	It follows from the proof of Theorem~\ref{T:ri-main} (revisit \eqref{eq:V_prob_lower_bnd_dge4}) that
	\begin{equation}\label{eq:V0Ndge4}
		\P\big[ ( {\rm V}_{0, N}(\overline{Z}_{0, N}^u(\nu = 0)) )^c \big] \le C(u) e^{-c(u) N}
	\end{equation}
	for all $u\in (0,u_*)$ and $N \ge 1$. Now it follows from \eqref{eq:incl_tr_slu} and 
	Lemma~\ref{lem:VzSLU_inclusion} applied with $L=N/\sqrt{d}$  (using the inclusion 
	$B_{N/\sqrt{d}} \subset B_N^2$)  that for all $N \ge 10\sqrt{d}$,
	\begin{equation}\label{eq:includeVzdge4}
		\big\{0 \stackrel{}{\longleftrightarrow}\partial B_N^2, 
		\nlr{}{}{0}{\infty} \text{ in } \mathcal V^u\big\} \subset \textstyle  \bigcup_{z \in B_{2N}^2} \big({\rm 
			V}_{z, {{N}}/{100\sqrt{d}}} \, \cap \, \big\{B_{{N}/{4 \sqrt{d}}} 
		\stackrel{}{\longleftrightarrow}\infty \text{ in } \mathcal V^u\} \big)^c 
	\end{equation}
	where ${\rm V}_{z, L} = {\rm V}_{z, L}(\overline{Z}_{z, L}^u(\nu = 0))$. 
	Following the steps leading to the bound (5.73) in \cite{gosrodsev2021radius}, we deduce 
	from \eqref{eq:V0Ndge4} and \cite[Theorem~7.3]{MR3602841} that $\P[\nlr{}{}{B_{{N}/{4 \sqrt{d}}}}{\infty} \text{ in } \mathcal V^u]$ decays super-exponentially in $N$. Together with \eqref{eq:V0Ndge4}, 
	this implies \eqref{eq:d=4-sup} via \eqref{eq:includeVzdge4} and a union bound. The inclusion 
	\eqref{eq:includeVzdge4} continues to hold with the event ${\rm LocUniq}(N, u)^c$ (and 
	hence also $2$-${\rm arms}(N, u)$, see 
	below \eqref{def:locuniq}) on the left-hand side, as follows from  \eqref{def:locuniq}, hence  the same bound for ${\rm LocUniq}(N, u)^c$ and $2$-${\rm arms}(N, u)$.
\end{proof}

The case $d = 3$ of Theorem~\ref{thm:d=3}, i.e.~\eqref{eq:d=3-sup}, requires several 
rounds of bootstrapping owing to the refined nature of the bounds involved. Content of the 
first round is summarized in our next lemma.
\begin{lemma}[Bootstrapping $\mathcal G_{z}^{{\rm I}}$; $d = 3$]\label{lem:script_G}
	Suppose that (cf.~\eqref{eq:mathcalGtrigger})
	\begin{equation}\label{eq:mathcalGtrigger2}
		\begin{split}
			&\P\big[\mathcal{G}_{0, L}^{\rm I}\big(\overline{Z}_{\mathbb{L}(L)}, \delta, \bm u; a^{(1)}\big) 
			\big] \ge 1 - \theta' e^{- L^{\theta}},
			\quad L \ge 1 
		\end{split}
	\end{equation}
	for some 
	$\theta \in (0, 1)$, $\theta' <\infty$, 
	$a^{(1)} \geq 1$, $L_0$ 
	and 
	$\bm u = (u_0, u_1, u_2, u_3)$ satisfying 
	\eqref{eq:params_RI} and all $\delta \in 
	[0, \delta']$ for some $\delta'  \in (0, \frac12)$. There exist 
	$\delta'' = \delta''(\bm u, \delta') \in (0, \tfrac12)$ 
	such that, with $\bm u(k, \varepsilon)$ as below \eqref{eq:V_to_V_L_supcrit_rnd1},
	\begin{equation}
		\label{eq:mscrG_z}
		\begin{split}
			&\P\big[\mathcal{G}_{0, N}^{\rm I}\big(\overline{Z}_{\mathbb{L}(N)}, \delta, \bm u(4, \varepsilon); a^{(2)}( N) =
			{c'N}{(\log N)^{-C(\theta)} } \cdot a^{(1)}\big) 
			\big] \ge 1 - 
			C' e^{- {N}{(1 \vee \log  N)^{-C(\theta)}}},
		\end{split}
	\end{equation}
	for all $N \ge 1$, 
	$\delta \in [0, \delta'']$, $\varepsilon \in (0, 1)$ satisfying the last two of the three conditions 
	in \eqref{eq:params_RIep} 
	and some $C' = C'(\bm u, 
	L_0, \delta', \theta, \theta', \varepsilon) < \infty$ and $c' = c'(\bm u, 
	\theta, \theta', \varepsilon) > 0$.
\end{lemma}
The proof of Lemma~\ref{lem:script_G} is postponed for a few lines. From now on until the 
end of this section we assume that $d=3$. We start by explaining how 
Lemma~\ref{lem:script_G} leads to a bound for the event $({\rm V}_{0,N}^{\rm I}(\nu))^c$ 
similar to \eqref{eq:SLU-reduc1} but with a 
{\em larger} value of $\nu$ and a better error. The need for a larger value of 
$\nu$ arises from the change in the form of 
${\rm V}_z= {\rm V}_{z,L}(\overline{Z}_{z,L}^u(\nu))$  (see above \eqref{eq:V-first-inclusion})
across any 
inclusion of the type \eqref{eq:RI_basic_coupling2} (see \eqref{eq:supcrit_inclusion_Vztilde} 
below) which is essential for 
further improving the error bound in view of \eqref{eq:bootstrap_prob_coupling}. 

To the effect of improving over {\eqref{eq:SLU-reduc1}}, for any given $u \in (0, u_\ast)$, let 
\begin{equation}\label{def:ep}
	\varepsilon = ((1 - (\tfrac{u}{u_\ast})^{\frac1{30}}) 
	\wedge \tfrac1{20} \text{ and } \delta' = \Cr{c:delta}(L_0(\bm u))
\end{equation}
where $\bm u = (u_0, u_1, u_2, u_3)$ is given by
\begin{equation}\label{eq:uchoice1}
	u_0 = \tfrac{\Cr{c:u_4}}{10}, u_1 = u_\ast(1 - \varepsilon)^{21}, u_2 = u_\ast(1 - \varepsilon)^{20}  \mbox{ and } u_3 = u_\ast(1 - \varepsilon) 
\end{equation}
(cf.~\eqref{eq:uchoice}). In view of Lemma~\ref{lem:mathcalG}, we see that the conditions of 
Lemma~\ref{lem:script_G} are satisfied with $\bm u$, $\varepsilon$ and $\delta'$ as above, 
$a^{(1)} = 1$ and $\theta = c$, $\theta' = C(u)$ and $L_0 = L_0(\bm u)$ 
from Lemma~\ref{lem:mathcalG}. Thus \eqref{eq:mscrG_z} holds with $\varepsilon$ as in 
\eqref{def:ep}  and $c',C$ depending effectively only on $u$ with the above choices.

Now we notice from \eqref{def:ep} and \eqref{eq:uchoice1} that the conditions 
\eqref{eq:params_RIep0} and \eqref{eq:params_RIep} are satisfied by $\bm u(4, \varepsilon)$ instead of $\bm u$
as well and consequently we can follow the steps leading to 
\eqref{eq:V_prob_lower_bnd_dge4} in the proof of Theorem~\ref{T:ri-main} starting from  
\eqref{eq:mscrG_z} in place of \eqref{eq:mathcalGtrigger}, which feeds into \eqref{eq:V_to_V_L_supcrit_rnd1} and the subsequent estimates. 
In particular, when reaching the point in the argument leading to 
\eqref{eq:V_prob_lower_bnd_dge4} at which Theorem~\ref{prop:exploration_RI-I} is applied, we can now afford to choose $a=a^{(2)}(L)$ owing 
to \eqref{eq:mscrG_z} when applying \eqref{eq:truncated_RI-I}. 
Moreover, we are free to choose any value of $\nu$ for which these bounds remain 
meaningful; that is, with $K(u)= K(\varepsilon,\bm  u)$ as above 
\eqref{eq:supcrit_GN_lower_bnd0d=3} for the choices of $\varepsilon,\bm  u$ from 
\eqref{def:ep}-\eqref{eq:uchoice1} and $L= L(N)$ as above 
\eqref{eq:supcrit_GN_lower_bnd0d=3}, we pick
\begin{equation*}\nu \coloneqq\big( c(u)a^{(2)}(L) \, 
	\tfrac{N}{h(K(u) L)} \big) \big\vert_{L=L(N)} 
	\stackrel{(d=3)}{=} c(u) \, a^{(2)}(L(N)) \,\tfrac N{K(u)L(N)} 
	\stackrel{\eqref{eq:mscrG_z}}{\geq} c(u) \tfrac N{( 1\vee \log \log N)^C}
\end{equation*}
when applying Theorems~\ref{prop:exploration_RI-I} (recall from \S\ref{subsec:admissible} 
that $h(x) = x$ when $d = 3$). The ${\rm i} = {\rm II}$ case follows similarly from the 
corresponding version of Lemma~\ref{lem:script_G} and Theorem~\ref{prop:exploration_RI-II} 
in lieu of Theorem~\ref{prop:exploration_RI-I}. All in all, we thus 
obtain, similarly as \eqref{eq:SLU-reduc1}, that for all $u \in (0, u_\ast)$, ${\rm i} = {\rm 
	I}, {\rm II}$ and $N \geq 2$,
\begin{equation}\label{eq:mathrmVtrigger0}
	\P[({\rm V}_{0,N}^{\rm i}(\nu;  \bar{u}_0 , u))^c] \le C(u) e^{- {N}/{( \log N)^{C'}}},
\end{equation}
for some absolute constant $C' \in (0, \infty)$, with $\nu$ as above and $\bar u_{0} = u_0(1 + \varepsilon)^3$ with $\varepsilon$ and $u_0$ defined in \eqref{def:ep}-\eqref{eq:uchoice1}.

The bound \eqref{eq:mathrmVtrigger0} brings us to 
the final round of bootstrapping where we derive the optimal upper bound on the probability 
of the $2$-${\rm arms}$ event. 

\begin{lemma}[Bootstrapping ${\rm V}_z$ to $2$-${\rm arms}$; $d = 3$]\label{lem:Vto2arms}
	Suppose that for all $u \in (0, u_\ast)$, we have 
	\begin{equation}\label{eq:mathrmVtrigger}
		\sup_{L \ge C(u)}L^{-\frac34}\log \P\big[({\rm V}_{0, L}(\overline{Z}_{0, L}^u(\nu_L)))^c \big] \le -1
	\end{equation}
	(see \eqref{eq:V-boosted} for the event in question), for some $C(u) < \infty$ and $\nu_L  \ge  L(1 \vee \log 
	L)^{-1/{4}}$. Then for any 
	$\Lambda_N \in \mathcal S_N$ (recall \eqref{eq:scriptS_N}) and $u \in (0, u_\ast)$, we have (see below \eqref{def:locuniq} for notation)
	\begin{equation}\label{eq:2armsd=3}
		\limsup_{N \to \infty} \frac{\log N}{N}\log \P[ \text{$2$-}{\rm arms}(\Lambda_N, u)] \leq -\frac{\pi}{3}(1 
		- \sigma)(\sqrt{u} - \sqrt{u_*})^2.
	\end{equation}
\end{lemma}	 
Assuming Lemma~\ref{lem:Vto2arms} for a moment, we are now ready to conclude the proof of \eqref{eq:d=3-sup}, thereby completing 
the proof of Theorem~\ref{thm:d=3}, contingent on Lemmas~\ref{lem:script_G} and~\ref{lem:Vto2arms} which are proved below.
\begin{proof}[Proof of \eqref{eq:d=3-sup}] 	
	Combining the two
	estimates \eqref{eq:mathrmVtrigger0} for ${\rm i} = {\rm 
		I}, {\rm II}$ with \eqref{eq:V-first-inclusion} and using \eqref{eq:Nuz_tail_bnd} to bound the Poisson deviation appearing in \eqref{eq:V-first-inclusion}, we  readily deduce that \eqref{eq:mathrmVtrigger} is satisfied with $\nu_L$ as defined above \eqref{eq:mathrmVtrigger0} with $L$
	in place of $N$. This choice satisfies $\nu_L  \ge  L(1 \vee \log 
	L)^{-1/{4}}$ for $L \geq C(u)$ hence Lemma~\ref{lem:Vto2arms} is in force and thus \eqref{eq:2armsd=3} 
	holds for all $u \in (0, u_\ast)$ and $\Lambda_N \in \mathcal S_N$. Now observe that,
	\begin{equation}\label{eq:2armsinclude}
		\big\{\lr{}{}{0}{\partial B_N^2}, \nlr{}{}{0}{\infty} \text{ in } \mathcal V^u\big\} \cap \big\{\lr{}{}{B_{\sigma N}^2}{\infty} \text{ in } \mathcal V^u\big\} \subset 
		\text{$2$-}{\rm arms}\left(B_N^2 \setminus B_{\sigma N}^2, u\right)	
	\end{equation}
	for any $\sigma \in (0, 1)$. Mimicking the proof of (5.73) in \cite{gosrodsev2021radius}, we obtain from 
	\eqref{eq:2armsd=3} applied to $\Lambda_N = \tilde D_{0, N} \setminus \tilde C_{0, N}$, the 
	disconnection estimate in \cite[Theorem~7.3]{MR3602841} which holds for all $u < \bar u$ and 
	the equality of $u_\ast$ and $\bar u$ in Theorem~1.2 of \cite{RI-I} that $({\log N})N^{-1}	\log \P\big[ \{ \lr{}{\mathcal V^u}{B_{\sigma N}^2}{\infty} \}^c\big] \to -\infty$ as $N \to \infty$ for all $u \in [0, u_\ast)$ and $d=3$. Jointly with \eqref{eq:2armsinclude} this implies via a union bound that the left-hand side of  \eqref{eq:d=3-sup} is bounded by $-\frac{\pi}{3}(1 - \sigma)(\sqrt{u} - \sqrt{u_*})^2$for {\em any} $\sigma \in (0, 1)$ and  $u \in (0, u_\ast)$, and  \eqref{eq:d=3-sup} follows upon letting $\sigma \downarrow 0$. 
	
	The corresponding bound for the event $\text{$2$-}{\rm arms}(N, u)$ follows directly from 
	\eqref{eq:2armsd=3}. As to the event ${\rm LocUniq}(N, u)^c$, recall from \eqref{def:locuniq} 
	that
	${\rm LocUniq}(N, u)^c$ is contained in the union of $\text{$2$-}{\rm arms}(N, u)$ and $\{\nlr{}{\mathcal V^u}{B_{N}^2}{B_{2N}^2}\}$. The probability $\P[\text{$2$-}{\rm arms}(N, u)]$ yields the desired contribution to the upper bound and similarly as before, one obtains by combining the results from \cite{MR3602841} and \cite{RI-I} that for all $u \in (0,u_*)$, the above disconnection probability decays exponentially in $N^{d-2}=N$ as $N \to \infty$ when $d=3$.
\end{proof}

We now give the pending proofs of Lemma~\ref{lem:script_G} and \ref{lem:Vto2arms} 
starting with the: 
\begin{proof}[Proof of Lemma~\ref{lem:script_G}]
	Let $\mathbb{L} = L\Z^d$. We will introduce slightly modified versions of the event $\mathscr{G}_{z,N}^{\rm I}$. To this end we let, for any $\varepsilon \in (0, 1)$, $\delta \in [0, \frac12)$, and $\bm u(k, \varepsilon)$ as below 
	\eqref{eq:V_to_V_L_supcrit_rnd1},
	\begin{equation}\label{def:mathscrG_RI-genbar}
		{\overline{\mathscr{G}}}_{z, N}^{\rm I} (Z_{\mathbb L}, \delta,
		\bm u, \varepsilon; a^{(1)}) 
		\coloneqq\mathscr G\big(\tilde D_{z, N} \setminus \tilde C_{z, N},  
		\mathcal{G}^{\rm I}, \mathcal F_L^{{\bm u(1, \varepsilon)}', {\bm u}'} ; \rho = 
		\tfrac1{2\Cr{C:rho}}\big)
	\end{equation}
	(cf.~\eqref{def:mathscrG_RI-gen} and \eqref{def:mathscrG_RI-gen2}), where $\mathcal{G}^{\rm I}= \{\mathcal{G}_{z'}^{\rm I}(Z_{\mathbb L},  \delta, 
		\bm u; a^{(1)}): z' \in \mathbb L\}$ and $\bm u'$ is 
	defined as above \eqref{eq:supcrit_inclusion0} for any $\bm u$. 
	%
	%
	%
	%
	%
	%
	%
	%
	Now mimicking the derivation of \eqref{eq:supcrit_GN_lower_bnd0d=3} in the proof of 
	Lemma~\ref{lem:mathcalG} with \eqref{eq:mathcalGtrigger2} instead of 
	\eqref{eq:mathcalGtrigger} as the corresponding seed estimate, we obtain from an 
	application of \eqref{eq:bootstrapped_limit_prob1d=3} at the final stage that with any $\varepsilon > 0$ small enough depending solely on $\bm u$ and $K = 
	K(\bm u, 
	\theta, \theta', \varepsilon)$ and $L(N) = \lfloor (\log N)^{C(\theta)} \rfloor$,
	\begin{equation}\label{eq:supcrit_GN_lower_bnd1d=3}
		\P[\overline{\mathscr G}_{0, N}^{\rm I}(Z_{\mathbb L}, 
		\delta, \bm u(3, \varepsilon), \varepsilon; a^{(1)})]	\ge 1 - C'\exp\{- N{(1 \vee \log N)^{-C(\theta)}}\} 
	\end{equation}
	for any $\delta \in [0, \delta']$ and $N \ge 1$, and some
	$C'$ with a dependence on parameters as specified below \eqref{eq:mscrG_z}. 
	
	We also need to 
	introduce 
	versions of 
	the events $\mathcal G_z^{{\rm I}}$ 
	from \eqref{def:good_events_supcrit} 
	with excursions at scale $N$ instead of $L$. These will carry a superscript ``$0$.''
	Thus, for $z \in \mathbb L$, recalling Definition~\ref{def:good_events}, let
	\begin{equation}\label{def:good_events_supcritZy}
		\mathcal{G}_{z, L}^{\rm I, 0}(\widehat{Z}_{0, N}, \delta, u_0, u_1, u_2, u_3; a) \coloneqq\mathcal G_z({\rm V}^0, {\rm W}^{{\rm I}, 0}, \mathscr C^0 ; a),
	\end{equation}
	where ${\rm V}_z^0 = {\rm 
		V}_{z}(\widehat{Z}_{0, N}^{u_1}, 
	\widehat{Z}_{0, N}^{u_2}, \widehat{Z}_{0, N}^{u_3}, \delta)$ (see above \eqref{def:Fz}), ${\rm W}^{\rm I, 0}_{z, y} = {\rm 
		FE}_y\big( (\widehat{Z}_{0, N}^{u_1})_+ ( \hat{\nu}_{0, N}( \tfrac{u_0}{8}))\big)$ 
	(cf.~\eqref{eq:WI}), and $\mathscr C_z = \mathscr C_z(\widehat{Z}_{0, N}^{u_1}, 
	\widehat{Z}_{0, N}^{u_3}, \delta)$ (see around \eqref{eq:script_C}). 
	We now claim that 
	\begin{equation}\label{eq:supcrit_inclusion}
		\mathcal G_{z, L}^{\rm I}({Z}_{\mathbb L}, \delta, \bm u(3, \varepsilon); a^{(1)}) \cap \mathcal 
		F_{z, L}^{{\bm u(4, \varepsilon)}', {\bm u(3, \varepsilon)}'} \subset \mathcal G_{z, 
			L}^{\rm I, 0}
		(%
		\overline{Z}_{0, N}, \delta, \bm u(4, \varepsilon); a^{(1)}), 
	\end{equation}
	for 
	any $z \in \mathbb L$ such that $ D_{z, L} \subset  D_{0, N}$ and $ U_{z, L} 
	\subset  U_{0, N}$; for later orientation, the event $\mathcal G_{z, L}^{\rm I}$ with 
	arguments as on the left-hand side of \eqref{eq:supcrit_inclusion} belongs precisely to the 
	family used to declare the event $\overline{\mathscr G}_{0, N}^{\rm I}$ appearing in 
	\eqref{eq:supcrit_GN_lower_bnd1d=3} in view of \eqref{def:mathscrG_RI-genbar}.
	
	The inclusion \eqref{eq:supcrit_inclusion} 
	follows from similar arguments as those leading to 
	\eqref{eq:supcrit_inclusion0} in the proof of Lemma~\ref{L:loca-supercrit}, except that 
	some caution is needed as the event on the right-hand side now involves excursions between $ D_{0, N}$ and 
	$ U_{0, N}$ instead of $ D_{z, L}$ and $ U_{z, L}$. We now highlight these changes. Since $ D_{z, L} 
	\subset  D_{0, N}$ and $ U_{z, L} \subset  U_{0, N}$, we get from 
	\eqref{eq:nested_excursion} in 
	\S\ref{subsec:excursion} that {on the event} $\mathcal F_{z, L}^{\bm u(4, \varepsilon)', \bm u(3, \varepsilon)'}$, 
	$\mathcal I\big(\overline{Z}_{0, N}^{u_k(1 - \varepsilon)^4}\big) \cap  D_{z, L}$ is contained in
			 $\mathcal I\big({Z}_{z, L}^{u_k(1 - \varepsilon)^3}\big) \cap  D_{z, L}$ for $k = 1, 3$, and
			$\mathcal I\big({Z}_{z, L}^{u_2(1 + \varepsilon)^3}\big) \subset \mathcal 
			I\big(\overline{Z}_{0, N}^{u_2(1 + \varepsilon)^4}\big)$. Moreover, for any $Z \in ({Z}_{0, N}^{u_1(1 - \varepsilon)^4})_+\big(\bar{\nu}_{0, N}( \tfrac{u_0(1 + \varepsilon)^4}{8})\big)$, there exists $Z' \in 
	({Z}_{z, L}^{u_1(1 - \varepsilon)^3})_+({\nu}_{z, L}\big( \tfrac{u_0(1 + \varepsilon)^3}{8})\big)$ 
	satisfying $\mathcal I(Z) \cap  D_{z, L} = \mathcal I(Z') \cap  D_{z, L}$ and 
	$\ell_x(Z) = \ell_x(Z')$ for all $x \in  D_{z, L}$. But these are 
	enough to deduce \eqref{eq:supcrit_inclusion} following the arguments in the proof of 
	\eqref{eq:supcrit_inclusion0} owing to 
	the defintions of $\mathscr C_{z, L}$ and the events ${\rm V}_{z, L}$ and ${\rm 
		W}^{{\rm I}}_{z, y}$ (revisit \eqref{eq:script_C}, \eqref{def:Fz} and \eqref{eq:WI}). 

	Now in view of \eqref{eq:supcrit_inclusion}, whereby condition \eqref{eq:bootstrap_inclusion} 
	of Proposition~\ref{prop:bootstrap_events} is satisfied for the 
	pair of events $(\mathcal G_{z, L}^{\rm I}({Z}_{\mathbb L}, \delta, \bm u(3, \varepsilon); 
	a^{(1)}), \mathcal G_{z, L}^{\rm I, 0}(\overline{Z}_{0, N}, \delta, \bm u(4, 
	\varepsilon); a^{(1)}))$ and $\mathcal F_{z, L} = \mathcal F_{z, L}^{{\bm u(4, \varepsilon)}', 
		{\bm u(3, \varepsilon)}'}$, we obtain by application of \eqref{eq:grand_inclusion2} 
	that there exista a (random) non-empty set $\mathcal O^{{\rm I}}$ 
	satisfying \eqref{eq:O-properties} such that
	\begin{equation}\label{eq:grand_inclusionA}
		\overline{\mathscr G}_{0, N}^{\rm I}(Z_{\mathbb L}, \delta, \bm u(3, \varepsilon), \varepsilon; a^{(1)}) 
		\subset \mathcal G_{0, N}({\rm V}^{2, {\rm I}}, {\rm W}^{2, {\rm I}}, \mathscr{C}^{2, {\rm I}}; a^{(2)}), 
	\end{equation}
	(recall \eqref{def:mathscrG_RI-genbar} and that $\bm u(4, \varepsilon) = (\bm u(3, 
	\varepsilon))(1, \varepsilon)$ in the notation from below \eqref{eq:V_to_V_L_supcrit_rnd1}), where $a^{(2)} 
	{=}  \lfloor\tfrac{c N}
	{KL(N)}\rfloor \cdot a^{(1)}$ with $K$ as above \eqref{eq:supcrit_GN_lower_bnd1d=3} (recall 
	from 
	\S\ref{subsec:admissible} that $h(x) = x$ when $d = 3$), and the triplets $({\rm V}^{2, {\rm I}}, {\rm W}^{2, {\rm I}}, \mathscr{C}^{2, {\rm I}})$ are specified as follows:
	$$\textstyle{{\rm V}}^{2, {\rm I}}_{0} {=} \bigcap_{z \in \mathcal O^{\rm I}} {\rm V}_{z, 
		L}\big(\overline{Z}_{0, N}^{u_1(1 - \varepsilon)^4}, \overline{Z}_{0, N}^{u_2(1 + 
		\varepsilon)^4}, \overline{Z}_{0, N}^{u_3(1 - \varepsilon)^4}, \delta\big) \stackrel{ \eqref{def:Fz}}{=} \textstyle\bigcap_{z \in 
		\mathcal O^{\rm I}} {\rm V}_{z, L}(\overline{Z}_{\mathbb L(N)}, \delta, \bm u(4, \varepsilon)),$$
	${{\rm W}}^{2, {\rm I}}_{0,y}= {\rm W}^{{\rm I}, 
		0}_{0,y}$ 
	for all $y \in \mathbb{L}_0$ (see below \eqref{def:good_events_supcritZy}) 
	and ${\mathscr C}^{2, {\rm I}}_{0} 
	= \bigcup_{z \in \mathcal O} {\mathscr C}_{z, L}(\overline{Z}_{0, N}^{u_1(1 - 
		\varepsilon)^4}, \overline{Z}_{0, N}^{u_3(1 - \varepsilon)^4}, \delta)$. 
	In particular, these choices entail that \eqref{eq:grand_inclusion2} indeed applies in 
	deducing \eqref{eq:grand_inclusionA}.

	Abbreviating $\text{Conn}= \big\{\lr{}{}{C_{0, N}}{\partial D_{0, N}} \text{ in } \mathsf 
	N_{2\delta}({\mathcal V}^{u_3(1-\varepsilon)^4})\big\} 
	$ we have
	\begin{equation}\label{eq:mathscrGoNIinclude}
		\begin{split}	
			& \, 
			\mathcal G_{0, N}\big({\rm V}^{2, {\rm I}}, {\rm W}^{2, {\rm I}}, \mathscr{C}^{2, {\rm I}}; 
			a^{(2)} \big) \cap \text{Conn} \subset  \, \mathcal{G}_{0, N}^{\rm I}\big(\overline{Z}_{\mathbb{L}(N)}, \delta, \bm u(4, 
			\varepsilon); a^{(2)} \big),
		\end{split}
	\end{equation}
	which also follows readily from the definition of $\mathcal G_{0, N}^{\rm 
		I}(\overline{Z}_{\mathbb L(N)}, \delta, \bm u; a^{(2)})$ in \eqref{def:good_events_supcrit} 
	{provided} one has ${\rm V}_{0}^{2, {\rm I}} \subset 
	{\rm V}_{0, N}(\overline{Z}_{\mathbb L(N)}, \delta, \bm u(4, \varepsilon))$ and $\mathscr 
	C_{0}^{2, {\rm I}} \subset \mathscr C_{0, N}(\overline{Z}_{\mathbb L(N)}, \delta, \bm u(4, 
	\varepsilon))$ {\em on} the event $\text{Conn}$ (
	recall that the event $\mathcal G_z({\rm V}, {\rm W}, \mathscr C; a)$ is increasing in ${\rm 
		V}_z$ and $\mathscr C_z$). 
	Both of these inclusions follow from standard gluing arguments inherent in the definition of 
	the events ${\rm V}_{z, L}$ and already used in the proof of 
	Lemma~\ref{lem:VzSLU_inclusion}. 
	For a precise verification, we refer the reader to 
	the arguments used 
	in 
	\S\ref{subsec:exploration} 
	to derive \eqref{eq:onearm_twoarm} in the 
	course of proving Theorem~\ref{prop:exploration_RI-I}. 
	Finally, \eqref{eq:mathscrGoNIinclude}, 
	\eqref{eq:supcrit_GN_lower_bnd1d=3} and the 
	upper bound on the 
	disconnection probability from 
	\eqref{eq:connection_delta_d=3} 
		for $\delta \le \tfrac{\Cr{c:delta}(u_3)}2$ (recall \eqref{def:noised_set_inclusion}) together imply \eqref{eq:mscrG_z} 
		via a simple union bound. 
	\end{proof}
	Next we give the:
	\begin{proof}[Proof of Lemma~\ref{lem:Vto2arms}]
		Let us start with an inclusion of events. For any $\Lambda_N \in \mathcal S_N$ as in \eqref{eq:scriptS_N}, all $\rho \in (0, 1)$, 
		$0 < u < v < u_\ast$ and $\nu \ge 0$, we have
		\begin{equation}\label{eq:include_Gtwoarms}
			\mathscr G(\Lambda_N, {\rm V}_L, \mathcal F_L^{u, v}; \rho) \subset (\text{$2$-}{\rm arms}(\Lambda_N, 
			u))^c
		\end{equation}
		where $ {\rm V}_L= \{ {\rm V}_{z}: z \in 
		\mathbb{L}\}$ and $  {\rm V}_z = {\rm V}_z(Z_z^v(\nu))$ for $z \in \mathbb L$. To see this, first note that the sequence $\overline{Z}_z^u$ lies in the family $Z_z^v(\nu)$ 
		on the event $\mathcal F_z^{u, v}$ (revisit \eqref{def:Z_tr_RI} and \eqref{eq:F} for relevant 
		definitions) and therefore by \eqref{eq:V-boosted}, ${\rm V}_z \cap \mathcal F_z^{u, v} \subset {\rm V}_z(\overline{Z}, u)$. Thus condition \eqref{eq:bootstrap_inclusion} of Proposition~\ref{prop:bootstrap_events} is satisfied for the pair of events $({\rm V}_z, {\rm 
			V}_z(\overline{Z}, u))$ and $\mathcal F_{z, L} = \mathcal F_z^{u, v}$, and hence by \eqref{eq:grand_inclusion} there exists a 
		(random) non-empty, $\ast$-connected set $\mathcal O'$ satisfying \eqref{eq:O-properties} such that 
			$\mathscr G(\Lambda_N, {\rm V}_L, \mathcal F_L^{u, v}; \rho) \subset \bigcap_{z \in \mathcal O'} {\rm V}_z(\overline{Z}, u)$.
		From this and the definition of the event ${\rm V}_z(\overline Z, u)$ given below \eqref{def:Fz}
		it follows by elementary gluing considerations (see also below \eqref{eq:z'z''} in 
		\S\ref{subsec:reduction}) that on the event $\mathscr G(\Lambda_N, {\rm V}, \mathcal 
		F_L^{u, v}; \rho)$,
		\begin{equation}\label{eq:connected_interface1}
			\begin{split}
				&\text{there exists a component $\mathscr C_{\mathcal O'}$ of $\Lambda_N  \cap 
					\mathcal V^u$ which contains}\\ 
				&\text{all crossing clusters of 
					$\tilde D_z \setminus \tilde C_z$ in $D_z \cap \mathcal V^{u}$ for each $z \in \mathcal O'$.}
			\end{split}
		\end{equation}
		Moreover, writing $\Lambda_N =  V_N \setminus U_N $, since $\{0\} \cup U_N \cap \mathbb L \preceq \mathcal O' \preceq \partial_{\mathbb 
			L}(V_N \cap \mathbb L)$ by \eqref{eq:O-properties} (see \eqref{eq:prec} for definition), any 
		crossing of $\Lambda_N$ in $\mathcal V^u$ must lie in the {\em same} component of $\Lambda_N \cap 
		\mathcal V^u$ as $\mathscr C_{\mathcal O'}$ on the event \eqref{eq:connected_interface1}, 
		thus yielding \eqref{eq:include_Gtwoarms} (the definition of the $2$-$\text{arms}$ event 
		appears below \eqref{def:locuniq}).

		In view of \eqref{eq:include_Gtwoarms}, it suffices to obtain \eqref{eq:2armsd=3} with 
		$\text{$2$-}{\rm arms}(\Lambda_N, u)$ replaced by the event $(\mathscr G(\Lambda_N, {\rm V}_L, \mathcal 
		F_L^{u, v} ; \rho))^c$ for `suitable' values of the parameters $v$, $\rho$ and $K$ (recall 
		\eqref{def:script_G}). Obviously, we will use Proposition~\ref{prop:bootstrap_prob} to this end. 
		We start just like in the proof of Theorem~\ref{T:ri-main} with the events ${\rm 
			V}_z(\widehat{Z}_z^v(\nu))$ in place of $\mathcal G_z^{{\rm i}}(\widehat Z_{\mathbb L}, \delta, \bm u; a)$. By similar arguments as those yielding 
		\eqref{eq:supcrit_inclusion0}, we have that
		\begin{equation}\label{eq:supcrit_inclusion1}
			{\rm V}_z(\overline{Z}_z^v(\nu)) \cap \mathcal F_z^{v, v(1 - \varepsilon)} \subset {\rm 
				V}_z({Z}_z^{v(1 - \varepsilon)}(\nu))
		\end{equation}
		where $\varepsilon \in (0, 1)$. In place of \eqref{eq:supcrit_inclusion_Ztilde}, on the other 
		hand, we have that under any coupling $\mathbb Q$ of $\mathbb P$ and 
		$\widetilde {\mathbb P}_{z}$, 
		\begin{equation}\label{eq:supcrit_inclusion_Vztilde}
				{\rm V}_z(\widetilde{Z}_z^{v}(\nu)) \cap {\rm Incl}_{z}^{\frac{\varepsilon}{10}, 
					\frac1{20}(\nu \wedge v\,{\rm cap}(D_z))} \subset {\rm V}_z({Z}_z^{v(1 - \varepsilon)}(\nu - 2u_\ast\,{\rm cap}(D_z)\,\varepsilon))
				\text{,}\\
		\end{equation} 
		and the same with $Z$ and $\widetilde{Z}$ exchanged
		whenever $\nu \ge 4u_\ast\,{\rm cap}(D_z) \, \varepsilon$, for $\varepsilon \in (0, \frac12)$ 
		and $L \ge \tfrac{C}{v\varepsilon}$. In view of \eqref{def:boosted}, 
		\eqref{eq:supcrit_inclusion_Vztilde} follows readily from the inclusions 
		$\{{Z}_z^{v(1 - \varepsilon)}(\nu - 2u_\ast\,{\rm cap}(D_z)\,\varepsilon)\} \subset 
			\{\widetilde{Z}_z^{v}(\nu)\}$ and $\{\widetilde{Z}_z^{v(1 - \varepsilon)}(\nu - 2u_\ast\,{\rm 
				cap}(D_z)\,\varepsilon)\} \subset \{{Z}_z^{v}(\nu)\}$,
		which hold on the event ${\rm 
			Incl}_{z}^{\frac{\varepsilon}{10}, \frac1{20}(\nu \wedge v\,{\rm cap}(D_z))}$ owing to the 
		definition of the latter in \eqref{eq:RI_basic_coupling2} and the family $Z(\nu)$ in 
		\eqref{def:Z_tr_RI}. 

		Equipped with \eqref{eq:supcrit_inclusion1} and \eqref{eq:supcrit_inclusion_Vztilde}, we can 
		now follow in the steps of the proof of Theorem~\ref{T:ri-main} starting from 
		\eqref{eq:mathrmVtrigger} instead of \eqref{eq:mathcalGtrigger} as the corresponding seed 
		estimate and with $\nu =  \nu_L$. In particular, we obtain that the conditions 
		\eqref{eq:bootstrap_prob_coupling}--\eqref{eq:initial_limit_prob} of 
		Proposition~\ref{prop:bootstrap_prob} are satisfied by the events 
		\begin{equation}\label{eq:supcrit_GtildeG2arms}
			\begin{split}
				&\widetilde {\mathcal G}_{z, L} =  {\rm V}_z({Z}_z^{v(1 - \varepsilon)^2}(\nu_L - 2u_\ast \, {\rm cap}(D_z) \, 
				\varepsilon_L)), \quad
				{\mathcal G}_{z, L} = {\rm V}_z({Z}_z^{v(1 - \varepsilon)^3}(\nu_L - 4u_\ast \, {\rm cap}(D_z) \, 
				\varepsilon_L))
			\end{split}
		\end{equation}
		for $\varepsilon_L = c  (1 \vee \log 
		L)^{-1/{4}}$, $m_L = cL (1 \vee \log 
		L)^{-1/{4}}$, whence 
			$\widetilde \P_z[(\mathcal U_z^{\varepsilon_L, m_L})^c] $ is bounded using \eqref{eq:bnd_Uzepm} by $C \varepsilon_L^{-2} e^{-c m_L\varepsilon_L^2}  \le C e^{-\frac L{1 \vee \log L}}$
		(cf.~\eqref{eq:V_L_to_tildeV_L_supcrit_rnd1} and 
		\eqref{eq:bootstrap_prob_coupling}--\eqref{eq:initial_limit_prob}); $\beta' = \frac34 - \frac18$ ($ > \frac12$), $K_0 = \frac C{\varepsilon_L} = C (1 \vee \log 
		L)^{\frac14}$ and any $\varepsilon \in (0, 1)$; $L_0$ (from \eqref{eq:initial_limit_prob}) and $L$ 
		sufficiently large depending only on $v$ and $\varepsilon$. 
		The estimate
		\eqref{eq:bootstrapped_limit_prob1d=3} now yields with the choice $(\bm u, \bm v) = (u, v(1 - 
		\varepsilon)^3)$, 
		$\widetilde{\mathcal G}_{z, L}$ and $\mathcal G_{z, L}$ as in 
		\eqref{eq:supcrit_GtildeG2arms}, $\delta \in (0, 1)$, $\rho \in (0, 1)$, $L(N) = \lfloor (\log 
		N)^\alpha \rfloor$ for some $\alpha \in (0, \infty)$ and $K = \sqrt{ 1 \vee \log \log N}$ that for all $v \in (0, u_\ast)$,
		\begin{equation*}\label{eq:two_arms_lower_bnd0d=3}
			\begin{split}
				\limsup_{N \to \infty} \frac{\log N}{N}&\log \P[\mathscr G^c(\Lambda_N, {\rm V}_{L(N)}, \mathcal F_{L(N)}^{u, v(1- 
					\varepsilon)^3}; \rho)]	\\
				&\le -(1 - \delta)(1 - \sigma)(1 - \Cr{C:rho}\rho) \frac{\pi}3(\sqrt{u} - \sqrt{v(1 - \varepsilon)^3} )^2.
			\end{split}
		\end{equation*}	
		Sending $\delta$, $\rho$ and $\varepsilon$ to $0$ and subsequently 
		$v$ to $u_\ast$, we obtain \eqref{eq:2armsd=3} in view of \eqref{eq:include_Gtwoarms}.
	\end{proof}
	\begin{proof}[Proof of Theorem~\ref{T:ri-2point}]
		Theorems~\ref{thm:d=4} and \ref{subsubsec:supcrit} immediately imply 
		Theorem~\ref{T:ri-2point} on account of the inclusion $\{\lr{}{}{0}{x} \text{ in }\mathcal V^u\} \subset 
		\big\{\lr{}{}{0}{\partial B_{ |x| }^2}  \text{ in }\mathcal V^u\}$	(see \eqref{eq:tau-RI} and below to recall 
		the definition of $\tau_u^{\rm tr}(x)$). 
	\end{proof}

\noindent \textbf{Acknowledgement.}  SG's research was partially supported by the SERB 
grant SRG/2021/000032, a grant from the Department of Atomic Energy, Government of India, 
under project 12R\&DTFR5.010500 \sloppy and in part by a grant from the Infosys Foundation as a 
member of the Infosys-Chandrasekharan virtual center for Random Geometry. The research of PFR
is supported by the European Research Council (ERC) under the European Union’s Horizon Europe research and innovation program (grant agreement No 101171046). This work was supported in part
by EPSRC grant EP/Z000580/1, and PFR thanks the Newton Institute for its hospitality.
 YS's work was 
supported by EPSRC Centre for Doctoral Training in Mathematics of Random Systems: Analysis, Modelling and Simulation (EP/S023925/1).

\appendix

\stoptoc

\section{Small excursion packets}\label{subsec:smallu}

In this Appendix we present an analogue of Theorem~\ref{prop:exploration_RI-I} for 
$\P[({\rm V}_{z,N}^{\rm II})^c]$, Theorem~\ref{prop:exploration_RI-II} below, as needed 
in view of \eqref{eq:V-first-inclusion}. It is proved by  adapting {\em parts} of the proof of 
Theorem~\ref{prop:exploration_RI-I}. 
The proof is comparatively simpler in this case 
owing to the stronger conditioning permitted by Lemma~\ref{lem:conditional_prob2} which is 
related to the fact that the sequences of excursions within the purview of this lemma are meant 
to be `small'. 

Our starting point is the following simplification of Proposition~\ref{lem:conditional_prob1} in the regime of `small' excursion packets $Z$. Essentially, this result allows the stronger conditionning on the configuration immediately outside a box, cf.~\eqref{def:FEsigma_algebra_small_u} and~\eqref{def:FEsigma_algebra1}. But this comes at the 
(serious) cost of holding on a good event which is only likely when the underlying interlacement 
set is very small.
\subsection{Insertion tolerance at low intensity}\label{appendix:typeII_fe}
Let $\square(0, L)$ denote the set of all 
points in $C_{0, L}$ (see \eqref{def:CDU}) such that at least two of their coordinates lie in the set $\{0, 1, 2, L-3, L-2, 
L-1\}$ and $\square(z, L) = z + \square(0, L)$ for any $z \in \Z^d$. For $y \in 
\mathbb L_0 = L_0 \Z^d$ and a sequence of excursions $Z = (Z_j)_{1 \le j \le n_Z}$, let 
\begin{equation}\label{def:Gx-1}
{\rm W}_{y, L_0}^-(Z)\stackrel{{\rm def.}}{=} \left\{
		\begin{array}{c}
			\mbox{$\square( z, L_{0}) \subset \mathcal V({Z})$ and $\sum_{x \in \partial C_{z, L_{0}}} \ell_x({Z}) \le (L_{0})^{d-1}$} \\
			\mbox{for all $ z \in \mathbb L_{0}$ satisfying $| z - y |_{\infty} \le L_{0}$}
		\end{array}
		\right\}.
\end{equation}
Using this we can define the (random) set
\begin{equation}\label{def:OyZ}
	\mathcal O_0^-({Z})  = \{ y \in \mathbb L_{0} :  {\rm W}_{y, L_0}^-(Z) \mbox{ occurs}\}.
\end{equation}
Clearly the event ${\rm W}_{y, L_0}^-(Z_J)$ is measurable 
relative to the $\sigma$-algebra
\begin{equation}\label{def:FEsigma_algebra_small_u}
	{\mathcal F}^-_{y, L_0}(Z_J) = \sigma\big(\mathcal O_0^-(Z_J), \mathcal I(Z_J)_{|C_{y, L_{0}}^c}\big).
\end{equation}
\begin{lemma}\label{lem:conditional_prob2}
Let $L_0 \geq 100$. There exists $c' = c'(L_{0}) \in (0, 1)$ such that for all $K \geq 100$, $z \in N\Z^d$ for some $N \ge 10^3 L_0$,
	 $y \in L_0 \Z^d$ such that $ D_{y, L_0} \subset  D_{z, N}$, $x \in \partial^{{\rm out}} C_{y, 
	 L_{0}}$ and any $J \subset \N^\ast$ deterministic and finite,  abbreviating $Z_J = 
	 Z_J^{ D_{z, N},  U_{z, N}}$  we have
	\begin{equation}\label{eq:conditional_prob2}
		\P\big[ \lr{}{\mathcal V(Z_J)}{x}{\square(y, L_{0})}  
		\, \big | \,  {\mathcal F}^-_{y, L_0}(Z_J) \big] \ge c' \, 1_{G'},
	\end{equation}
	with the `good' event $G'=\{x \in \mathcal V(Z_J)\} \, \cap \, {\rm W}_{y, L_0}^-({Z}_J).$
\end{lemma}
 To appreciate the utility of \eqref{eq:conditional_prob2}, one should imagine the sets 
$\square(y, L)$ being contained in $\mathcal V(Z_J)$ for many neighboring points $y \in 
\mathbb{L}_0$, thus forming an ambient cluster, and \eqref{eq:conditional_prob2} gives a 
conditional probability on a point $x$ at the doorstep of the box $C_{y,L_0}$ to connect locally 
to this ambient cluster. 
\begin{proof}
	\eqref{eq:conditional_prob2} follows from a straighforward adaptation of the argument underlying the proof of Lemma~5.10 in \cite{MR3269990}, using a similar 
	computation as in the proof of Lemma~5.13 therein. 
\end{proof}	

\subsection{Type II estimates} \label{appendix:typeII} As with type ${\rm I}$, the 
estimate for $\P[({\rm 
V}_{z,N}^{\rm II})^c]$ will bring into play an event  $\mathscr{G}_z^{\rm II}$ built using different events 
${\rm W}^{\rm II}$ involving ${\rm  W}_{y^-}^-(Z) = {\rm  W}_{y^-, L_0^-}^-(Z)$ from \eqref{def:Gx-1}. 
We introduce a scale $L_0^- \geq 1$ with 
			$L_0 > 100L_{0}^{-} $ (cf.~\eqref{eq:params_RI}). As with 
$L_0$, to keep notations reasonable and because $L_0^-$ does not change throughout our 
bootstrap argument, we keep its dependence implicit. Recall  $\hat{\nu}_z$ from 
\eqref{eq:nu-hat}. 

\begin{itemize}
\item 
\textbf{The events ${\rm W}^{\rm II} = \{{\rm W}_{z, y}^{\rm II} : z \in \mathbb L, y \in \mathbb L_0\}$.} Let
\begin{equation} \label{eq:WII}
{\rm W}_{z, y}^{\rm II} \equiv {\rm W}_{z, y}^{\rm II}(\widehat{Z}, u_4 )  
\coloneqq \mathcal G_{y}^-(\widehat{Z}_z^{u_4}),
\end{equation}
where (see Definition~\ref{def:good_events} for notation) $\mathcal G_{y}^-(Z) \coloneqq\mathcal{G}_{y, 
L_{0}, L_{0}^{-}}({\rm V} = \{\Omega : y \in \mathbb L_0\}, {\rm W}^-(Z), \mathscr 
C = \{\Z^d: y \in \mathbb L_0\}; a = 1)$ and for $y \in \mathbb L_0, y^- \in \mathbb L_{0}^{-} = L_0^-\Z^d$, ${\rm W}_{y,y^-}^-(Z) \equiv {\rm W}_{y^-}^-(Z)$. 
\end{itemize}

Somewhat in the same way as \eqref{def:good_events_supcrit} and 
\eqref{def:mathscrG_RI-gen}, 
this leads to events 
\begin{equation}\label{def:good_events_supcrit2}
\mathcal{G}_z^{\rm II}(\widehat{Z}, u_4; a) = \mathcal{G}_z({\rm V}, {\rm W}^{{\rm II}}, 
\mathscr C; a)
\end{equation}
with ${\rm V}_z = \Omega$, ${\rm W}^{{\rm II}} = {\rm W}^{{\rm II}}(\widehat Z, u_4)$ given by 
\eqref{eq:WII} and $\mathscr C_z = \Z^d$, and 
subsequently
\begin{equation}\label{def:mathscrG_RI-gen2}
	\mathscr{G}_{z,N}^{\rm II} (\widehat{Z}, u_4; a) \coloneqq\mathscr G\big(\tilde D_{z, N} \setminus \tilde C_{z, N},  
	\mathcal{G}^{\rm II}= \big\{\mathcal{G}_{z'}^{\rm II}(\widehat{Z}, u_4; a): z' \in 
	\mathbb L\big\}, 
\mathcal F_L^{\frac{3u_0}{2}, u_4}; \rho = \tfrac1{2\Cr{C:rho}}\big).
\end{equation}
Finally we let
\begin{equation}\label{def:mathscrG_RI2}
\mathscr{G}_{z,N}^{\rm II} \coloneqq 
\mathscr{G}_{z,N}^{\rm II} ({Z}_{\mathbb L}, u_0, u_4; a), \quad z \in N\Z^d
\end{equation}
(cf.~\eqref{def:mathscrG_RI}). 
The analogue of Theorem~\ref{prop:exploration_RI-I} reads as follows.

\begin{theorem}[Coarse-graining for ${\rm V}_{z,N}^{\rm II}$]
\label{prop:exploration_RI-II}
With the choice of parameters as in \eqref{eq:params_RI} 
as well as 
$\nu \ge 0$, $2u_0 <  u_4 < u_\ast$ 
and $L \ge C(u_0)$, there exists
$c = c(L_{0}^{-}) > 0$ such that, abbreviating $\bar{u}_0=3u_0/2$, 
 \begin{multline}\label{eq:truncated_RI-II}
	\P[({\rm V}_{z,N}^{\rm II})^c] 
	\le \P\big[(\mathscr G_{z,N}^{\rm II})^c\big]  	+ 
	\P\big[\nlr{}{}{( C_{z,N} \cap \mathbb L_{0}^{-})}{\partial_{\mathbb 
			L_{0}^{-}} ( D_{z,N} \cap \mathbb L_{0}^{-})} \text{ in }\mathcal O_0^-(\overline Z^{\bar{u}_0}_{z,N})\big] 
		\\ + e^{-c(a
		 	\frac N{h(KL)}\wedge N)  +  C(\nu + \log N)},
\end{multline}
where the set $\mathcal O_0^-(Z)$ was defined in \eqref{def:OyZ} in 
Section~\ref{sec:insertion} and the connectivity in $\mathcal O_0^-(\overline Z_z^v)$ 
is w.r.t. the graph structure on 
$\mathbb L_{0}^{-}$, i.e.~$z_1, z_2 \in \mathbb L_{0}^{-}$ are neighbors if and only if 
$|z_1 - z_2| = L_{0}^{-}$.
\end{theorem}


\begin{proof}[Proof of Theorem~\ref{prop:exploration_RI-II}]
For $A \subset \Z^d$, let $A^-= A \cap \mathbb L_{0}^{-}$. Following the proof of Theorem~\ref{prop:exploration_RI-I}, the following inequality is analogous to \eqref{eq:explore_hard}. For any finite deterministic $J \subset \N^\ast$, we have
	\begin{equation}\label{eq:explore_easy}
			\P\big [({\rm V}_z(Z_J))^c \cap \mathscr G_z^{{\rm II}} \cap \big\{\lr{}{}{C_{z}^{-}}{(\partial_{\mathbb L_{0}^{-}} D_{z}^{-})} \text{ in } \mathcal 
				O_0^{-}(\overline{Z}_z^{\bar{u}_0}), J \subset [1, N_z^{\bar{u}_0}]\big\} \big] 
			\le C N^{d-1}e^{- c(L_{0}^{-}) am}
	\end{equation}
(recall the set $\mathcal O_0^-(Z)$ from \eqref{def:OyZ}). 
Theorem~\ref{prop:exploration_RI-II} follows from this using a similar line of argument as 
used to deduce Theorem~\ref{prop:exploration_RI-I} from \eqref{eq:explore_hard}. We will deduce \eqref{eq:explore_easy} from 
a related statement: namely, 
\begin{equation}\label{eq:explore_easyII}
\P\big [({\rm V}_z(Z_J))^c \cap \big\{\lr{}{}{C_{z}^{-}}{(\partial_{\mathbb L_{0}^{-}} D_{z}^-)} \text{ in } \mathcal 
O_0^{-}(Z_J), k( \mathcal O_0^-(Z_J)) \ge k\big\} \big]  \le 
C N^{d-1}e^{- c(L_{0}^{-}) k},
\end{equation}
for any $k \ge 0$, where the functional $k(\Sigma)$ is defined similarly as below \eqref{def:S} 
in 
\S\ref{subsec:reduction} 
with $L$ and $\mathbb L$ replaced by $L_0^-$ and $\mathbb L_0^-$ respectively. The bound
\eqref{eq:explore_easy} follows from this and the inclusion of events
\begin{equation}\label{eq:inclusion_GII}
\mathscr G_z^{{\rm II}} \subset \big\{k( \mathcal 
O_0^-(Z_z^{\bar{u}_0})) \ge c am\big\}
\end{equation}
for some $c \in (0, \infty)$ in view of monotonicity of the set $\mathcal O_0^-(Z)$ in $Z$ (
see \eqref{def:OyZ} and \eqref{def:Gx-1}) and also of 
the functional $k(\cdot)$ (revisit its definition 
below \eqref{def:S}). We will show \eqref{eq:inclusion_GII} at the end after giving a proof of 
\eqref{eq:explore_easyII}. To this end, let $O'_1, \ldots, O'_\ell \subset \mathcal 
O_0^-(Z_z^{\bar{u}_0})$ denote the $\ast$-connected sets satisfying properties~(a)--(b) in Proposition~\ref{lem:surroundingInterfaces} (as subsets of the 
coarse-grained lattice $\mathbb L_0^-$) with $V = \mathbb L_0^-(\tilde D_z)$, $U = \mathbb 
L_0^-(\tilde C_z)$ and $\Sigma = \mathcal O_0^-(Z_z^{\bar{u}_0}) \cap (V \setminus U)$.  
Note that $\ell = 0$ when $k( \mathcal O_0^-(Z_z^{\bar{u}_0})) = 0$. 
As we now explain, we can reduce \eqref{eq:explore_easyII} to a similar statement (cf.~ 
\eqref{eq:explore_hard_tilde_x} in the proof of Theorem~\ref{prop:exploration_RI-I}), whereby $V_z(Z_J)$ is replaced by the event $\overline {\rm V}_{z}(Z_J) \coloneqq 
\bigcap_{x \in \partial \tilde C_z} \overline {\rm V}_{z, x}(Z_J)$ with (see above \eqref{def:Gx-1} for notation)
\begin{equation}\label{def:overineVx}
\overline {\rm V}_{z, x}({Z}_J) \coloneqq 
\big\{\nlr{}{}{x}{\partial \tilde D_z} \text{ in } \mathcal V({Z}_J)\big\} \cup \big\{\lr{}{}{x}{ \textstyle \big( \bigcup_{y \in \bigcup_{1 \le j \le \ell} \, O'_j} \square( y, L_{0}^{-})\big)} \text{ in } D_z \cap \mathcal 
V({Z}_J) \big\}.
\end{equation}
Indeed \eqref{eq:explore_easyII} follows from its version for $\overline {\rm V}_z(\cdot)$ and the inclusion \begin{equation}\label{eq:onearm_twoarm_easy}
\overline {\rm V}_{z}(Z_J) \cap \big\{\lr{}{}{(C_{z}^-}{(\partial_{\mathbb L_{0}^{-}} D_{z}^{-})} \text{ in }\mathcal O_0^{-}(Z_J)\big\}
\subset {\rm V}_z(Z_J).
\end{equation}
However, \eqref{eq:onearm_twoarm_easy} can be obtained in a similar way as 
\eqref{eq:onearm_twoarm} in 
\S\ref{subsec:reduction} in view of following analogue of 
\eqref{eq:z'z''} which is a direct consequence of the definition of the set $\square(y, L_0^-)$ above \eqref{def:Gx-1} and the event ${\rm W}_{y}^-(Z) =  {\rm W}_{y, L_0^-}^-(Z)$ in \eqref{def:Gx-1}: $\square(y, L_0)$ and $\square(y', L_0)$ are 
			connected in $(\tilde C_{y, L_0^-} \cup \tilde C_{y', L_0^-}) \cap \mathcal V(Z_J)$ whenever $|y - y'|_{\infty} \le L_0^-$ and ${\rm W}_{y}^-(Z_J) \cap {\rm W}_{y'}^-(Z_J)$ occurs.

Let us get back to \eqref{eq:explore_easyII} in its version for $\overline {\rm V}_z(\cdot)$, which remains to be shown. By definition of the event 
$\overline {\rm V}_z(Z_J)$ above \eqref{def:overineVx}, we can deduce this from the estimate
\begin{equation}\label{eq:explore_easy_tilde_x}
		 \P\big [(\overline{\rm V}_{z, x}(Z_J))^c \cap \big\{\lr{}{}{C_{z}^{-}}{\partial_{\mathbb L_{0}^{-}} (D_{z}^{-})} \text{ in }\mathcal O_0^{-}(Z_J), k( \mathcal O_0^-(Z_J)) \ge k\big\} \big]
		\le  \, e^{- c(L_{0}^{-}) k}
\end{equation}
for each $x \in \partial \tilde C_z$ via a union bound over $x$. To show 
\eqref{eq:explore_easy_tilde_x}, 
we will 
construct a sequence of `good' random times $(\tau_l)_{l \ge 1}$ coupled to the exploration sequence $(w_n)_{n \ge 1}$ 
revealing the cluster $\mathscr C_J(x)$ (see Section~\ref{prop:reduction}) 
and satisfying the following two properties (cf.~properties~\eqref{property:reduction1}--\eqref{property:reduction3} in 
Proposition~\ref{prop:reduction}).
\begin{align}
	\text{\begin{minipage}{0.9\textwidth}  If $\tau_l < \infty$, then 
			$w_{\tau_l} \in \partial^{{\rm out}} C_{Y_l, L_0} \cap \mathcal V(Z_J)$ for some $ \textstyle Y_l \in \bigcup_{1 \le j \le \ell}\,O_j'$. Conversely, if $y \in \textstyle \bigcup_{1 \le j \le \ell}\,O_j'$ and $\mathscr C_J(x) \cap \partial^{{\rm out}} C_{y, L_0} \ne \emptyset$, there exists $l 
			\ge 1$ such that $\tau_l < \infty$ with $y = Y_l$. \end{minipage}}\label{property:reduction_easy1}\\[0.5em]	
	\text{\begin{minipage}{0.9\textwidth} For any $z \in {\mathbb L}_{0}^-$ and $l \ge 0$, 
	the event $\textstyle \bigcap_{1 \le j \le l}\big\{\tau_j < \infty, \nlr{}{S_j}{w_{\tau_j}}{\square(Y_j, L_0^-)}\big\} \cap \{\tau_{l + 1} < \infty, Y_{l+1} = y\}$ is measurable relative to the $\sigma$-algebra 
	$\overline{\mathcal F}_{y, L_0^-}(Z_J)$ defined in \eqref{def:FEsigma_algebra_small_u};
	\end{minipage}}\label{property:reduction_easy3}
\end{align}
 here and in the sequel $S_j$ is the set $(\{w_{\tau_j}\} \cup C_{Y_j, L_0}) \cap \mathcal V(Z_J)$.
Let us suppose for the moment that such a sequence of random times exists. Then it follows from the 
definition of $k( \mathcal O_0^-(Z_J))$ below \eqref{def:S} and 
property~\eqref{property:reduction_easy1} above that
\begin{multline}\label{eq:VEk}
(\overline{\rm V}_{z, x}(Z_J))^c \cap \big\{\lr{}{}{C_{z}^-}{(\partial_{\mathbb L_{0}^{-}} D_{z}^{-})} \text{ in }\mathcal O_0^{-}(Z_J), k(\tilde D_z 
\setminus \tilde C_z, \mathcal O_0^-(Z_J)) \ge k\big\} \\
\subset \big\{ \tau_j < \infty, \nlr{}{}{w_{\tau_j}}{\square(Y_j, L_0^-)}\text{ in } S_j,
\mbox{ for all } j \le k\big\} \stackrel{{\rm def.}}{=} \overline{\mathcal E}_k. 
\end{multline}
Now using similar arguments as used to derive \eqref{eq:explore_tilde2} in the proof of 
Theorem~\ref{prop:exploration_RI-I} 
with property~\eqref{property:reduction_easy3} and Lemma~\ref{lem:conditional_prob2} in 
lieu of \eqref{property:reduction3} and Proposition~\ref{lem:conditional_prob1}, respectively, 
we obtain
$\P[\overline{\mathcal E}_k] \le e^{-c(L_0^-)k}$ for any $k \ge 0$ thus yielding \eqref{eq:explore_easy_tilde_x} in view of \eqref{eq:VEk}. 
Coming back to the 
sequence $(\tau_l)_{l \ge 1}$, recalling $(w_n)_{n \ge 1}$ 
from the beginning of 
\S\ref{subsec:reduction}, 
define for each $l \ge 1$, with $\tau_0 = 0$,
\begin{equation*}\label{def:tauk_easy}
	\tau_l = 	\inf \left\{n > \tau_l :  \text{ \begin{minipage}{0.50\textwidth}
			$w_n \in \mathcal V(Z_J)\cap \partial^{{\rm out}} C_{y, L_0^-}$ for some $y \textstyle \in \bigcup_{1 \le j \le \ell}\, O_j'$ such that $\textstyle\bigcup_{1 \le i < n} \{w_i\} \cap \mathcal V(Z_J) \cap \partial^{{\rm out}} C_{y, L_0^-} = \emptyset$
	\end{minipage}} \right\}.
\end{equation*}
Properties~\eqref{property:reduction_easy1} and \eqref{property:reduction_easy3} follow from 
this definition and that of $(w_n)_{n \ge 1}$ in a straightforward manner.

It remains to verify \eqref{eq:inclusion_GII}. 
Since the event ${\rm W}_{y^-}^-(Z)$ in decreasing in $Z$ (see \eqref{def:Gx-1}) and $2u_0 < 
u_4$ (see the statement of Theorem~\ref{prop:exploration_RI-II}), we have 
by \eqref{eq:F} that ${\rm W}_{y}^-(Z_{z', L}^{u_4}) \cap \mathcal F_{L}^{\bar{u}_0, u_4} 
{\subset}{\rm W}_{y}^-(\overline{Z}_{z}^{\bar{u}_0})$ for any $z' \in \mathbb L$ satisfying $ D_{z', L_0} \subset  D_z$ and 
$ U_{z', L_0} \subset  U_z$ (cf.~\eqref{eq:WII}). Now recalling the definition of 
the event $\mathcal G_{z'}^{{\rm 
II}}(\widehat Z, u_4; a)$ from \eqref{def:good_events_supcrit2} and, as part of that, the event 
${\rm W}_{z', y}^{{\rm II}}$ from \eqref{eq:WII} (see also Definition~\ref{def:good_events} for 
the generic events $\mathcal G_{z'}(\cdot)$), we get that \begin{equation}\label{eq:GII_inclusion}
\big(\mathcal G_{z'}^{{\rm II}}(Z_{\mathbb L}, u_4; a) \cap \mathcal F_L^{\bar{u}_0, u_4}\big) 
\subset \mathcal G_{z'}({\rm V} = \{\Omega : z'' \in \mathbb L\}, \widetilde{\rm W}^{{\rm II}}, \mathscr C = \{\Z^d: z'' \in \mathbb L\}; a)
\end{equation}
for any $z' \in \mathbb L$ such that $ D_{z', L} \subset  D_z$ and $ 
U_{z', L} \subset  U_z$ where $\widetilde{\rm W}^{\rm II} = \{\widetilde{\rm W}_{z', 
y'}^{\rm II} : z' \in \mathbb L, y' \in \mathbb L_0\}$ with $\widetilde{\rm W}_{z', y'}^{\rm II} = 
\mathcal G_{y'}^{-}(\overline{Z}_{z}^{\bar{u}_0})$. 
Hence from \eqref{eq:grand_inclusion2} in Proposition~\ref{prop:bootstrap_events} and 
\eqref{def:mathscrG_RI2} we obtain that, on the event $\mathscr G_z^{{\rm II}}$, any crossing of $\tilde D_z \setminus \tilde C_z$ intersects at least $c a m$-many boxes 
$C_{y, L_0}$ such that $y \in \mathbb L_0^-$ and ${\rm W}_{y}^-(\overline{Z}_{z}^{\bar{u}_0})$ occurs 
(condition \eqref{eq:bootstrap_inclusion} is satisfied owing to \eqref{eq:GII_inclusion} and 
the observation that $ D_{z', L} \subset  D_z$ and $ 
U_{z', L} \subset  U_z$ as soon as $z' \in (\tilde D_z \setminus \tilde C_z) \cap \mathbb 
L$, a consequence of \eqref{eq:params_RI} and \eqref{eq:L_k_descending}). But the above 
statement directly implies \eqref{eq:inclusion_GII} in view of the definition of $\mathcal 
O^-_0(\overline Z_{z}^{\bar{u}_0})$ in \eqref{def:OyZ} and the functional $k(\Sigma)$ below \eqref{def:S}. 
\end{proof}

\section{Crossings and blocking interfaces}\label{sec:dual_surface}
In this appendix we state a basic result, Proposition~\ref{lem:surroundingInterfaces} below, 
which is topological and of independent interest. It connects the minimum number of times 
any path crossing an annular region $V \setminus U$ (where $U \subset V \subset \Z^d$) 
intersects a set $\Sigma$ to the density of certain dual `blocking' interfaces in $\Sigma$. 
This result is a refinement of \cite[Lemma~2.1]{gosrodsev2021radius}, in that it also 
establishes a certain maximality property of the interfaces (see item~(c) below) which is 
crucial for our proof of Theorem~\ref{prop:exploration_RI-I}.

We first introduce the necessary notation. 
For any $U \subset \subset \Z^d$, we let ${{\rm Fill}}(U)$ denote the union of 
$U$ and all its holes, where a hole is~any finite component of $U^c$. The set ${{\rm Fill}}(U)$ is  
($\ast$-)connected whenever the set $U$ is ($\ast$-)connected. Since $U$ is finite, there exists a unique infinite 
connected component $U_\infty^c$ of $U^c$ and we define the {\em exterior 
	boundary} of $U$ as $\partial^{{\rm ext}} U \coloneqq \partial U_\infty^c$, which equals $\partial^{{\rm out}}{{\rm Fill}}(U).$ For any two sets $U, \Sigma \subset \subset \Z^d$, we say $U$ is surrounded by $\Sigma$, denoted as
\begin{equation}\label{eq:prec}
	\text{$U \preceq \Sigma$, if any infinite 
		connected set $\gamma$ intersecting $U$ also intersects  $\Sigma$.}
\end{equation}

Following is the main result of this section. Property~(c) is the most delicate of the three 
stated below and, as mentioned already, it is also the main feature of this result compared 
to \cite[Lemma~2.1]{gosrodsev2021radius}. 
\begin{prop}[Blocking interfaces] \label{lem:surroundingInterfaces}
	Let $V \subset \Z^d$ be a box, $U\subset V$ a $\ast$-connected set and $k \geq 1$. Suppose that $\Sigma \subset (V\setminus U)$ is such that any path $\gamma$ connecting $U$ and $\partial V$ 
	intersects $\Sigma$ in at least 
	$k (\geq 1) $ points. Then there exists an integer $\ell \geq 1$ and $\ast$-connected sets 
	$O_1, \ldots, {O}_\ell \subset \Sigma$ 
	such that:
	\begin{enumerate}
		\item[(a)] $U \preceq {O}_1 \preceq \ldots \preceq {O}_\ell \preceq \partial V$.
		\item[(b)] Any path connecting $U$ and $\partial V$ intersects $O \coloneqq \bigcup_{1 \le i \le \ell} \, O_i$ in at least $k$ points.
		\item[(c)] The sets $O_1, \ldots, {O}_\ell$ are maximal in the following sense. If for some $j \in \{1,\dots,\ell\}$ and $j' \in \{j, ( j + 1) \wedge \ell \}$  two points $x_j \in \overline{O_j}$ and 
		$x_{j'} \in \overline{O_{j'}}$ are connected in $V \setminus O$, then they are connected 
		in $V \setminus \Sigma$. Similarly if $ x_{\ell} \in \overline{O_{\ell}}$ is connected to $\partial V$ 
		in $V \setminus O$, then $x_{\ell}$ is connected 
		to $\partial V$ 
		in $V \setminus \Sigma$.
	\end{enumerate}
\end{prop}
The following lemma captures an essential feature that will be used to prove (c) above. For a 
path $\gamma=(\gamma(n))_{0\leq n \leq k}$, $k \geq 0$, we denote by ${\gamma}^\circ 
=\bigcup_{0< n < k} \{\gamma(n)\}$ its range with endpoints omitted, also referred to as the 
{\em interior} of $\gamma$.
\begin{lemma}\label{lem:annular}
	Let $\Sigma \subset \subset \Z^d$ and $U \subset \Sigma$ be a $\ast$-component of $\Sigma$. 
	Let $W \subset \subset \Z^d$ be 
	either i) a connected set, or ii) a $\ast$-component of $\Sigma$, and assume $W$ is such that 
	\begin{align}
		\label{eq:W-prop1}
		&U \preceq W, \text{ and }\\
		&\label{eq:W-prop2} \text{$\overline{U}$ is connected to $\overline{W}$ 
			in $\Sigma^c$.}
	\end{align} Then any point in ${U}$  
	that is connected to $W$ by a path $\pi$ with ${\pi}^\circ \cap U=\emptyset$, is also connected to $W$ by a path $\tilde\pi$ with $\tilde\pi^\circ \cap \Sigma=\emptyset$.
\end{lemma}

\begin{proof}
	Throughout the proof, we refer to a point $z \in \mathbb{Z}^d$ as having \textit{property $\text{NC}$} if $z$
	\begin{equation}\label{eq:property-P}
		\text{is \textit{not} connected to $W$ by a path $\pi$ with $\pi^{\circ} \cap \Sigma =\emptyset$}.
	\end{equation}
	Let $x \in U$ be connected to $W$ by a path $\pi$ with ${\pi}^\circ \cap U=\emptyset$. We assume that ${\pi}^\circ  \neq \emptyset$ else choosing $\tilde\pi=\pi$ works. Our aim is to show 
	that $x$ does {\em not} have property $\text{NC}$. Suppose for the sake of contradiction 
	that it did. Consider the component $\mathscr C_x$ of $\Sigma^c \cup \{x\}$ containing the point 
	$x$. By definition, $\mathscr C_x$ is a connected set. Since $x$ is assumed to satisfy 
	\eqref{eq:property-P}, it necessarily holds that
	\begin{equation}\label{eq:annular_empty_intersect}
		\big((\mathscr C_x \cup \partial^{{\rm ext}}\mathscr C_x) \cap W \big)   \, {  \subset \big( \overline{\mathscr C_x} \cap W \big)} = \emptyset;
	\end{equation}  
	indeed the inclusion is obvious since $\partial^{{\rm ext}} A \subset \partial^{{\rm out}}A$ for any set $A$, and a 
	non-empty intersection of $\overline{\mathscr C_x} \cap W$ would imply the existence of a path 
	$\pi$ with the above (precluded) properties since $\mathscr{C}_x$ is connected. Next, since 
	$\mathscr C_x  \cup \partial^{{\rm ext}}\mathscr C_x$ is a connected set containing $x $ and 
	$\{x\} \preceq W$ by \eqref{eq:W-prop1} and the transitivity of $\preceq$ (using that $\{x\} \subset U$ which implies that $\{x\} \preceq U$), it follows from 
	\eqref{eq:annular_empty_intersect} in view of \eqref{eq:prec} that
	\begin{equation}\label{eq:WsurrC}
		(\mathscr C_x \cup \partial^{{\rm ext}}\mathscr C_x) \preceq W.
	\end{equation}
	Moreover, by definition of $\mathscr C_x$ and \eqref{eq:annular_empty_intersect},
	\begin{equation*}
		\begin{split}
			\mbox{every point in $\mathscr C_x$ has property $\text{NC}$.}
		\end{split}
	\end{equation*}
	Now recall the definition of ${{\rm Fill}}(\mathscr C_x)$ from the beginning of this section. We claim 
	that, in fact,
	\begin{equation}\label{eq:P-claim}
		\text{every point in ${{\rm Fill}}(\mathscr C_x)$ has property~$\text{NC}$}
	\end{equation} 
	and first finish the proof of the lemma assuming this claim by deriving the desired contradiction. By hypothesis in \eqref{eq:W-prop2}, $U$ contains a point 
	$y$ (say) that does not have property~$\text{NC}$: indeed with $\gamma' \subset \Sigma^c (\subset U^c)$ a path starting in $\overline{U}$ and ending in $\overline{W}$, any neighbor $y \in U$ of $\gamma'(0) (\in \partial^{\text{out}} U)$ will do (note that $\gamma'$ can always be extended by addition of at most one point so as to intersect $W$). Since $U$ is $\ast$-connected, there is a $\ast$-path 
	$\gamma \subset U$ connecting $x$ and $y$. Also since $y \notin {{\rm Fill}}(\mathscr C_x)$ on 
	account of \eqref{eq:P-claim}, it follows from the definition of $\ast$-connectivity that $\gamma$ 
	contains a point that either lies in $\partial^{{\rm out}}  {{\rm Fill}}(\mathscr C_x) = \partial^{{\rm ext}} \mathscr C_x \subset \Sigma$ or is a $\ast$-neighbor of $\partial^{{\rm ext}} \mathscr C_x$. {The inclusion in $\Sigma$ is immediate since $\mathscr C_x$ is a component in $\Sigma^c \cup \{x\}$}. 

	We will now argue that the stronger inclusion $\partial^{{\rm ext}} 
	\mathscr C_x \subset U (\subset \Sigma)$ holds true.
	Indeed, 
	$\mathscr C_x$ is connected and therefore 
	$\partial^{{\rm ext}} \mathscr C_x$ is $\ast$-connected by 
	\cite[Lemma~2.1-(i)]{DeuschelPisztora96}.  Consequently, the set $ \gamma \cup 
	\partial^{{\rm ext}} \mathscr C_x$ is itself a $\ast$-connected subset of $\Sigma$. Since 
	$U$ is a $\ast$-component of $\Sigma$ and $\gamma \subset U$, we thus obtain $\partial^{{\rm ext}} \mathscr C_x \subset U$. On the other hand, since $x$ can be connected to $W$ by a path $\pi$ with $\pi^{\circ} \cap U=\emptyset$ and $\overline{\mathscr C_x} \cap W = \emptyset$, the latter on account of \eqref{eq:annular_empty_intersect}, it must be the case that $\pi \cap \partial^{{\rm ext}} \mathscr C_x \ne \emptyset$ and hence $\partial^{{\rm ext}} \mathscr C_x \cap 
	U^c \ne \emptyset$ {
		(recall that $\pi^\circ \neq \emptyset$) }which leads to a contradiction.

	

	It remains to prove \eqref{eq:P-claim}. 
	{ To this end let $w \in {{\rm Fill}}(\mathscr C_x)$.} We first argue that if $w$ does not have property~$\text{NC}$, then it is necessarily the case that 
	\begin{equation}\label{eq:weird-W}
		\text{$W$ intersects some finite component of $\mathscr C_x^c$.}
	\end{equation}
	Indeed since 
	{ $w \in {{\rm Fill}}(\mathscr C_x)$, any path $\gamma$ connecting $w$ to a point in $({{\rm Fill}}(\mathscr C_x) \cup \partial^{{\rm ext}}\mathscr C_x)^c$ must satisfy $\gamma^\circ \cap \partial^{{\rm ext}} \mathscr C_x \ne \emptyset$.} 
	But if \eqref{eq:weird-W} does not occur { then, since $\overline{\mathscr C_x} \cap W = \emptyset$ in view of \eqref{eq:annular_empty_intersect}, it further holds that $({{\rm Fill}}(\mathscr C_x) \cup \partial^{{\rm ext}}\mathscr C_x) \cap W = \emptyset$. The last two observations together imply that} 
	$\pi^{\circ}$ must intersect {$\partial^{{\rm ext}}\mathscr C_x$ and hence $\Sigma$ for any path $\pi$ connecting $w$ to $W$}, i.e.~$w$ 
	has Property $\text{NC}$. All in all, \eqref{eq:weird-W} thus follows.
	{ Now notice that \eqref{eq:WsurrC} and \eqref{eq:annular_empty_intersect} together imply
		\begin{equation}\label{eq:Wintersects}
			({{\rm Fill}}(\mathscr C_x) \cup \partial^{{\rm ext}}\mathscr C_x)^c \cap W \ne \emptyset.	
		\end{equation}	
		Hence if $W$ is connected and \eqref{eq:weird-W} holds, then $W$ must intersect $\partial^{{\rm ext}}\mathscr C_x$ as any path between a point in some finite component of $\mathscr C_x^c$ and $({{\rm Fill}}(\mathscr C_x) \cup \partial^{{\rm ext}}\mathscr C_x)^c$ has to intersect $\partial^{{\rm ext}}\mathscr C_x$. But this contradicts \eqref{eq:annular_empty_intersect} and thus 
		\eqref{eq:weird-W} is not possible in this case}. On the other hand, if $W$ is a $\ast$-{ 
		connected set and \eqref{eq:weird-W} holds, then it follows from \eqref{eq:Wintersects} and 
		the definition of $\ast$-connectivity that $W$ contains a point which is either an element of 
		$\partial^{{\rm ext}}\mathscr C_x$ or a $\ast$-neighbor of $\partial^{{\rm ext}}\mathscr C_x$. 
		Therefore if $W$ is a $\ast$-component of $\Sigma$, we get $\partial^{{\rm ext}}\mathscr C_x \subset W$ since $\partial^{{\rm ext}}\mathscr C_x \subset \Sigma$ as we already noted above.} 
	{ But this also }
	contradicts \eqref{eq:annular_empty_intersect}. Thus the conclusion holds for both types of $W$ considered. 
\end{proof}

We now turn to the:

\begin{proof}[Proof of Proposition~\ref{lem:surroundingInterfaces}]
	Since $k \ge 1$, $U$ is not connected to $\partial V$ in $\Z^d \setminus \Sigma \, (\supset U)$. Then by \cite[Lemma~2.1-(i)]{DeuschelPisztora96}, the exterior boundary $\partial^{{\rm ext}}\mathscr C_{U}$ of the $\ast$-component $\mathscr C_{U}$ of $U$ in 
	$\Z^d \setminus \Sigma$ is a non-empty $*$-connected subset of $\Sigma$. Let $O_1$ be defined as 
	the $\ast$-component of $\partial^{{\rm ext}} \mathscr C_{U}$ in $\Sigma$. Thus by definition, $U 
	\preceq O_1$ and $O_1 \subset \Sigma \subset V$. Also observe that $\mathscr C_U \cup O_1 
	\subset V$ is $\ast$-connected.

	Now the hypothesis of the proposition is clearly satisfied with $U_1 = \mathscr C_U \cup O_1$, $\Sigma_1 
	= \Sigma \setminus O_1$ and $k_{U_1, \Sigma_1}$ {
		(in case the latter is $\geq 1$)} substituting for $U$, $\Sigma$ and $k$ 
	respectively, where for any two disjoint subsets $U'$ and $\Sigma'$ of $V$, we denote
	\begin{equation*}
		k_{U', \Sigma'} = \min\{|\gamma \cap \Sigma'| :  \mbox{$\gamma$ is any path between $U'$ and $\partial V$}\} \ge 0.
	\end{equation*}
	Iterating the construction in the previous paragraph over successive rounds until the first $\ell$ such 
	that $k_{U_\ell, \Sigma_\ell} = 0$ by letting $O_k$ be the $\ast$-component of $\partial^{{\rm ext}}U_{k-1}$ in $\Sigma_{k-1}$, $U_k = \mathscr C_{U_{k-1}} \cup O_k$ and $\Sigma_k = \Sigma_{k-1} \setminus O_k$ in each round $2 \le k \le \ell$, we obtain a sequence of $\ast$-components $(U \preceq) \, O_1 \preceq \ldots \preceq O_\ell \preceq \partial V$ in $\Sigma$. This is readily verified 
	inductively. In particular, the collection $\{O_1, \ldots, O_\ell\}$ satisfies~(a).

	Next, consider any path $\gamma$ connecting $U$ and $\partial V$ with exactly one (end-)point in $\partial V$. Let $\gamma_1, \ldots, \gamma_m$ denote the maximal non-trivial segments of 
	$\gamma$ whose interiors lie outside $O$. By 
	our construction of the sets $O_1, \ldots, O_\ell$, any such segment must 
	necessarily have either both its endpoints in $\{O_{j}, O_{j+1}\}$ for some $j \in \{1, \ldots, \ell - 1\}$ or one 
	endpoint in $O_\ell$ and the other in $\partial V$. Now assuming the sets $O_1, \ldots, O_\ell$ also satisfy~(c),  we can replace each $\gamma_i$ with a suitable segment  
	whose interior is disjoint from $\Sigma$ such that the resulting sequence $\gamma'$ is also a 
	path between $U$ and $\partial V$ 
	satisfying $\gamma' \cap \Sigma = \gamma' \cap O = \gamma \cap O$. Hence 
	$|\gamma \cap O| \ge k$ by the hypothesis of our proposition applied to the path $\gamma'$, yielding item~(b).

	It only remains to verify~(c). 
	Let us suppose that some $x \in O_j$ is connected to some $y \in O_j \cup O_{j+1}$, or $O_j \cup \partial V$ in case $j = \ell$, by a path 
	whose interior lies in $V \setminus O$. 

	We will first consider the case $y \in O_j$. 
	Let $\mathscr C_x$ denote the component of $x$ in $(V 
	\setminus \Sigma) \cup \{x\}$. If $x$ and $y$ can {\em not} be connected 
	by a path whose interior lies in $V 
	\setminus \Sigma$, then $y$ necessarily lies in a component, say $\mathscr C_y^x$, of 
	$V \setminus \mathscr C_x$ such that $y \in \mathscr C_y^x \setminus \partial_V \mathscr C_y^x$ and 
	\begin{equation}\label{eq:C_ySigma}
		\partial_V \mathscr C_y^x \subset \Sigma
	\end{equation}	
	Also since $x, y$ are connected by a path whose interior lies in $V \setminus O$ and $y \in \mathscr C_y^x \setminus \partial_V \mathscr C_y^x$, it follows that 
	\begin{equation}\label{eq:C_yO}
		\partial_V \mathscr C_y^x \cap O^c  \ne \emptyset.
	\end{equation}
	On the other hand, $x$ and $y$ are connected by a $\ast$-path $\gamma$ in $\Sigma$ as $O_j 
	\subset \Sigma$ is $\ast$-connected by definition. Hence by the definition of $\ast$-connectivity, 
	$\gamma$ either intersects $\partial_V \mathscr  C_y^x$ or contains a point that is a $\ast$-neighbor of $\partial_V \mathscr C_y^x$. Therefore the set $\gamma \cup \partial_V \mathscr C_y^x$ is $\ast$-connected {\em provided} $\partial_V \mathscr C_y^x$ is also $\ast$-connected which turns out to be a consequence 
	of \cite[Lemma~2.1-(ii)]{DeuschelPisztora96} as the set $\mathscr C_x$ is a connected subset of $V$. But $\gamma \cup \partial_V \mathscr C_y^x \subset \Sigma$ (recall \eqref{eq:C_ySigma} as well as that $\gamma \subset \Sigma$) and intersects $\{x, y\} 
	\subset O_j$. Since $O_j$ is a $\ast$-component by construction, the previous two observations 
	imply that $\partial_V \mathscr C_y^x \subset \gamma \cup \partial_V \mathscr C_y^x \subset O_j$. However, this contradicts \eqref{eq:C_yO} and thus property~(b) is satisfied in this case.
	
	Next let us consider the case $y \in O_{j+1}$ where $j < \ell$ and let $\mathscr C_x$ denote the component of $x$ in $(V 
	\setminus \Sigma) \cup \{x\}$, as before. Since $O_{j} \subset \Sigma$ and $(\{x\} \subset)\, O_j \preceq O_{j+1}$	 by~(a), it follows from Lemma~\ref{lem:annular} applied with $(U, \Sigma, W) = (O_j, \Sigma, O_{j+1})$ that there is a point $z \in \partial_{V, 
		{\rm out}} \mathscr C_x \cap O_{j + 1}$. Since $z \in \partial_{V}^{ {\rm out}} \mathscr C_x$, it must be the case that $z \in \partial_V \mathscr C_y^x$ if $z \in \partial_{V}\mathscr C_y^x$. Thus any $\ast$-path contained in $O_{j+1}$ connecting $y$ and $z$, which necessarily exists as $O_{j + 1}$ is $\ast$-connected, must either intersect or be a $\ast$-neighbor of $\partial_{V}\mathscr C_y^x$. Now we repeat the same argument as in the previous case with such a $\ast$-path $\gamma$.
	
	Finally the case where $x \in O_\ell$ and $y \in \partial V$ follows almost immediately from 
	Lemma~\ref{lem:annular} with $(U, \Sigma, W) = (O_\ell, \Sigma, \partial V')$ where $V'$ is any box 
	containing $V$ in its interior.
\end{proof}

\section{Proofs of Lemmas~\ref{L:BWG} and~\ref{L:BWG'}} \label{sec:aux-lemmas}

Lemmas~\ref{L:BWG} and~\ref{L:BWG'} essentially follow from the defining properties of the ``algorithm'' $\mathscr{A}=(\mathscr{A}_n)_{ n \geq 0}$ introduced in \S\ref{subsec:exploration}. We include full proofs here since the definition of $\mathscr{A}$ is somewhat involved.

\begin{proof} [Proof of Lemma~\ref{L:BWG}]
We start with~\eqref{eq:BW0}. For $n = 0$, $\mathscr{A}_n=(\mathsf{B}_n, \mathsf{W}_n, \mathsf{G}_n)$ is a partition of $\Z^d$ by \eqref{eq:BWG0}. For 
	general $n$, this is deduced inductively by following the update rule for the sets $\mathscr{A}_n$. The second part is an immediate consequence of 
	 definitions \eqref{eq:B_n}--\eqref{eq:W_n} and \eqref{eq:GBW2}, together with the definitions of the sets $\mathsf{B}_n', \mathsf{W}_n'$ and $\mathsf{G}_n'$ in the paragraph preceding the display \eqref{eq:G_n}.
	
	\smallskip
	
	We now show \eqref{eq:BW}. This is obvious except for the restrictions of the sets $\mathsf S_n$ and $\widetilde{\mathsf S}_n$ to $D_{\tilde Y_k, L_0}$ for 
	any $\mathsf S \in \{\mathsf B, \mathsf W\}$ and $k \ge 1$ such that $\tilde \tau_k < \infty$. It follows from the update rule for the 
	triplets $\mathscr{A}_n$ and $\tilde{\mathscr{A}}_n$ and \eqref{eq:GBW2} that 
	 if $\tilde Y_k = \tilde Y_l$ for some $l < k$ and \eqref{eq:BW} holds for $n=\tau_l$, 
	then it also holds for $n=\tau_k$.  If $\tilde Y_k \ne \tilde Y_l$ for any $l < k$, we are in the 
	same situation as $k = 1$ (see above \eqref{eq:GBW2}), hence it suffices to verify the 
	property for $n = \tilde \tau_1$. Moreover, in view of our treatment of Cases~I and II below 
	\eqref{eq:W_n}, it is enough to verify the inclusions for $\widetilde{\mathsf W}_n \cap D$ 
	and $\widetilde{\mathsf B}_n \cap D$ with $D=D_{\tilde Y_1, L_0}$. In the sequel, we will 
	implicitly mean their intersection with $D$ when referring to the sets $\mathsf B_n$ and 
	$\widetilde{\mathsf W}_n$ .
	
	It is clear from \eqref{eq:B_n} that
	$\widetilde{\mathsf{B}}_n \subset \mathcal I$. On the other hand, since $\mathsf W_n' 
	\subset \mathcal V$, we have $\widetilde{\mathsf{W}}_n 
	\subset \mathcal V$ in view of \eqref{eq:W_n} {\em provided} we also have $\mathsf 
	{G}_n' \setminus \widetilde{\mathsf {G}}_n \subset \mathcal V$. To this end, let $x' \in 
	\mathcal I \cap \mathsf{G}_n'$ and we will show $x' \in \widetilde{\mathsf{G}}_n$. 
	
	Let us first observe that the cluster $\mathscr C(x')$ of $x'$ in $\mathcal I \cap (D \setminus 
	C)$, $C=C_{\tilde{Y}_1, L_0}$, is necessarily a subset of $\mathsf{G}_n'$ and is disjoint from $\partial 
	D$. This is because, by definition, $(\mathsf{B}_n', \mathsf{W}_n', \mathsf{G}_n')$ forms a partition of $D$ with $\mathsf{W}_n' \subset \mathcal V$ and 
	$\mathsf{B}_n'$ comprising the clusters of $\partial D$ in $\mathcal I \cap (D \setminus C)$. 
	But any (non-empty) component of $\mathcal I \cap D = \mathcal I(Z_J) \cap D$ must 
	intersect $\partial D$ as the sequence $Z_J$ consists of excursions $Z_j^{ D_z,  U_z}$'s 
	between $ 
	D_z$ and $\partial^{{\rm out}} U_z$ with $D \subset  D_z$. 
	Since $\mathscr C(x')$ is a component of $\mathcal I \cap (D \setminus C)$ disjoint from $\partial D$, the previous 
	observation implies that $\mathscr C(x') \cap \partial^{{\rm out}} C \ne \emptyset$. Together with the fact that both $\mathscr C(x')$ and $C$ are subsets of 
	$\mathsf{G}_n'$ (see above \eqref{eq:G_n} for the latter), this implies $x'$ lies in the 
	component of $C$ in $\mathsf{G}_n'$, i.e.~$\widetilde{\mathsf{G}}_n$ by \eqref{eq:G_n}, and the 
	proof is complete. For use in the proof of \eqref{eq:CsubsetV}, let us also note important conclusion that we can draw from the arguments in this paragraph: the cluster of $x'$ 
	in $\mathcal I \cap D$ 
	must intersect $C$.

\smallskip

The proofs of~\eqref{eq:w2} and~\eqref{eq:CsubsetV} both rely on Lemma~\ref{L:BWG'}. The 
argument is not circular, as the latter only requires knowledge of \eqref{eq:BW0} and 
\eqref{eq:BW}, which have already been shown. The proof of \eqref{eq:w2} is in fact contained 
in the proof of \eqref{eq:w}, see the paragraph following \eqref{eq:taul_tauk} below. 
As to~\eqref{eq:CsubsetV}, assuming \eqref{eq:w} to hold, we only need to verify this when 
$\tilde Y_k \ne \tilde Y_l$ for any $l < k$, which is similar to the case $k = 1$. But then it is 
precisely the statement at the end of the proof of \eqref{eq:BW} in the previous paragraph. \end{proof}
	
\begin{proof}[Proof of Lemma~\ref{L:BWG'}]
It clearly suffices to prove \eqref{eq:w} with $k$ replaced by $K_1$ where $K_1 \coloneqq \inf\{l \ge 1: \tilde Y_l = \tilde Y_k\}$ and therefore we can 
	assume, without any loss of generality, that $\tilde Y_k \ne \tilde Y_l$ for any $l < k$. Now suppose that 
	the statement, i.e. \eqref{eq:w} holds for some $n \ge \tilde \tau_k$ such 
	that $\mathsf{G}_{n + 1} \cap D_{\tilde Y_k, L_0} \ne \emptyset$. We will verify that the statement 
	also holds at time $n + 1$. The claim then follows by induction. To this end, first 
	note that if $w_{n+1} \in \Z^d \setminus D_{\tilde Y_k, L_0}$, then the statement clearly holds at time 
	$n + 1$ as no vertex in $D_{\tilde Y_k, L_0}$ gets inspected in this case. So suppose that $w_{n + 
		1} \in D_{\tilde Y_k, L_0}$. We now consider two possibilities. Firstly, $w_{n + 1} \in D_{\tilde Y_k, 
		L_0} \setminus \partial D_{\tilde Y_k, L_0}$ in which case one performs a generic step of the 
	exploration (see the start of the paragraph containing \eqref{eq:GBW2}). Since $\mathsf{W}_{n} 
	\subset \mathcal V$ and $\mathsf{B}_{n} \subset \mathcal I$ by \eqref{eq:BW}, the only way 
	the statement may fail to hold in this case is if $w_{n + 1} \in \mathsf{G}_{n}$ on account of \eqref{eq:BW0}. Since $w_{n+1}$ 
	is selected from the outer boundary of a connected set in $\mathsf{W}_n$, which is the 
	explored part of $\mathscr C_J(x)$ at time $n$, and this set intersects $\partial D_{\tilde Y_k , L_0}$ 
	as it contains the point $w_{\tilde \tau_k} \in \partial D_{\tilde Y_k , L_0}$ (see~\eqref{eq:x-tilde-k}), it 
	follows that $w_{n+1}$ is adjacent to a boundary component of $\mathsf{W}_n \cap D_{\tilde Y_k, 
		L_0}$. The points in this component, in particular the ones that also lie in $\partial D_{\tilde Y_k, 
		L_0}$, were revealed at some time $m \le n$. In view of \eqref{eq:x-tilde-k}, it then follows that 
	that $\tilde \tau_l \le n$ for some $l \ge 1$ such that $\tilde Y_l = \tilde Y_k$ and $w_{n + 1}$ has a neighbor 
	connected to $\tilde X_l \in \partial D_{\tilde Y_l, L_0}$ inside $\mathsf{W}_n \cap D_{\tilde Y_k, L_0}$. 
	Since $\tilde Y_k \ne \tilde Y_{l'}$ for any $l' < k$ by our assumption at the beginning of the proof, we in 
	fact have $l \ge k$. All in all we get that {\em either} \eqref{eq:w} holds in this case at time 
	$n+1$ {\em or}
	\begin{equation}\label{eq:taul_tauk}
		\text{\begin{minipage}{0.9\textwidth}
				$w_{n+1} \in \mathsf{G}_n$ is conn.~to $\tilde X_{l}$ in $(\mathsf{G}_n \cup \mathsf{W}_n) \cap D_{\tilde Y_k, L_0}$ for some $l \ge 1$ s.t.~$\tilde \tau_l \in [\tilde \tau_k, n]$ and 
				$\tilde Y_{l} = \tilde Y_k$.
		\end{minipage}}
	\end{equation}
	We will now show that \eqref{eq:taul_tauk} cannot occur. Recall that $\mathsf{G}_{n + 1} \cap D_{\tilde Y_k, L_0} \ne \emptyset$ by assumption, which implies $\mathsf{G}_{n'} \cap D_{\tilde Y_k, L_0} \ne \emptyset$ for any $n' \le n$ by monotonicity (recall \eqref{eq:wninclude2}). Therefore, $\mathsf{G}_{\tilde \tau_l} \cap D_{\tilde Y_k, L_0} = 
	\mathsf{G}_{\tilde \tau_l} \cap D_{\tilde Y_l, L_0} \ne \emptyset$ as $\tilde \tau_l \le n$ and $\tilde Y_l = \tilde Y_k$ by 
	\eqref{eq:taul_tauk}. Consequently, the triplet $(\widetilde{\mathsf{B}}_{\tilde \tau_l} \cap D_{\tilde Y_l, L_0}, \widetilde{\mathsf{W}}_{\tilde \tau_l} \cap D_{\tilde Y_l, L_0}, \widetilde{\mathsf{G}}_{\tilde \tau_l} \cap D_{\tilde Y_l, L_0})$ satisfies the condition 
	leading to Case~I below \eqref{eq:W_n}, for otherwise we would have 
	$\mathsf{G}_{\tilde \tau_l} \cap D_{\tilde Y_l, L_0} = \emptyset$. 
	But in Case~I, there exists {\em no} path in $(\widetilde{\mathsf{G}}_{\tilde \tau_l} \cup 
	\widetilde{\mathsf{W}}_{\tilde \tau_l}) \cap D_{\tilde Y_l, L_0} = (\mathsf{G}_{\tilde \tau_l} 
	\cup \mathsf{W}_{\tilde \tau_l}) \cap D_{\tilde Y_l, L_0}$ connecting $\tilde X_l$ to 
	$C_{\tilde Y_l, L_0}$. As $\mathsf{G}_{\tilde \tau_l} \cap D_{\tilde Y_l, L_0}$ is a connected 
	set containing $C_{\tilde Y_l, L_0}$ 
	(see~\eqref{eq:G_n} and the line above \eqref{eq:B_n}), the previous fact 
	also implies that there is {\em no} path in $(\mathsf{G}_{\tilde \tau_l} \cup \mathsf{W}_{\tilde \tau_l}) \cap 
	D_{\tilde Y_l, L_0}$ connecting $\tilde X_l$ to any point in $\mathsf{G}_{\tilde \tau_l} \cap D_{\tilde Y_l, L_0}$. 
	However, this contradicts \eqref{eq:taul_tauk} as $\mathsf{G}_n \cap D_{\tilde Y_k, L_0} = 
	\mathsf{G}_{\tilde \tau_l} \cap D_{\tilde Y_k, L_0}$ and $\mathsf{W}_n \cap D_{\tilde Y_k, L_0} = 
	\mathsf{W}_{\tilde \tau_l} \cap D_{\tilde Y_k, L_0}$ for any $l \ge 1$ satisfying $\tilde \tau_l \in [\tilde \tau_k, n]$ 
	according to our induction hypothesis. Thus property~\eqref{eq:w} holds in this case at time 
	$n+1$.
	
	The remaining possibility is that $w_{n+1} \in 
	\partial D_{\tilde Y_k, L_0}$, i.e.~$n + 1 = \tilde \tau_l$ and $w_{n+1} = \tilde X_{l}$ for some $l > k$ with 
	$\tilde Y_l = \tilde Y_k$. So we are in the situation considered in \eqref{eq:GBW2}. 
	As $\mathsf{G}_{n+1} 
	\cap D_{\tilde Y_k, L_0} = \mathsf{G}_{\tilde \tau_l} \cap D_{\tilde Y_k, L_0} \ne \emptyset$ by induction 
	hypothesis, it is clear as before that $(\mathsf{B}_{\tilde \tau_l} \cap D_{\tilde Y_l, L_0}, \mathsf{W}_{\tilde \tau_l}\cap D_{\tilde Y_l, L_0}, \mathsf{G}_{\tilde \tau_l}\cap D_{\tilde Y_l, L_0})$ satisfies the condition of Case~I at time $\tilde \tau_l = n+1$, 
	hence $\mathsf{S}_{n+1} \cap D_{\tilde Y_l, L_0} = \widetilde{\mathsf{S}}_{n+1} \cap D_{\tilde Y_l, L_0}$ for all 
	$\mathsf{S} \in \{\mathsf{W}, \mathsf{B}, \mathsf{G}\}$. 
	But $\widetilde{\mathsf{S}}_{n+1} \cap D_{\tilde Y_l, L_0} =  
	\mathsf{S}_{n} \cap D_{\tilde Y_l, L_0}$ 
	due to \eqref{eq:GBW2}, hence \eqref{eq:w} holds in this case as well at 
	time $n+1$.
\end{proof}

\resumetoc

\bibliography{rodriguez,biblicomplete}
\bibliographystyle{abbrv}

\end{document}